\documentclass[reqno,10pt]{amsart}
\usepackage{amsmath,mathrsfs}
\usepackage{amssymb, amsmath}
\usepackage{graphicx}
\usepackage{pdfsync}
\usepackage[colorlinks,
linkcolor=red,
anchorcolor=blue,
citecolor=green
]{hyperref}
\usepackage{colonequals}
\usepackage{multicol}
\usepackage{enumerate}
\usepackage{color}
\usepackage{mathtools}
\usepackage{ulem}
\usepackage[numbers,sort&compress]{natbib}
\usepackage{tikz}
\usetikzlibrary{matrix}
\usepackage{float}

\def\inte#1{
	\displaystyle\mathop{#1\kern0pt}^\circ }

\let\pa=\partial

\let\f=\frac

\def\pa{\partial}

\def\virgp{\raise 2pt\hbox{,}}
\def\cdotpv{\raise 2pt\hbox{;}}

\def\C{\mathop{\mathbb C\kern 0pt}\nolimits}
\def\DD{\mathop{\mathbb D\kern 0pt}\nolimits}
\def\EE{\mathop{{\mathbb E \kern 0pt}}\nolimits}
\def\K{\mathop{\mathbb K\kern 0pt}\nolimits}
\def\N{\mathop{\mathbb N\kern 0pt}\nolimits}
\def\Q{\mathop{\mathbb Q\kern 0pt}\nolimits}
\def\R{\mathop{\mathbb R\kern 0pt}\nolimits}
\def\SS{\mathop{\mathbb S\kern 0pt}\nolimits}
\def\ZZ{\mathop{\mathbb Z\kern 0pt}\nolimits}
\def\TT{\mathop{\mathbb T\kern 0pt}\nolimits}
\def\P{\mathop{\mathbb P\kern 0pt}\nolimits}

\def\na{\nabla}

\newcommand{\beq}{\begin{equation}}
	\newcommand{\eeq}{\end{equation}}
\newcommand{\ben}{\begin{eqnarray}}
	\newcommand{\een}{\end{eqnarray}}
\newcommand{\beno}{\begin{eqnarray*}}
	\newcommand{\eeno}{\end{eqnarray*}}

\newtheorem{thm}{Theorem}[section]
\newtheorem{lem}{Lemma}[section]
\newtheorem{col}{Corollary}[section]
\newtheorem{prop}{Proposition}[section]
\renewcommand{\theequation}{\thesection.\arabic{equation}}

\theoremstyle{definition}
\newtheorem{defi}{Definition}[section]

\newtheorem{rmk}{Remark}[section]

\numberwithin{equation}{section}

\begin{document}
	\title[Grazing limit]
	{A new understanding of grazing limit}
	
	\author[T. Yang]{Tong Yang}
	\address[T. Yang]{Department of Applied Mathematics, 
		The Hong Kong Polytechnic University,
		Hong Kong, P.R. China.}
	\email{t.yang@polyu.edu.hk}
	
	\author[Y.-L. Zhou]{Yu-long Zhou}
	\address[Y.-L. Zhou]{School of Mathematics, Sun Yat-Sen University, Guangzhou, 510275, P. R.  China.} \email{zhouyulong@mail.sysu.edu.cn}

	\begin{abstract} The grazing limit of the Boltzmann equation to Landau equation is well-known and has been justified by using cutoff near the grazing angle with some suitable scaling. In this paper, we will provide a new understanding by simply applying a natural scaling on the Boltzmann operator without angular cutoff. The proof is based on a new well-posedness theory on the Boltzmann equation without angular cutoff in the regime with optimal ranges of parameters so that the grazing limit can be justified directly for any $\gamma>-5$ that includes the Coulomb potential 
	corresponding to $\gamma=-3$. With this new understanding, the scaled Boltzmann operator in fact can be decomposed into two components. The first one converges to the Landau operator when the singular parameter $s$ of interaction angle tends to $1^{-}$ and the second one vanishes in this limit.
	\end{abstract}


	\maketitle
	
		\setcounter{tocdepth}{1}
	
	\tableofcontents


	
	
	\noindent {\sl AMS Subject Classification (2010):} {35Q20, 35R11, 75P05.}

	\renewcommand{\theequation}{\thesection.\arabic{equation}}
	\setcounter{equation}{0}

	\section{Introduction}
The Boltzmann and Landau equations are the two most classical kinetic equations. Regarding to
the Cauchy problem,
there has been extensive work in different frameworks, e.g.~\cite{ukai1974existence,ukai1982cauchy,caflisch1980boltzmann1,caflisch1980boltzmann2,diperna1989cauchy,lions1994boltzmann,villani1996cauchy,villani1998new,guo2002landau,guo2003classical,guo2004boltzmann,desvillettes2005trend,mouhot2006explicit,strain2006almost,strain2008exponential,gressman2011global,alexandre2012boltzmann,he2012well,he2018sharp,duan2021global}. In fact,  the Landau  equation was derived by Landau in 1936
 from the Boltzmann  equation with cutoff Rutherford
 cross section. Mathematical justification  of the grazing collision limit has proved to be successful
 since  1990s by adding a cutoff angle with suitable scaling parameter to the Boltzmann
 cross-section,
 cf.~\cite{alexandre2004landau,arsen1991connection,degond1992fokker,desvillettes1992asymptotics,goudon1997boltzmann,he2014asymptotic,he2014well,he2021boltzmann,duan2023solutions}.

 This paper aims to provide a new approach to justisfy this limit so that the relation between the 
 Boltzmann equation and Landau equation can be understood from a different angle. Precisely, we study the grazing limit directly starting from the Boltzmann equation with angular non-cutoff kernels $B^{s,\gamma}(v-v_{*}, \sigma) \sim \theta^{-2-2s} |v-v_{*}|^{\gamma}$
 originating from the 
 inverse power law potentials. Roughly speaking, 
 with a proper scaling to the Boltzmann cross-section, when parameter of 
  the angular singularity  $s \to 1^{-}$,
 we  naturally justify the limit to  the Landau equation for any $\gamma>-5$. We point out
  that this new approach is related to but very different from the existing arguments for grazing limit. 

\subsection{A natural scaling}	
	Consider the Cauchy problem of the Landau  equation
	\ben \label{Cauchy-Landau} \left\{ \begin{aligned}
		&\partial _t F +  v \cdot \nabla_{x} F=Q_{L}^{\gamma}(F,F), ~~t > 0, x \in \mathbb{T}^{3}, v \in\R^3 ,\\
		&F|_{t=0} = F_{0}.
	\end{aligned} \right.
	\een
Here the Landau operator $Q_{L}^{\gamma}(g,h)$ is defined by
	\ben \label{oroginal-definition-Laudau-oprator} Q_{L}^{\gamma}(g,h)(v) :=
	\nabla_{v}\cdot\bigg\{\int_{\R^3}a^{\gamma}(v-v_{*})[g(v_{*})\nabla_{v}h(v)-\nabla_{v_{*}}g(v_{*})h(v)]\mathrm{d}v_{*}\bigg\}.
	\een 
where the symmetric matrix $a^{\gamma}$ is given by
	\ben\label{matrix}
	a^{\gamma}(z)  = \Lambda |z|^{\gamma+2}\Pi(z), \quad \Pi(z) :=
	(I_{3} - \frac{z \otimes z}{|z|^{2}}).
	\een
	Here, $I_{3}$ is the $3 \times 3$ identity matrix and $\Lambda$ is a constant.

	We will show that the solution $F_{L}^{\gamma}$ of \eqref{Cauchy-Landau} can be derived from that of the Boltzmann equation with a natural scaling. 
	Let $F^{s,\gamma}_{B}$ be the solution to
	the Boltzmann equation 
	\ben \label{original-Boltzmann-with-ipl}
	\partial _t F +  v \cdot \nabla_{x} F=Q^{s,\gamma}_{B}(F,F), \quad F|_{t=0} = (1-s)F_{0},
	\een
	where $Q^{s,\gamma}_{B}$ is the Boltzmann operator defined by
	\ben \label{Boltzmann-operator-2}
	Q^{s,\gamma}_{B}(g,h)(v):=
	\int_{\R^3}\int_{\mathbb{S}^{2}}B^{s,\gamma}(v-v_*,\sigma)\left(g^{\prime}_{*} h^{\prime}-g_{*}h\right)\mathrm{d}\sigma \mathrm{d}v_{*}.
	\een
	Here, $B^{s,\gamma}$ is the angular non-cutoff kernel derived from inverse power law potentials, given by
	\ben \label{kernel-ipl-gamma-s}
	B^{s,\gamma}(v-v_*,\sigma) = C_{s,\gamma} (\sin\f{\theta}{2})^{-2-2s} |v-v_{*}|^{\gamma}.
	\een
	The main result in this paper is to rigorously prove for $\gamma>-5$, 
	\ben \label{limit-from-B-2-L}
	F^{s,\gamma}_{B}/(1-s) \to F_{L}^{\gamma}, \quad s \to 1^{-}. \een
	This can be stated  mathematically as that  the Boltzmann equation with a proper scaling tends to the Landau equation as $s \to 1^{-}$. 

	Recall that \eqref{kernel-ipl-gamma-s} is derived from the inverse power law potential $U(r) =r^{-p}$. For $p \geq 1$, one has $s = 1/p, \gamma = 1-4s$. The Coulomb potential 
	corresponds to  $p=s=1, \gamma=-3$. In order to treat the Coulomb potential physically, we can take $\gamma =1-4s$ and study the limit $s \to 1^{-}$.
	Our result directly  yield  limit to the Landau equation with $\gamma=-3$ in \eqref{limit-from-B-2-L}. 
	To be more general in mathematics, in the following discussion,  we will regard $\gamma$ as a fixed constant in the range $\gamma>-5$ for  Landau and $\gamma>-3-2s$ for Boltzmann.

In the existing literatures on the grazing limit, for instance in \cite{desvillettes1992asymptotics}, 
by introducing a 
cutoff $\theta \gtrsim \epsilon$
and by suitably adding the scaling factor $\epsilon^{2s-2}$, Desvillettes
considered the scaled Boltzmann kernel
$B^{\epsilon,s,\gamma} (v-v_*,\sigma)  = \epsilon^{2s-2} (\sin\f{\theta}{2})^{-2-2s} \mathrm{1}_{0 \leq \theta \leq \epsilon} |v-v_{*}|^{\gamma}$ and studied the limit $\epsilon \to 0$. See also \cite{he2014asymptotic} and \cite{duan2023solutions} for further development
in this direction.
However, this kind of revised artificial Boltzmann kernels do not correspond to any physical potential as shown  in Figure \ref{flow}. 

We now explain Figure \ref{flow}. Let $\gamma>-5$ be fixed.
In the middle column, the Landau equation 
 with initial datum $F_0$ yields a solution $F^{\gamma}_{L}$.
Our new approach corresponds to the right column  in Figure \ref{flow}.
That is,  the Boltzmann equation with the scaled initial datum $(1-s)F_0$ directly
gives a solution $F^{s,\gamma}_{B}$. Then the  scaling $F^{s,\gamma}_{B}/(1-s)$ with limit $s \to 1^{-}$ gives the  solution $F^{\gamma}_{L}$ of the  Landau equation.
Note that the Boltzmann equation is considered with  the physical cross-section \eqref{kernel-ipl-gamma-s} with parameter $s$. We also remark that in the previous
argument for  grazing limit, the parameter $s$ is fixed while the artificial  cutoff  parameter $\epsilon$ plays a role in the limit. However, in our approach,  the limit of $s$  to $1$ yields the
limit of the Boltzmann equation to Landau equation.

\begin{figure}[htbp]
	\large
	\centering
	\includegraphics[scale=1.2]{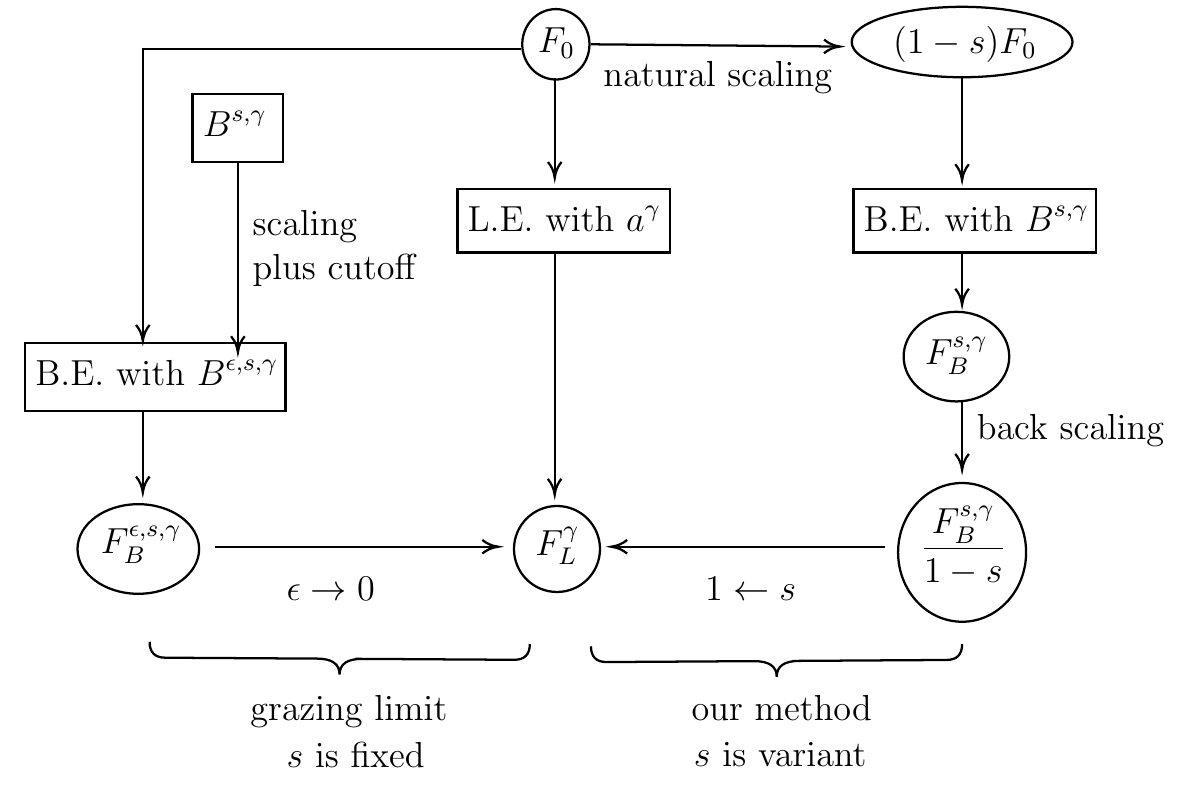}
	\caption{Grazing limit VS our method}
	\label{flow}
\end{figure}

In the following, we will explain in more details about the new appraoch.
Let $\tilde{F}^{s,\gamma}_{B} = F^{s,\gamma}_{B}/(1-s)$. Then  $\tilde{F}^{s,\gamma}_{B}$ is the solution to
\ben \label{scaled-B}
\partial _t F +  v \cdot \nabla_{x} F= \tilde{Q}^{s,\gamma}_{B}(F, F), \quad F|_{t=0} = F_{0}.
\een
Here $\tilde{Q}^{s,\gamma}_{B}$ is the Boltzmann operator defined with the kernel
\ben \label{kernel-ipl-gamma-s-with-factor}
\tilde{B}^{s,\gamma} (v-v_*,\sigma)  = (1-s)B^{s,\gamma} (v-v_*,\sigma) .
\een
In order to prove the limit \eqref{limit-from-B-2-L}, it is equivalent to 
prove that 
\ben \label{convergence}
\tilde{F}^{s,\gamma}_{B} \to F_{L}^{\gamma}, \quad s \to 1^{-}.
\een
The factor $(1-s)$ that naturally appears 
in \eqref{kernel-ipl-gamma-s-with-factor} corresponds to the grazing limit. This is because only small deviation  collisions contribute in the limit 
$s \to 1^{-}$.

In this paper, we will investigate
the above limit in the near-equilibrium framework where the 
unique global  classical solution can be constructed. We first recall some relevant results on the
well-posedness theories in this framework.
For the non-cutoff kernel \eqref{kernel-ipl-gamma-s},
Gressman-Strain in \cite{gressman2011global} established global well-posedness of the Boltzmann equation  
	 in the following parameter range 
	\ben \label{gamma-s-range-old-2}  
	\gamma > -3, \quad 0<s<1. 
	\een
	Independently,  Alexandre-Morimoto-Ukai-Xu-Yang  \cite{alexandre2012boltzmann} proved the same result in the range \eqref{gamma-s-range-old-2} with a constraint $\gamma+2s> -\f{3}{2}$ to obtain a better estimate on the nonlinear operators. In order to consider the grazing limit for  Coulomb potential $\gamma=-3$, we need  to obtain some uniform estimates for $\gamma <-3$.
		There are some discussions about this
		in the previous works.  For instance, the weak solutions were constructed for 
		$\gamma\geq-3$ for the two types of equation in the classical work \cite{villani1998new} 
		by Villani in which a remark on page 284 says that  'one could take $\gamma >-4$'.  
		As for Landau equation, Guo \cite{guo2002landau} firstly proved global well-posedness of
		 for $\gamma\geq-3$ and he  pointed out on page 394  in \cite{guo2002landau} that 
		  'Although our theorem is still valid for certain 
		  $\gamma$ 
		  even below $-3$'.
		  
Motivated by the above  works, we will consider the range 
$\gamma \le -3$ in this paper.
The new contribution of this paper to  the Boltzmann equation
is to contain the parameters range for $\gamma$ and $s$ in  the triangle $0<s<1, -3 -2s <\gamma \le -3$, see the region formed by the three red dash lines in Figure \ref{range}.
 On one hand, we justisfy the well-posedness of the Boltzmann equation in a region below
 $\gamma=-3$.
  On the other hand,  more importantly, the uniform estimates are obtained to the left of the vertical line $s=1$  so  that the limit $s \to 1^{-}$ can be considered for any $\gamma>-5$. This then
obvious includes  the Coulomb potential $\gamma=-3$ and the cases mentioned in
the previous literatures. Hence, as a byproduct, our new contribution for the Landau equation
is the well-posedness for $-5 <\gamma < -3$. 

\begin{figure}[htbp]
	\large
	\centering
	\includegraphics[scale=1.2]{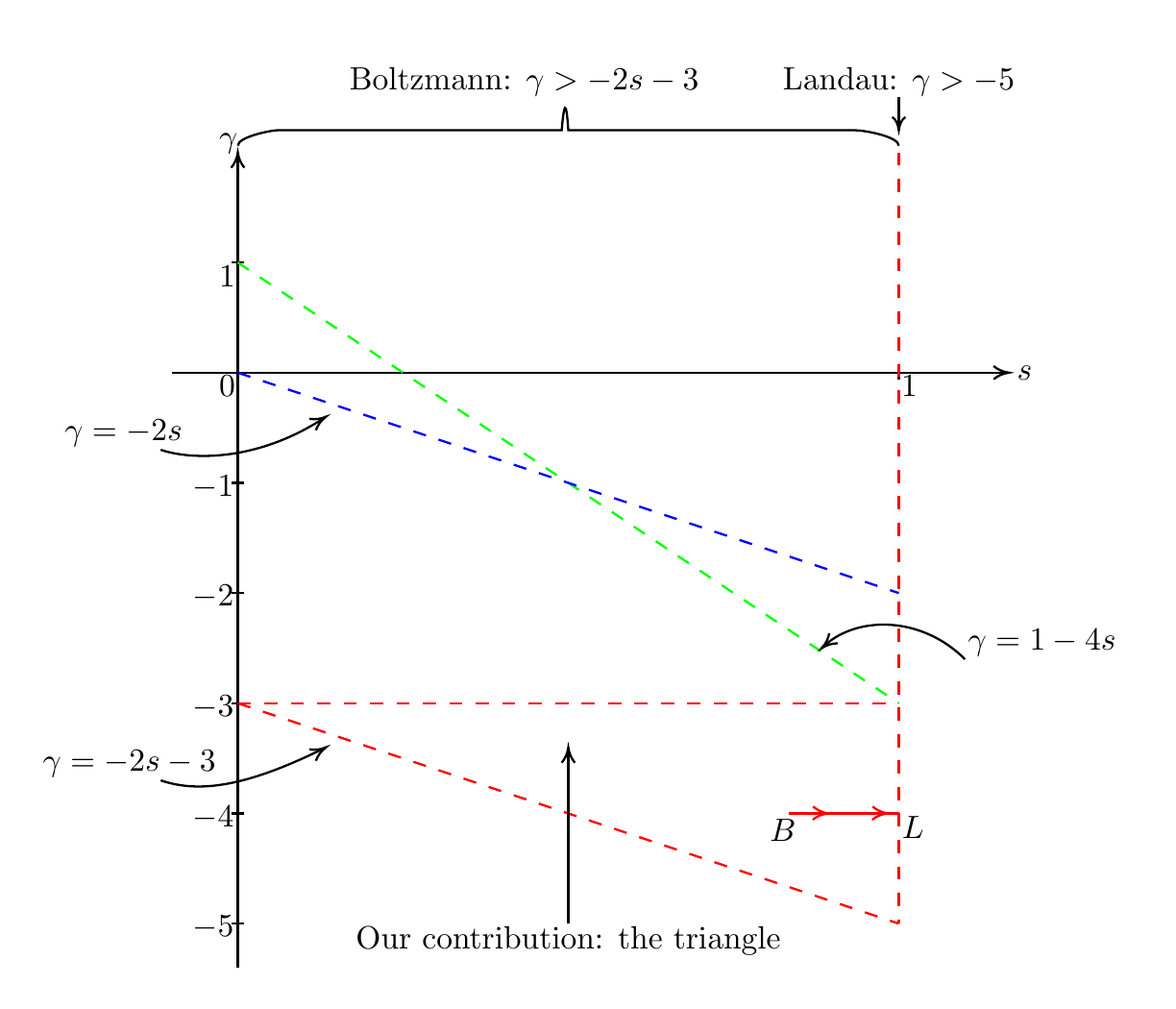}
	\caption{The parameters domain}
	\label{range}
\end{figure}


Now let us explain why $\gamma=-2s-3$ is a valid threshold in three dimensional space. It is now well known that
the Boltzmann operator $Q^{s,\gamma}_{B}(g, \cdot)$ in \eqref{Boltzmann-operator-2} behaves like 
a fractional Laplace operator $-C_{g}(-\Delta)^{s}$. Recall that $-(-\Delta)^{s}$ in three dimensions can be defined by a 
singular integral
\ben \label{fractional-Laplace}
-(-\Delta)^{s} f (v) := C_{s}
\lim_{r \to 0^{+}}  \int_{\mathbb{R}^3 \setminus B(v,r)} \f{f(v)-f(v_*)}{|v-v_{*}|^{3+2s}} \mathrm{d}v_*.
\een
This implies that $\gamma>-2s-3$ for suitably general function spaces. Moreover, 
there exists a  universal constant $c>0$ such that
\beno
\lim_{s \to 1^{-}}  \f{C_{s}}{1-s} = c.
\eeno
This implies the scaling factor $1-s$  naturally appears in \eqref{limit-from-B-2-L}. Then as $s \to 1^{-}$, 
$-(-\Delta)^{s} \to \Delta$ which is the main part of
the Landau operator \eqref{oroginal-definition-Laudau-oprator}. 


	\subsection{Main results} We will state the main results  in this subsection. Consider
	the following Cauchy problem of the Boltzmann equation
	\begin{equation}\label{Cauchy-Boltzmann} \left\{ \begin{aligned}
			&\partial _t F +  v \cdot \nabla_{x} F=Q^{s,\gamma}_{B}(F,F), ~~t > 0, x \in \mathbb{T}^{3}, v \in \R^3 ,\\
			&F|_{t=0} = F_{0}.
		\end{aligned} \right.
	\end{equation}
	Here $F(t,x,v)\geq 0$ is the density
	function of particles with velocity
	$v\in\R^3$ at time $t\geq 0$, position $x \in \mathbb{T}^{3} :=[-\pi,\pi]^{3}$. The Boltzmann operator is defined as
	\ben \label{Boltzmann-operator}
	Q^{s,\gamma}_{B}(g,h)(v):=
	\int_{\R^3}\int_{\mathbb{S}^{2}}B^{s,\gamma}(v-v_*,\sigma)\left(g^{\prime}_{*} h^{\prime}-g_{*}h\right)\mathrm{d}\sigma \mathrm{d}v_{*}.
	\een
	Here,  $h=h(v)$, $g_*=g(v_*)$,
	$h'=h(v')$, $g'_*=g(v'_*)$ where $v'$, $v_*'$ are given by
	\begin{eqnarray}\label{velocity-before-after}
		v'=\frac{v+v_{*}}{2}+\frac{|v-v_{*}|}{2}\sigma,
	\quad	v'_{*}=\frac{v+v_{*}}{2}-\frac{|v-v_{*}|}{2}\sigma,\quad \sigma\in\mathbb{S}^{2}.
	\end{eqnarray}
	Recalling \eqref{kernel-ipl-gamma-s} and \eqref{kernel-ipl-gamma-s-with-factor}, from now on
	we take
	\ben \label{kernel-studied}
	B^{s,\gamma}(v-v_{*},\sigma)= (1-s) (\sin\f{\theta}{2})^{-2-2s} \mathrm{1}_{0 \leq \theta \leq \pi/2} |v-v_{*}|^{\gamma}, 
	\een
	where $\cos\theta= \frac{v-v_{*}}{|v-v_{*}|}\cdot \sigma$. That is, 
	the angular function is 
	\beno
	b^{s}(\theta) = (1-s) (\sin\f{\theta}{2})^{-2-2s} \mathrm{1}_{0 \leq \theta \leq \pi/2}. 
	\eeno
	Note that the angle variable is restricted  to
	$0 \leq \theta \leq \pi/2$ by symmetry as other papers. 
	Thanks to the factor $1-s$, the mean moment transfer is finite 
	by computing
	\ben \label{mean-momentum-transfer}
	\int_{\mathbb{S}^{2}}  b^{s}(\theta) \sin^{2} \frac{\theta}{2}  \mathrm{d}\sigma = 4 \pi \times  2^{s-1}.
	\een
	In accordance with \eqref{mean-momentum-transfer}, we take $\Lambda = \pi$ in \eqref{matrix}.
	In fact, we will show that
	$$Q^{s,\gamma}_{B} =  2^{s-1} Q^{\gamma}_{L} + O(1-s),$$
	 and so
	$Q^{s,\gamma}_{B} \to Q^{\gamma}_{L}$ as $s \to 1^{-}$. See \eqref{Q-into-three-terms} for  details.

	To construct the global-in-time classical solution in the spatially inhomogeneous case, one usually consider the near equilibrium framework.
	Recall that the solution of \eqref{Cauchy-Landau}
	and \eqref{Cauchy-Boltzmann} conserves mass, momentum and energy. We assume $F_{0}$ is
	a small perturbation of the  equilibrium $\mu$  where $\mu(v) := (2\pi)^{-3/2}e^{-|v|^{2}/2}$.
	Let us recall the linearized versions of \eqref{Cauchy-Boltzmann} and  \eqref{Cauchy-Landau}.
	Set $F=\mu +\mu^{1/2}f$, then \eqref{Cauchy-Boltzmann} is reduced to	
	\begin{equation}\label{Cauchy-linearizedBE} \left\{ \begin{aligned}
			&\partial_{t}f + v\cdot \nabla_{x} f + \mathcal{L}^{s,\gamma}_{B}f= \Gamma^{s,\gamma}_{B}(f,f), ~~t > 0, x \in \mathbb{T}^{3}, v \in\R^3,\\
			&f|_{t=0} = f_{0}.
		\end{aligned} \right.
	\end{equation}
	Here the linearized Boltzmann operator $\mathcal{L}^{s,\gamma}_{B}$ and the nonlinear term $\Gamma^{s,\gamma}_{B}$ are defined by
	\ben\label{Def-nonlinear-linear} \Gamma^{s,\gamma}_{B}(g,h):= \mu^{-1/2} Q^{s,\gamma}_{B}(\mu^{1/2}g,\mu^{1/2}h), \quad
	\mathcal{L}^{s,\gamma}_{B}g:= -\Gamma^{s,\gamma}_{B}(\mu^{1/2},g) - \Gamma^{s,\gamma}_{B}(g, \mu^{1/2}).
	\een
	
	With the same decomposition, set $F=\mu +\mu^{1/2}f$. Then  \eqref{Cauchy-Landau} becomes
	\begin{equation}\label{Cauchy-linearizedLE} \left\{ \begin{aligned}
			&\partial_{t}f + v\cdot \nabla_{x} f + \mathcal{L}_{L}^{\gamma}f= \Gamma_{L}^{\gamma}(f,f), t > 0, x \in \mathbb{T}^{3}, v \in\R^3,\\
			&f|_{t=0} = f_{0}.
		\end{aligned} \right.
	\end{equation}
	The linearized Landau operator $\mathcal{L}_{L}^{\gamma}$ and the nonlinear term $\Gamma_{L}^{\gamma}$  are defined by
	\ben\label{DefL}\Gamma_{L}^{\gamma}(g,h) := \mu^{-1/2} Q_{L}^{\gamma}(\mu^{1/2}g,\mu^{1/2}h), \quad
	\mathcal{L}_{L}^{\gamma}g := -\Gamma_{L}^{\gamma}(\mu^{1/2},g) - \Gamma_{L}^{\gamma}(g, \mu^{1/2}). \een

Note that the conservation laws imply that for all $t\ge0$,
	\ben\label{conserveq}  \int_{\mathbb{T}^{3} \times \R^3} F(t,x,v)\phi(v)\mathrm{d}x\mathrm{d}v=\int_{\mathbb{T}^{3} \times \R^3} F(0,x,v)\phi(v)\mathrm{d}x\mathrm{d}v,\quad \phi(v)=1,v_{j},|v|^2, \quad j=1,2,3.\een
Up to suitable choice of the physical parameters in the equilibrium state, without loss of generality, we
assume the initial data satisfy
	\ben
	\label{initial-condition} \int_{\mathbb{T}^{3} \times \R^3} \sqrt{\mu}f_0\phi \mathrm{d}x\mathrm{d}v=0, \quad \phi(v)=1,v_{j},|v|^2, \quad j=1,2,3,\quad F_0 = \mu+\mu^{\f12}f_0\ge 0.
	\een
Then
	\ben \label{Nuspace}  \int_{\mathbb{T}^{3} \times \R^3} \sqrt{\mu}f(t)\phi \mathrm{d}x\mathrm{d}v=0, \quad \phi(v)=1,v_{j},|v|^2, \quad j=1,2,3.\een
%
	
The case for hard potential with  $\gamma +2s \geq 0$ is relatively easy because the linearized Boltzmann operator has a spectrum gap. This corresponds to the region
above the blue line in Figure \ref{range}.
Therefore, we only consider the soft potentials in this paper when 
\ben \label{gamma-s-range-new}  0<s<1, \quad -3 <  \gamma + 2s \leq 0. \een
Note that this corresponds to the region inside the parallelogram in Figure \ref{range}.
 To overcome the lack of spectrum gap, the following weighted energy space is introduced by Guo for the global well-posedness
	\ben \label{energy-functional}
\mathcal{E}_{N,l}^{s,\gamma}(f) := \sum_{j=0}^{N}\|f\|^{2}_{H^{N-j}_{x}
	\dot{H}^{j}_{l+j(\gamma+2s)}}.
\een
If $s=1$, we sometimes write $\mathcal{E}_{N,l}^{\gamma}(f) = \mathcal{E}_{N,l}^{1,\gamma}(f)$ which is the functional space for the Landau equation, cf. Subsection \ref{notation}. 

There are three main results given in the following theorem. The first one is the global well-posedness of the Boltzmann equation \eqref{Cauchy-linearizedBE} in the 
parameter range \eqref{gamma-s-range-new}. The second one is the global well-posedness of the Landau equation \eqref{Cauchy-linearizedLE} for $\gamma>-5$. The last one is about the grazing
limit of   the Boltzmann equation \eqref{Cauchy-linearizedBE} to the Landau equation \eqref{Cauchy-linearizedLE} by proving a 
 global-in-time asymptotic formula for the limit $s \to 1^{-}$.

\begin{thm}\label{asymptotic-result}
	{\textbf{[Well-posedness of the Boltzmann equation]}}
	Let $0<s<1, -3<\gamma+2s \leq 0$.	Let $N \geq 4, l\geq -N(\gamma+2s)$.
	There is a  constant $\delta_{s,\gamma,N,l}>0$ 
	such that, if
	\ben \label{initial-condition-smallness}
	\mathcal{E}^{s,\gamma}_{N,l}(f_{0}) \leq \delta_{s,\gamma,N,l}, \een
	then \eqref{Cauchy-linearizedBE} admits a unique global  solution $f^{s,\gamma}$ satisfying $\mu+\mu^{\f12}f^{s,\gamma}\ge0$ and
	\ben \label{uniform-controlled-by-initial} \sup_{t\ge 0} \mathcal{E}^{s,\gamma}_{N,l}(f^{s,\gamma}({t}))\leq Z_{s,\gamma,N,l} \mathcal{E}^{s,\gamma}_{N,l}(f_{0}),\een
	for some constant $Z_{s,\gamma,N,l}$.  Here $\delta_{s,\gamma,N,l} = \f12 \eta_{s,\gamma,N,l}^2$ where 
	$\eta_{s,\gamma,N,l}$ is given in Theorem
	\ref{a-priori-estimate-LBE}.
	See \eqref{c-relation-1} for the definition of $Z_{s,\gamma,N,l}$.
	For any fixed $N,l$, there are two functions
	$\delta_{N,l}, Z_{N,l} : (0,1] \times (0, 3] \to (0, \infty)$
 satisfying
	\ben \label{dependence-s-gamma}
	\delta_{s,\gamma,N,l} = \delta_{N,l}(s, \gamma+2s+3), \quad Z_{s,\gamma,N,l} = Z_{N,l}(s, \gamma+2s+3),
	\een
	and
	\ben \label{property-of-delta-N-l}
	\delta_{N,l}(x_1, x_2)  \text{ is non-decreasing w.r.t. each argument, and vanishes as } x_1 \to 0^{+} \text{ or } x_2 \to 0^{+},
	\\ \label{property-of-Z-N-l}
	Z_{N,l}(x_1, x_2)  \text{ is non-increasing w.r.t. each argument and tends to infinity  as } x_1 \to 0^{+} \text{ or } x_2 \to 0^{+}.
	\een

{\textbf{[Well-posedness of the Landau equation]}}	Taking $s=1$ in the above,  the global well-posedness of 
	\eqref{Cauchy-linearizedLE} holds true. We state this result in details for later discussion.
	Let $-3<\gamma+2 < 0$. Let $N \geq 4, l\geq -N(\gamma+2)$.
	Suppose $f_0$ verify \eqref{initial-condition}. Let $\delta_{\gamma,N,l} = \delta_{1,\gamma,N,l}, Z_{\gamma,N,l} = Z_{1,\gamma,N,l}$.
	If
	\ben \label{initial-condition-smallness-L}
	\mathcal{E}^{\gamma}_{N,l}(f_{0}) \leq \delta_{\gamma,N,l}, \een
	then \eqref{Cauchy-linearizedLE} admits a unique global  solution $f^{\gamma}$ satisfying $\mu+\mu^{\f12}f^{\gamma}\ge0$ and
	\ben \label{uniform-controlled-by-initial-L} \sup_{t\ge 0} \mathcal{E}^{\gamma}_{N,l}(f^{\gamma}({t}))\leq Z_{\gamma,N,l} \mathcal{E}^{\gamma}_{N,l}(f_{0}).\een
	
	{\textbf{[Asymptotic formula of the grazing limit]}}
	Fix $-5<\gamma < -2$. Let $N \geq 4,  l\geq -N(\gamma+2)$.
	Let $s_{*} : = \f{1}{2}(1 - \f{\gamma+3}{2})$.
	Assume
	\ben \label{initial-condition-smallness-for-asy}
	\mathcal{E}^{\gamma}_{N+3,l+2N-3\gamma+5}(f_{0}) \leq \delta_{s_*,\gamma,N+3,l+2N-3\gamma+5}. \een
	Since for any $s_*\leq s \leq 1$,
	\beno
	\mathcal{E}^{s,\gamma}_{N+3,l+2N-3\gamma+5}(f_{0}) \leq	\mathcal{E}^{\gamma}_{N+3,l+2N-3\gamma+5}(f_{0}) \leq \delta_{s_*,\gamma,N+3,l+2N-3\gamma+5} \leq \delta_{s,\gamma,N+3,l+2N-3\gamma+5},
	\eeno
	then
	by the above well-posedness result, \eqref{Cauchy-linearizedBE} has a unique solution $f^{s,\gamma}$ for 
	$s_* \leq s < 1$, and \eqref{Cauchy-linearizedLE} has a unique solution $f^{\gamma}$. 
	Moreover, 
	the family of solutions $\{f^{s,\gamma}\}_{s_* \leq s < 1}$ satisfy
	\ben \label{error-function-uniform-estimate} \sup_{t\ge 0} \mathcal{E}^{\gamma}_{N,l}(f^{s,\gamma}(t)-f^{\gamma}(t))\leq (1-s)^2 Z_{s_*,\gamma,N,l} \exp \left(C_{N,l} Z_{s_*,\gamma,N,l}^3 C_{\gamma}^2 \mathcal{E}_{N+3,l+2N-3\gamma+5}^{1,\gamma}(f_{0})\right).\een
\end{thm}

In the following we will give some remarks on the above result. First of all, 
we will  keep track of the dependence on the parameters $s,\gamma$. This kind
of dependence gives the precise condition on the parameters for  the global well-posedness theory of the Boltzmann and Landau equation. In particular, we have the explicit relation between $\delta, Z$ and $s,\gamma$
in \eqref{dependence-s-gamma}. On one hand, the region of parameters
for the well-posedness is non-empty as long as $0 <s<1, \gamma>-2s-3$. On the other hand, according to \eqref{property-of-delta-N-l},
the region indicated by
the "radius" $\delta_{s,\gamma,N,l}>0$ 
may shrink
as $s \to 0^{+}$ or $\gamma \to (-2s-3)^{+}$.


The estimate \eqref{error-function-uniform-estimate} implies that 
\ben \label{asy-f-to-f}
F^{s,\gamma}_{B} = F^{\gamma}_{L} + (1-s) F_{R}^{s,\gamma},
\een
where $F^{s,\gamma}_{B}$ and $F^{\gamma}_{L}$ are solutions to \eqref{Cauchy-Boltzmann} and  \eqref{Cauchy-Landau} respectively. Here the error term $F_{R}^{s,\gamma}$ is uniformly bounded in some function space. The formula validates our approach shown in Figure \ref{flow}.


A few more remarks are given as follows.

\begin{rmk}  If $s$ and $\gamma$ are regarded as two
	independent parameters, the well-posedness of the Boltzmann equation holds when $\gamma>-3-2s$ and $0<s<1$. In fact, the angular singularity $s \to 1^{-}$  is the essential reason
	for being possible ill-posedness.
\end{rmk}


\begin{rmk} The  global well-posedness in the lower regularity function spaces, such as   the one introduced in \cite{duan2021global} and time decay rates as  
	 in \cite{strain2006almost,strain2008exponential,duan2021global} for different
	 non-optimal ranges of the parameters $\gamma$ and $s$ can be obtained for the parameters in the
	 optimal range as in Theorem \ref{asymptotic-result}. For brevity and to focus on the
	 key points in this paper, we will not pursue these analysis here.
\end{rmk}

	\begin{rmk}
		In the previous
		studies, the reason that the condition
			 $\gamma>-3$ is needed is  because  the following estimate is used, 
	\ben \label{ingegral-bounded-by-C-gamma}
	\int_{\mathbb{R}^{3}} \mu(v_{*}) |v-v_{*}|^{\gamma} \mathrm{d}v_{*} \leq C_{\gamma} \langle v \rangle^{\gamma},
	\een
	where $C_{\gamma}$ is a constant for any $v \in \mathbb{R}^{3}$. 
	To understand why $\gamma>-3-2s$ is sufficient for well-posedness, intuitively, one can
	note  the
	angular singularity in the cross-section leads to a fractional derivative of order $2s$.
	Thus, we
	can expect to recover of an extra $2s$ in the range for $\gamma$ below $\gamma=-3$ by
	sacrificing some regularity of the solution under consideration.
\end{rmk}

	\subsection{Notations} \label{notation}
	We list some notations that are used in the paper.
	
	{\it {Common notations.}}
	We denote the multi-index $\beta =(\beta_1,\beta_2,\beta_3)$ with
	$|\beta |=\beta _1+\beta _2+\beta _3$.  $a\lesssim b$ means that  there is a
	uniform constant $C,$ which may be different on different lines,
	such that $a\leq Cb$.  We use the notation $a\sim b$ when $a\lesssim b$ and $b\lesssim
	a$. The bracket $\langle \cdot \rangle$ 
	is defined by $\langle \cdot \rangle := (1+|\cdot|^2)^{1/2}$.
	The weight function $W_l$ is defined by $W_l(v) := \langle v\rangle^l$.
	We denote $C(\lambda_1,\lambda_2,\cdots, \lambda_n)$ or $C_{\lambda_1,\lambda_2,\cdots, \lambda_n}$  by a constant depending on   parameters $\lambda_1,\lambda_2,\cdots, \lambda_n$.
	The notations  $\langle f,g\rangle:= \int_{\R^3}f(v)g(v)\mathrm{d}v$ and $(f,g):= \int_{\R^3\times\TT^3} fg\mathrm{d}x\mathrm{d}v$
	are used to denote the inner products for the $v$ variable and for the $x,v$ variables respectively.
	As usual, $\mathrm{1}_A$ is the characteristic function of the set $A$. If $A,B$ are two operators, then $[A,B]:= AB-BA$.

	{\it {Function spaces.}}	
		 For simplicity,  set $\partial^{\alpha}:=\partial^{\alpha}_x, \partial_{\beta}:= \pa^{\beta}_v, \partial^{\alpha}_{\beta}:=\partial^{\alpha}_x\pa^{\beta}_v$. We will use the following 
		 function spaces.
	\begin{itemize}
		\item  For   real number $n, l $,
		\begin{equation*}
			H^{n}_l :=\bigg\{f(v)\big| |f|^2_{H^n_l}:= |\langle D\rangle^n W_{l} f|^{2}_{L^{2}} = \int_{\R^3} |(\langle D\rangle^n W_{l} f)(v)|^2 \mathrm{d}v
			<+\infty\bigg\}.
		\end{equation*} Here $a(D)$ is a   differential operator with the symbol
		$a(\xi)$ defined by
		\beno  \big(a(D)f\big)(v):=\f1{(2\pi)^3}\int_{\R^3}\int_{\R^3} e^{i(v-y)\xi}a(\xi)f(y)\mathrm{d}y\mathrm{d}\xi.\eeno
		\item
		For $n \in \N, l \in \R$, 
		\beno H^{n}_{l} := \bigg\{f(v) \big| |f|^{2}_{H^{n}_{l}} :=
		\sum_{|\beta| \leq n}  |\pa_{\beta}f|^{2}_{L^{2}_l}< \infty  \bigg\},\eeno
		where $|f|_{L^{2}_{l}} := |W_l f|_{L^{2}}$ is the usual $L^{2}$ norm with weight $W_l$.

		\item For $n \in \N, l \in \R$, 
		\ben \label{pure-order} \dot{H}^{n}_l := \bigg\{f(v) \big| |f|^{2}_{\dot{H}^{n}_{l}} :=
		\sum_{|\beta| = n}  |\pa_{\beta} f|^{2}_{L^{2}_l}< \infty  \bigg\}.\een

		\item For $m\in\N$, 
		\begin{equation*} H^{m}_{x} :=\bigg\{f(x)\big| |f|^{2}_{H^{m}_{x}}:= \sum_{|\alpha | \leq m}|\partial^{\alpha} f|^{2}_{L^{2}_{x}}<\infty\bigg\}.
		\end{equation*}

		\item For $m,n \in \N, l \in \R$, 
		\beno H^{m}_xH^{n}_{l} := \bigg\{f(x,v) \big| \|f\|^{2}_{ H^{m}_xH^{n}_{l}} :=
		\sum_{|\alpha| \leq m, |\beta| \leq n}  || \partial^{\alpha}_{\beta} f|_{L^{2}_l} |^{2}_{L^{2}_{x}} < \infty\bigg\}.\eeno
		We write $\|f\|_{H^{m}_{x}L^{2}_{l}} := \|f\|_{ H^{m}_xL^{2}_{l} }$ if $n=0$ and   $\|f\|_{L^{2}_{x}L^{2}_{l}} := \|f\|_{ H^{0}_xH^{0}_{l}}$ if $m=n=0$. The space $H^{m}_x\dot{H}^{n}_{l}$ can be defined similarly.
	\end{itemize}
	
	Finally in the introduction, let us recall the dissipation norm of the linearized operators $\mathcal{L}^{s,\gamma}_{B}$ and 
	$\mathcal{L}^{\gamma}_{L}$. More precisely, for $l \in \mathbb{R}$, set
	\ben \label{norm-definition}
	|f|^{2}_{s,l} :=|W_{s}((-\Delta_{\mathbb{S}^{2}})^{1/2}) W_{l}f|^{2}_{L^{2}} + |W_{s}(D)W_{l}f|^2_{L^{2}} + |W_{s} W_{l}f|^{2}_{L^{2}}. \een
	Here $W_{s}(D)$ is the pseudo-differential operator with symbol $W_{s}$.
	The operator $W_{s}((-\Delta_{\mathbb{S}^{2}})^{1/2})$ is defined as follows. If $v = r \sigma, r \geq 0, \sigma \in \mathbb{S}^{2}$, then
	\ben\label{DeltaWe} (W_{s}((-\Delta_{\mathbb{S}^{2}})^{1/2})f)(v) := \sum_{l=0}^\infty\sum_{m=-l}^{l} (1+l(l+1))^{\frac{s}{2}}
	Y^{m}_{l}(\sigma)f^{m}_{l}(r),
	\een
	where
	$ f^{m}_{l}(r) = \int_{\mathbb{S}^{2}} Y^{m}_{l}(\sigma) f(r \sigma) \mathrm{d}\sigma$,
	and $Y_l^m, -l\le m\le l$ are the real spherical harmonics satisfying 
	$ (-\Delta_{\mathbb{S}^2})Y_l^m=l(l+1)Y_l^m.$
	Note that the function $W_{s}$ is the common weight gain in the three individual norms. 
The dissipation norm $|\cdot|_{s,\gamma/2}$ characterizes $\mathcal{L}^{s,\gamma}_{B}$, see Proposition \ref{part-l} and
Theorem
\ref{micro-dissipation} for details. Similarly, the dissipation norm $|\cdot|_{1,\gamma/2}$ characterizes the $\mathcal{L}^{\gamma}_{L}$.

	From time to time,  we also write $|f|_{L^{2}_{s,l}} = |f|_{s,l}$. For functions defined on $\mathbb{T}^{3}\times\R^{3}$, the space $H^{m}_xH^{n}_{s,l}$   with $m,n\in\N$ is defined  by
		\beno  H^{m}_xH^{n}_{s,l}:=\bigg\{f(x,v)\big|
		\|f\|^{2}_{H^{m}_xH^{n}_{s,l}} := \sum_{|\alpha| \leq m, |\beta| \leq n} || \partial^{\alpha}_{\beta} f|_{L^{2}_{s,l}} |^{2}_{L^{2}_{x}} <\infty\bigg\}.
		\eeno
	Set $\|f\|_{H_x^{m}L^2_{s,l}}:= \|f\|_{H^{m}_xH^{0}_{s,l}}$ if $n=0$ and $\|f\|_{L^{2}_{x}L^2_{s,l}}:= \|f\|_{H^{0}_xH^{0}_{s,l}}$ if $m=n=0$. Again, the space $H^{m}_x\dot{H}^{n}_{s,l}$ can be defined accordingly.
	
	We sometimes omit the range of some frequently used variables in the integrals.
	Usually, $\sigma \in \mathbb{S}^{2}, v, v_{*}, u, \xi \in \mathbb{R}^{3}$.
	For example,  $\int (\cdots) \mathrm{d}\sigma := \int_{\mathbb{S}^{2}} (\cdots) \mathrm{d}\sigma,  \int (\cdots) \mathrm{d}\sigma \mathrm{d}v \mathrm{d}v_{*}
	:= \int_{\mathbb{S}^{2} \times \mathbb{R}^{3} \times \mathbb{R}^{3}} (\cdots) \mathrm{d}\sigma \mathrm{d}v \mathrm{d}v_{*}$.
	Integration w.r.t. other variables is understood in a similar way.
	Whenever a new variable appears, we will specify its range. 
	
When there is no confusion, we drop the subscripts $B$ and $L$ in the Boltzmann and Landau operators for brevity.
	
	\subsection{Organization of the paper}
	In Section \ref{Upper-Bound-Estimate-near}, we derive  the 
	upper bound estimates of operators in the singular region. Section \ref{Upper-Bound-Estimate-away} contains the upper bound estimates in the regular region. We 
	will show
the	coercivity estimate in Section \ref{coercivity-spectral}. 
 The commutator estimates and weighted upper bound estimates are given in Section \ref{Commutator-Estimate}. The proof Theorem
 \ref{asymptotic-result} is given in Section \ref{proof-main-theorem}. In the Appendix \ref{Operator-difference}, for  completeness,
 we prove  the operator convergence stated in Proposition \ref{limit-Botltzmann-to-Landau}.


	\section{Upper bound estimate in the singular region} \label{Upper-Bound-Estimate-near}
	
	In this section, we will derive the upper bound estimates on the collision operators in the singular region
	$|v-v_*| \lesssim \eta$. For this, we first recall the dyadic decomposition.

	\subsection{Dyadic decomposition} 
	Let  $\varphi$ be a smooth function on $\mathbb{R}_{+}$ satisfying 
	\beno
	\varphi=0 \text{ on } [0, 3/4], \quad \varphi \text{ is strictly increasing on } [3/4, 4/3], \quad \varphi = 1 \text{ on } [4/3, 3/2],
	\\
	\varphi \text{ is strictly decreasing on } [3/2, 8/3], \quad \varphi = 0 \text{ on } [8/3, \infty), \quad
	|\varphi^{\prime}| \leq 4 .
	\eeno
	Moreover, $\varphi$ is chosen such that
	the functions $\{\varphi_{k}(\cdot)\colonequals\varphi(\cdot/2^{k})\}_{ k \in \mathbb{Z}}$ is a partition of unit on $(0, \infty)$. That is,
	\ben \label{partition-of-unity}
	\sum_{k= -\infty}^{\infty} \varphi_{k} = 1 \text{ on } (0, \infty).
	\een
	Set  $\psi(r) = \sum_{k= -\infty}^{-1} \varphi_{k}(r) $ for $r >0$ and $\psi(0) =1$.  Then $\psi$ is a non-increasing smooth function on $\mathbb{R}_{+}$  and satisfies
	\ben \label{smooth-psi}
	\psi=1 \text{ on } [0, 3/4], \quad \psi \text{ is strictly decreasing on } [3/4, 4/3], \quad \psi = 0 \text{ on } [4/3, \infty), \quad |\psi^{\prime}| \leq 4 .
	\een
	Then we have the identity 
	\ben \label{function-identity}
	\psi + \sum_{k=0}^{\infty} \varphi_{k} \equiv 1 \text{ on } [0, \infty).
	\een
	With a little abuse of notation, we define radial functions $\varphi(v)\colonequals\varphi(|v|), \psi(v)\colonequals\psi(|v|), \varphi_{k}(v)\colonequals\varphi_{k}(|v|)$ for $ v \in \R^{3}$.

	Given a general Boltzmann kernel $B = B(v-v_{*},\sigma) = B(|v-v_{*}|,\cos\theta)$ where $\cos\theta = \f{v-v_{*}}{|v-v_{*}|} \cdot \sigma$, let $Q$ be the  Boltzmann operator with kernel $B$. That is,
	\ben \label{def-Boltzmann-operator-general}
	Q(g,h) \colonequals  \int_{\mathbb{S}^{2}} \int_{\mathbb{R}^{3}} B(|v-v_{*}|, \f{v-v_{*}}{|v-v_{*}|} \cdot \sigma) (g^{\prime}_{*}h^{\prime} - g_{*}h)
	\mathrm{d}\sigma \mathrm{d}v_{*}.
	\een

	
		Let $\mathcal{U}_{k}  \colonequals  \sum_{j \leq k} \varphi_{j}, \tilde{\varphi}_{k} \colonequals \sum_{|j-k| \leq N_{0}-1} \varphi_{j}$ for $k \in \mathbb{Z}$ for some fixed integer $N_{0} \geq 4$.
	Suppose the relative velocity satisfies $|v-v_{*}| \leq \frac{4}{3}$. Let $v_{*} \in \{\frac{3}{4} \times  2^{j} \leq |v_{*}| \leq \frac{8}{3} \times  2^{j}\}$. Then the fact that 
	\beno
	\frac{\sqrt{2}}{2} |v-v_{*}| \leq |v^{\prime}-v_{*}| \leq |v-v_{*}|,
	\eeno
	yields
	\begin{itemize}
		\item If $j \geq N_{0}-1$, then $|v|,|v^{\prime}| \in [(\frac{3}{4} - \frac{8}{3} \times 2^{-N_{0}})2^{j}, \frac{8}{3} (1 + 2^{-N_{0}})2^{j}] \subset \mathrm{Supp} \tilde{\varphi}_{j}$.
		\item If $j \leq N_{0}-2$, then $|v|,|v^{\prime}| \leq \frac{3}{2} \times 2^{N_{0}-1} \subset \mathrm{Supp} \mathcal{U}_{N_{0}-1}$.
	\end{itemize}
	Hence, if $Q$ is localized in $|v-v_{*}| \leq \frac{4}{3}$, then
	\ben 	\label{decomposition-phase-near-0}
	\langle Q(g,h), f\rangle =  \sum_{j \geq N_{0}-1} \langle Q(\varphi_{j} g, \tilde{\varphi}_{j}h), \tilde{\varphi}_{j}f\rangle + \langle Q(\mathcal{U}_{N_{0}-2} g, \mathcal{U}_{N_{0}-1}h), \mathcal{U}_{N_{0}-1}f\rangle.
	\een

	Let us recall the Bobylev formula which is about the Fourier transform of the Boltzmann operator. By \eqref{def-Boltzmann-operator-general}, the Bobylev formula reads
	\beno
	\mathcal{F}(Q(g,h))(\xi) = \int_{\mathbb{S}^{2}} \int_{\mathbb{R}^{3}} \big(\hat{B}(|\eta_{*}-\xi^{-}|, \f{\xi}{|\xi|}\cdot\sigma)
	-\hat{B}(|\eta_{*}|, \f{\xi}{|\xi|}\cdot\sigma)\big) \mathcal{F} g(\eta_{*}) \mathcal{F}h(\xi - \eta_{*}) \mathrm{d}\sigma \mathrm{d}\eta_{*},
	\eeno
	where
	\beno
	\hat{B}(|\xi|, \cos\theta)  \colonequals   \int_{\mathbb{R}^{3}}   B(|q|, \cos\theta) e^{- i q \cdot \xi}      \mathrm{d}q.
	\eeno
	Here $\mathcal{F}$ is the Fourier transform operator. As usual, denote $\hat{f} = \mathcal{F}f$, then
	\ben
	\langle Q(g,h), f \rangle &=& \langle \widehat{Q(g,h)}, \hat{f} \rangle
	\nonumber \\&=& \int_{\mathbb{S}^{2}} \int_{\mathbb{R}^{3}} \int_{\mathbb{R}^{3}} \big(\hat{B}(|\eta_{*}-\xi^{-}|, \f{\xi}{|\xi|}\cdot\sigma)
	-\hat{B}(|\eta_{*}|, \f{\xi}{|\xi|}\cdot\sigma)\big)  \hat{g}(\eta_{*}) \hat{h}(\xi - \eta_{*}) \bar{\hat{f}}(\xi) \mathrm{d}\sigma \mathrm{d}\eta_{*} \mathrm{d}\xi. \label{inner-product-into-frequency}
	\een
	Note that 
	\beno
	||\xi| - |\eta_{*}|| \leq |\xi - \eta_{*}| \leq |\xi| + |\eta_{*}|.
	\eeno
	
	Fix $j, p \in \mathbb{Z}$, suppose $\frac{3}{4} \times 2^{p}\leq |\eta_{*}| \leq \frac{8}{3} \times 2^{p}$ and $\frac{3}{4} \times 2^{j}\leq |\xi - \eta_{*}| \leq \frac{8}{3} \times 2^{j}$. Since $N_{0} \geq 4$, then
	\begin{itemize}
		\item If $p \leq j-N_{0}$, then $|\xi| \in [(\frac{3}{4} - \frac{8}{3} \times 2^{-N_{0}})2^{j}, \frac{8}{3} (1 + 2^{-N_{0}})2^{j}] \subset \mathrm{Supp} \tilde{\varphi}_{j}$.
		\item If $p \geq j+N_{0}$, then $|\xi| \in [(\frac{3}{4} - \frac{8}{3} \times 2^{-N_{0}})2^{p}, \frac{8}{3} (1 + 2^{-N_{0}})2^{p}] \subset \mathrm{Supp} \tilde{\varphi}_{p} $.
		\item If $|p-j| < N_{0}$, then $|\xi| \in [0,\frac{3}{2} \times 2^{p+N_{0}}]  \cap [0,\frac{3}{2} \times 2^{j+N_{0}}] \subset \mathrm{Supp} \mathcal{U}_{p+N_{0}} \cap \mathrm{Supp} \mathcal{U}_{j+N_{0}}$.
	\end{itemize}
	
	Define
	\ben \label{DefFj} \mathfrak{F}_{-1}f(x) \colonequals \psi(D) f, \quad
	\mathfrak{F}_{j}f(x) \colonequals \varphi_{j}(D) f, j \geq 0.
	\een
	 $\mathfrak{F}_{-1}$ and $\mathfrak{F}_{j}$ localize the frequency of function $f$ in the region $|\xi| \lesssim 1$ and  $|\xi| \sim 2^{j}$ respectively. By \eqref{function-identity}, 
	the dyadic decomposition in frequency space reads
	\beno f = \sum_{j \geq -1} \mathfrak{F}_j f.
	\eeno
	Note that $\mathfrak{F}_{-1}$ has symbol $\psi$ instead of $\varphi_{-1}$. Set
	$\tilde{\mathfrak{F}}_{k} := \tilde{\varphi}_{k}(D)$.
	Then we have
	\ben \label{decomposition-in-frequency-space}
	\langle Q(g,h), f\rangle
	&=& \sum_{j \geq N_{0} -1} \sum_{-1 \leq p \leq j-N_{0}} \langle Q(\mathfrak{F}_{p} g, \mathfrak{F}_{j} h), \tilde{\mathfrak{F}}_{j} f\rangle 	+ \sum_{p \geq N_{0} -1} \sum_{-1 \leq j \leq p-N_{0}}  \langle Q(\mathfrak{F}_{p} g, \mathfrak{F}_{j}h), \tilde{\mathfrak{F}}_{p}f\rangle
	\\ \nonumber && 	+\sum_{p,j \geq -1, |p-j| < N_{0}}  \sum_{q \leq j+N_{0}} \langle Q(\mathfrak{F}_{p} g, \mathfrak{F}_{j}h), \mathfrak{F}_{q}f\rangle.
	\een
	
	For completeness, we  recall the definition of symbol class $S^{m}_{1, 0}$ as follows.
	\begin{defi}\label{psuopde} A smooth function $W(v,\xi)$ is said to be a symbol of type $S^{m}_{1, 0}$ if for any multi-indices $\alpha$ and $\beta$,
		\beno |(\pa^\alpha_\xi\pa^\beta_v W)(v,\xi)|\le C_{\alpha,\beta} \langle \xi\rangle^{m-|\alpha|}, \eeno
		where $C_{\alpha,\beta}$ is a constant depending only on   $\alpha$ and $\beta$.
	\end{defi}

	\begin{lem} \label{operatorcommutator1}
		Let $l, m, r \in \R, M(v,\xi) = M(\xi) \in S^{r}_{1,0}$ and $\Phi(v,\xi) = \Phi(\xi) \in S^{l}_{1,0}$. Then there exists a constant $C$ such that
		\beno
		|[M(D), \Phi]f|_{H^{m}} \leq C|f|_{H^{r+s-1}_{l-1}}.
		\eeno
	\end{lem}
	See Lemma 5.3 in \cite{he2018sharp} for the proof of Lemma \ref{operatorcommutator1}.
	Based on Lemma \ref{operatorcommutator1},  one directly has
	\ben \label{h-m-l-norm}
	|f|_{H^{m}_{l}}^2 \sim \sum_{j,k=-1}^{\infty} 2^{2mk} 2^{2lj}
	|\varphi_{k}(D) \varphi_{j}f|_{L^{2}}^2
	\sim \sum_{j,k=-1}^{\infty} 2^{2mk} 2^{2lj}
	|\varphi_{j} \varphi_{k}(D)f|_{L^{2}}^2.
	\een
	Here for simplicity, we take $\varphi_{-1} := \psi$. By \eqref{h-m-l-norm}, we use both $|W_{l}f|_{H^{m+s}}$ and $|W_{l} W_{s}(D)f|_{H^m}$ in the rest of the paper.

		\subsection{Operator splitting}	\label{ope-spl}
	We divide the relative velocity into two parts 
	\beno
	B(v-v_{*}, \sigma) = B_{\eta}^{s,\gamma}(v-v_{*}, \sigma)
	+ B^{s, \gamma, \eta}(v-v_{*}, \sigma),
	\eeno
	where
	\ben \label{kernel-split}
	B_{\eta}^{s,\gamma}(v-v_{*}, \sigma) := \psi_{\eta}(|v-v_{*}|) B(v-v_{*}, \sigma), \quad 
	B^{s,\gamma, \eta}(v-v_{*}, \sigma) := (1- \psi_{\eta}(|v-v_{*}|)) B(v-v_{*}, \sigma).
	\een
	Note that $B^{s,\gamma}_{\eta}$ is supported in $|v-v_{*}| \leq 4\eta/3$ so that it is singular when $|v-v_{*}| \to 0$, while  $B^{s,\gamma,\eta}$ is supported in $|v-v_{*}| \geq 3\eta/4$ 
	without singularity.

	Let $Q^{s,\gamma}_{\eta}$ and $Q^{s,\gamma, \eta}$ be the Boltzmann operators defined with kernel 
	$B^{s,\gamma}_{\eta}$ and $B^{s,\gamma,\eta}$ respectively. 
	And then let $\Gamma^{s,\gamma}_{\eta}$ and $\Gamma^{s,\gamma,\eta}$ be the nonlinear terms defined with kernel 
	$B^{s,\gamma}_{\eta}$ and $B^{s,\gamma,\eta}$ respectively. 
	
	Recall that the nonlinear term $\Gamma$ for a general kernel $B$ is defined by
	\beno
	\Gamma(g,h) &=& \mu^{-1/2} Q(\mu^{1/2}g,\mu^{1/2}h)
	= \int B(v-v_*,\sigma) \mu^{1/2}_{*} \big(g^{\prime}_{*} h^{\prime}-g_{*} h\big) \mathrm{d}\sigma \mathrm{d}v_{*}
	\\&=& \int B(v-v_*,\sigma) \big((\mu^{1/2}g)^{\prime}_{*} h^{\prime}-(\mu^{1/2}g)_{*} h\big)\mathrm{d}\sigma \mathrm{d}v_{*}
	\\&&+ \int B(v-v_*,\sigma) \big(\mu^{1/2}_{*}-(\mu^{1/2})^{\prime}_{*}\big)g^{\prime}_{*} h^{\prime} \mathrm{d}\sigma \mathrm{d}v_{*}
	= Q(\mu^{1/2}g,h) + I(g,h),
	\eeno
	where for brevity,  we define
	\ben \label{I-ep-ga-geq-eta}
	I(g,h) :=\int B(v-v_*,\sigma) \big(\mu^{1/2}_{*}-(\mu^{1/2})^{\prime}_{*}\big)g^{\prime}_{*} h^{\prime} \mathrm{d}\sigma \mathrm{d}v_{*} .
	\een
	Let $I^{s,\gamma}, I^{s,\gamma}_{\eta}$ and $I^{s,\gamma,\eta}$ be the bi-linear operators defined according to \eqref{I-ep-ga-geq-eta} with kernel 
	$B^{s,\gamma}, B^{s,\gamma}_{\eta}$ and $B^{s,\gamma,\eta}$ respectively.  Other operators
	with such subscripts and  superscripts are understood in the same way.
	With these notations, we have
	\ben
	\label{Gamma-ep-ga-into-IQ}
	\Gamma^{s,\gamma}(g,h) &=& Q^{s,\gamma}(\mu^{1/2}g,h) + I^{s,\gamma}(g,h),
	\\
	\label{Gamma-ep-ga-geq-eta-into-IQ}
	\Gamma^{s,\gamma,\eta}(g,h) &=& Q^{s,\gamma,\eta}(\mu^{1/2}g,h) + I^{s,\gamma,\eta}(g,h),
	\\
	\label{Gamma-ep-ga-leq-eta-into-IQ}
	\Gamma^{s,\gamma}_{\eta}(g,h) &=& Q^{s,\gamma}_{\eta}(\mu^{1/2}g,h)+ I^{s,\gamma}_{\eta}(g,h),
	\\
	\label{Q-ep-ga-sep-eta}
	Q^{s,\gamma}(g,h) &=& Q^{s,\gamma,\eta}(g,h) + Q^{s,\gamma}_{\eta}(g,h),
	\\
	\label{Gamma-ep-ga-sep-eta}
	\Gamma^{s,\gamma}(g,h) &=& \Gamma^{s,\gamma,\eta}(g,h) + \Gamma^{s,\gamma}_{\eta}(g,h),
	\\
	\label{I-ep-ga-sep-eta}
	I^{s,\gamma}(g,h) &=& I^{s,\gamma,\eta}(g,h) + I^{s,\gamma}_{\eta}(g,h).
	\een

	To implement the energy estimates for the nonlinear equations, we need to take derivatives.
	By binomial expansion, we have
	\ben \label{alpha-beta-on-Gamma} \pa^{\alpha}_{\beta}\Gamma^{s,\gamma}(g,h) = \sum _{\beta_{0}+\beta_{1}+\beta_{2}= \beta,\alpha_{1}+\alpha_{2}=\alpha} C^{\beta_{0},\beta_{1},\beta_{2}}_{\beta} C^{\alpha_{1},\alpha_{2}}_{\alpha} \Gamma^{s,\gamma}(\pa^{\alpha_{1}}_{\beta_{1}}g,\pa^{\alpha_{2}}_{\beta_{2}}h;\beta_{0}),\een
	where
	\ben \label{Gamma-beta}
	\Gamma^{s,\gamma}(g,h;\beta)(v):=
	\int_{\SS^{2} \times \R^3} B^{s,\gamma}(v-v_*,\sigma)(\pa_{\beta}\mu^{1/2})_{*}(g'_*h'-g_*h)\mathrm{d}\sigma \mathrm{d}v_{*}.
	\een
	Note that
	\ben \label{Gamma-ep-ga-geq-eta-into-IQ-inner-beta}
	\Gamma^{s,\gamma}(g,h;\beta) =   Q^{s,\gamma}(g\partial_{\beta}\mu^{1/2},h) +
	I^{s,\gamma}(g,h;\beta),
	\een
	where
	\ben \label{I-ep-ga-geq-eta-beta}
	I^{s,\gamma}(g,h;\beta) :=\int B^{s,\gamma}(v-v_*,\sigma) \big((\pa_{\beta}\mu^{1/2})_{*}-(\pa_{\beta}\mu^{1/2})^{\prime}_{*}\big)g^{\prime}_{*} h^{\prime} \mathrm{d}\sigma \mathrm{d}v_{*}.
	\een
	Thus,  in general we need to consider $ I^{s,\gamma}(g,h;\beta)$. This is again divided into two parts:
	$I^{s,\gamma,\eta}(g,h;\beta)$ and $ I^{s,\gamma}_{\eta}(g,h;\beta)$. 
	
	Recall 
	\beno 
	\mathcal{L}^{s,\gamma} f
	= - \Gamma^{s,\gamma}(\mu^{\f12}, f) - \Gamma^{s,\gamma}(f, \mu^{\f12}).
	\eeno
	By binomial expansion, we have
	\ben \label{alpha-beta-on-L} \pa^{\alpha}_{\beta} \mathcal{L}^{s,\gamma} f = \sum _{\beta_{0}+\beta_{1}+\beta_{2}= \beta} C^{\beta_{0},\beta_{1},\beta_{2}}_{\beta}  \mathcal{L}^{s,\gamma}(\pa^{\alpha}_{\beta_2}f; \beta_0, \beta_1),\een
	where
	\ben \label{beta-version-L-s-gamma}
	\mathcal{L}^{s,\gamma}(f; \beta_0, \beta_1) := - \Gamma^{s,\gamma}(\pa_{\beta_{1}}\mu^{1/2}, f;\beta_{0})- \Gamma^{s,\gamma}(f, \pa_{\beta_{1}}\mu^{1/2};\beta_{0}).
	\een
	We also define
	\ben \label{beta-version-L-1-2-def}
	\mathcal{L}^{s,\gamma}_{1}(f; \beta_0, \beta_1) := - \Gamma^{s,\gamma}(\pa_{\beta_{1}}\mu^{1/2}, f;\beta_{0}), \quad
	\mathcal{L}^{s,\gamma}_{2}(f; \beta_0, \beta_1) := - \Gamma^{s,\gamma}(f, \pa_{\beta_{1}}\mu^{1/2};\beta_{0}).
	\een
	In the same way, we can define
	$\mathcal{L}^{s,\gamma,\eta}(\cdot; \beta_0, \beta_1),
	\mathcal{L}^{s,\gamma,\eta}_{1}(\cdot; \beta_0, \beta_1), \mathcal{L}^{s,\gamma,\eta}_{2}(\cdot; \beta_0, \beta_1)$ with kernel $B^{s,\gamma,\eta}$
	and $\mathcal{L}^{s,\gamma}_{\eta}(\cdot; \beta_0, \beta_1),
	\mathcal{L}^{s,\gamma}_{\eta, 1}(\cdot; \beta_0, \beta_1), \mathcal{L}^{s,\gamma}_{\eta, 2}(\cdot; \beta_0, \beta_1)$ with kernel $B^{s,\gamma}_{\eta}$.

	\subsection{Taylor expansion and symmetry} 
	
	When evaluating the difference $f'-f$ (or $f'_*-f_*$) before and after collision, Taylor expansion is applied. We first denote the 1st-order expansion by
	\ben
	\label{Taylor1-order-1}
	f'-f=\int_0^1 (\na f)(v(\kappa)) \cdot (v'-v)\mathrm{d}\kappa, \quad 
	f'_*-f_*= \int_0^1 (\na f)(v_{*}(\iota)) \cdot (v'_*-v_*)\mathrm{d}\iota.
	\een

	To cancel the angular singularity, the second order expansion is needed:
	\ben
	&&\label{Taylor1}
	f'-f=(\na f)(v)\cdot(v'-v)+\int_0^1(1-\kappa) (\na^2 f)(v(\kappa)):(v'-v)\otimes(v'-v)\mathrm{d}\kappa,
	\\
	&&\label{Taylor2} f'-f=(\na f)(v')\cdot(v'-v)-\int_0^1 \kappa(\na^2 f)(v(\kappa)):(v'-v)\otimes(v'-v)\mathrm{d}\kappa.
	\een 
	Thanks to the symmetry property of $\sigma$-integral,  we have
	\ben \label{cancell1} &&\int B(|v-v_*|, \f{v-v_*}{|v-v_*|}\cdot\sigma) (v'-v) \mathrm{d}\sigma = \int B(|v-v_*|, \f{v-v_*}{|v-v_*|}\cdot\sigma)  \sin^{2}\f{\theta}{2} (v_* - v)  \mathrm{d}\sigma, 
	\\
	\label{cancell2} &&\int  B(|v-v_*|,\f{v-v_*}{|v-v_*|} \cdot \sigma)  (v'-v) h(v') \mathrm{d}\sigma \mathrm{d}v =0. \een
	Here,  the formula \eqref{cancell1} holds for fixed $v, v_*$ and \eqref{cancell2} holds for fixed $v_*$. 

	
	We now recall a useful formula in the following lemma on the change of variables $v \to v(\kappa)$ and $v_{*} \to v_{*}(\iota)$
	where for $\kappa, \iota \in[0,1]$, 
	\ben\label{Defkappav} v(\kappa)=\kappa v'+(1-\kappa)v,  \quad 
	v_{*}(\iota)=\iota v'_*+(1-\iota)v_*. \een 
		\begin{lem}\label{usual-change} For $a \in [0, 2]$, let us define
		\ben \label{general-Jacobean}
		\psi_{a}(\theta) \colonequals  (\cos^{2}\frac{\theta}{2}+(1-a)^{2}\sin^{2}\frac{\theta}{2})^{-1/2}.
		\een
		For any $0 \leq \kappa, \iota \leq 1$, it holds that
		\ben \label{change-of-variable-2}
		&& \int_{\R^{3}} \int_{\R^{3}} \int_{\SS^{2}_{+}}  B(|v-v_{*}|, \cos\theta)
		g(\iota(v_{*})) f(v(\kappa))  \mathrm{d}v \mathrm{d}v_*  \mathrm{d}\sigma 
		\\ \nonumber &=& \int_{\R^{3}} \int_{\R^{3}} \int_{\SS^{2}_{+}} B(|v-v_{*}|\psi_{\kappa+\iota}(\theta), \cos\theta)
		g(v_{*}) f(v) \psi_{\kappa+\iota}^{3}(\theta) \mathrm{d}v \mathrm{d}v_*  \mathrm{d}\sigma .
		\een
		Here $\SS^{2}_{+}$ stands for $(v-v_{*}) \cdot \sigma \geq 0$.
	\end{lem}
	
	Before giving the proof of this lemma, we firstly note that the above formula is  general as it  simultaneously deals with the two changes $v \to v(\kappa)$ and $v_{*} \to v_{*}(\iota)$. This will be used  in the  proof of Proposition \ref{limit-Botltzmann-to-Landau}.
	If $\kappa = \iota =0$, then $\psi_{\kappa + \iota}(\theta) = \psi_{0}(\theta) =1$
	and it corresponds to the identity transformation. If $\kappa = \iota =1$, then $\psi_{\kappa + \iota}(\theta) = \psi_{2}(\theta) =1$
	and it corresponds to the change of velocities for pre-post collision: $(v, v_{*}, \sigma) \rightarrow (v^{\prime}, v_{*}^{\prime}, \sigma'=(v-v_{*})/|v-v_{*}|)$. If $\kappa =1, \iota = 0$ or $\kappa =0, \iota =1$, then $\psi_{\kappa + \iota}(\theta) = \psi_{1}(\theta) =\cos^{-1}\frac{\theta}{2}$
	and it corresponds to $v \to v'$ or $v_* \to v'_*$
	respectively. This is consistent with the cancellation lemma given in \cite{alexandre2000entropy}. If $\iota = 0$ or $\kappa =0$, then it corresponds to the individual change $v \to v(\kappa)$ or $v_* \to v_*(\iota)$ respectively.
	
	Note that $1 \leq \psi_{a}(\theta) \leq \sqrt{2}$
	for $a \in [0, 2], \theta \in [0, \pi/2]$.  Thanks to Lemma \ref{usual-change} and $1 \leq \psi_{a}(\theta) \leq \sqrt{2}$, considering the kernel \eqref{kernel-studied},
	we can skip the details regarding the above mentioned change of variables. As a result,
	in most part of this paper, 
	$v(\kappa)$ and $v_{*}(\iota)$ will be replaced by $v$ and $v_{*}$ respectively at the cost of a multiplicative constant.
	
	\begin{proof}[Proof of Lemma \ref{usual-change}] 
The case  $\kappa = \iota =1$ is obviously given by the standard change of variable $(v, v_*, \sigma) \to (v', v'_*, \sigma')$ where $\sigma' = (v - v_*)/|v - v_*|$. This change has unit Jacobian.

Now we deal with the case  $\kappa + \iota < 2$. Recalling  \eqref{Defkappav}, it is direct to check
\ben \label{relative-relation}
|v-v_{*}| =  |\kappa(v) - \iota(v_*)| \psi_{\kappa+\iota}(\theta).
 \een 
Let $\beta$ be the angle between $\kappa(v) - \iota(v_*)$ and $\sigma$, then $
\cos \beta = \varphi_{\kappa+\iota}(\sin\frac{\theta}{2})$ where 
\beno
\varphi_{a}(x) :=
\frac{ 1 - x^{2}+(a-1)x^2}
{\left( 1 - x^{2} +(1-a)^{2}x^{2} \right)^{1/2} }.
\eeno
Let $\delta_{a} \colonequals  
\arccos( \f{\sqrt{2}} {2} \frac{ a }
{\sqrt{1 + (1-a)^{2}  }  })$.
If  $\kappa + \iota < 2$, then $\delta_{\kappa+\iota}>0$ 
and
the function:
$
\theta \in [0, \f{\pi}{2}]  \to 
\beta_{\kappa+\iota} \in [0,\delta_{\kappa + \iota}] 
$
is a bijection.   It holds that
		\ben \label{Jacobean}
		\det (\frac{\partial (\kappa(v), \iota(v_*))}{\partial (v, v_*)}) =  \alpha_{\kappa+\iota}(\theta),
		\een
		where for $0 \leq a \leq 2$,
			\ben \label{def-alpha}
\alpha_{a}(\theta)	\colonequals  (1-\frac{a}{2})^2 \left( (1-\frac{a}{2})+\frac{a}{2} \cos\theta \right) .
		\een  
By \eqref{relative-relation} and \eqref{Jacobean}, with $\mathrm{d}\sigma = \sin \beta  
 \mathrm{d}\beta \mathrm{d}\mathbb{S}$, we have
			\beno 
		&& \int_{\R^{3}} \int_{\R^{3}} \int_{\SS^{2}_{+}}  B(|v-v_{*}|, \cos\theta)
		g(\iota(v_{*})) f(v(\kappa))  \mathrm{d}v \mathrm{d}v_*  \mathrm{d}\sigma 
		\\ \nonumber &=& 2 \pi \int_{\R^{3}} \int_{\R^{3}} \int_{0}^{\delta_{\kappa+\iota}} B(|v-v_{*}|\psi_{\kappa+\iota}(\theta), \cos\theta)
		g(v_{*}) f(v)  \alpha_{\kappa+\iota}^{-1}(\theta) \sin \beta  
		\mathrm{d}v \mathrm{d}v_* \mathrm{d}\beta.
		\eeno
It is directly to check that
\beno
\alpha_{\kappa+\iota}^{-1}(\theta) \sin \beta
\mathrm{d}\beta = - \alpha_{\kappa+\iota}^{-1}(\theta) \mathrm{d} \cos \beta
=  -\f{1}{4}  \varphi^{\prime}_{\kappa+\iota}(\sin\f{\theta}{2}) \sin^{-1}\f{\theta}{2} \alpha_{\kappa+\iota}^{-1}(\theta) \sin \theta
\mathrm{d}\theta =   \psi_{\kappa+\iota}^{3}(\theta)
\sin \theta
\mathrm{d}\theta.
\eeno
Then we go back from $\beta$ to $\theta$ and use $\mathrm{d}\sigma = \sin \theta  
\mathrm{d}\theta \mathrm{d}\mathbb{S}$
to get \eqref{change-of-variable-2}.
	\end{proof}

	\subsection{Upper bound of $Q^{s,\gamma}_{\eta}$}
	We give the  upper bound  of $ Q^{s,\gamma}_{\eta}$ in the following proposition.

	\begin{prop}\label{ubqepsilon-singular} 
		Let $0<s<1, -2s-3<\gamma \leq 0, 0<\eta \leq 1$. Let $l_1, l_2, l_3 \in \R$ satisfying $l_1 + l_2 + l_3 = 0$.
		For any fixed small $1/2 \geq \delta>0$, for any combination $a_1, a_2, a_3 \geq 0, a_1 + a_2 \geq s, a_1 + a_3 \geq 2s, a_2 + a_3 \geq 2s$ satisfying the constraint
		$a_1 + a_2 + a_3 = 2s+\f32 + \delta$, we have
		\beno
		|\langle Q^{s,\gamma}_{\eta}(g,h), f\rangle| \lesssim C_{\delta, s,\gamma,\eta}  |g|_{H^{a_{1}}_{l_1}} |h|_{H^{a_{2}}_{l_2}}
		|f|_{H^{a_3}_{l_3}},
		\eeno
		where 
		\ben \label{constant-C-delta-s-gamma-eta}
		C_{\delta,s,\gamma,\eta} := \delta^{-\f12} C_{s,\gamma,\eta}, \quad
		C_{s,\gamma,\eta} := \f{1}{s} \frac{\eta^{\gamma+2s+3}}{\gamma+2s+3}.
		\een
		We point out that the constant associated to $\lesssim$ in the above inequality
		depends only on the upper bound of  $|l_1|, |l_2|, |l_3|$.
	\end{prop}

	\begin{proof}  Recalling the decomposition in frequency space, we have
		\beno
		\langle Q^{s,\gamma}_{\eta}(g,h), f\rangle
		&=& \sum_{j \leq k-N_{0}} \langle Q^{s,\gamma}_{\eta}(\mathfrak{F}_{j} g, \mathfrak{F}_{k} h), \tilde{\mathfrak{F}}_{k} f\rangle
		+\sum_{|j-k| < N_{0}}\sum_{l \leq k+N_{0}} \langle Q^{s,\gamma}_{\eta}(\mathfrak{F}_{j} g, \mathfrak{F}_{k}h), \mathfrak{F}_{l}f\rangle
		\\&&+ \sum_{j \geq k+N_{0}} \langle Q^{s,\gamma}_{\eta}(\mathfrak{F}_{j} g, \mathfrak{F}_{k}h), \tilde{\mathfrak{F}}_{j}f\rangle.
		\eeno
		We estimate the second quantity in details for illustration. That is, when $|j-k| < N_{0}, l \leq k+N_{0}$. Note that
		\beno
		\langle Q^{s,\gamma}_{\eta}(\mathfrak{F}_{j} g, \mathfrak{F}_{k}h), \mathfrak{F}_{l}f\rangle = \int B^{s,\gamma}_{\eta} (\mathfrak{F}_{j} g)_{*} \mathfrak{F}_{k}h ((\mathfrak{F}_{l}f)^{\prime}-\mathfrak{F}_{l}f) \mathrm{d}V,
		\eeno
		where for brevity of notiation, $\mathrm{d}V =  \mathrm{d}v \mathrm{d}v_{*} \mathrm{d}\sigma$.
		Motivated by \eqref{cancell2}, we write
		\beno
		\langle Q^{s,\gamma}_{\eta}(\mathfrak{F}_{j} g, \mathfrak{F}_{k}h), \mathfrak{F}_{l}f\rangle &=& 
		\int B^{s,\gamma}_{\eta} (\mathfrak{F}_{j} g)_{*} (\mathfrak{F}_{k}h)' ((\mathfrak{F}_{l}f)^{\prime}-\mathfrak{F}_{l}f) \mathrm{d}V
		\\&&+  \int B^{s,\gamma}_{\eta} (\mathfrak{F}_{j} g)_{*} (\mathfrak{F}_{k}h - (\mathfrak{F}_{k}h)') ((\mathfrak{F}_{l}f)^{\prime}-\mathfrak{F}_{l}f) \mathrm{d}V := \mathcal{I}_1 + \mathcal{I}_2.
		\eeno
		
		For $\mathcal{I}_1$,
		let $E:=\sin\frac{\theta}{2} \leq 2^{-l}|v-v_{*}|^{-1}\wedge \sqrt{2}/2$ and
		write
		\beno
		\mathcal{I}_1 = \mathcal{I}_{1, \leq}(\mathfrak{F}_{j} g,\mathfrak{F}_{k}h,\mathfrak{F}_{l}f) + \mathcal{I}_{1, \geq}(\mathfrak{F}_{j} g,\mathfrak{F}_{k}h,\mathfrak{F}_{l}f),
		\\
		\mathcal{I}_{1, \leq}(\mathfrak{F}_{j} g,\mathfrak{F}_{k}h,\mathfrak{F}_{l}f) := \int B^{s,\gamma}_{\eta}\mathrm{1}_{E} (\mathfrak{F}_{j} g)_{*} (\mathfrak{F}_{k}h)' ((\mathfrak{F}_{l}f)^{\prime}-\mathfrak{F}_{l}f) \mathrm{d}V,
		\\
		\mathcal{I}_{1, \geq}(\mathfrak{F}_{j} g,\mathfrak{F}_{k}h,\mathfrak{F}_{l}f) := \int B^{s,\gamma}_{\eta}\mathrm{1}_{E^{c}} (\mathfrak{F}_{j} g)_{*} (\mathfrak{F}_{k}h)'
		((\mathfrak{F}_{l}f)^{\prime}-\mathfrak{F}_{l}f)
		\mathrm{d}V.
		\eeno
		For term $\mathcal{I}_{1, \leq}(\mathfrak{F}_{j} g,\mathfrak{F}_{k}h,\mathfrak{F}_{l}f)$, we apply
		\eqref{Taylor2} and \eqref{cancell2} to get 
		\beno
		\mathcal{I}_{1, \leq}(\mathfrak{F}_{j} g,\mathfrak{F}_{k}h,\mathfrak{F}_{l}f) = \int B^{s,\gamma}_{\eta}\mathrm{1}_{E} (\mathfrak{F}_{j} g)_{*} (\mathfrak{F}_{k}h)' \left(\int_0^1 \kappa (\na^2 \mathfrak{F}_{l}f )(v(\kappa)):(v'-v)\otimes(v'-v)\mathrm{d}\kappa \right) \mathrm{d}V.
		\eeno
		By using $|\nabla^{2} \mathfrak{F}_{l}f|_{L^{\infty}} \lesssim 2^{\frac{7}{2}l}|\mathfrak{F}_{l}f|_{L^{2}}$,  the change of variable $v \to v(\kappa)$, 
		and $\int \sin^{2}\frac{\theta}{2} b^{s}(\theta)\mathrm{1}_{E} \mathrm{d}\sigma \lesssim 2^{2sl-2l}|v-v_{*}|^{2s-2}$, we have 
		\beno
		|\mathcal{I}_{1, \leq}(\mathfrak{F}_{j} g,\mathfrak{F}_{k}h,\mathfrak{F}_{l}f)| \lesssim
		2^{2sl} 2^{\frac{3}{2}l}|\mathfrak{F}_{l}f|_{L^{2}}  \int 
		|v-v_{*}|^{\gamma+2s} \mathrm{1}_{|v-v_{*}|\leq 4\eta/3}
		(\mathfrak{F}_{j} g)_{*} (\mathfrak{F}_{k}h)  \mathrm{d}v \mathrm{d}v_{*}.
		\eeno
		By using the fact that
		\ben \label{key-reason-near-origin}
		| |\cdot|^{\gamma+2s} \mathrm{1}_{|\cdot| \leq 4\eta/3}|_{L^{1}} \lesssim \frac{\eta^{\gamma+2s+3}}{\gamma+2s+3},
		\een
		we get
		\beno
		|\mathcal{I}_{1, \leq}(\mathfrak{F}_{j} g,\mathfrak{F}_{k}h,\mathfrak{F}_{l}f)| \lesssim
		\frac{\eta^{\gamma+2s+3}}{\gamma+2s+3}2^{2sl+\frac{3}{2}l}|\mathfrak{F}_{j} g|_{L^{2}}|\mathfrak{F}_{k}h|_{L^{2}}|\mathfrak{F}_{l}f|_{L^{2}}.
		\eeno

		Since $\int b^{s}(\theta)\mathrm{1}_{E^{c}} \mathrm{d}\sigma \lesssim s^{-1}2^{2sl}|v-v_{*}|^{2s}$, then
		\beno
		|\mathcal{I}_{1, \geq}(\mathfrak{F}_{j} g,\mathfrak{F}_{k}h,\mathfrak{F}_{l}f)| \lesssim
		C_{s,\gamma,\eta} 
		2^{2sl}|\mathfrak{F}_{j} g|_{L^{2}}|\mathfrak{F}_{k}h|_{L^{2}}|\mathfrak{F}_{l}f|_{L^{\infty}} \lesssim C_{s,\gamma,\eta}2^{2sl+\frac{3}{2}l}|\mathfrak{F}_{j} g|_{L^{2}}|\mathfrak{F}_{k}h|_{L^{2}}|\mathfrak{F}_{l}f|_{L^{2}}.
		\eeno
		Combining these estimates on $\mathcal{I}_{1}$, we get
		\beno
		|\mathcal{I}_{1}| \lesssim C_{s,\gamma,\eta}2^{2sl+\frac{3}{2}l}|\mathfrak{F}_{j} g|_{L^{2}}|\mathfrak{F}_{k}h|_{L^{2}}|\mathfrak{F}_{l}f|_{L^{2}}.
		\eeno
		
		Now we estimate
		$\mathcal{I}_{2}$. Let $F:=\sin\frac{\theta}{2} \leq 2^{-l/2-k/2}|v-v_{*}|^{-1}\wedge \sqrt{2}/2$.
		We write
		\beno
		\mathcal{I}_2 = \mathcal{I}_{2, \leq}(\mathfrak{F}_{j} g,\mathfrak{F}_{k}h,\mathfrak{F}_{l}f) + \mathcal{I}_{1, \geq}(\mathfrak{F}_{j} g,\mathfrak{F}_{k}h,\mathfrak{F}_{l}f),
		\\
		\mathcal{I}_{2, \leq}(\mathfrak{F}_{j} g,\mathfrak{F}_{k}h,\mathfrak{F}_{l}f) := - \int B^{s,\gamma}_{\eta}\mathrm{1}_{F} (\mathfrak{F}_{j} g)_{*} ((\mathfrak{F}_{k}h)' -\mathfrak{F}_{k}h) ((\mathfrak{F}_{l}f)^{\prime}-\mathfrak{F}_{l}f) \mathrm{d}V,
		\\
		\mathcal{I}_{2, \geq}(\mathfrak{F}_{j} g,\mathfrak{F}_{k}h,\mathfrak{F}_{l}f) := - \int B^{s,\gamma}_{\eta}\mathrm{1}_{F^{c}} (\mathfrak{F}_{j} g)_{*} ((\mathfrak{F}_{k}h)' -\mathfrak{F}_{k}h)
		((\mathfrak{F}_{l}f)^{\prime}-\mathfrak{F}_{l}f)
		\mathrm{d}V.
		\eeno
		By the 1st-order Taylor expansion \eqref{Taylor1-order-1}, using $|\nabla \mathfrak{F}_{l}f|_{L^{\infty}} \lesssim 2^{\frac{5}{2}l}|\mathfrak{F}_{l}f|_{L^{2}}$, the change of variable in Lemma \ref{general-Jacobean},
		and  $\int \sin^{2}\frac{\theta}{2} b^{s}(\theta)\mathrm{1}_{F} \mathrm{d}\sigma \lesssim 2^{sl+sk-l-k}|v-v_{*}|^{2s-2}$, we get
		\beno
		|\mathcal{I}_{2, \leq}(\mathfrak{F}_{j} g,\mathfrak{F}_{k}h,\mathfrak{F}_{l}f)| \lesssim
		2^{sl +sk -k} 2^{\frac{3}{2}l}|\mathfrak{F}_{l}f|_{L^{2}}  \int 
		|v-v_{*}|^{\gamma+2s} \mathrm{1}_{|v-v_{*}|\leq 4\eta/3}
		(\mathfrak{F}_{j} g)_{*} (\na \mathfrak{F}_{k}h)  \mathrm{d}v \mathrm{d}v_{*}.
		\eeno
		Then \eqref{key-reason-near-origin} implies
		\beno
		|\mathcal{I}_{2, \leq}(\mathfrak{F}_{j} g,\mathfrak{F}_{k}h,\mathfrak{F}_{l}f)| \lesssim
		\frac{\eta^{\gamma+2s+3}}{\gamma+2s+3}2^{sl+sk-k+\frac{3}{2}l}|\mathfrak{F}_{j} g|_{L^{2}}|\na \mathfrak{F}_{k}h|_{L^{2}}|\mathfrak{F}_{l}f|_{L^{2}}
		\\ \lesssim
		\frac{\eta^{\gamma+2s+3}}{\gamma+2s+3}2^{sl+sk+\frac{3}{2}l}|\mathfrak{F}_{j} g|_{L^{2}}|\mathfrak{F}_{k}h|_{L^{2}}|\mathfrak{F}_{l}f|_{L^{2}}.
		\eeno
		Since $\int b^{s}(\theta)\mathrm{1}_{F^{c}} \mathrm{d}\sigma \lesssim s^{-1} 2^{sl+sk}|v-v_{*}|^{2s}$, then
		\beno
		|\mathcal{I}_{2, \geq}(\mathfrak{F}_{j} g,\mathfrak{F}_{k}h,\mathfrak{F}_{l}f)| \lesssim
		C_{s,\gamma,\eta}2^{sl+sk}|\mathfrak{F}_{j} g|_{L^{2}}|\mathfrak{F}_{k}h|_{L^{2}}|\mathfrak{F}_{l}f|_{L^{\infty}} \lesssim C_{s,\gamma,\eta}2^{sl+sk+\frac{3}{2}l}|\mathfrak{F}_{j} g|_{L^{2}}|\mathfrak{F}_{k}h|_{L^{2}}|\mathfrak{F}_{l}f|_{L^{2}}.
		\eeno
		Combining these estimates on $\mathcal{I}_{2}$, we have
		\beno
		|\mathcal{I}_{2}| \lesssim C_{s,\gamma,\eta}2^{sl+sk+\frac{3}{2}l}|\mathfrak{F}_{j} g|_{L^{2}}|\mathfrak{F}_{k}h|_{L^{2}}|\mathfrak{F}_{l}f|_{L^{2}}.
		\eeno
		
		Therefore, when  $l \leq k+ N_0$, we obtain
		\ben \label{term-type-1}
		|\langle Q^{s,\gamma}_{\eta}(\mathfrak{F}_{j} g, \mathfrak{F}_{k}h), \mathfrak{F}_{l}f\rangle| \lesssim C_{s,\gamma,\eta}2^{sl+sk+\frac{3}{2}l}|\mathfrak{F}_{j} g|_{L^{2}}|\mathfrak{F}_{k}h|_{L^{2}}|\mathfrak{F}_{l}f|_{L^{2}}.
		\een
		Finally, to
		\beno
		\sum_{|j-k| < N_{0}}\sum_{l \leq k+N_{0}} \langle Q^{s,\gamma}_{\eta}(\mathfrak{F}_{j} g, \mathfrak{F}_{k}h), \mathfrak{F}_{l}f \rangle,
		\eeno
		let $a+b=s+\f32$ for $a \geq 0$.
		For any fixed $k$, the sum over $-1 \leq l \leq k+N_{0}$ can be estimated by using Cauchy-Schwarz inequality as
		\ben \label{Cauchy-Schwarz-directly}
		\sum_{l \leq k+N_{0}} 2^{sl+\frac{3}{2}l} |\mathfrak{F}_{l}f|_{L^{2}}
		\leq \left(\sum_{l \leq k+N_{0}} 2^{2al} |\mathfrak{F}_{l}f|_{L^{2}}^2\right)^{\f12} \left(\sum_{l \leq k+N_{0}} 2^{2bl} \right)^{\f12}
		\lesssim |f|_{H^{a}} C_{b, k},
		\een
		where 
		for $b \neq 0$,
		\beno
		C_{b, k}^2 = \f{2^{2b(k+N_0+1)} - 2^{-2b}}{2^{2b} -1}.
		\eeno
	For $b$ close to $0$,  by allowing an extra $\delta$-order regularity,
		we conclude that
		\ben \label{sum-2-final-result}
		\sum_{|j-k| < N_{0}} \sum_{l \leq k+N_{0}} \langle Q^{s,\gamma}_{\eta}(\mathfrak{F}_{j} g, \mathfrak{F}_{k}h), \mathfrak{F}_{l}f \rangle 
		\lesssim C_{\delta, s,\gamma,\eta} |g|_{H^{a_1}} |h|_{H^{a_2}} |f|_{H^{a_3}},
		\een
		where $a_1, a_2, a_3 \geq 0, a_1 + a_2 \geq s$ satisfying the constraint
		$a_1 + a_2 + a_3 = 2s+\f32 + \delta$ for any fixed small $\delta>0$. Indeed, recalling \eqref{Cauchy-Schwarz-directly}, in which we can take $a+b = s + \f32   + \delta$ and $b \geq \f{\delta}{2}$, then $	C_{b, k} \lesssim \delta^{-1/2} 2^{k \delta/2}$. Since $|j-k|<N_0$,
		 we get \eqref{sum-2-final-result} for $a_1 + a_2 \geq s + \f{\delta}{2}$. 
		In \eqref{Cauchy-Schwarz-directly}
		we can also take $ = s + \f32 + \f{\delta}{2}, b = - \f{\delta}{2}$ and so  $	C_{b, k} \lesssim \delta^{-1/2}$, then we get \eqref{sum-2-final-result} for $a_1 + a_2 = s,
		a_3 = s + \f32 + \f{\delta}{2}$.  This obviously implies that \eqref{sum-2-final-result} holds for $s \le a_1 + a_2 
		\leq  s + \f{\delta}{2}$.

		Similar argument can be  applied to $\langle Q^{s,\gamma}_{\eta}(\mathfrak{F}_{j} g, \mathfrak{F}_{k} h), \tilde{\mathfrak{F}}_{k} f\rangle$ for $j \leq k-N_{0}$
		to obtain
		\ben \label{term-type-2}
		|\langle Q^{s,\gamma}_{\eta}(\mathfrak{F}_{j} g, \mathfrak{F}_{k} h), \tilde{\mathfrak{F}}_{k} f\rangle| \lesssim C_{s,\gamma,\eta}2^{2sk+\frac{3}{2}j}|\mathfrak{F}_{j} g|_{L^{2}}|\mathfrak{F}_{k}h|_{L^{2}}|\tilde{\mathfrak{F}}_{k} f|_{L^{2}}.
		\een
		Here we take $L^{\infty}$ on $\mathfrak{F}_{j} g$ so that there is a factor $2^{\frac{3}{2}j}$. We also apply Taylor expansions to $\mathfrak{F}_{k} h$ and $\tilde{\mathfrak{F}}_{k} f$ to obtain the factor $2^{2sk}$.
		Let $0<\delta \leq \f12$, for any combination $a_1, a_2, a_3 \geq 0, a_2 + a_3 \geq 2s$ satisfying the constraint
		$a_1 + a_2 + a_3 = 2s+\f32 + \delta$, it holds that
		\beno
		|\sum_{j \leq k-N_{0}} \langle Q^{s,\gamma}_{\eta}(\mathfrak{F}_{j} g, \mathfrak{F}_{k} h), \tilde{\mathfrak{F}}_{k} f\rangle|
		\lesssim \delta^{-1/2}  C_{s,\gamma,\eta}|g|_{H^{a_{1}}}|h|_{H^{a_{2}}}|f|_{H^{a_3}}.
		\eeno
		
		Again similar argument can be  applied to $\langle Q^{s,\gamma}_{\eta}(\mathfrak{F}_{j} g, \mathfrak{F}_{k}h), \tilde{\mathfrak{F}}_{j}f\rangle$ for $j \geq k+N_{0}$
		to get
		\ben \label{term-type-3}
		|\langle Q^{s,\gamma}_{\eta}(\mathfrak{F}_{j} g, \mathfrak{F}_{k}h), \tilde{\mathfrak{F}}_{j}f\rangle| \lesssim C_{s,\gamma,\eta}2^{2sj+\frac{3}{2}k}|\mathfrak{F}_{j} g|_{L^{2}}|\mathfrak{F}_{k}h|_{L^{2}}|\tilde{\mathfrak{F}}_{j} f|_{L^{2}}.
		\een
		Here we also apply Taylor expansions to $\mathfrak{F}_{k}h$ and $\tilde{\mathfrak{F}}_{j}f$ to get the factor $2^{2sj}$ or $2^{sk + sj} \leq 2^{2sj}$.
		Here we take $L^{\infty}$ on $\mathfrak{F}_{k}h$ or $\na \mathfrak{F}_{k}h$ so that
		 at the end there is a factor $2^{\frac{3}{2}k}$.  Let $0<\delta \leq \f12$, for any combination $a_1, a_2, a_3 \geq 0, a_1 + a_3 \geq 2s$ satisfying the constraint
		$a_1 + a_2 + a_3 = 2s+\f32 + \delta$, it holds that
		\beno
		|\sum_{j \geq k+N_{0}} \langle Q^{s,\gamma}_{\eta}(\mathfrak{F}_{j} g, \mathfrak{F}_{k}h), \tilde{\mathfrak{F}}_{j}f\rangle|
		\lesssim \delta^{-1/2}  C_{s,\gamma,\eta}|g|_{H^{a_{1}}}|h|_{H^{a_{2}}}|f|_{H^{a_3}}.
		\eeno
		In summary, we prove the desired estimate with $l_1 = l_2 =l_3 =0$.

		By \eqref{decomposition-phase-near-0} and \eqref{h-m-l-norm}, we can freely transfer weight among $g, h, f$ so that the estimate in the proposition holds for  $l_1 + l_2 +l_3 =0$.
	\end{proof}
	
	Based on the proof of the above proposition, we will derive another
	 version of cancellation lemma introduced in \cite{alexandre2000entropy}.
	The idea is to 
	 gain $|v-v_{*}|^{2s}$  at the price of $2s$-order derivatives on the functions. 
	
	\begin{prop} \label{new-version-cancellation}
		Let $0<s<1, \gamma+2s+3>0$.
		Let $a_1, a_2, l_1, l_2 \in \R$ satisfying
		$a_1 + a_2 =2s, l_1 + l_2 =0$, then
		\beno
		|\int B^{s,\gamma}_{\eta} g_{*}(f^{\prime}-f)\mathrm{d}V| \lesssim  C_{s,\gamma,\eta} |g|_{H^{a_1}_{l_1}}|f|_{H^{a_2}_{l_2}}.
		\eeno
	\end{prop}
	\begin{proof} 
		Recalling the decomposition in frequency space, we have
		\beno
\int B^{s,\gamma}_{\eta} g_{*}(f^{\prime}-f)\mathrm{d}V =		\langle Q^{s,\gamma}_{\eta}(g, 1), f\rangle
		=
		\sum_{|j-l| < N_{0}} \langle Q^{s,\gamma}_{\eta}(\mathfrak{F}_{j} g, 1), \mathfrak{F}_{l}f\rangle.
		\eeno
		This is because frequency of $g$ and $f$ lies in the same region when $h=1$.
		Following the estimate on $\mathcal{I}_1$ in the  proof of Proposition \ref{ubqepsilon-singular}, we have
		\beno
		|\langle Q^{s,\gamma}_{\eta}(\mathfrak{F}_{j} g, 1), \mathfrak{F}_{l}f\rangle| \lesssim C_{s,\gamma,\eta}2^{2sl}|\mathfrak{F}_{j} g|_{L^{2}}|\mathfrak{F}_{l}f|_{L^{2}}.
		\eeno
		Since $|j-l| < N_{0}$, by the Cauchy-Schwarz inequality,
		we can estimate the sum as
		\beno
		\sum_{|j-l| < N_{0}} 2^{2sl}|\mathfrak{F}_{j} g|_{L^{2}}|\mathfrak{F}_{l}f|_{L^{2}} \lesssim 
		|g|_{H^{a_1}}|f|_{H^{a_2}}.
		\eeno
		By \eqref{decomposition-phase-near-0}, we can freely transfer weight among $g, f$ so that the proof of the proposition is completed.
	\end{proof}

	\subsection{Upper bound of $\langle I^{s,\gamma}_{\eta}(g,h;\beta), f\rangle$} 
		In this subsection, we will estimate the upper bound of $\langle I^{s,\gamma}_{\eta}(g,h;\beta), f\rangle$ where $I^{s,\gamma}_{\eta}(g,h;\beta)$ is defined by 
	\eqref{I-ep-ga-geq-eta-beta} with 	$B^{s,\gamma}_{\eta}$.

	\begin{prop} \label{I-less-eta-upper-bound} 
		Let $l\geq 0$. Let  $(a_{1},a_{2})=(\f{3}{2} + \delta, s)$ or $(0, \f{3}{2} + \delta)$. Then it holds 
		\beno 
		\langle I^{s,\gamma}_{\eta}(g,h;\beta), f\rangle  \lesssim_{l} C_{\delta, s,\gamma,\eta} |g|_{H^{a_{1}}_{-l}}|h|_{H^{a_{2}}_{-l}}
		|f|_{H^{s}_{-l}}. 
		\eeno
	\end{prop}
	\begin{proof} We only need to consider the case when $\beta=0$ because
		 the following argument  also holds when we replace $\mu^{1/2}$  by  $P_{\beta}\mu^{1/2}$.
		Recall
		\ben \label{I-s-ga-eta-ghf}
		\langle I^{s,\gamma}_{\eta}(g,h), f\rangle =  \int B^{s,\gamma}_{\eta}((\mu^{1/2})_{*}^{\prime} - \mu_{*}^{1/2})g_{*} h f^{\prime} \mathrm{d}V  =  \mathcal{I}_{1} + \mathcal{I}_{2},
		\een
		where
		\beno
		\mathcal{I}_{1}:= \int B^{s,\gamma}_{\eta}((\mu^{1/2})_{*}^{\prime} - \mu_{*}^{1/2})g_{*} h (f^{\prime} - f)\mathrm{d}V, \quad
		\mathcal{I}_{2}:= \int B^{s,\gamma}_{\eta}((\mu^{1/2})_{*}^{\prime} - \mu_{*}^{1/2})g_{*} h f \mathrm{d}V.
		\eeno
		
		Firstly, for  $\mathcal{I}_{1}$,
		by the Cauchy-Schwarz inequality, we have
	$
		|\mathcal{I}_{1}| \leq \mathcal{I}_{1,1}^{1/2}\mathcal{I}_{1,2}^{1/2}
	$,
		where
		\beno
		\mathcal{I}_{1,1}:= \int B^{s,\gamma}_{\eta}((\mu^{1/4})_{*}^{\prime} + \mu_{*}^{1/4})^{2} (f^{\prime} - f)^{2}\mathrm{d}V, \quad 
		\mathcal{I}_{1,2}:=\int B^{s,\gamma}_{\eta}((\mu^{1/4})_{*}^{\prime} - \mu_{*}^{1/4})^{2} g_{*}^{2} h^{2} \mathrm{d}V.
		\eeno
		Using $((\mu^{1/4})_{*}^{\prime} + \mu_{*}^{1/4})^{2} \leq 2 ((\mu^{1/2})_{*}^{\prime} + \mu_{*}^{1/2})$, by the   change of variable $(v, v_{*}, \sigma) \rightarrow (v^{\prime}, v_{*}^{\prime}, \sigma')$,
		we have
		\beno
		\mathcal{I}_{1,1} \leq 4 \int B^{s,\gamma}_{\eta} \mu_{*}^{1/2} (f^{\prime} - f)^{2}\mathrm{d}V = 4 \mathcal{N}^{s,\gamma}_{\eta}(\mu^{1/4},f),
		\eeno
		where
		\ben  \label{definition-of-N-s-ga-eta-leq}
		\mathcal{N}^{s, \gamma}_{\eta}(g,h) := \int B^{s, \gamma}_{\eta}
		g^{2}_{*} (h^{\prime}-h)^{2} \mathrm{d}V.\een
	Using $(f^{\prime} - f)^{2} = (f^{2})^{\prime} - f^{2} - 2f(f^{\prime}-f)$, we get
		\ben \label{functional-N-decomposition}
		\mathcal{N}^{s,\gamma}_{\eta}(\mu^{1/4},f) = \int B^{s,\gamma}_{\eta} \mu_{*}^{1/2} ((f^{2})^{\prime} - f^{2})\mathrm{d}V - 2 \langle Q^{s,\gamma}_{\eta}(\mu^{1/2},f),f \rangle.
		\een
		By Taylor expansion of $(\mu^{1/2})'_* - (\mu^{1/2})_*$, we have
		\beno
		|\int B^{s,\gamma}_{\eta} \mu_{*}^{1/2} ((f^{2})^{\prime} - f^{2})\mathrm{d}V| &=& |\int B^{s,\gamma}_{\eta} ((\mu^{1/2})'_* - (\mu^{1/2})_*)  f^{2} \mathrm{d}V|
		\\ &\lesssim& \int \mathrm{1}_{|v-v_{*}| \leq 4\eta/3} |v-v_{*}|^{\gamma+2} \mu^{\f14}f^{2} \mathrm{d}v_{*} \mathrm{d}v
		\lesssim \frac{\eta^{\gamma+5}}{\gamma+5} |\mu^{\f18}f|_{L^{2}}^{2}.
		\eeno
		By Proposition \ref{ubqepsilon-singular}, we have
		\beno
		|\langle Q^{s,\gamma}_{\eta}(\mu^{1/2},f),f \rangle|
		\lesssim_l C_{s,\gamma,\eta} |f|_{H^{s}_{-l}}^{2}.
		\eeno
		Since $\gamma+5 \geq \gamma+2s+3$,
		\ben \label{functional-N-lower-eta}
		\mathcal{I}_{1,1} \lesssim \mathcal{N}^{s,\gamma}_{\eta}(\mu^{1/4},f) \lesssim_l C_{s,\gamma,\eta} |f|_{H^{s}_{-l}}^{2}.
		\een
		It is straightforward to have
		\beno
		\mathcal{I}_{1,2} \lesssim \int \mathrm{1}_{|v-v_{*}| \leq  4 \eta/3}|v-v_{*}|^{\gamma+2} \mu^{1/8} \mu^{1/8}_{*} g_{*}^{2} h^{2} \mathrm{d}v_{*} \mathrm{d}v \lesssim \delta^{-1} \frac{\eta^{\gamma+5}}{\gamma+5} |\mu^{1/16}g|_{H^{a_{1}}}^{2}|\mu^{1/16}h|_{H^{a_{2}}}^{2},
		\eeno
		where $a_1 + a_2 = \f{3}{2} + \delta$.
		Combining the estimates on  $\mathcal{I}_{1,1}$ and $\mathcal{I}_{1,2}$ gives
		\beno
		\mathcal{I}_{1} \lesssim_{l} C_{\delta, s,\gamma,\eta} |\mu^{1/16}g|_{H^{a_{1}}}|\mu^{1/16}h|_{H^{a_{2}}}|f|_{H^{s}_{-l}}.
		\eeno

		We now turn to estimate $\mathcal{I}_{2}$. 
		By  Prop. \ref{ubqepsilon-singular}, we directly have 
		\beno
		|\mathcal{I}_{2}| = |\int B^{s,\gamma}_{\eta}((\mu^{1/2})^{\prime} - \mu^{1/2}) g (h f)_{*} \mathrm{d}V| 
		=
		|\langle Q^{s,\gamma}_{\eta}(hf, g), \mu^{1/2}\rangle|
		\\ \lesssim C_{s,\gamma,\eta} |hf|_{H^{s}_{-2l}} |g|_{L^{2}_{-l}} |\mu^{\frac{1}{2}}|_{H^{s+2}_{3l}}
		\lesssim_{l} C_{\delta, s,\gamma,\eta}  |g|_{L^{2}_{-l}}
		|h|_{H^{\f{3}{2}+\delta}_{-l}} |f|_{H^{s}_{-l}}.
		\eeno
		Similarly, by  Lemma \ref{a-special-term} to be proved later, we have 
		\beno
		|\mathcal{I}_{2}|
		=
		|\langle Q^{s,\gamma}_{\eta}(hf, g), \mu^{1/2}\rangle|
		\lesssim_{l} C_{\delta, s,\gamma,\eta}  |g|_{H^{\f{3}{2}+\delta}_{-l}}
		|h|_{H^{s}_{-l}} |f|_{H^{s}_{-l}}.
		\eeno
		Combining the estimates on $\mathcal{I}_{1}$ and $\mathcal{I}_{2}$
		completes the proof of the proposition.
	\end{proof}

	\begin{lem}\label{a-special-term} It holds that
		\beno 
		|\langle Q^{s,\gamma}_{\eta}(f_2 f_3, f_1), \mu^{1/2}\rangle|
		\lesssim_{l} C_{\delta, s,\gamma,\eta}  |f_1|_{H^{\f{3}{2}+\delta}_{-l}}
		|f_2|_{H^{s}_{-l}} |f_3|_{H^{s}_{-l}}.
		\eeno
	\end{lem}
	\begin{proof} We will follow the proof of Prop. \ref{ubqepsilon-singular} by taking $g = f_2 f_3, h = f_1, f = \mu^{1/2}$. 
		Recall \eqref{term-type-1}	for 
		$l \leq k+ N_0$ that
		\beno
		|\langle Q^{s,\gamma}_{\eta}(\mathfrak{F}_{j} g, \mathfrak{F}_{k}h), \mathfrak{F}_{l}f\rangle| \lesssim C_{s,\gamma,\eta}2^{sl+sk+\frac{3}{2}l}|\mathfrak{F}_{j} g|_{L^{2}}|\mathfrak{F}_{k}h|_{L^{2}}|\mathfrak{F}_{l}f|_{L^{2}}.
		\eeno	
		Note that
		\ben \label{f-j-g-f2f3-trick}
		|\mathfrak{F}_{j} g|_{L^{2}} = |\varphi_{j} (\hat{f_2} * \hat{f_3})|_{L^{2}} \leq 
		|\varphi_{j}|_{L^{r}}
		|\hat{f_2}|_{L^{q}}
		|\hat{f_3}|_{L^{q}}
		\lesssim 2^{\frac{3}{r}j} |\hat{f_2}|_{L^{q}} |\hat{f_3}|_{L^{q}}.
		\een	
		Here $r \geq 2, q \leq 2$ satisfy
		\beno
		\f{1}{r} + \f{2}{q} = 1 + \f{1}{2}, \quad \f{3}{r} = \f{3}{2} -s.
		\eeno
		For the  chosen $r$, when  $|j-k| < N_{0}$, we have
	$
		2^{\frac{3}{r}j} 2^{sk} \lesssim
		2^{\f{3}{2}k}
	$.
	For $1/q = 1/2 + 1/p$, we have
		\beno
		|\hat{f_2}|_{L^{q}} \leq |W_s \hat{f_2}|_{L^{2}}
		|W_{-s}|_{L^{p}} \lesssim |f_2|_{H^{s}},
		\eeno
		because $	s p = 6$. Then we obtain
		\beno
		\sum_{|j-k| < N_{0}} \sum_{l \leq k+N_{0}} \langle Q^{s,\gamma}_{\eta}(\mathfrak{F}_{j} g, \mathfrak{F}_{k}h), \mathfrak{F}_{l}f \rangle 	
		\lesssim 
		\sum_{|j-k| < N_{0}} \sum_{l \leq k+N_{0}} 
		C_{s,\gamma,\eta} 2^{sl+\frac{3}{2}l+\frac{3}{2}k}
		|f_2|_{H^{s}} |f_3|_{H^{s}} |\mathfrak{F}_{k}h|_{L^{2}}|\mathfrak{F}_{l}f|_{L^{2}}
		\\
		\lesssim C_{s,\gamma,\eta} 	|f_2|_{H^{s}} |f_3|_{H^{s}} |f|_{H^{s+2}}
		\sum_{|j-k| < N_{0}}   2^{\frac{3}{2}k}
		|\mathfrak{F}_{k}h|_{L^{2}}
		\\ 
		\lesssim C_{\delta, s,\gamma,\eta} 	|f_2|_{H^{s}} |f_3|_{H^{s}} |f|_{H^{s+2}} |h|_{H^{3/2+\delta}}
		\lesssim C_{\delta, s,\gamma,\eta} |f_1|_{H^{3/2+\delta}} 	|f_2|_{H^{s}} |f_3|_{H^{s}}.
		\eeno

		Recalling \eqref{term-type-2} for $j \leq k-N_{0}$,  and using \eqref{f-j-g-f2f3-trick} for $r=q=2$, we have
		\beno
		|\langle Q^{s,\gamma}_{\eta}(\mathfrak{F}_{j} g, \mathfrak{F}_{k} h), \tilde{\mathfrak{F}}_{k} f\rangle| \lesssim C_{s,\gamma,\eta}
		2^{2sk+3j} |f_2|_{L^{2}} |f_3|_{L^{2}}|\mathfrak{F}_{k}h|_{L^{2}}|\tilde{\mathfrak{F}}_{k} f|_{L^{2}},
		\eeno
	which yields
		\beno
		|\sum_{j \leq k-N_{0}} \langle Q^{s,\gamma}_{\eta}(\mathfrak{F}_{j} g, \mathfrak{F}_{k} h), \tilde{\mathfrak{F}}_{k} f\rangle|
		\lesssim  C_{s,\gamma,\eta} |f_2|_{L^{2}} |f_3|_{L^{2}} \sum_{j \leq k-N_{0}} 2^{2sk+3j} |\mathfrak{F}_{k}h|_{L^{2}}|\tilde{\mathfrak{F}}_{k} f|_{L^{2}}
		\\ \lesssim 
		C_{s,\gamma,\eta} |f_2|_{L^{2}} |f_3|_{L^{2}} \sum_{k} 2^{2sk+3k} |\mathfrak{F}_{k}h|_{L^{2}}|\tilde{\mathfrak{F}}_{k} f|_{L^{2}}
		\lesssim 
		C_{s,\gamma,\eta} |f_2|_{L^{2}} |f_3|_{L^{2}} |h|_{L^{2}}|f|_{H^{2s+3}}
		\\ \lesssim 
		C_{s,\gamma,\eta} |f_1|_{L^{2}} |f_2|_{L^{2}} |f_3|_{L^{2}}.
		\eeno
		
		Recalling \eqref{term-type-3}  
		for $j \geq k+N_{0}$, and using \eqref{f-j-g-f2f3-trick} for $r=q=2$, 
		we have
		\beno
		|\langle Q^{s,\gamma}_{\eta}(\mathfrak{F}_{j} g, \mathfrak{F}_{k}h), \tilde{\mathfrak{F}}_{j}f\rangle| \lesssim C_{s,\gamma,\eta}
		2^{2sj+\frac{3}{2}j+\frac{3}{2}k} |f_2|_{L^{2}} |f_3|_{L^{2}} |\mathfrak{F}_{k}h|_{L^{2}}|\tilde{\mathfrak{F}}_{j} f|_{L^{2}},
		\eeno
	which yields
		\beno
		\sum_{j \geq k+N_{0}} |\langle Q^{s,\gamma}_{\eta}(\mathfrak{F}_{j} g, \mathfrak{F}_{k}h), \tilde{\mathfrak{F}}_{j}f\rangle|
		\lesssim  C_{s,\gamma,\eta} |f_2|_{L^{2}} |f_3|_{L^{2}} \sum_{j \geq k+N_{0}} 2^{2sj+\frac{3}{2}j+\frac{3}{2}k} |\mathfrak{F}_{k}h|_{L^{2}}|\tilde{\mathfrak{F}}_{k} f|_{L^{2}}
		\\ \lesssim 
		C_{s,\gamma,\eta} |f_2|_{L^{2}} |f_3|_{L^{2}} |h|_{L^{2}} \sum_{j} 2^{2sj+3j} |\tilde{\mathfrak{F}}_{j} f|_{L^{2}}
		\\ \lesssim 
		C_{s,\gamma,\eta} |f_2|_{L^{2}} |f_3|_{L^{2}} |h|_{L^{2}} |f|_{H^{2s+4}} \lesssim 
		C_{s,\gamma,\eta} |f_1|_{L^{2}} |f_2|_{L^{2}} |f_3|_{L^{2}}.
		\eeno
		
		Combining the above estimates gives
		\beno 
		|\langle Q^{s,\gamma}_{\eta}(f_2 f_3, f_1), \mu^{1/2}\rangle|
		\lesssim C_{\delta, s,\gamma,\eta}  |f_1|_{H^{\f{3}{2}+\delta}}
		|f_2|_{H^{s}} |f_3|_{H^{s}}.
		\eeno
		By \eqref{decomposition-phase-near-0},  since we can freely transfer
		weight among $f_1, f_2, f_3, \mu^{\f12}$, then we put the negative order $-l$ weight on each of $f_1 f_2, f_3$ and the positive order $3l$ weight on $\mu^{\f12}$ to have the final estimate.
	\end{proof}

	\subsection{Upper bound of $\langle \Gamma^{s,\gamma}_{\eta}(g,h;\beta), f\rangle$} 	
	
	Note that
	\ben \label{Gamma-into-Q-I-leq-eta}
	\Gamma^{s,\gamma}_{\eta}(g,h;\beta) =   Q^{s,\gamma}_{\eta}(g\partial_{\beta}\mu^{1/2},h) +
	I^{s,\gamma}_{\eta}(g,h;\beta).
	\een
	By Propositions  \ref{ubqepsilon-singular} and \ref{I-less-eta-upper-bound}, we have the following
	proposition.

	\begin{prop} \label{Gamma-less-eta-upper-bound}  Let  $(a_{1},a_{2})=(\f{3}{2} + \delta, s)$ or $(s, \f{3}{2} + \delta)$. Then it holds that
		\beno 
		\langle \Gamma^{s,\gamma}_{\eta}(g,h;\beta), f\rangle  \lesssim_{l} C_{\delta, s,\gamma,\eta} |g|_{H^{a_{1}}_{-l}}|h|_{H^{a_{2}}_{-l}}
		|f|_{H^{s}_{-l}}. 
		\eeno
	\end{prop}

	\subsection{ Upper bound of  $\langle \mathcal{L}^{s,\gamma}_{\eta}f , h\rangle$ } 
	Recall that
	\beno
\mathcal{L}^{s,\gamma}_{\eta}(f; \beta_0, \beta_1) = \mathcal{L}^{s,\gamma}_{\eta,1}(f; \beta_0, \beta_1) + \mathcal{L}^{s,\gamma}_{\eta,2}(f; \beta_0, \beta_1),
	\\
	\mathcal{L}^{s,\gamma}_{\eta,1}(f; \beta_0, \beta_1) := - \Gamma^{s,\gamma}_{\eta}(\pa_{\beta_{1}}\mu^{1/2}, f;\beta_{0}), \quad
	\mathcal{L}^{s,\gamma}_{\eta,2}(f; \beta_0, \beta_1) := - \Gamma^{s,\gamma}_{\eta}(f, \pa_{\beta_{1}}\mu^{1/2};\beta_{0}).
	\eeno
	If $|\beta_0| = |\beta_1| = 0$, the operators are reduced to 
	\beno
	\mathcal{L}^{s,\gamma}_{\eta}f = \mathcal{L}^{s,\gamma}_{\eta}(f; 0, 0), \quad	\mathcal{L}^{s,\gamma}_{\eta,1}f = \mathcal{L}^{s,\gamma}_{\eta,1}(f; 0, 0), \quad
	\mathcal{L}^{s,\gamma}_{\eta,2}f = \mathcal{L}^{s,\gamma}_{\eta,2}(f; 0, 0).
	\eeno
	
	By  Proposition \ref{Gamma-less-eta-upper-bound}, we have the following proposition.
	\begin{prop}\label{less-eta-part-l1} It holds that
		\beno |\langle \mathcal{L}^{s,\gamma}_{\eta,1}(f; \beta_0, \beta_1), h\rangle| \lesssim_{l}  C_{ s,\gamma,\eta} |f|_{H^{s}_{-l}}
		|h|_{H^{s}_{-l}}.  \eeno
	\end{prop}
	
Note that 
	for any $0 \leq \kappa, \iota \leq 1$
	\beno
	(1-\frac{\sqrt{2}}{2})(|v|^{2}+|v_{*}|^{2}) \leq |v(\kappa)|^{2}+|v_{*}(\iota)|^{2} \leq (1+\frac{\sqrt{2}}{2})(|v|^{2}+|v_{*}|^{2}),
	\eeno
	which yields
	\ben \label{mu-weight-result}
	\mu^{2}(v) \mu^{2}(v_{*}) \leq \mu(v(\kappa)) \mu(v_{*}(\iota)) \leq \mu^{\frac{1}{4}}(v) \mu^{\frac{1}{4}}(v_{*}).
	\een
	Thus, this  estimate keeps a $\mu$-type weight in the upper bound of  $\mathcal{L}^{s,\gamma}_{\eta,2}$.
	\begin{prop}\label{less-eta-part-l2} It holds that
		\beno |\langle \mathcal{L}^{s,\gamma}_{\eta,2}(f; \beta_0, \beta_1), h\rangle| \lesssim  C_{s,\gamma,\eta} |\mu^{1/32}f|_{H^{s}}|\mu^{1/32}h|_{H^{s}}.  \eeno
	\end{prop}
	\begin{proof} For simplicity, we only consider $|\beta_0| = |\beta_1| = 0$.
		Note that
		\beno \langle -\mathcal{L}^{s,\gamma}_{\eta, 2}f , h\rangle &=&
		\langle \mu^{-1/2} Q^{s,\gamma}_{\eta}(\mu^{1/2}f,\mu), h\rangle
		= \int B^{s,\gamma}_{\eta} (\mu^{1/2}f)_{*}\mu((\mu^{-1/2}h)^{\prime}-\mu^{-1/2}h) \mathrm{d}V
		\\&=& \int B^{s,\gamma}_{\eta} (\mu^{1/2}f)_{*} \mu^{1/2} (h^{\prime}-h) \mathrm{d}V +\int B^{s,\gamma}_{\eta} f_{*}\mu^{1/2}h^{\prime}((\mu^{1/2})^{\prime}_{*}-\mu^{1/2}_{*}) \mathrm{d}V
		:= Y_{1}+Y_{2}.
		\eeno

		We first estimate $Y_{1}$. Observe that
		\beno Y_{1}&=& \int B^{s,\gamma}_{\eta} (\mu^{1/2}f)_{*} \mu^{1/2} (h^{\prime}-h) \mathrm{d}V
		\\&=&\int B^{s,\gamma}_{\eta} (\mu^{1/2}f)_{*}  ((\mu^{1/2}h)^{\prime}-\mu^{1/2}h) \mathrm{d}V+\int B^{s,\gamma}_{\eta} (\mu^{1/2}f)_{*} (\mu^{1/2}- (\mu^{1/2})^{\prime})h^{\prime} \mathrm{d}V
		:= Y_{1,1} +Y_{1,2}.
		\eeno
		Proposition \ref{new-version-cancellation} implies
		\beno |Y_{1,1}|\lesssim C_{s,\gamma,\eta} |\mu^{1/2}f|_{H^{s}}|\mu^{1/2}h|_{H^{s}}.
		\eeno
		By \eqref{Taylor2} and \eqref{cancell2}, using
		\eqref{mu-weight-result}, and  the change of variable $v \to v'$, we have
		\beno
		|Y_{1,2}| \lesssim \int \mathrm{1}_{|v-v_{*}| \leq 4\eta/3}|v-v_{*}|^{\gamma+2} \mu^{1/12}_{*} f_{*}  \mu^{1/12} h  \mathrm{d}v_{*} \mathrm{d}v \lesssim \frac{\eta^{\gamma+5}}{\gamma+5} |\mu^{1/16}f|_{L^{2}}|\mu^{1/16}h|_{L^{2}}.
		\eeno
		We now turn to $Y_2$. Note that
		\beno Y_{2} &=& \int B^{s,\gamma}_{\eta} f_{*}\mu^{1/2}h^{\prime}((\mu^{1/2})^{\prime}_{*}-\mu^{1/2}_{*}) \mathrm{d}V
		\\&=& \int B^{s,\gamma}_{\eta} f_{*}(\mu^{1/2}-(\mu^{1/2})^{\prime})h^{\prime}((\mu^{1/2})^{\prime}_{*}-\mu^{1/2}_{*}) \mathrm{d}V  +\int B^{s,\gamma}_{\eta} f_{*}(\mu^{1/2})^{\prime}h^{\prime}((\mu^{1/2})^{\prime}_{*}-\mu^{1/2}_{*}) \mathrm{d}V
		\\&:=& Y_{2,1} +Y_{2,2}.
		\eeno
		Similar to the estimate on $Y_{1,2}$, we get
		\beno |Y_{2,2}| \lesssim  \frac{\eta^{\gamma+5}}{\gamma+5}|\mu^{1/16}f|_{L^{2}}|\mu^{1/16}h|_{L^{2}}.
		\eeno
		By the Cauchy-Schwarz inequality, we get
		\beno |Y_{2,1}| &\leq&  \big(\int B^{s,\gamma}_{\eta} f^{2}_{*}(\mu^{1/4}-(\mu^{1/4})^{\prime})^{2}
		((\mu^{1/4})^{\prime}_{*}+\mu^{1/4}_{*})^{2}
		\mathrm{d}V \big)^{1/2}
		\\&& \times \big(\int B^{s,\gamma}_{\eta}(h^{\prime})^{2}
		((\mu^{1/4})^{\prime}_{*}-\mu^{1/4}_{*})^{2} (\mu^{1/4}+(\mu^{1/4})^{\prime})^{2}\mathrm{d}V \big)^{1/2}
		\\&=&\big(\int B^{s,\gamma}_{\eta} f^{2}_{*}(\mu^{1/4}-(\mu^{1/4})^{\prime})^{2}
		((\mu^{1/4})^{\prime}_{*}+\mu^{1/4}_{*})^{2}
		\mathrm{d}V \big)^{1/2}
		\\&& \times \big(\int B^{s,\gamma}_{\eta} h^{2}_{*}(\mu^{1/4}-(\mu^{1/4})^{\prime})^{2}
		((\mu^{1/4})^{\prime}_{*}+\mu^{1/4}_{*})^{2}
		\mathrm{d}V \big)^{1/2}.
		\eeno
		By Taylor expansion, 
		using
		\eqref{mu-weight-result} to obtain the $\mu$-type weight,  we get
		\beno \int B^{s,\gamma}_{\eta} f^{2}_{*}(\mu^{1/4}-(\mu^{1/4})^{\prime})^{2}
		((\mu^{1/4})^{\prime}_{*}+\mu^{1/4}_{*})^{2}
		\mathrm{d}V \lesssim
		\frac{\eta^{\gamma+5}}{\gamma+5}|\mu^{1/32}f|^{2}_{L^{2}},\eeno
		which gives
		$|Y_{2,1}| \lesssim (\gamma+5)^{-1} \eta^{\gamma+5} |\mu^{1/32}f|_{L^{2}} |\mu^{1/32}h|_{L^{2}}. $
		Combining  the above estimates completes the proof of the proposition.
	\end{proof}
	
	By Propositions \ref{less-eta-part-l1} and \ref{less-eta-part-l2}, we also have the following
	proposition.
	\begin{prop}\label{less-eta-part-l} It holds that
		\beno |\langle \mathcal{L}^{s,\gamma}_{\eta}(f; \beta_0, \beta_1), h\rangle| \lesssim  C_{s,\gamma,\eta} |f|_{H^{s}_{-l}}
		|h|_{H^{s}_{-l}}.  \eeno
	\end{prop}
	
	
	\section{Upper bound estimate in the regular region} \label{Upper-Bound-Estimate-away}
	
	In this section, we will  derive some upper bound
	estimates in the regular region $|v-v_{*}| \gtrsim \eta$. 
	Recall\eqref{Gamma-ep-ga-geq-eta-into-IQ} that 
	\ben \label{Gamma-ep-ga-geq-eta-into-IQ-inner}
	\langle \Gamma^{s,\gamma,\eta}(g,h), f\rangle =  \langle Q^{s,\gamma,\eta}(\mu^{1/2}g,h), f\rangle +
	\langle I^{s,\gamma,\eta}(g,h), f\rangle.
	\een
	We now estimate $\langle Q^{s,\gamma,\eta}(g,h), f\rangle$ and $\langle I^{s,\gamma,\eta}(g,h), f\rangle$ in the following two subsections. With these estimates, we 
	will combine them suitably according to the operator splitting in subsection \ref{ope-spl} to 
	obtain several operator estimates for later use.

	\subsection{Upper bound for $ Q^{s,\gamma,\eta}$}
	Set $u = v-v_{*}$, the 
	$v = v_* + u, v' = v_* + u^{+}$ where $u^{+} = \frac{|u|\sigma + u}{2}$. Define the translation operator $T_{v_{*}}$ by $(T_{v_{*}}f)(v) = f(v_{*}+v)$. We recall the geometric decomposition into radial and spherical parts
	\ben
	f(v') - f(v) &=&  \left(((T_{v_{*}}f)(u^{+})-(T_{v_{*}}f)(|u|\frac{u^{+}}{|u^{+}|}))\right)  + \left(
	(T_{v_{*}}f)(|u|\frac{u^{+}}{|u^{+}|})- (T_{v_{*}}f)(u)
	\right)
	\nonumber
	\\ \label{sphere-radius} &=&  \text{ radial part} + \text{ spherical part}.
	\een

	\subsubsection{Radial part} We first derive some estimates on the radial part. Note that $|u^{+} - |u|\frac{u^{+}}{|u^{+}|}| = |u|(1-\cos \f{\theta}{2}) = 2 |u|\sin^2 \f{\theta}{4}$, which yields an order-2 cancellation in the angular singularity.  The radial part in \eqref{sphere-radius} can be  controlled by gain of $W_{s}$   in the phase and frequency space, namely the two Sobolev norms $L^{2}_{s}$ and $H^{s}$.
	We show this by using some localization technique introduced in \cite{he2022asymptotic}.
	\begin{lem}\label{crosstermsimilar} Let $0<\eta \leq 1$.
		Set 
		$$
		\mathcal{Y}^{s,\gamma,\eta}(h,f) := \int b^{s}(\frac{u}{|u|}\cdot\sigma)|u|^\gamma 
		\psi^{\eta}(u) h(u)[f(u^{+}) - f(|u|\frac{u^{+}}{|u^{+}|})] \mathrm{d}\sigma \mathrm{d}u,
		$$
		then
		\beno
		|\mathcal{Y}^{s,\gamma,\eta}(h,f)| \lesssim s^{-1} \eta^{\gamma} (|W_{\gamma/2}h|_{L^{2}_{s}}+|W_{\gamma/2}h|_{H^{s}})
		(|W_{\gamma/2}f|_{L^{2}_{s}}+|W_{\gamma/2}f|_{H^{s}}).
		\eeno
	\end{lem}
	\begin{proof}We divide the proof into two steps.
		
		{\it Step 1: Without the term $|u|^\gamma  \psi^{\eta}(u)$.} We first denote $$\mathcal{X}(h,f):=\int b^{s}(\frac{u}{|u|}\cdot\sigma)  h(u)[f(u^{+}) - f(|u|\frac{u^{+}}{|u^{+}|})] \mathrm{d}\sigma \mathrm{d}u .$$
		 By applying dyadic decomposition in the phase space, we have
		\beno
		\mathcal{X}(h,f) = \sum_{k=-1}^{\infty}\int b^{s}(\frac{u}{|u|}\cdot\sigma) (\tilde{\varphi}_{k}h)(u) [(\varphi_{k}f)(u^{+})- (\varphi_{k}f)(|u|\frac{u^{+}}{|u^{+}|})] \mathrm{d}\sigma \mathrm{d}u
		:= \sum_{k=-1}^{\infty} \mathcal{X}_{k}.
		\eeno
		where $\tilde{\varphi}_{k} = \sum_{|l-k|\leq 3} \varphi_{l}$. By using Proposition 5.2 in \cite{he2022asymptotic}, we have
		\ben \label{fourier-transform-cross-term}
		\int_{\R^3\times\mathbb{S}^2} b(\f{u}{|u|}\cdot \sigma) h(u)(\bar{f}(u^+)-\bar{f}(\f{|u|}{|u^+|}u^+)) \mathrm{d}\sigma \mathrm{d}u
		=\int_{\R^3\times\mathbb{S}^2} b(\f{\xi}{|\xi|}\cdot \sigma)  (\widehat{h}(\xi^+)-\widehat{h}(\f{|\xi|}{|\xi^+|}\xi^+))\overline{\widehat{f}}(\xi) \mathrm{d}\sigma \mathrm{d}\xi.
		\een
		By \eqref{fourier-transform-cross-term} and  the dyadic decomposition in the frequency space, we have
		\beno
		\mathcal{X}_{k} &=& \int b^{s}(\frac{\xi}{|\xi|}\cdot\sigma)  [\widehat{\tilde{\varphi}_{k}h}(\xi^{+})- \widehat{\tilde{\varphi}_{k}h}(|\xi|\frac{\xi^{+}}{|\xi^{+}|})] \overline{\widehat{\varphi_{k}f}}(\xi)  
		\mathrm{d}\sigma \mathrm{d}\xi
		\\&=& \sum_{l=-1}^{\infty} \int b^{s}(\frac{\xi}{|\xi|}\cdot\sigma)[({\varphi}_{l}\widehat{\tilde{\varphi}_{k}h})(\xi^{+})- ({\varphi}_{l}\widehat{\tilde{\varphi}_{k}h})(|\xi|\frac{\xi^{+}}{|\xi^{+}|})] (\tilde{\varphi}_{l}\overline{\widehat{\varphi_{k}f}})(\xi)  \mathrm{d}\sigma \mathrm{d}\xi
		:= \sum_{l=-1}^{\infty} \mathcal{X}_{k,l}.
		\eeno
		Therefore,
		\beno
		\mathcal{X}(h,f) = \sum_{k=-1}^{\infty}\sum_{l=-1}^{\infty} \mathcal{X}_{k,l}.
		\eeno
		By dividing $\theta \in [0,\pi/2]$  according to $\sin \frac{\theta}{2} \geq 2^{-\frac{k+l}{2}}$ and $\sin \frac{\theta}{2} \leq 2^{-\frac{k+l}{2}}$, we get
		\beno
		\mathcal{X}_{k,l} &=& \int b^{s}(\frac{\xi}{|\xi|}\cdot\sigma)\mathrm{1}_{\sin \frac{\theta}{2} \geq 2^{-\frac{k+l}{2}}}[({\varphi}_{l}\widehat{\tilde{\varphi}_{k}h})(\xi^{+})- ({\varphi}_{l}\widehat{\tilde{\varphi}_{k}h})(|\xi|\frac{\xi^{+}}{|\xi^{+}|})] (\tilde{\varphi}_{l}\overline{\widehat{\varphi_{k}f}})(\xi) \mathrm{d}\sigma \mathrm{d}\xi
		\\&& + \int b^{s}(\frac{\xi}{|\xi|}\cdot\sigma)\mathrm{1}_{\sin \frac{\theta}{2} \leq 2^{-\frac{k+l}{2}}}[({\varphi}_{l}\widehat{\tilde{\varphi}_{k}h})(\xi^{+})- ({\varphi}_{l}\widehat{\tilde{\varphi}_{k}h})(|\xi|\frac{\xi^{+}}{|\xi^{+}|})] (\tilde{\varphi}_{l}\overline{\widehat{\varphi_{k}f}})(\xi) \mathrm{d}\sigma \mathrm{d}\xi
		:= \mathcal{X}_{k,l,1} + \mathcal{X}_{k,l,2}.
		\eeno
		Similarly, we have
		$$
		|\mathcal{X}_{k,l,1}| \lesssim s^{-1} 2^{sk+sl}|{\varphi}_{l}\widehat{\tilde{\varphi}_{k}h}|_{L^{2}}|\tilde{\varphi}_{l}\widehat{\varphi_{k}f}|_{L^{2}}.
		$$
		Therefore,  we obtain
		\beno
		\sum_{k,l} |\mathcal{X}_{k,l,1}| \lesssim s^{-1} \bigg(\sum_{k,l} 2^{2sl}|{\varphi}_{l}\widehat{\tilde{\varphi}_{k}h}|^{2}_{L^{2}}\bigg)^{1/2}
		\bigg(\sum_{k,l}2^{2sk}|\tilde{\varphi}_{l}\widehat{\varphi_{k}f}|^{2}_{L^{2}}\bigg)^{1/2}
		\lesssim s^{-1}|W_{s}(D)h|_{L^{2}}|W_{s}f|_{L^{2}}.
		\eeno
		By Taylor expansion,
		$$
		({\varphi}_{l}\widehat{\tilde{\varphi}_{k}h})(\xi^{+})- ({\varphi}_{l}\widehat{\tilde{\varphi}_{k}h})(|\xi|\frac{\xi^{+}}{|\xi^{+}|}) = (1-\frac{1}{\cos\theta})\int_{0}^{1}
		(\nabla (\varphi_{l}\widehat{\tilde{\varphi}_{k}h}))(\xi^{+}(\iota))\cdot \xi^{+} \mathrm{d}\iota,
		$$
		where $\xi^{+}(\iota) = (1-\iota)|\xi|\frac{\xi^{+}}{|\xi^{+}|} + \iota \xi^{+}$,
		 we have
		\beno
		|\mathcal{X}_{k,l,2}| &=& | \int_{[0,1]\times \R^{3} \times \mathbb{S}^{2}}  b^{s}(\frac{\xi}{|\xi|}\cdot\sigma)(1-\frac{1}{\cos\theta})\mathrm{1}_{\sin \frac{\theta}{2} \leq 2^{-\frac{k+l}{2}} }(\tilde{\varphi}_{l}\overline{\widehat{\varphi_{k}f}})(\xi)
		(\nabla (\varphi_{l}\widehat{\tilde{\varphi}_{k}h})(\xi^{+}(\iota)))\cdot \xi^{+} \mathrm{d}\iota \mathrm{d}\sigma \mathrm{d}\xi |
		\\&\lesssim&  (\int_{0}^{2^{-\frac{k+l}{2}}} \theta^{1-2s} |\tilde{\varphi}_{l}\widehat{\varphi_{k}f}(\xi)|^{2} \mathrm{d}\theta \mathrm{d}\xi)^{1/2}
		(\int_{0}^{2^{-\frac{k+l}{2}}} \theta^{1-2s} |u|^{2}|\nabla (\varphi_{l}\widehat{\tilde{\varphi}_{k}h})(u)|^{2} \mathrm{d}\theta \mathrm{d}u )^{1/2}
		\\&\lesssim& 2^{(s-1)(k+l)}
		|\tilde{\varphi}_{l}\widehat{\varphi_{k}f}|_{L^{2}}(\int |u|^{2}|\nabla (\varphi_{l}\widehat{\tilde{\varphi}_{k}h})(u)|^{2}\mathrm{d}u )^{1/2}
		\\&\lesssim& 2^{(s-1)(k+l)}
		|\tilde{\varphi}_{l}\widehat{\varphi_{k}f}|_{L^{2}}(|\tilde{\varphi}_{l}\widehat{\tilde{\varphi}_{k}h}|_{L^{2}} + 2^{l}|\varphi_{l} \widehat{v\tilde{\varphi}_{k}h}|_{L^{2}}),
		\eeno
		where we have used  the change of variable $\xi \rightarrow u = \xi^{+}(\iota)$ and the fact
		that
		\beno
		\int |u|^{2}|\nabla (\varphi_{l}\widehat{\tilde{\varphi}_{k}h})(u)|^{2}\mathrm{d}u\ \lesssim |\tilde{\varphi}_{l}\widehat{\tilde{\varphi}_{k}h}|^2_{L^{2}} + 2^{2l}|\varphi_{l} \widehat{v\tilde{\varphi}_{k}h}|^{2}_{L^{2}}.
		\eeno
		Since $2^{(s-1)(k+l)} \lesssim 1$, we have
		\beno \sum_{k,l}2^{(s-1)(k+l)}|\tilde{\varphi}_{l}\widehat{\varphi_{k}f}|_{L^{2}}|\tilde{\varphi}_{l}\widehat{\tilde{\varphi}_{k}h}|_{L^{2}} \lesssim |f|_{L^{2}}|h|_{L^{2}}.\eeno
		It is straightforward to check that
		\beno  &&\sum_{k,l}2^{(s-1)(k+l)} 2^{l}|\varphi_{l} \widehat{v\tilde{\varphi}_{k}h}|_{L^{2}}|\tilde{\varphi}_{l}\widehat{\varphi_{k}f}|_{L^{2}} \\&\lesssim& \bigg( \sum_{k,l} 2^{2sl}|\tilde{\varphi}_{l}\widehat{\varphi_{k}f}|^{2}_{L^{2}} \bigg)^{1/2}
		\bigg( \sum_{k,l} 2^{2sk-2k} |\varphi_{l} \widehat{v\tilde{\varphi}_{k}h}|^{2}_{L^{2}}\bigg)^{1/2}
		\\&\lesssim& \bigg( \sum_{k}|W_{s}(D)\varphi_{k}f|^{2}_{L^{2}} \bigg)^{1/2} \bigg( \sum_{k} 2^{2sk-2k} |v\tilde{\varphi}_{k}h|^{2}_{L^{2}}\bigg)^{1/2}
		\lesssim |W_{s}(D)f|_{L^{2}} |W_{s}h|_{L^{2}}.
		\eeno
		Thus, we have
		$
		\sum_{k,l}|\mathcal{X}_{k,l,2}| \lesssim |W_{s}(D)f|_{L^{2}} |W_{s}h|_{L^{2}}.
		$
		Therefore, 
		\ben \label{gamma=0-clear}
		|\mathcal{X}(h,f)| \lesssim s^{-1} (|W_{s}h|_{L^{2}}+|W_{s}(D)h|_{L^{2}})(|W_{s}f|_{L^{2}}+|W_{s}(D)f|_{L^{2}}).
		\een

		{\it Step 2: With the term $|u|^\gamma\psi^{\eta}(u)$.} We write
		\beno \mathcal{Y}^{s,\gamma}(h,f) =
		\mathcal{Y}_1^{s,\gamma}(h,f) + \mathcal{Y}_2^{s,\gamma}(h,f),
		\eeno
		according to $\psi^{\eta}(u)
		=(1-\psi(u))+(\psi(u)-\psi_{\eta}(u))$.

		{\it{Estimate of $\mathcal{Y}_1^{s,\gamma}(h,f)$ containing $|u|^\gamma(1-\psi(u))$.}}
		Set $H(u) :=h\langle u \rangle^{-\gamma} |u|^\gamma(1-\psi(u))$ and $w = |u|\frac{u^{+}}{|u^{+}|}$, then $W_{\gamma/2}(u) = W_{\gamma/2}(w)$ and
		\beno
		\langle u \rangle^{\gamma} H(u)[f(u^{+}) - f(w)] &=& (W_{\gamma/2}H)(u)[(W_{\gamma/2}f)(u^{+})-(W_{\gamma/2}f)(w)]
		\\&& +(W_{\gamma/2}H)(u)(W_{\gamma/2}f)(u^{+})(W_{\gamma/2}(w)W_{-\gamma/2}(u^{+}) - 1).
		\eeno
		From this, we have
		\beno
		\mathcal{Y}_1^{s,\gamma}(h,f) &=& \mathcal{X}(W_{\gamma/2}H, W_{\gamma/2}f)+ \int b^{s}(\frac{u}{|u|}\cdot\sigma)(W_{\gamma/2}H)(u) \\
		&&\times(W_{\gamma/2}f)(u^{+})(W_{\gamma/2}(w)W_{-\gamma/2}(u^{+}) - 1) \mathrm{d}\sigma \mathrm{d}u
		:= \mathcal{X}(W_{\gamma/2}H, W_{\gamma/2}f) + \mathcal{A}.
		\eeno
		Thanks to  \eqref{gamma=0-clear} in {\it Step 1} and Lemma \ref{operatorcommutator1} with $M=W^s \in S^{s}_{1,0}$ and $\Phi= \langle \cdot \rangle^{-\gamma}  |\cdot|^\gamma(1-\psi(\cdot)) \in S^{0}_{1,0}$, we have
		\beno |\mathcal{X}(W_{\gamma/2}H, W_{\gamma/2}f)|\lesssim s^{-1} (|W_{s}W_{\gamma/2}h|_{L^{2}}+|W_{s}(D)W_{\gamma/2}h|_{L^{2}})
		(|W_{s}W_{\gamma/2}f|_{L^{2}}+|W_{s}(D)W_{\gamma/2}f|_{L^{2}}).\eeno
		By noting 
		$
		|W_{\gamma/2}(u)W_{-\gamma/2}(u^{+}) - 1| \lesssim \theta^{2}$
		and using the change of variable $u \rightarrow u^{+}$,
		we have
		\beno
		|\mathcal{A}| &=&  \big(\int b^{s}(\frac{u}{|u|}\cdot\sigma)
		|(W_{\gamma/2}H)(u)|^{2}|W_{\gamma/2}(w)W_{-\gamma/2}(u^{+}) - 1|\mathrm{d}\sigma \mathrm{d}u\big)^{1/2}
		\\&&\times\big(\int b^{s}(\frac{u}{|u|}\cdot\sigma)
		|(W_{\gamma/2}f)(u^{+})|^{2}|W_{\gamma/2}(w)W_{-\gamma/2}(u^{+}) - 1| \mathrm{d}\sigma \mathrm{d}u\big)^{1/2}
		\\	&\lesssim& |W_{\gamma/2}H|_{L^{2}}|W_{\gamma/2}f|_{L^{2}}
		\lesssim |W_{\gamma/2}h|_{L^{2}}|W_{\gamma/2}f|_{L^{2}}.
		\eeno
		Combining  the estimates of $ \mathcal{X}(W_{\gamma/2}H, W_{\gamma/2}f)$ and $\mathcal{A}$ gives
		\beno |\mathcal{Y}_1^{s,\gamma}(h,f)|\lesssim s^{-1} (|W_{s}W_{\gamma/2}h|_{L^{2}}+|W_{s}(D)W_{\gamma/2}h|_{L^{2}})
		(|W_{s}W_{\gamma/2}f|_{L^{2}}+|W_{s}(D)W_{\gamma/2}f|_{L^{2}}).\eeno

		{\it{Estimate of $\mathcal{Y}_2^{s,\gamma}(h,f)$ containing 
			$|u|^\gamma(\psi(u)-\psi_{\eta}(u))$.}}
		Since the support of $|u|^\gamma(\psi(u)-\psi_{\eta}(u))$ belongs to $\f{3}{4}\eta \leq |u| \leq \f{4}{3}$, we write
		\ben\label{Y2epgaestimate}  \mathcal{Y}_2^{s,\gamma}(h,f)=\int b^{s}(\frac{u}{|u|}\cdot\sigma)  \tilde{W}(u) \tilde{H}(u)[\tilde{F}(u^{+}) - \tilde{F}(|u|\frac{u^{+}}{|u^{+}|})] \mathrm{d}\sigma \mathrm{d}u, \een
		where $\tilde{W}(u):=|u|^\gamma(\psi(u)-\psi_{\eta}(u)), \tilde{H}:=\psi_{2}h$, $\tilde{F}:=\psi_{2}f$. By \eqref{gamma=0-clear} in {\it Step 1}, we have
		\beno|\mathcal{Y}_2^{s,\gamma}(h,f)|\lesssim s^{-1} (|W_{s}\tilde{W}\tilde{H}|_{L^{2}}+|W_{s}(D)\tilde{W}\tilde{H}|_{L^{2}})
		(|W_{s}\tilde{F}|_{L^{2}}+|W_{s}(D)\tilde{F}|_{L^{2}}). \eeno
		By $|\tilde{W}(u)| \lesssim \psi(u)\eta^{\gamma}$, we have $|W_{s}\tilde{W}\tilde{H} |_{L^{2}} \lesssim \eta^{\gamma}|W_{s}W_{\gamma/2}h|_{L^{2}}$.
		Note that  $|\tilde{W}|_{L^{2}} \lesssim \eta^{\gamma+\frac{3}{2}}, |\tilde{W}|_{H^{1}} \lesssim \eta^{\gamma+\frac{1}{2}}, |\tilde{W}|_{H^{2}} \lesssim \eta^{\gamma-\frac{1}{2}}$. This gives
		\beno
		|\tilde{W}\tilde{H}|_{H^{s}} \lesssim |\tilde{W}|_{H^{1}}^{1/2}|\tilde{W}|_{H^{2}}^{1/2} |\tilde{H}|_{H^{s}} \lesssim \eta^{\gamma} |\tilde{H}|_{H^{s}}.
		\eeno
		It is straightforward to check that  $|\tilde{W}|_{L^{2}} \lesssim \eta^{\gamma+\frac{3}{2}}, |\tilde{W}|_{H^{1}} \lesssim \eta^{\gamma+\frac{1}{2}}, |\tilde{W}|_{H^{2}} \lesssim \eta^{\gamma-\frac{1}{2}}$.
		Note the supports of $\tilde{H}$ and $\tilde{F}$ belong
		to $|v| \leq 8/3$. By Lemma \ref{operatorcommutator1}, we finally have \beno
		|\mathcal{Y}_2^{s,\gamma}(h,f)|\lesssim s^{-1} \eta^{\gamma}(|W_{s}W_{\gamma/2}h|_{L^{2}}+|W_{s}(D)W_{\gamma/2}h|_{L^{2}})
		(|W_{s}W_{\gamma/2}f|_{L^{2}}+|W_{s}(D)W_{\gamma/2}f|_{L^{2}}).\eeno
		
Combining the above estimates completes the proof of the lemma.
	\end{proof}

	The following Lemma deals with the quadratic term
	in
	 the case when $\gamma=0$.
	\begin{lem}\label{gammanonzerotozero}
		Let $$
		\mathcal{Z}^{s}(f) := \int b^{s}(\frac{u}{|u|}\cdot\sigma)|f(|u|\frac{u^{+}}{|u^{+}|}) - f(u^{+})|^{2} \mathrm{d}\sigma \mathrm{d}u.$$ Then
		\beno
		\mathcal{Z}^{s}(f) \lesssim s^{-1} (|W_{s}(D)f|^{2}_{L^{2}} + |W_{s}f|^{2}_{L^{2}}).
		\eeno
	\end{lem}
	\begin{proof} 
		By the change of variable $(u, \sigma) \rightarrow (r, \tau, \varsigma)$ with $u=r\tau$ and $\varsigma=\f{\sigma+\tau}{|\sigma+\tau|}$, we have
		\beno
		\mathcal{Z}^{s}(f) = 4 \int b^{s}(2(\tau\cdot\varsigma)^{2} - 1)|f(r\varsigma) - f((\tau\cdot\varsigma)r\varsigma)|^{2} (\tau\cdot\varsigma) r^{2} \mathrm{d}r \mathrm{d}\tau \mathrm{d}\varsigma.
		\eeno
		Let $v = r\varsigma$, and $\theta$ be the angle between $\tau$ and $\varsigma$. Since $b^{s}(2(\tau\cdot\varsigma)^{2} - 1) = b^{s}(\cos2\theta) \lesssim (1-s)\theta^{-2-2s} \mathrm{1}_{0 \leq \theta \leq \pi/4}$, and $r^{2} \mathrm{d}r \mathrm{d}\tau \mathrm{d}\varsigma = \sin\theta \mathrm{d}v \mathrm{d}\theta \mathrm{d}\mathbb{S}$, we have
		\beno
		\mathcal{Z}^{s}(f) \lesssim  (1-s)\int_{\R^{3}} \int_{0}^{\pi/4} \theta^{-1-2s}|f(v) - f(v\cos\theta)|^{2} \mathrm{d}v \mathrm{d}\theta
		\lesssim s^{-1} (|W_{s}(D) f|^{2}_{L^{2}} + |W_{s} f|^{2}_{L^{2}}).
		\eeno
		Here,  the last inequality follows from the same argument given in 
		step 1 in the  proof of Lemma \ref{crosstermsimilar}. More precisely, 
		we can consider $f(v)(f(v) - f(v \cos\theta))$ and $f(v \cos\theta)(f(v) - f(v \cos\theta))$
		separately and apply the localization techniques used in Lemma \ref{crosstermsimilar}.
		Thus,  we omit the details.
	\end{proof}

	\subsubsection{Spherical part} We now derive some estimates related to the spherical part.
	In the following Lemma, we recall a preliminary result on the characterization of norm $|(1-\Delta_{\SS^2})^{s/2}f|_{L^2(\SS^2)}$. The proof of this lemma can be found in Lemma 5.5 of \cite{he2018sharp}. Note that we add an factor $s(1-s)$ for 
	consideration of the limit 
	$s \to 1^{-}$.
	\begin{lem}\label{antin1} Let $f$ be a smooth function defined on $\SS^2$. If $0<s<1$, then it holds that
		\ben \label{anisotropic-norm-with-s}
		|f|_{L^2(\SS^2)}^2+ s(1-s) \iint \f{|f(\sigma)-f(\tau)|^2}{|\sigma-\tau|^{2+2s}}\mathrm{d}\sigma \mathrm{d}\tau\sim |(1-\Delta_{\SS^2})^{s/2}f|^2_{L^2(\SS^2)}.
		\een
		The constant associated to $\sim$ is independent of $s$.
	\end{lem}

	\begin{lem}\label{spherical-part} Let $\mathcal{A}^{s}(f) := s\int b^{s}(\frac{u}{|u|} \cdot \sigma)|f(u) - f(|u|\frac{u^{+}}{|u^{+}|})|^{2} \mathrm{d}\sigma \mathrm{d}u $ where $u^{+} = \frac{u + |u|\sigma}{2}$, then
		\ben \label{spherical-part-result}
		\mathcal{A}^{s}(f) + |f|_{L^2}^2
		\sim  |W_{s}((-\Delta_{\mathbb{S}^{2}})^{1/2})f|^{2}_{L^{2}}.
		\een
	\end{lem}
	\begin{proof}
		Let $ r= |u|, \tau = u/|u|$ and $\varsigma = \frac{\tau+\sigma}{|\tau+\sigma|}$, then $\frac{u}{|u|} \cdot \sigma = 2(\tau\cdot\varsigma)^{2} - 1$ and $|u|\frac{u^{+}}{|u^{+}|} = r \varsigma$. For the change of variable $(u, \sigma) \rightarrow (r, \tau, \varsigma)$, one has
		$
		\mathrm{d}u \mathrm{d}\sigma = 4  (\tau\cdot\varsigma) r^{2} \mathrm{d}r \mathrm{d}\tau \mathrm{d}\varsigma.
		$ Let $\theta$ be the angle between $\tau$ and $\sigma$, then $|\tau-\sigma|=2 \sin\frac{\theta}{2}, |\tau - \varsigma| = 2 \sin\frac{\theta}{4}$ and thus
		$\sin\frac{\theta}{2} = \frac{1}{2} |\tau-\sigma| \leq |\tau - \varsigma| \leq |\tau-\sigma| = 2\sin\frac{\theta}{2}$.
		Therefore
		\beno
		|\tau - \varsigma|^{-2-2s} \mathrm{1}_{|\tau - \varsigma| \leq \sqrt{2}/2} \leq   (\sin\frac{\theta}{2})^{-2-2s} \mathrm{1}_{\sin\frac{\theta}{2} \leq \sqrt{2}/2}
		\leq   2^{2+2s} |\tau - \varsigma|^{-2-2s} \mathrm{1}_{|\tau - \varsigma| \leq \sqrt{2}}.
		\eeno
		By \eqref{anisotropic-norm-with-s}, we have
		\beno
		\mathcal{A}^{s}(f)  &=& 4 s \int_{0}^{\infty} \int_{\mathbb{S}^2 \times \mathbb{S}^2} b^{s}(\theta)|f(r\tau) - f(r\varsigma)|^{2} (\tau\cdot\varsigma) r^{2} \mathrm{d}r \mathrm{d}\tau \mathrm{d}\varsigma 
		\\&\lesssim& s (1-s) \int |f(r\tau) - f(r\varsigma)|^{2}|\tau - \varsigma|^{-2-2s}\mathrm{1}_{|\tau-\varsigma| \leq \sqrt{2}}  r^{2} \mathrm{d}r \mathrm{d}\tau \mathrm{d}\varsigma
		\lesssim |W_{s}((-\Delta_{\mathbb{S}^{2}})^{1/2})f|^{2}_{L^{2}}.
		\eeno
		In the last inequality, we have used the fact that
		the	Fourier transform and the operator $W_{s}((-\Delta_{\mathbb{S}^{2}})^{1/2})$ are commutable  and the Plancherel's  Theorem.  This gives the direction of $\lesssim$  in  \eqref{spherical-part-result}.

		On the other hand,
		\beno
		\mathcal{A}^{s}(f)  
		\gtrsim s(1-s) \int |f(r\tau) - f(r\varsigma)|^{2}|\tau - \varsigma|^{-2-2s}\mathrm{1}_{|\tau-\varsigma| \leq \sqrt{2}/2}  r^{2} \mathrm{d}r \mathrm{d}\tau \mathrm{d}\varsigma.
		\eeno
		It is straightforward to check that
		\beno
		\int |f(r\tau) - f(r\varsigma)|^{2}|\tau - \varsigma|^{-2-2s}\mathrm{1}_{|\tau-\varsigma| \geq \sqrt{2}/2}  r^{2} \mathrm{d}r \mathrm{d}\tau \mathrm{d}\varsigma \lesssim
		|f|_{L^2}^2.
		\eeno
		From this and \eqref{anisotropic-norm-with-s}, we get
		\beno
		\mathcal{A}^{s}(f)  
		\gtrsim   s(1-s) \int |f(r\tau) - f(r\varsigma)|^{2}|\tau - \varsigma|^{-2-2s}   r^{2} \mathrm{d}r \mathrm{d}\tau \mathrm{d}\varsigma - C|f|_{L^2}^2
		\gtrsim   |W_{s}((-\Delta_{\mathbb{S}^{2}})^{1/2})f|^{2}_{L^{2}} - C|f|_{L^2}^2.
		\eeno
		This gives the  direction of  $\gtrsim$ in \eqref{spherical-part-result}. Hence, this completes
		the proof of  the lemma.
	\end{proof}

	\subsubsection{Upper bound of $\langle Q^{s,\gamma,\eta}(g,h), f\rangle$}
	We now turn  to prove the following upper bound  of $ Q^{s,\gamma,\eta}$.
	\begin{prop}\label{ubqepsilonnonsingular} It holds that
		\beno
		|\langle Q^{s,\gamma,\eta}(g,h), f\rangle| \lesssim s^{-1} \eta^{\gamma} |g|_{L^{1}_{|\gamma|+2s}}|h|_{s,\gamma/2}|f|_{s,\gamma/2}.
		\eeno
	\end{prop}
	\begin{proof} 
		By geometric decomposition in the phase space, we have
		$\langle Q^{s,\gamma,\eta}(g,h), f\rangle = \mathcal{D}_{1} + \mathcal{D}_{2}$,
		where
		\beno \mathcal{D}_{1} := \int b^{s}(\frac{u}{|u|}\cdot\sigma)|u|^{\gamma} \psi^{\eta}(u)g_{*} (T_{v_{*}}h)(u) ((T_{v_{*}}f)(u^{+})-(T_{v_{*}}f)(|u|\frac{u^{+}}{|u^{+}|})) \mathrm{d}\sigma \mathrm{d}v_{*} \mathrm{d}u, \\
		\mathcal{D}_{2}:=\int b^{s}(\frac{u}{|u|}\cdot\sigma)|u|^{\gamma} \psi^{\eta}(u)g_{*} (T_{v_{*}}h)(u)((T_{v_{*}}f)(|u|\frac{u^{+}}{|u^{+}|})- (T_{v_{*}}f)(u)) \mathrm{d}\sigma \mathrm{d}v_{*} \mathrm{d}u.
		\eeno
Observe that $\mathcal{D}_{1}$ and  $\mathcal{D}_{2}$ contain the radial and spherical parts respectively.
		
		{\it Step 1: Estimate of $\mathcal{D}_{1}$.}
		By   Lemma \ref{crosstermsimilar}, we have
		\beno
		|\mathcal{D}_{1}| &\lesssim& s^{-1} \eta^{\gamma} \int |g_{*}| (|W_{s}W_{\gamma/2}T_{v_{*}}h|_{L^{2}}+|W_{s}(D)W_{\gamma/2}T_{v_{*}}h|_{L^{2}})
		\\&& \quad  \quad \quad \times(|W_{s}W_{\gamma/2}T_{v_{*}}f|_{L^{2}}+|W_{s}(D)W_{\gamma/2}T_{v_{*}}f|_{L^{2}}) \mathrm{d}v_{*}.
		\eeno
		Since $\langle u+v\rangle^{s} \leq \langle u\rangle^{s}\langle v\rangle^{s}$, then for $u \in \mathbb{R}^{3}$, we have
		\ben \label{translation-out-cf}
		|W_{s}T_{u}f|_{L^{2}} \lesssim W_{s}(u) |W_{s}f|_{L^{2}}.
		\een
		For $u \in \mathbb{R}^{3}, l \in \mathbb{R}$, $
		(T_{u} W^{l})(v) = \langle v + u \rangle^{l} \lesssim C(l)\langle u \rangle^{|l|} \langle v \rangle^{l}.$
		As a result, we have
		\ben \label{translation-out-weight}
		|T_{u}f|_{L^{2}_{l}} \lesssim \langle u \rangle^{|l|} |f|_{L^{2}_{l}}.
		\een
		By \eqref{translation-out-cf} and \eqref{translation-out-weight}, we have
		\begin{eqnarray}\label{tvstartonovstar1}
			|W_{s}W_{\gamma/2}T_{v_{*}}h|_{L^{2}} \lesssim  W_{|\gamma|/2+s}(v_{*})|W_{s}W_{\gamma/2}h|_{L^{2}}.
		\end{eqnarray}
		Since $W_{s} \in S^{s}_{1,0}, W_{\gamma/2} \in S^{\gamma/2}_{1,0}$, by Lemma \ref{operatorcommutator1}, we have
		\begin{eqnarray}\label{tvstartonovstar2}
			|W_{s}(D)W_{\gamma/2}T_{v_{*}}h|_{L^{2}} &\lesssim& |W_{\gamma/2}W_{s}(D)T_{v_{*}}h|_{L^{2}} + |T_{v_{*}}h|_{L^{2}_{\gamma/2-1}}
			\\&=& |W_{\gamma/2}T_{v_{*}}W_{s}(D)h|_{L^{2}} + |T_{v_{*}}h|_{L^{2}_{\gamma/2-1}}
			\nonumber\\&\lesssim& W_{|\gamma|/2}(v_{*})(|W_{\gamma/2}W_{s}(D)h|_{L^{2}} + |h|_{L^{2}_{\gamma/2-1}})
			\nonumber\\&\lesssim& W_{|\gamma|/2}(v_{*})|W_{s}(D)W_{\gamma/2}h|_{L^{2}}, \nonumber
		\end{eqnarray}
		where we have used the fact  that $T_{v_{*}}$ and $W_{s}(D)$ are commutable, inequality \eqref{translation-out-weight} and Lemma \ref{operatorcommutator1}.
		By \eqref{tvstartonovstar1} and \eqref{tvstartonovstar2},
		we get
		\beno
		|\mathcal{D}_{1}| \lesssim s^{-1} \eta^{\gamma} | g|_{L^{1}_{|\gamma|+2s}}( |W_{s}(D)W_{\gamma/2}h|_{L^{2}} + |W_{s}W_{\gamma/2}h|_{L^{2}})
		( |W_{s}(D)W_{\gamma/2}f|_{L^{2}} + |W_{s}W_{\gamma/2}f|_{L^{2}}).
		\eeno

		{\it Step 2: Estimate of $\mathcal{D}_{2}$.}
		The spherical part $\mathcal{D}_{2}$ has some symmetry and essentially is an 2nd-order term.
		Let $u = r \tau$ and $\varsigma = \frac{\tau+\sigma}{|\tau+\sigma|}$, then $\frac{u}{|u|} \cdot \sigma = 2(\tau\cdot\varsigma)^{2} - 1$ and $|u|\frac{u^{+}}{|u^{+}|} = r \varsigma$. By the change of variable $(u, \sigma) \rightarrow (r, \tau, \varsigma)$, one has
		$
		\mathrm{d}\sigma \mathrm{d}u = 4  (\tau\cdot\varsigma) r^{2} \mathrm{d}r \mathrm{d}\tau \mathrm{d}\varsigma.
		$
		Then
		\beno 
		\mathcal{D}_{2} &=& 4 \int r^\gamma \psi^{\eta}(r) b^{s}(2(\tau\cdot\varsigma)^{2} - 1)(T_{v_*}h)(r\tau)\big((T_{v_*}f)(r\varsigma) - (T_{v_*}f) (r\tau)\big) (\tau\cdot\varsigma) r^{2} \mathrm{d}r \mathrm{d}\tau \mathrm{d}\varsigma \mathrm{d}v_{*}\\
		&=& 2 \int r^\gamma \psi^{\eta}(r) b^{s}(2(\tau\cdot\varsigma)^{2} - 1)\big((T_{v_*}h)(r\tau) - (T_{v_*}h) (r\varsigma)\big)\\ &&\times \big((T_{v_*}f)(r\varsigma) - (T_{v_*}f) (r\tau)\big) (\tau\cdot\varsigma) r^{2} \mathrm{d}r \mathrm{d}\tau \mathrm{d}\varsigma \mathrm{d}v_{*}\\&=&
		-\frac{1}{2}\int b^{s}(\frac{u}{|u|}\cdot\sigma)|u|^{\gamma} \psi^{\eta}(u)g_{*} ((T_{v_{*}}h)(|u|\frac{u^{+}}{|u^{+}|})-(T_{v_{*}}h)(u))\\ &&\times
		((T_{v_{*}}f)(|u|\frac{u^{+}}{|u^{+}|})-(T_{v_{*}}f)(u)) \mathrm{d}\sigma \mathrm{d}v_{*} \mathrm{d}u.
		\eeno
		Then by the Cauchy-Schwarz inequality and the fact $|u|^{\gamma} \psi^{\eta}(u) \lesssim \eta^{\gamma} \langle u \rangle^{\gamma}$, we have
		\beno
		|\mathcal{D}_{2}| &\lesssim& \eta^{\gamma} \big(\int b^{s}(\frac{u}{|u|}\cdot\sigma)\langle u \rangle^{\gamma}|g_{*}| ((T_{v_{*}}h)( |u|\frac{u^{+}}{|u^{+}|})-(T_{v_{*}}h)( u))^{2} \mathrm{d}\sigma \mathrm{d}v_{*} \mathrm{d}u\big)^{1/2}
		\\&& \times \big(\int b^{s}(\frac{u}{|u|}\cdot\sigma)\langle u \rangle^{\gamma}|g_{*}|
		((T_{v_{*}}f)(|u|\frac{u^{+}}{|u^{+}|})-(T_{v_{*}})f(u))^{2} \mathrm{d}\sigma \mathrm{d}v_{*} \mathrm{d}u\big)^{1/2}
		:= \eta^{\gamma} (\mathcal{D}_{2,1})^{1/2}(\mathcal{D}_{2,2})^{1/2}.
		\eeno
		Note that $\mathcal{D}_{2,1}$ and $\mathcal{D}_{2,2}$ have exactly the same structure. Hence it suffices to estimate  $\mathcal{D}_{2,1}$. For this, by Lemma \ref{spherical-part},
		we have
		\beno
		\mathcal{D}_{2,1} =  s^{-1} \int |g_{*}| \mathcal{A}^{s}(W_{\gamma/2}T_{v_{*}}h) \mathrm{d}v_{*} \lesssim s^{-1} 
		|g|_{L^{1}_{|\gamma|+2s}} |f|^{2}_{s,\gamma/2}.
		\eeno
		Here we have used
		\beno
		|W_{s}((-\Delta_{\mathbb{S}^{2}})^{1/2})W_{\gamma/2}T_{v_{*}}h|_{L^{2}}
		&=&	| W_{\gamma/2} W_{s}((-\Delta_{\mathbb{S}^{2}})^{1/2})
		T_{v_{*}}h|_{L^{2}}
		\\&\lesssim& W_{|\gamma|/2+s}(v_{*})
		(|W_{s}((-\Delta_{\mathbb{S}^{2}})^{1/2})W_{\gamma/2}h|_{L^{2}} + |h|_{H^{s}_{\gamma/2}}), 
		\eeno
		because $W_{s} \in S^{s}_{1,0}, W_{\gamma/2} \in S^{\gamma/2}_{1,0}$ are radial functions.

		Thus, 
		$$|\mathcal{D}_{2}| \lesssim s^{-1} \eta^{\gamma} (\mathcal{D}_{2,1})^{1/2}(\mathcal{D}_{2,2})^{1/2} \lesssim s^{-1} \eta^{\gamma} |g|_{L^{1}_{|\gamma|+2s}}|h|_{s,\gamma/2}|f|_{s,\gamma/2}.$$
		Combining the estimates on $\mathcal{D}_{1}$ and $\mathcal{D}_{2}$ completes the
		proof of the proposition.
	\end{proof}

	\subsection{{Upper bound the operator $ I^{s,\gamma,\eta}$}}
	In this subsection, we study the upper bound of $\langle I^{s,\gamma,\eta}(g,h;\beta), f\rangle$ where $I^{s,\gamma,\eta}(g,h;\beta)$ is defined by 
	\eqref{I-ep-ga-geq-eta-beta} with 	$B^{s,\gamma,\eta}$.
	For this, we first derive a slightly revised version of the
	cancellation lemma introduced in \cite{alexandre2000entropy} for the kernel
	$B^{s,\gamma,\eta}$. 
	\begin{lem} [Revised cancellation lemma for relative velocity away from origin] \label{cancellation-lemma-general-gamma-minus3-mu} Let $0 < s <1, \gamma \leq 0$. For  $|a| \leq -\gamma$ and $0<\eta \leq 1$,  we have
		\ben \label{B-s-gamma-eta-large-constant-result}
		|\int B^{s,\gamma, \eta}(|v-v_{*}|,\cos\theta) g_{*} (h^{\prime}-h) \mathrm{d}V| \leq C_{\gamma} \eta^{\gamma}|g|_{L^{1}_{a}} |h|_{L^{1}_{-a}}  .
		\een
		The constant $C_{\gamma}$  depends on $\gamma$ but is uniformly bounded for $-5 \leq \gamma \leq 0$.
	\end{lem}
	\begin{proof} By the cancellation lemma in \cite{alexandre2000entropy},  recalling $B^{s,\gamma,\eta}(|v-v_{*}|,\cos\theta) = |v-v_{*}|^{\gamma} b^{s}(\theta) \psi^{\eta}(|v-v_{*}|)$ where $\psi^{\eta} = 1- \psi_{\eta}$,
		we have
		\ben \label{cancelation-result-B-s-gamma-eta}
		\int B^{s,\gamma,\eta}(|v-v_{*}|,\cos\theta) g_{*} (h^{\prime}-h) \mathrm{d}V = \int S^{s,\gamma,\eta}(v-v_{*}) g_{*} h \mathrm{d}v_{*} \mathrm{d}v,
		\een
		where
		\ben \label{convolution-kernel-B-s-gamma-eta}
		S^{s,\gamma,\eta}(v-v_{*}) = \int |v-v_{*}|^{\gamma}b^{s}(\theta)\left( \cos^{-\gamma-3}\frac{\theta}{2} \psi^{\eta}(\frac{|v-v_{*}|}{\cos\frac{\theta}{2}})  - \psi^{\eta}(|v-v_{*}|) \right)
		\mathrm{d}\sigma.
		\een
	By 
	 using \eqref{mean-momentum-transfer}, we get
		\beno
		|S^{s,\gamma,\eta}(v-v_{*})| \lesssim \mathrm{1}_{|v-v_{*}|/\eta \geq 3/(4\sqrt{2})} |v-v_{*}|^{\gamma}.
		\eeno
		Observe that
		\beno
	\mathrm{1}_{|v-v_{*}|/\eta \geq 3/(4\sqrt{2})} |v-v_{*}|^{\gamma} \lesssim_{\gamma} \mathrm{1}_{1 \geq |v-v_{*}| \geq \eta \times 3/(4\sqrt{2}) } \eta^{\gamma} \langle v_{*} \rangle^{a} \langle v \rangle^{-a} +\mathrm{1}_{|v-v_{*}| \geq 1} \langle v_{*} \rangle^{a} \langle v \rangle^{-a},
		\eeno
		which yields \eqref{B-s-gamma-eta-large-constant-result} and this completes the proof of
		the lemma.
	\end{proof}
	
	Recalling the decomposition \eqref{functional-N-decomposition},
	by Lemma \ref{cancellation-lemma-general-gamma-minus3-mu} and Proposition \ref{ubqepsilonnonsingular}, we have
	the following  upper bound estimate on $\mathcal{N}^{s,\gamma,\eta}$.
	
	\begin{prop}\label{functional-N-up-geq-eta} It holds uniformly for $1/4 \leq a \leq 1/2$ that
		\beno \mathcal{N}^{s,\gamma,\eta}(\mu^{a},f) \lesssim s^{-1} \eta^{\gamma} |f|^{2}_{s,\gamma/2}.
		\eeno
	\end{prop}

	The following lemma is about an  integral 
	over the sphere $\mathbb{S}^{2}$.
	\begin{lem}\label{symbol}
		Recall $b^{s}(\theta) = (1-s)\sin^{-2-2s}\frac{\theta}{2} \mathrm{1}_{0 \leq \theta \leq \pi/2}$ for $0< s <1$. Denote
		\beno A_{s}(\xi):= \int_{\mathbb{S}^2} b^{s}(\theta)\min\{ |\xi|^2\sin^2(\theta/2),1\} \mathrm{d}\sigma.\eeno
		Then 
		\beno A_{s}(\xi)  = \mathrm{1}_{|\xi| \leq \sqrt{2}} 4 \pi \times 2^{s-1} |\xi|^{2} + \mathrm{1}_{|\xi| > \sqrt{2}} 4 \pi \times (|\xi|^{2s} + \f{1-s}{s}(|\xi|^{2s} - 2^{s})).\eeno
		Hence, 
		\beno A_{s}(\xi) + 1  \gtrsim  \langle \xi \rangle^{2s}, 
		\quad A_{s}(\xi)  \lesssim \mathrm{1}_{|\xi| \leq \sqrt{2}} |\xi|^{2} + \mathrm{1}_{|\xi| > \sqrt{2}} \max\{\frac{1-s}{s}, 1\} |\xi|^{2s} \lesssim s^{-1}  \langle \xi \rangle^{2s}.\eeno
	\end{lem}
	\begin{proof}  Using  $\mathrm{d}\sigma = 4 \sin(\theta/2) \mathrm{d} \sin(\theta/2) \mathrm{d} \mathbb{S}$,
		we have 
		\beno A_{s}(\xi)=8 \pi (1-s) \int_0^{\pi/2} \sin^{-1-2s}(\theta/2) \min\{|\xi|^2\sin^2(\theta/2),1\} 
		\mathrm{d} \sin(\theta/2). \eeno
		By the change of variable: $t=\sin(\theta/2)$, we get
		\beno A_{s}(\xi) = 8 \pi  (1-s)  \int_0^{\sqrt{2}/{2}} t^{-1-2s} \min\{ |\xi|^2 t^2,1\}\mathrm{d}t.
		\eeno
		When $|\xi| \leq \sqrt{2}$, it holds
		\beno A_{s}(\xi) = 8 \pi  (1-s) |\xi|^2 \int_0^{\sqrt{2}/{2}} t^{1-2s} \mathrm{d}t = 8 \pi  (1-s)  \times \frac{2^{s-1}}{2-2s}|\xi|^{2}
		= 4 \pi    \times 2^{s-1} |\xi|^{2}.
		\eeno
		When $|\xi| > \sqrt{2}$, it holds
		\beno A_{s}(\xi) &=& 8 \pi (1-s) |\xi|^2 \int_0^{|\xi|^{-1}} t^{1-2s} \mathrm{d}t +
		8 \pi (1-s) \int_{|\xi|^{-1}}^{\sqrt{2}/{2}} t^{-1-2s} \mathrm{d}t
		\\&=& 8 \pi (1-s) \times (\frac{|\xi|^{2s}}{2-2s} + \frac{|\xi|^{2s} - 2^{s}}{2s}) =4 \pi \times (|\xi|^{2s} + \f{1-s}{s}(|\xi|^{2s} - 2^{s})).
		\eeno
		Combining the estimates in  two cases completes the proof of the lemma.
	\end{proof}

As for Proposition 4.2 in \cite{he2021boltzmann} and Proposition 3.2 in \cite{duan2023solutions}, by applying
Proposition \ref{ubqepsilonnonsingular} and
Lemma \ref{cancellation-lemma-general-gamma-minus3-mu},
we can derive  the upper bound of $\langle I^{s,\gamma,\eta}(g,h;\beta), f\rangle$ stated in the following  proposition. Note that the factor $s^{-1}$ comes from Lemma \ref{symbol} and the factor $\eta^{\gamma}$ comes from
	\ben \label{v-minus-vstar-lower-no-singularity}
|v-v_{*}|^{\gamma} \psi^{\eta}(|v-v_{*}|) \lesssim \eta^{\gamma} \langle v -v_{*} \rangle^{\gamma}. \een

	\begin{prop}\label{upforI-ep-ga-et}
		It holds that
		\beno |\langle I^{s,\gamma,\eta}(g,h; \beta), f\rangle|  \lesssim s^{-1} \eta^{\gamma} |g|_{L^{2}}|h|_{s,\gamma/2}|W_{s}f|_{L^{2}_{\gamma/2}}.\eeno
		The  constant associated to $\lesssim$ may depend on $|\beta|$ but not on $s,\gamma$.
	\end{prop}

	\subsection{Upper bound of $\langle \Gamma^{s,\gamma, \eta}(g,h;\beta), f\rangle$} 	
	
	Note that
	\ben \label{Gamma-into-Q-I-geq-eta}
	\Gamma^{s,\gamma,\eta}(g,h;\beta) =   Q^{s,\gamma,\eta}(g\partial_{\beta}\mu^{1/2},h) +
	I^{s,\gamma,\eta}(g,h;\beta).
	\een
	By Propositions \ref{ubqepsilonnonsingular} and \ref{upforI-ep-ga-et},
	we have the following estimate.
	\begin{prop}\label{upGammagh-geq-eta}
		It holds that
		\beno
		|\langle \Gamma^{s,\gamma,\eta}(g,h; \beta), f\rangle| \lesssim s^{-1} \eta^{\gamma} |g|_{L^{2}}|h|_{s,\gamma/2}|f|_{s,\gamma/2}.
		\eeno
	\end{prop}

	\subsection{Upper bound of  $\langle \mathcal{L}^{s,\gamma,\eta}f , h\rangle$ } 
	Recall that
	\beno
	\mathcal{L}^{s,\gamma,\eta}(f; \beta_0, \beta_1)
	= \mathcal{L}^{s,\gamma,\eta}_{1}(f; \beta_0, \beta_1) + \mathcal{L}^{s,\gamma,\eta}_{2}(f; \beta_0, \beta_1),
	\\
	\mathcal{L}^{s,\gamma,\eta}_{1}(f; \beta_0, \beta_1) := - \Gamma^{s,\gamma,\eta}(\pa_{\beta_{1}}\mu^{1/2}, f;\beta_{0}), \quad
	\mathcal{L}^{s,\gamma,\eta}_{1}(f; \beta_0, \beta_1) := - \Gamma^{s,\gamma,\eta}(f, \pa_{\beta_{1}}\mu^{1/2};\beta_{0}).
	\eeno
	If $|\beta_0| = |\beta_1| = 0$,  we have 
	\beno
	\mathcal{L}^{s,\gamma,\eta}f = \mathcal{L}^{s,\gamma,\eta}(f; 0, 0), \quad	\mathcal{L}^{s,\gamma,\eta}_{1}f =  \mathcal{L}^{s,\gamma,\eta}_{1}(f; 0, 0), \quad
	\mathcal{L}^{s,\gamma,\eta}_{2}f = \mathcal{L}^{s,\gamma,\eta}_{2}(f; 0, 0).
	\eeno
	
	By using Proposition \ref{upGammagh-geq-eta}, we have the following estimates.
	
	\begin{prop}\label{geq-eta-part-l1} It holds that
		\beno |\langle \mathcal{L}^{s,\gamma,\eta}_{1}(f; \beta_0, \beta_1), h\rangle| \lesssim  s^{-1} \eta^{\gamma} |h|_{s,\gamma/2}|f|_{s,\gamma/2}.  \eeno
	\end{prop}

	%
	%
	%

	\begin{prop} \label{geq-eta-part-l2}  It holds that
		\ben \label{full-L2} |\langle \mathcal{L}^{s,\gamma,\eta}_{2}(f; \beta_0, \beta_1), h\rangle| \lesssim  \eta^{\gamma} |\mu^{1/32}f|_{L^{2}}|\mu^{1/32}h|_{L^{2}}.  \een
	\end{prop}
	\begin{proof} Since the proof is similar to the one  of Proposition \ref{less-eta-part-l2},  we omit the details. Here, we do not have  the factor $s^{-1}$  because of the cancellation Lemma \ref{cancellation-lemma-general-gamma-minus3-mu}.
	\end{proof}

	Propositions  \ref{geq-eta-part-l1} and \ref{geq-eta-part-l2} give the following estimate.
	
	\begin{prop}\label{geq-eta-part-l} It holds that
		\beno |\langle \mathcal{L}^{s,\gamma,\eta}(f; \beta_0, \beta_1), h\rangle| \lesssim   s^{-1} \eta^{\gamma} |h|_{s,\gamma/2}|f|_{s,\gamma/2}.  \eeno
	\end{prop}

	\subsection{Upper bound of $Q^{s,\gamma}, I^{s,\gamma}, \Gamma^{s,\gamma}, \mathcal{L}^{s,\gamma}$} 	
	By applying Propositions \ref{ubqepsilon-singular} and \ref{ubqepsilonnonsingular}, 
	by taking $\eta=1$ and $C_{\delta, s,\gamma} = \delta^{-1/2} s^{-1} (\gamma+2s+3)^{-1}$,
	we have the following upper bound estimate on $Q^{s,\gamma}$.
	\begin{thm}\label{Q-full-up-bound} 
		Let $l_1, l_2, l_3 \in \R$ satisfying $l_1 + l_2 + l_3 = 0$. Let $0<\delta \leq \f12$, for any combination $a_1 , a_2 \geq s$ satisfying the constraint
		$a_1 + a_2 = s+\f32 + \delta$, 
		it holds that
		\beno
		|\langle Q^{s,\gamma}(g,h), f\rangle| \lesssim C_{\delta, s,\gamma}
		|g|_{H^{a_{1}}_{l_1}} |h|_{H^{a_{2}}_{l_2}}
		|f|_{H^{s}_{l_3}} + s^{-1}
		|g|_{L^{1}_{|\gamma|+2s}}|h|_{s,\gamma/2}|f|_{s,\gamma/2}.
		\eeno
	\end{thm}
	
	By applying Propositions \ref{I-less-eta-upper-bound} and \ref{upforI-ep-ga-et} and
	taking $\eta=1$,
	we have  the following upper bound estimate on  $I^{s,\gamma}$.
	
	\begin{thm}\label{upforI-total} Let $0<\delta \leq \f12$, let $(a_{1},a_{2})=(\f{3}{2} + \delta, s)$ or $(0, \f{3}{2} + \delta)$. Then
		\beno
		|\langle I^{s,\gamma}(g,h; \beta) , f \rangle|  \lesssim C_{\delta, s,\gamma} |g|_{H^{a_{1}}_{-l}}|h|_{H^{a_{2}}_{-l}}
		|f|_{H^{s}_{-l}} + s^{-1} |g|_{L^{2}}|h|_{s,\gamma/2}|f|_{s,\gamma/2}.
		\eeno
	\end{thm}
	
	Also Propositions \ref{Gamma-less-eta-upper-bound}  and \ref{upGammagh-geq-eta} with  $\eta=1$ give 
	the following upper bound estimate on  $\Gamma^{s,\gamma}$.	
	
	\begin{thm}\label{Gamma-full-up-bound} For $0<\delta \leq \f12$, let  $(a_{1},a_{2})=(\f{3}{2} + \delta, s)$ or $(s, \f{3}{2} + \delta)$. Then 
		\ben
		\label{Gamma-full-on-g}
		|\langle \Gamma^{s,\gamma}(g,h;\beta), f\rangle| \lesssim_l C_{\delta, s,\gamma}|g|_{H^{a_{1}}_{-l}}|h|_{H^{a_{2}}_{-l}}
		|f|_{H^{s}_{-l}}+ s^{-1} |g|_{L^{2}}|h|_{s,\gamma/2}|f|_{s,\gamma/2}	.
		\een
	\end{thm}

	And Prop. \ref{less-eta-part-l}  and \ref{geq-eta-part-l} with  $\eta=1$ give the following upper bound estimate on  $\mathcal{L}^{s,\gamma}$.	
	
	\begin{prop}\label{part-l} Set
			\ben \label{def-C-s-gamma}
		C_{s, \gamma} = s^{-1} (\gamma+2s+3)^{-1}.
		\een
		It holds that
		\beno |\langle \mathcal{L}^{s,\gamma}(f; \beta_0, \beta_1), h\rangle| \lesssim   C_{s,\gamma} |h|_{s,\gamma/2}|f|_{s,\gamma/2}.  \eeno
	\end{prop}
	
	The following result will be used in Section \ref{proof-main-theorem} to obtain  dissipation 
	estimate on the macroscopic component.
	\begin{prop}\label{Gamma-g-h-mu} Let $P$ be a polynomial function.
		For any combination $a_1 , a_2 \geq 0$ satisfying the constraint
		$a_1 + a_2 = s$, 
		it holds that
		\ben
		\label{Gamma-g-h-mu-up-total-s}
		|\langle \Gamma^{s,\gamma}(g,h), \mu^{\f12}P\rangle| \lesssim C_{ s,\gamma}|\mu^{\f14}g|_{H^{a_1}} |\mu^{\f14}h|_{H^{a_2}} + s^{-1} |\mu^{\f14}g|_{L^{2}} |\mu^{\f14}h|_{L^{2}}.
		\een
	\end{prop}
	\begin{proof} First note that
		\beno
		\langle \Gamma^{s,\gamma}(g,h), \mu^{\f12}P\rangle = \langle Q^{s,\gamma}(\mu^{\f12}g, \mu^{\f12}h), P\rangle = \langle Q^{s,\gamma}_{1}(\mu^{\f12}g, \mu^{\f12}h), P\rangle + \langle Q^{s,\gamma,1}(\mu^{\f12}g, \mu^{\f12}h), P\rangle.
		\eeno
		Applying Proposition \ref{ubqepsilon-singular} with $a_3 = s+ \f32+ \delta$ and taking $l_3$ small enough relative to the degree of $P$,  we get
		\beno
		|\langle Q^{s,\gamma}_{1}(\mu^{\f12}g, \mu^{\f12}h), P\rangle|	\lesssim	 C_{ s,\gamma}|\mu^{\f14}g|_{H^{a_1}} |\mu^{\f14}h|_{H^{a_2}}.
		\eeno
		By using \eqref{Taylor1} to $P'-P$ and 
		\eqref{cancell1}, thanks to the factor $\mu^{\f12}\mu^{\f12}_{*}$,
		we have
		\beno
		|\langle Q^{s,\gamma,1}(\mu^{\f12}g, \mu^{\f12}h), P\rangle|	\lesssim s^{-1}|\mu^{\f14}g|_{L^{2}} |\mu^{\f14}h|_{L^{2}}.
		\eeno
		And this completes the proof of the proposition.
	\end{proof}

	\section{Coercivity estimate} \label{coercivity-spectral}
	In this section, we will prove coercivity estimate of the linear operator $\mathcal{L}^{s,\gamma,\eta}$ for some $\eta>0$. This is a linear counterpart of the famous H-theorem near Maxwellians.
	Unless otherwise specified, the parameter range is $-5 \leq \gamma \leq 0, 0<s<1$. The parameter $\gamma$ actually can tend to $-\infty$ because we consider in the domain $|v-v_{*}| \gtrsim \eta>0$.
	
	The proof contains two parts. One is a rough coercivity estimate capturing the  norm $|\cdot|_{s,\gamma/2}$ with a lower order correction norm $|\cdot|_{L^{2}_{\gamma/2}}$. The other is a spectrum-gap type estimate to recover the lower order norm $|\cdot|_{L^{2}_{\gamma/2}}$. Accordingly, we divide this section into two subsections.

	\subsection{Rough coercivity estimate}
	In this subsection, we will prove the rough coercivity estimate of $\mathcal{L}^{s, \gamma, \eta}$ for small $\eta > 0$ in Theorem \ref{strong-coercivity}.
	The  strategy  relies on the following relation (see \eqref{L-dominate-lower-bound} in the proof of Theorem \ref{strong-coercivity}):
	\ben \label{equivalence-relation} \langle \mathcal{L}^{s, \gamma, \eta} f,f \rangle+ \eta^{\gamma}
	|f|_{L^2_{\gamma/2}}^2 \gtrsim \mathcal{N}^{s, \gamma, \eta}(\mu^{1/2},f) + \mathcal{N}^{s, \gamma, \eta}(f,\mu^{1/2}), \een
	where the functional $\mathcal{N}^{s, \gamma, \eta}$ is defined by 
	\ben  \label{definition-of-N-s-ga-eta}
	\mathcal{N}^{s, \gamma, \eta}(g,h) := \int B^{s, \gamma, \eta}
	g^{2}_{*} (h^{\prime}-h)^{2} \mathrm{d}\sigma 
	\mathrm{d}v \mathrm{d}v_{*}.\een
	If $\eta=0$, then $\psi^{\eta} = 1$ and we write
	$\mathcal{N}^{s,\gamma} = \mathcal{N}^{s,\gamma,0}$. If $\gamma = \eta=0$, we write
	$\mathcal{N}^{s} = \mathcal{N}^{s,0,0}$ for simplicity.
	Thanks to \eqref{equivalence-relation},
	to obtain the coercivity estimate of $\mathcal{L}^{s, \gamma, \eta}$, it suffices to estimate from below the two functionals
	$\mathcal{N}^{s, \gamma, \eta}(\mu^{1/2},f)$  and $ \mathcal{N}^{s, \gamma, \eta}(f,\mu^{1/2})$.


	\subsubsection{Gain of weight from $\mathcal{N}^{s,\gamma,\eta}(f,\mu^{1/2})$} \label{gain-weight} 
	The functional  $\mathcal{N}^{s,\gamma,\eta}(f,\mu^{1/2})$ produces weight $W_{s}$ in the phase space. 
	\begin{prop}\label{lowerboundpart1} Let $0 \leq \eta \leq 1$, then
		\beno
		\mathcal{N}^{s,\gamma,\eta}(f,\mu^{1/2}) + |f|^{2}_{L^{2}_{\gamma/2}} \geq  C|f|^{2}_{L^{2}_{\gamma/2+s}},
		\eeno
		where $C>0$ is a generic constant.
	\end{prop}
	\begin{proof}  Let $0<\delta \leq 1$.
		We consider the set $A(\delta) := \{(v_{*},v,\sigma): |v| \leq 2, |v_{*}| \geq 4,  \sin(\theta/2)
		\leq \delta|v_{*}|^{-1}\}$. Since $|v-v_{*}| \geq 2 \geq 4/3$ in the set $A(\delta)$,
		we get
		\ben
		\label{reduce-to-set-A}
		\mathcal{N}^{s,\gamma,\eta}(f,\mu^{1/2}) \geq \int B^{s,\gamma} \mathrm{1}_{A(\delta)} f_{*}^{2}
		((\mu^{1/2})^{\prime}-\mu^{1/2})^{2} \mathrm{d}V.
		\een
		Note that $\nabla \mu^{1/2} = -\frac{\mu^{1/2}}{2} v$ and $\nabla^{2} \mu^{1/2} = \frac{\mu^{1/2}}{4} (-2I+v \otimes v)$. By Taylor expansion \eqref{Taylor1},
	using the basic inequality $(A-B)^{2} \geq A^{2}/2 - B^{2}$, we have
		\beno
		(\mu^{1/2}(v^{\prime}) - \mu^{1/2}(v))^{2} \geq \frac{\mu(v)}{8} |v \cdot (v^{\prime}-v)|^{2} - \int_{0}^{1}  |(\nabla^{2} \mu^{1/2}) (v(\kappa))|^{2}|v^{\prime}-v|^{4} \mathrm{d}\kappa.
		\eeno
		Plugging this into \eqref{reduce-to-set-A}, we get
		\begin{eqnarray}\label{vsmallvstarsmall}
			\mathcal{N}^{s,\gamma,\eta}(f,\mu^{1/2}) &\geq& \frac{1}{8}\int B^{s,\gamma} \mathrm{1}_{A(\delta)} \mu(v)|v \cdot (v^{\prime}-v)|^{2} f_{*}^{2}  \mathrm{d}V
			\\&& -  \int B^{s,\gamma} \mathrm{1}_{A(\delta)} |(\nabla^{2} \mu^{1/2}) (v(\kappa))|^{2}|v^{\prime}-v|^{4} f_{*}^{2}  \mathrm{d}V \mathrm{d}\kappa \nonumber
			\\&:=& \frac{1}{8}\mathcal{I}_{1}^{s,\gamma} (\delta) - \mathcal{I}_{2}^{s,\gamma} (\delta). \nonumber
		\end{eqnarray}
		
		To estimate $\mathcal{I}_{1}^{s,\gamma} (\delta)$, for fixed $v, v_*$, we introduce an orthogonal basis $(h^{1}_{v,v_{*}},h^{2}_{v,v_{*}}, \frac{v-v_{*}}{|v-v_{*}|})$ such that $\mathrm{d}\sigma= \sin\theta \mathrm{d}\theta \mathrm{d}\phi$. Then one has
		\beno
		\frac{v^{\prime}-v}{|v^{\prime}-v|} = \cos\frac{\theta}{2}\cos\phi h^{1}_{v,v_{*}} + \cos\frac{\theta}{2}\sin\phi h^{2}_{v,v_{*}} -\sin\frac{\theta}{2} \frac{v-v_{*}}{|v-v_{*}|},
		\eeno
		and
		\beno
		\frac{v}{|v|} = c_{1} h^{1}_{v,v_{*}} + c_{2} h^{2}_{v,v_{*}} + c_{3} \frac{v-v_{*}}{|v-v_{*}|},
		\eeno
		where $c_{3}=\frac{v}{|v|}\cdot \frac{v-v_{*}}{|v-v_{*}|}$ and $c_{1}, c_{2}$ are  constants independent of $\theta$ and $\phi$. Then we have
		\beno
		\frac{v}{|v|} \cdot \frac{v^{\prime}-v}{|v^{\prime}-v|}  = c_{1}\cos\frac{\theta}{2}\cos\phi + c_{2}\cos\frac{\theta}{2}\sin\phi - c_{3}\sin\frac{\theta}{2}.
		\eeno
		Thus
		\beno
		|\frac{v}{|v|} \cdot \frac{v^{\prime}-v}{|v^{\prime}-v|}|^{2}  &=& c^{2}_{1}\cos^{2}\frac{\theta}{2}\cos^{2}\phi + c^{2}_{2}\cos^{2}\frac{\theta}{2}\sin^{2}\phi + c^{2}_{3}\sin^{2}\frac{\theta}{2}
		\\ && + 2c_{1}c_{2}\cos^{2}\frac{\theta}{2}\cos\phi\sin\phi - 2c_{3}\cos\frac{\theta}{2}\sin\frac{\theta}{2}(c_{1}\cos\phi + c_{2}\sin\phi).
		\eeno
		Since $|v^{\prime}-v| = |v-v_{*}| \sin \frac{\theta}{2}$ and $ \cos^{2}\frac{\theta}{2} \geq \f12$,  by integrating with respect to $\sigma$ and using $b^{s}(\theta) = (1-s) \sin^{-2-2s} \frac{\theta}{2} \mathrm{1}_{0 \leq  \theta \leq  \pi/2}$,
		we have
		\beno
		\int b^{s}(\theta)\mathrm{1}_{A(\delta)}|v \cdot (v^{\prime}-v)|^{2}\mathrm{d}\sigma &=& 4 \int_{0}^{\pi}\int_{0}^{2\pi}b^{s}(\theta) \sin \frac{\theta}{2} \sin\theta \mathrm{1}_{A(\delta)}|v \cdot (v^{\prime}-v)|^{2} \mathrm{d}\phi \mathrm{d}\sin \frac{\theta}{2}
		\\ &\geq& 2 \pi(c^{2}_{1}+c^{2}_{2})|v|^{2}|v-v_{*}|^{2} \int_{0}^{\pi} b^{s}(\theta) \sin^{3} \frac{\theta}{2} \mathrm{1}_{A(\delta)} \mathrm{d}\sin \frac{\theta}{2}
		\\ &\gtrsim&
		\delta^{2-2s}(c^{2}_{1}+c^{2}_{2})|v_{*}|^{2s-2}|v|^{2}|v-v_{*}|^{2}\mathrm{1}_{B(\delta)},
		\eeno
		where $B(\delta) = \{(v_{*},v):  |v_{*}| \geq 4, |v| \leq 2\}$.
		Note that
		\beno (1-c_3^2)|v-v_{*}|^{2}=(1-(\frac{v}{|v|}\cdot\frac{v_{*}}{|v_{*}|})^{2})|v_{*}|^{2}\eeno
		gives
		\beno
		\int b^{s}(\theta)\mathrm{1}_{A(\delta)}|v \cdot (v^{\prime}-v)|^{2}\mathrm{d}\sigma  \gtrsim
		\delta^{2-2s}(1-(\frac{v}{|v|}\cdot\frac{v_{*}}{|v_{*}|})^{2})|v_{*}|^{2s}|v|^{2}\mathrm{1}_{B(\delta)}.
		\eeno
		Plugging this  estimate in the definition of
		$\mathcal{I}_{1}^{s,\gamma} (\delta)$, we get
		\beno
		\mathcal{I}_{1}^{s,\gamma} (\delta) \gtrsim \int \delta^{2-2s}(1-(\frac{v}{|v|}\cdot\frac{v_{*}}{|v_{*}|})^{2})|v_{*}|^{2s}|v|^{2}\mathrm{1}_{B(\delta)} |v-v_{*}|^{\gamma} \mu(v) f_{*}^{2}\mathrm{d}v \mathrm{d}v_{*}.
		\eeno

		Note that in the region $B(\delta)$, one has
		\ben \label{v-vstar-is-vstar-norm}
		\frac{1}{2} | v_* | \leq |v-v_{*}| \leq \frac{3}{2} | v_* |.
		\een
		We then obtain
		\beno
		\mathcal{I}_{1}^{s,\gamma} (\delta) &\gtrsim& \delta^{2-2s}
		\int (1-(\frac{v}{|v|}\cdot\frac{v_{*}}{|v_{*}|})^{2})|v_{*}|^{\gamma+2s} |v|^{2}\mathrm{1}_{B(\delta)} \mu(v) f_{*}^{2}\mathrm{d}v \mathrm{d}v_{*}
		\\&\gtrsim& \delta^{2-2s}
		\int  |v_{*}|^{\gamma+2s}\mathrm{1}_{ |v_{*}| \geq 4}  f_{*}^{2} \mathrm{d}v_{*},
		\eeno
		where we have used the fact  that $\int (1-(\frac{v}{|v|}\cdot\frac{v_{*}}{|v_{*}|})^{2}) |v|^{2} \mu(v) \mathrm{1}_{|v|\leq 2}\mathrm{d}v > 0$ which is independent of $v_{*}$.
		
		We now turn to estimate $\mathcal{I}_{2}^{s,\gamma} (\delta)$.
		By \eqref{v-vstar-is-vstar-norm} and $|v(\kappa)-v_{*}| \leq |v-v_{*}|$, we have
		\beno
		\mathrm{1}_{A(\delta)} \leq \mathrm{1}_{|v_{*}| \geq 4} \mathrm{1}_{\sin(\theta/2) \leq \f{3}{2}\delta|v-v_{*}|^{-1}}.
		\eeno 
		Recalling $B^{s,\gamma} = |v-v_{*}|^{\gamma}b^{s}(\theta), |v^{\prime}-v| = |v-v_{*}|\sin(\theta/2)$, by \eqref{v-vstar-is-vstar-norm}, and by using $|\nabla^{2} \mu^{1/2}| \lesssim \mu^{1/4}$ and
		the change of variable $v \to v(\kappa)$ in Lemma \ref{usual-change},
		we have
		\beno
		\mathcal{I}_{2}^{s,\gamma} (\delta) &=& \int |v-v_{*}|^{\gamma} b^{s}(\theta) \mathrm{1}_{A(\delta)} |(\nabla^{2} \mu^{1/2}) (v(\kappa))|^{2}|v^{\prime}-v|^{4} f_{*}^{2}  \mathrm{d}V \mathrm{d}\kappa
		\\&\lesssim& \int b^{s}(\theta) \sin^4(\theta/2)
		\mathrm{1}_{|v_{*}| \geq 4} \mathrm{1}_{\sin(\theta/2) \leq \f{3}{2}\delta|v-v_{*}|^{-1}} |v-v_{*}|^{4-2s} \mu^{1/2} (v(\kappa)) |v_{*}|^{\gamma+2s} f_{*}^{2}  \mathrm{d}V \mathrm{d}\kappa
		\\&=& \int b^{s}(\theta) \sin^4(\theta/2)
		\mathrm{1}_{|v_{*}| \geq 4} \mathrm{1}_{\sin(\theta/2) \leq \f{3}{2}\delta|v-v_{*}|^{-1}\psi^{-1}_{\kappa}(\theta)} |v-v_{*}|^{4-2s} \psi^{7-2s}_{\kappa}(\theta) \mu^{1/2}(v)|v_{*}|^{\gamma+2s} f_{*}^{2}  \ \mathrm{d}V \mathrm{d}\kappa
		\\&\lesssim& \delta^{4-2s} \int 
		\mathrm{1}_{|v_{*}| \geq 4}   \mu^{1/2}(v)|v_{*}|^{\gamma+2s} f_{*}^{2}  
		\mathrm{d}v \mathrm{d}v_{*} 	\lesssim
		\delta^{4-2s}
		\left(  \int   \mathrm{1}_{|v_{*}| \geq 4}|v_{*}|^{\gamma+2s} f_{*}^{2}  \mathrm{d}v_{*} \right).
		\eeno

		Combining  the estimate on $\mathcal{I}_{1}^{s,\gamma} (\delta)$ and $\mathcal{I}_{2}^{s,\gamma} (\delta)$ gives
		\beno
		\mathcal{N}^{s,\gamma,\eta}(f,\mu^{1/2}) \geq \left(C_{1} -C_{2} \delta^{2}\right) \delta^{2-2s} \int \mathrm{1}_{|v_{*}| \geq 4}|v_{*}|^{\gamma+2s} f_{*}^{2} \mathrm{d}v_{*},
		\eeno
		for some generic constants $C_{1},C_{2}>0$. By choosing $\delta$ such that $C_{2} \delta^{2} = C_{1}/2$, and observing
		$|v_{*}|^{\gamma+2s} \sim \langle v_{*} \rangle^{\gamma+2s}$ for $|v_{*}| \geq 4$,
		we get
		\ben \label{lowerboundvstarsmall}
		\mathcal{N}^{s,\gamma,\eta}(f,\mu^{1/2}) \gtrsim 
		\int \mathrm{1}_{|v_{*}| \geq 4} \langle v_{*} \rangle^{\gamma+2s} f_{*}^{2} \mathrm{d}v_{*} =  (|f|_{L^{2}_{\gamma/2+s}} - \int \mathrm{1}_{|v_{*}| < 4} \langle v_{*} \rangle^{\gamma+2s} f_{*}^{2} \mathrm{d}v_{*}).
		\een
		If $|v_{*}| \leq 4$, then $\langle v_{*} \rangle^{\gamma+2s} \sim \langle v_{*} \rangle^{\gamma}$. Then the proof of the proposition is completed. 
	\end{proof}
	
	In the following, we focus on gain of regularity from  $\mathcal{N}^{s,\gamma,\eta}(\mu^{\f12},f)$. The strategy can be stated as follows.
	\begin{enumerate}
		\item Gain of regularity from $\mathcal{N}^{s}(g,f)$;
		
		\item  Gain of regularity from $\mathcal{N}^{s,0,\eta}(g,f)$ by reducing to  $\mathcal{N}^{s}(g,f)$;
		
		\item  Gain of regularity from $\mathcal{N}^{s,\gamma,\eta}(g,f)$ by reducing  to $\mathcal{N}^{s,0,\eta}(g,f)$.
	\end{enumerate}

	\subsubsection{Gain of  regularity from $\mathcal{N}^{s}(g,f)$.} We derive Sobolev regularity from $\mathcal{N}^{s}(g,f)$ by the following argument used  in \cite{alexandre2000entropy}. For $g\ge0$ with $|g|_{L^1}\ge\delta>0$ and $|g|_{L^1_1}\le \lambda < \infty$, there exists a  constant  $C(\delta, \lambda)>0$ such that
	\ben \label{Sobolev-regularity}
	\int b(\cos\theta)g_*(f'-f)^2\mathrm{d}\sigma \mathrm{d}v_{*}\mathrm{d}v+|f|_{L^2}^2\ge C(\delta, \lambda)|a(D)f|_{L^2}^2, \een
	where $a(\xi):= \int b(\f{\xi}{|\xi|}\cdot \sigma)\min\{ |\xi|^2\sin^2(\theta/2),1\} \mathrm{d}\sigma + 1$. By applying \eqref{Sobolev-regularity} to the angular function $b^{s}$
	and  using
	Lemma \ref{symbol}, we have the following lemma.
	
	\begin{lem}\label{lowerboundpart1-general-g}
		Let $g$ be a function such that $|g|_{L^{2}} \geq \delta >0, |g|_{L^{2}_{1/2}} \leq \lambda < \infty$, then there is a  constant  $C(\delta, \lambda)>0$ such that
		\ben\label{sobolev-regularity-general-g}
		\mathcal{N}^{s}(g,f)+  |f|^{2}_{L^{2}} \geq  C(\delta, \lambda) |f|^{2}_{H^{s}} .
		\een
	\end{lem}

	%

	We now extract
	the anisotropic norm  $|W_{s}((-\Delta_{\mathbb{S}^{2}})^{1/2})f|^{2}_{L^{2}_{\gamma/2}}$ from $\mathcal{N}^{s}(g,f)$ by Bobylev's formula and the upper bound of the radial part.

	\begin{lem}\label{lowerboundpart2-general-g}
		It holds that
		\ben\label{anisotropic-regularity-general-g} 
		\mathcal{N}^{s}(g,f) + |g|^{2}_{L^{2}_{s}}(|W_{s}(D)f|^{2}_{L^{2}}+ |W_{s}f|^{2}_{L^{2}}) \gtrsim  |g|^{2}_{L^{2}} |W_{s}((-\Delta_{\mathbb{S}^{2}})^{1/2})f|^{2}_{L^{2}}.
		\een
	\end{lem}
	\begin{proof}
		By Bobylev's formula, we have
		\beno
		\mathcal{N}^{s}(g,f) &=& \frac{1}{(2\pi)^{3}}\int  b^{s}(\frac{\xi}{|\xi|} \cdot \sigma)(\widehat{g^{2}}(0)|\hat{f}(\xi) - \hat{f}(\xi^{+})|^{2} + 2\Re((\widehat{g^{2}}(0) - \widehat{g^{2}}(\xi^{-}))\hat{f}(\xi^{+})\bar{\hat{f}}(\xi)) \mathrm{d}\sigma \mathrm{d}\xi
		\\ &:=& \frac{|g|^{2}_{L^{2}}}{(2\pi)^{3}}\mathcal{I}_{1} + \frac{2}{(2\pi)^{3}}\mathcal{I}_{2},
		\eeno
		where $\xi^{+} = \frac{\xi+|\xi|\sigma}{2}$ and $\xi^{-} = \frac{\xi-|\xi|\sigma}{2}$.
		Note that $\widehat{g^{2}}(0) - \widehat{g^{2}}(\xi^{-}) = \int (1-\cos(v \cdot \xi^{-}))g^{2}(v) \mathrm{d}v$ and
		$
		1-\cos(v \cdot \xi^{-}) \lesssim \min\{|v|^{2}|\xi|^{2}|\frac{\xi}{|\xi|} - \sigma|^{2},1\} \sim  \min\{|v|^{2}|\xi^{+}|^{2}|\frac{\xi^{+}}{|\xi^{+}|} - \sigma|^{2},1\}.
		$
		By the Cauchy-Schwarz inequality and 
		the change of variable $\xi \rightarrow \xi^{+}$, using
		Lemma \ref{symbol} and the fact that $W_{s}(|v||\xi|) \lesssim W_{s}(|v|)W_{s}(|\xi|)$,
		we have
		\ben \label{upper-I-2-another}
		|\mathcal{I}_{2}| \lesssim s^{-1}\int (W_{s})^{2}(|v||\xi|)|\hat{f}(\xi)|^{2}g^{2}(v)\mathrm{d}v\mathrm{d}\xi
		\lesssim  s^{-1} |W_{s}g|^{2}_{L^{2}} |W_{s}(D)f|^{2}_{L^{2}}.
		\een

		Now we turn to estimate $\mathcal{I}_{1}$. By the geometric decomposition
		\ben \label{geo-deco-frequency-space-another}
		\hat{f}(\xi) - \hat{f}(\xi^{+}) = \hat{f}(\xi) - \hat{f}(|\xi|\frac{\xi^{+}}{|\xi^{+}|})+ \hat{f}(|\xi|\frac{\xi^{+}}{|\xi^{+}|}) - \hat{f}(\xi^{+}),
		\een
		and using $(A+B)^{2} \geq \frac{1}{2}A^{2} - B^{2}$,
		we have $\mathcal{I}_{1} \geq	\frac{1}{2}\mathcal{I}_{1,1} - \mathcal{I}_{1,2}$
		where		
		\beno
		\mathcal{I}_{1,1} :=	 \int b^{s}(\frac{\xi}{|\xi|} \cdot \sigma)|\hat{f}(\xi) - \hat{f}(|\xi|\frac{\xi^{+}}{|\xi^{+}|})|^{2} \mathrm{d}\sigma \mathrm{d}\xi,
		\quad
		\mathcal{I}_{1,2} :=
		\int b^{s}(\frac{\xi}{|\xi|} \cdot \sigma)|\hat{f}(|\xi|\frac{\xi^{+}}{|\xi^{+}|}) - \hat{f}(\xi^{+})|^{2} \mathrm{d}\sigma \mathrm{d}\xi.
		\eeno	 
		By Lemma \ref{spherical-part}, we have
		\ben \label{upper-lower-I-11-another}
		\mathcal{I}_{1,1}+  s^{-1}|f|_{L^2}^2 \sim s^{-1}  |W_{s}((-\Delta_{\mathbb{S}^{2}})^{1/2})f|^{2}_{L^{2}}.
		\een
		By Lemma \ref{gammanonzerotozero},
		\ben \label{upper-I-12-another} \mathcal{I}_{1,2} \lesssim s^{-1} (|W_{s}(D)\hat{f}|^{2}_{L^{2}} +  |W_{s}\hat{f}|^{2}_{L^{2}}) = s^{-1} (|W_{s}(D)f|^{2}_{L^{2}} +  |W_{s}f|^{2}_{L^{2}}).\een
		Combining \eqref{upper-I-2-another}, \eqref{upper-I-12-another} and \eqref{upper-lower-I-11-another} gives \eqref{anisotropic-regularity-general-g}.
	\end{proof}

	\subsubsection{Gain of  regularity from $\mathcal{N}^{s,0,\eta}(g,f)$}
	We first introduce some notations. Recall $\psi_{R}(v):=\psi(v/R)$.
	Let $\psi_{r,u}(v) := \psi_{r}(v-u)$ and $\phi_{R,r,u} := \psi_{14R} -\psi_{4r,u}$ for some $r,R>0$ and $u \in \mathbb{R}^{3}$. The following lemma gives some 
	bound estimates on  $\mathcal{N}^{s,0,\eta}(g,f)$ by $\mathcal{N}^{s}(g,f)$ from below provided the distance between supports of $g$ and $f$ is suitably large.

	\begin{lem}\label{reduce-eta-to-0} For  $0 \leq \eta \leq 1 \leq R$, we have
		\ben \label{eta-to-0-part1}
		\mathcal{N}^{s,0,\eta}(g,f) +  |g|^{2}_{L^{2}}|f|^{2}_{L^{2}}  \gtrsim \mathcal{N}^{s}(\psi_{R}g,(1-\psi_{4R})f).
		\een
		For  $0 \leq \eta \leq r \leq 1 \leq R,  u \in B_{6R}$, we have
		\ben  \label{eta-to-0-part2}
		\mathcal{N}^{s,0,\eta}(g,f) +  r^{-2}R^{2}|g|^{2}_{L^{2}}|f|^{2}_{L^{2}}  \gtrsim \mathcal{N}^{s}(\phi_{R,r,u}g,\psi_{r,u}f).
		\een
	\end{lem}
	\begin{proof} We proceed in the spirit of \cite{he2014well}. Note that
		$\psi_{R}$ is supported in $|v| \leq \f{4}{3}R$ and equals to $1$ in $|v| \leq \f{3}{4}R$. 
		$1- \psi_{R}$ is supported in $|v| \geq \f{3}{4}R$ and equals to $1$ in $|v| \geq \f{4}{3}R$. 
		If $|v_{*}| \leq \f{4}{3}R$ and $|v| \geq 3R$, then $|v-v_{*}| \geq \f{5}{3} R \geq \f{4}{3} \eta$, which gives $\psi_{R}(v_*) (1-\psi_{4R}(v)) \leq \textrm{1}_{|v-v_{*}|\geq 4\eta/3}$.
		Hence,
		\beno  \mathcal{N}^{s,0,\eta}(g,f) &\geq& \int b^{s}(\theta) \textrm{1}_{|v-v_{*}|\geq 4\eta/3} g^{2}_{*} (f^{\prime}-f)^{2} \mathrm{d}V
		\\
		&\geq& \int b^{s}(\theta)  (\psi_{R}g)^{2}_{*} (f^{\prime}-f)^{2} (1-\psi_{4R})^{2} \mathrm{d}V
		\\&\geq& \frac{1}{2} \int b^{s}(\theta)  (\psi_{R}g)^{2}_{*}  (((1-\psi_{4R})f)^{\prime}-(1-\psi_{4R})f)^{2} \mathrm{d}V
		\\&&-\int b^{s}(\theta)   (\psi_{R}g)^{2}_{*} (f^{\prime})^{2} (\psi_{4R}^{\prime}-\psi_{4R})^{2} \mathrm{d}V
		:= \frac{1}{2}  \mathcal{I}_{1} - \mathcal{I}_{2}.
		\eeno
		Observe that $\mathcal{I}_{1}  = \mathcal{N}^{s}(\psi_{R}g,(1-\psi_{4R})f)$. Since $|\nabla \psi_{4R}|_{L^{\infty}} \lesssim R^{-1} |\nabla \psi|_{L^{\infty}} \lesssim R^{-1}$, we get $(\psi_{R})^{2}_{*}(\psi_{4R}^{\prime}-\psi_{4R})^{2} \lesssim R^{-2}|v^{\prime}-v|^{2} = R^{-2}|v-v_{*}|^{2}\sin^{2}(\theta/2)$. If $|v_{*}| \leq \f{4}{3}R \leq 2R, |v|\geq 20R, 0 \leq \theta \leq \pi/2$, we have
		\beno
		|v^{\prime}-v_{*}|=\cos(\theta/2)|v-v_{*}|\geq \cos(\theta/2)(|v|-|v_{*}|) \geq  9 \sqrt{2}R.
		\eeno
		Then we have
		$ |v^{\prime}| \geq |v^{\prime}-v_{*}| -|v_{*}| \geq 9 \sqrt{2}R -2R \geq 6R,
		$
		which gives $\psi_{4R}(v^{\prime}) =0 = \psi_{4R}(v)$. Since $\theta \leq \pi/2$, we have
		\beno
		(\psi_{R})^{2}_{*}(\psi_{4R}^{\prime}-\psi_{4R})^{2} \leq \mathrm{1}_{|v|\leq 20R, |v_{*}|\leq 2R} R^{-2}|\nabla \psi|^{2}_{L^{\infty}}|v-v_{*}|^{2}\sin^{2}(\theta/2) \lesssim \sin^{2}(\theta/2).
		\eeno
		By the change of variable $v \to v^{\prime}$ and using \eqref{mean-momentum-transfer}, we get
		\beno  \mathcal{I}_{2} \lesssim  \int g^{2}_{*}  f^{2} \mathrm{d}v \mathrm{d}v_{*}\lesssim  |g|^{2}_{L^{2}}|f|^{2}_{L^{2}}.\eeno
		This together with the fact that $\mathcal{I}_{1}  = \mathcal{N}^{s}(\psi_{R}g,(1-\psi_{4R})f)$ give \eqref{eta-to-0-part1}.

		If $v \in \mathrm{supp} \psi_{r,u}, v_{*} \in \mathrm{supp}\phi_{R,r,u}$, we claim $|v-v_{*}| \geq r \geq \eta$.
		In fact, if $v \in \mathrm{supp} \psi_{r,u}$, then $|v-u| \leq \f{4}{3} r$.  If $|v_*-u| \leq 3r = \f{3}{4} \times 4r$, then 
		$\psi_{4r,u}(v_*) = 1$. Moreover, 
		$|v_*| \leq |u|+ |u-v_{*}| \leq 6R + 3r \leq 9R \leq \f{3}{4} \times 14R$, then $\psi_{14R}(v_*) = 1$. As a result, $\phi_{R,r,u}(v_*) = 0$. From this if 
		$v_{*} \in \mathrm{supp}\phi_{R,r,u}$, then
		$|v_*-u| \geq 3r$. Therefore,
		if $v \in \mathrm{supp} \psi_{r,u}, v_{*} \in \mathrm{supp}\phi_{R,r,u}$, then $|v-v_{*}| \geq 
		|v_{*}-u| - |v-u| \geq 
		\f{5}{3}  r \geq \f{4}{3} \eta$. Thus, 
		\beno  \mathcal{N}^{s,0,\eta}(g,f) &=& \int b^{s}(\theta) \textrm{1}_{|v-v_{*}|\geq 4 \eta/3} g^{2}_{*} (f^{\prime}-f)^{2} \mathrm{d}V
		\\&\geq& \int b^{s}(\theta)   (\phi_{R,r,u}g)^{2}_{*} (f^{\prime}-f)^{2}  \psi^{2}_{r,u} \mathrm{d}V
		\\&\geq& \frac{1}{2} \int b^{s}(\theta)  (\phi_{R,r,u}g)^{2}_{*}  ((\psi_{r,u}f)^{\prime}-\psi_{r,u}f)^{2} \mathrm{d}V
		\\&&-\int b^{s}(\theta)   (\phi_{R,r,u}g)^{2}_{*} (f^{\prime})^{2} (\psi_{r,u}^{\prime}-\psi_{r,u})^{2} \mathrm{d}V
		:= \frac{1}{2}\mathcal{J}_{1} - \mathcal{J}_{2}.
		\eeno
		Observe that $\mathcal{J}_{1}  =  \mathcal{N}^{s}(\phi_{R,r,u}g,\psi_{r,u}f)$.
		Since $|\nabla \psi_{r,u} (v)| \lesssim r^{-1} |\nabla \psi|_{L^{\infty}}\mathrm{1}_{3r/4 \leq |v-u|\leq 4r/3}$, by Taylor expansion, we get \beno
		|\psi_{r,u}^{\prime}-\psi_{r,u}|^{2} = |\int_{0}^{1} \nabla \psi_{r,u} \left(v(\kappa)\right)\cdot(v^{\prime}-v) \mathrm{d}\kappa|^{2} \lesssim r^{-2}|v-v_{*}|^{2}\sin^{2}(\theta/2) \int_{0}^{1} \mathrm{1}_{3r/4 \leq |v(\kappa)-u|\leq 4r/3}
		\mathrm{d}\kappa.\eeno
		For $u \in B_{6R}, |v_{*}| \leq 20R, 3r/4 \leq |v(\kappa)-u|\leq 4r/3$, we have
		\beno|v-v_{*}| \leq \sqrt{2}|v(\kappa)-v_{*}| \leq \sqrt{2}|v(\kappa)-u| +\sqrt{2}|u-v_{*}| \leq 4\sqrt{2}r/3+\sqrt{2}(6R+20R) \leq 28\sqrt{2}R,\eeno
		and 
		\beno (\phi_{R,r,u})^{2}_{*}(\psi_{r,u}^{\prime}-\psi_{r,u})^{2} \lesssim r^{-2}R^{2}\sin^{2}(\theta/2).\eeno
		By the change of variable $v \rightarrow v^{\prime}$ and \eqref{mean-momentum-transfer}, we get
		\beno  \mathcal{J}_{2} \lesssim  r^{-2}R^{2} \int g^{2}_{*}  f^{2} \mathrm{d}v \mathrm{d}v_{*} \lesssim  r^{-2}R^{2}|g|^{2}_{L^{2}}|f|^{2}_{L^{2}}.\eeno
		This together with the fact that $\mathcal{J}_{1}  = \mathcal{N}^{s}(\phi_{R,r,u}g,\psi_{r,u}f)$ give \eqref{eta-to-0-part2}.
	\end{proof}

	\subsubsection{Gain of  regularity from $\mathcal{N}^{s,\gamma,\eta}(\mu^{1/2},f)$}

	We first establish a relation between $	\mathcal{N}^{s,\gamma,\eta}$ and $\mathcal{N}^{s,0,\eta}$.
	
	\begin{lem}\label{reduce-gamma-to-0-no-sigularity} For  $\gamma \leq 0 \leq \eta$, one has
		\beno 
		\mathcal{N}^{s,\gamma,\eta}(g,f)
		+ |g|^{2}_{L^{2}_{|\gamma/2|+1}}|f|^{2}_{L^{2}_{\gamma/2}} \gtrsim  \mathcal{N}^{s,0,\eta}(W_{\gamma/2}g,W_{\gamma/2}f).
		\eeno
	\end{lem}
	\begin{proof} Set $F= W_{\gamma/2}f$. If $\gamma \leq  0$, then $|v-v_{*}|^{\gamma} \geq \langle v-v_{*} \rangle^{\gamma}$, and thus 
		\beno 	 \mathcal{N}^{s,\gamma,\eta}(g,f) \geq  \int b^{s}(\theta) \psi^{\eta}(|v-v_{*}|)\langle v-v_{*} \rangle^{\gamma}  g^{2}_{*} ((W_{-\gamma/2}F)^{\prime}-W_{-\gamma/2}F)^{2} \mathrm{d}V.
		\eeno
		We apply the following decomposition
		\beno
		(W_{-\gamma/2}F)^{\prime}-W_{-\gamma/2}F = (W_{-\gamma/2})^{\prime} (F^{\prime}-F) + F(W_{-\gamma/2}^{\prime}-W_{-\gamma/2}) :=  A + B.
		\eeno
		Using $(A+B)^{2} \geq \frac{1}{2}A^{2} - B^{2}$, we get $\mathcal{N}^{s,\gamma,\eta}(g,f) \geq  \frac{1}{2}\mathcal{I}_{1}- \mathcal{I}_{2}$,
		where
		\beno
		\mathcal{I}_{1} &:=& \int b^{s}(\theta) \psi^{\eta}(|v-v_{*}|)\langle v-v_{*} \rangle^{\gamma}  g^{2}_{*} W_{-\gamma}^{\prime} (F^{\prime}-F)^{2} \mathrm{d}V,
		\\
		\mathcal{I}_{2} &:=& \int b^{s}(\theta) \psi^{\eta}(|v-v_{*}|)\langle v-v_{*} \rangle^{\gamma}  g^{2}_{*}F^{2} (W_{-\gamma/2}^{\prime}-W_{-\gamma/2})^{2} \mathrm{d}V.
		\eeno
		Since $ \langle v_{*} \rangle^{\gamma} \lesssim \langle v_{*}-v^{\prime} \rangle^{\gamma} \langle v^{\prime} \rangle^{-\gamma} \sim \langle v_{*}-v \rangle^{\gamma} \langle v^{\prime} \rangle^{-\gamma}$,
		we get $ \mathcal{I}_{1}  \gtrsim  \mathcal{N}^{s,0,\eta}(W_{\gamma/2}g,W_{\gamma/2}f)$.
		Taylor expansion implies that 
		\beno (W_{-\gamma/2}^{\prime}-W_{-\gamma/2})^{2} \lesssim |v-v_{*}|^{2}\sin^{2}(\theta/2)
		\int  \langle v(\kappa) \rangle^{-\gamma-2}
		\mathrm{d}\kappa.
		\eeno
		Note that
		\beno\langle v-v_{*} \rangle^{\gamma} |v-v_{*}|^{2} \langle v(\kappa) \rangle^{-\gamma-2}   \lesssim 
		\langle v(\kappa)-v_{*} \rangle^{\gamma+2}  v(\kappa) \rangle^{-\gamma-2}  \lesssim \langle v_{*} \rangle^{|\gamma|+2}.\eeno	
		By the above estimate and \eqref{mean-momentum-transfer},
		we get
		\beno  \mathcal{I}_{2} \lesssim   \int g^{2}_{*} \langle v_{*} \rangle^{|\gamma|+2} F^{2} \mathrm{d}v \mathrm{d}v_{*} \lesssim |g|^{2}_{L^{2}_{|\gamma/2|+1}}|F|^{2}_{L^{2}}.\eeno
		Combining the estimates on $\mathcal{I}_{1}$ and $\mathcal{I}_{2}$
		completes the proof of the lemma.
	\end{proof}

	We are now ready to derive gain of regularity from $\mathcal{N}^{s,\gamma,\eta}(\mu^{1/2},f)$.

	\begin{lem}\label{lowerboundpart1-gamma-eta} For $-5 \leq \gamma \leq 0 \leq \eta \leq  \eta_{1} := \f{9}{16} (2\pi)^{1/2} (\f{1}{4}W_{-5}(3/4)\mu(3/4))^{1/3} $, it holds that
		\ben \label{full-regularity-gamma-eta} \mathcal{N}^{s,\gamma,\eta}(\mu^{1/2},f) + |W_{s} W_{\gamma/2}f|_{L^2}^2\gtrsim |W_{s}((-\Delta_{\mathbb{S}^{2}})^{1/2})W_{\gamma/2}f|^{2}_{L^{2}} +  |W_{s}(D)W_{\gamma/2}f|_{L^2}^2. \een
	\end{lem}
	\begin{proof}  By Lemma \ref{reduce-gamma-to-0-no-sigularity}, we have
		\ben
		\label{reduce-gamma-to-0}
		\mathcal{N}^{s,\gamma,\eta}(g,f) +  |g|^{2}_{L^{2}_{|\gamma/2|+1}}|f|^{2}_{L^{2}_{\gamma/2}}  \geq  C\mathcal{N}^{s,0,\eta}(W_{\gamma/2}g,W_{\gamma/2}f),
		\een
		where $C$ is a generic constant.
		Taking $g=\mu^{1/2}$ in 
		\eqref{reduce-gamma-to-0}, we have
		\beno
		\mathcal{N}^{s,\gamma,\eta}(\mu^{1/2},f) +|f|^{2}_{L^{2}_{\gamma/2}}  \gtrsim  \mathcal{N}^{s,0,\eta}(W_{\gamma/2}\mu^{1/2},W_{\gamma/2}f) \geq \mathcal{N}^{s,0,\eta}(W_{-5/2}\mu^{1/2},W_{\gamma/2}f).
		\eeno
		Taking $g=W_{-5/2}\mu^{1/2}, f = F := W_{\gamma/2}f$ in Lemma \ref{reduce-eta-to-0}, we have for $\eta \leq r \leq 1 \leq R,  u \in B_{6R}$ that 
		\begin{eqnarray}\label{eta-to-0-part1-special-g}
			\mathcal{N}^{s,0,\eta}(W_{-5/2}\mu^{1/2}, F) + |F|^{2}_{L^{2}}  \gtrsim \mathcal{N}^{s}(\psi_{R}W_{-5/2}\mu^{1/2},(1-\psi_{4R})F),
			\\ \label{eta-to-0-part2-special-g}
			\mathcal{N}^{s,0,\eta}(W_{-5/2}\mu^{1/2},F) + r^{-2}R^{2}|F|^{2}_{L^{2}}  \gtrsim \mathcal{N}^{s}(\phi_{R,r,u}W_{-5/2}\mu^{1/2},\psi_{r,u}F).
		\end{eqnarray}
		From now on, take $R = 1$, then $\psi_{R} = \psi$, we get
		\beno |\psi_{R}W_{-5/2}\mu^{1/2}|^{2}_{L^{2}} = \int \psi^{2}W_{-5}\mu \mathrm{d}v  \geq 
		\frac{4 \pi}{3} (3/4)^{3}
		W_{-5}(3/4)\mu(3/4)
		:=\delta^{2}_{*}. \eeno
		Recalling $\phi_{R,r,u} = \psi_{14R} -\psi_{4r,u}$ and $\psi_{14R} \geq \psi_{R}$, we have
		\beno  \int \phi^{2}_{R,r,u}W_{-5}\mu \mathrm{d}v \geq \frac{1}{2} \int \psi^{2}_{14R}W_{-5}\mu \mathrm{d}v
		-  \int \psi^{2}_{3r,u}W_{-5}\mu \mathrm{d}v \geq \frac{1}{2}\delta^{2}_{*}
		-  \int \psi^{2}_{3r,u}W_{-5}\mu \mathrm{d}v. \eeno
		Note that $\int \psi^{2}_{3r,u}W_{-5}\mu \mathrm{d}v \leq \frac{4\pi}{3}(\f{4}{3}r)^{3} (2\pi)^{-\frac{3}{2}}  := C r^{3}$.
		By choosing $r$ such that $C r^{3} = \frac{1}{4}\delta^{2}_{*}$, we get
		\beno |\phi_{R,r,u}W_{-5/2}\mu^{1/2}|^{2}_{L^{2}} \geq \delta^{2}_{*}/4. \eeno
		Note that $r = \f{9}{16} (2\pi)^{1/2} (\f{1}{4}W_{-5}(3/4)\mu(3/4))^{1/3}$.
		Therefore, we have
		\ben \label{lower-bound-mu-l2}\min\{|\phi_{R,r,u}W_{-5/2}\mu^{1/2}|_{L^{2}}, |\psi_{R}W_{-5/2}\mu^{1/2}|_{L^{2}}\} \geq \delta_{*}/2.\een
		On the other hand, note that 
		\ben \label{upper-bound-mu-l21}
		\max\{|\psi_{R}W_{-5/2}\mu^{1/2}|_{L^{2}_{1}},|\phi_{R,r,u}W_{-5/2}\mu^{1/2}|_{L^{2}_{1}}\}  \leq |\mu|_{L^{1}_{2}}^{1/2} := \lambda_{*}. \een
		
		Thanks to \eqref{lower-bound-mu-l2} and \eqref{upper-bound-mu-l21}, by Lemma \ref{lowerboundpart1-general-g},
		we get
		\begin{eqnarray}\label{special-g-r-R-large-v} \mathcal{N}^{s}(\psi_{R}W_{-3/2}\mu^{1/2},(1-\psi_{4R})F) + |(1-\psi_{4R})F|^{2}_{L^{2}} \geq  C(\delta_{*}/2, \lambda_{*})|(1-\psi_{4R})F|^{2}_{H^{s}},
			\\ \label{special-g-r-R-small-v} \mathcal{N}^{s}(\phi_{R,r,u}W_{-3/2}\mu^{1/2},\psi_{r,u}F) + |\psi_{r,u}F|^{2}_{L^{2}} \geq  C(\delta_{*}/2, \lambda_{*})|\psi_{r,u}F|^{2}_{H^{s}}.
		\end{eqnarray}
		Note that $1-\psi_{4R}(v) =1$ if $|v| \geq 6 R \geq \f{16}{3} R$.
		There is a finite cover of $B_{6R}$ with open ball $B_{r}(u_{j})$ for $u_{j} \in B_{6R}$. More precisely,
		there exists $\{u_{j}\}_{j=1}^{N} \subset B_{6R}$ such that $B_{6R} \subset \cup_{j=1}^{N}B_{r}(u_{j})$, where $N \sim \frac{1}{r^{3}}$ is a generic constant. We then have $\psi_{4R} \leq \sum_{j=1}^{N}\psi_{r,u_{j}}$ and thus
		$
		|\psi_{4R}F|^{2}_{H^{s}} \leq N \sum_{j=1}^{N} |\psi_{r,u_{j}}F|^{2}_{H^{s}}.
		$
		From this together with \eqref{eta-to-0-part1-special-g}, \eqref{eta-to-0-part2-special-g}, \eqref{special-g-r-R-large-v}, \eqref{special-g-r-R-small-v},
		we get for any $0 \leq \eta \leq r$,
		\beno \mathcal{N}^{s,\gamma,\eta}(\mu^{1/2},f) +|f|_{L^2_{\gamma/2}}^2 \gtrsim r^{8} |W_{\gamma/2}f|_{H^{s}}^{2}. \eeno
		Since $r$ is a generic constant, we get 
		\ben \label{sobolev-regularity-gamma-eta} \mathcal{N}^{s,\gamma,\eta}(\mu^{1/2},f) +|f|_{L^2_{\gamma/2}}^2\gtrsim |W_{s}(D)W_{\gamma/2}f|_{L^2}^2.
		\een

		Thanks to \eqref{lower-bound-mu-l2} and \eqref{upper-bound-mu-l21},
		by \eqref{anisotropic-regularity-general-g} in Lemma \ref{lowerboundpart2-general-g}, we get
		\beno
		\mathcal{N}^{s}(\psi_{R}W_{-3/2}\mu^{1/2},(1-\psi_{4R})F) + \lambda^{2}_{*}(|(1-\psi_{4R})F|^{2}_{H^{s}}+|(1-\psi_{4R})F|^{2}_{L^{2}_{s}}) \\ \gtrsim  \delta^{2}_{*}|W_{s}((-\Delta_{\mathbb{S}^{2}})^{1/2}) (1-\psi_{4R})F|^{2}_{L^{2}}, \\
		\mathcal{N}^{s}(\phi_{R,r,u}W_{-3/2}\mu^{1/2},\psi_{r,u}F) + \lambda^{2}_{*}(|\psi_{r,u}F|^{2}_{H^{s}}+|\psi_{r,u}F|^{2}_{L^{2}_{s}}) \gtrsim \delta^{2}_{*}|W_{s}((-\Delta_{\mathbb{S}^{2}})^{1/2}) \psi_{r,u}f|^{2}_{L^{2}} .
		\eeno
		By applying the  similar argument, we also have
	\ben \label{anisotropic-regularity-gamma-eta} \mathcal{N}^{s,\gamma,\eta}(\mu^{1/2},f) +  |W_{s}(D)W_{\gamma/2}f|_{L^2}^2 + |W_{s} W_{\gamma/2}f|_{L^2}^2\gtrsim |W_{s}((-\Delta_{\mathbb{S}^{2}})^{1/2})W_{\gamma/2}f|^{2}_{L^{2}}. \een
	
	Finally a  suitable combination of \eqref{sobolev-regularity-gamma-eta} and  \eqref{anisotropic-regularity-gamma-eta} gives \eqref{full-regularity-gamma-eta}.
	\end{proof}


	\subsubsection{Rough coercivity estimate of $\langle \mathcal{L}^{s,\gamma,\eta}f,f\rangle$}
	By Lemma \ref{lowerboundpart1} and  \eqref{full-regularity-gamma-eta} in Lemma \ref{lowerboundpart1-gamma-eta},  we have the following estimate.
	
	\begin{lem}\label{two-parts-together} Let $-5 \leq \gamma \leq 0 \leq \eta \leq \eta_{1}$ where $\eta_{1}$ is the constant in  Lemma \ref{lowerboundpart1-gamma-eta}. We have
		\ben \label{two-parts-together-N-N}
		\mathcal{N}^{s,\gamma,\eta}(\mu^{1/2},f) + \mathcal{N}^{s,\gamma,\eta}(\mu^{1/2},f) + |f|^{2}_{L^{2}_{\gamma/2}} \gtrsim |f|_{s,\gamma/2}^2 . 
		\een
	\end{lem}
	Now we are ready to prove the following rough coercivity estimate.
	\begin{thm}\label{strong-coercivity} Let $-5 \leq \gamma \leq 0 < \eta \leq \eta_{1}$. We have
		\ben \label{rough-coercivity}
		\langle \mathcal{L}^{s,\gamma, \eta}f, f\rangle + \eta^{\gamma}|f|^{2}_{L^{2}_{\gamma/2}} \gtrsim  |f|_{s,\gamma/2}^2.
		\een
	\end{thm}
	\begin{proof}
		We recall that $\mathcal{N}^{s,\gamma,\eta}(\mu^{1/2},f) + \mathcal{N}^{s,\gamma,\eta}(f,\mu^{1/2})$ corresponds to
		the anisotropic norm $|||\cdot|||$ introduced in \cite{alexandre2012boltzmann}.
		By the proof of Proposition 2.16 in \cite{alexandre2012boltzmann}, 
		\beno
		\langle \mathcal{L}^{s,\gamma,\eta}_{1} f,f \rangle \geq \frac{1}{10}(\mathcal{N}^{s,\gamma,\eta}(\mu^{1/2},f) + \mathcal{N}^{s,\gamma,\eta}(f,\mu^{1/2})) - \frac{3}{10} \big|\int B^{s,\gamma,\eta} \mu_{*} (f^{2} -f^{\prime 2}) \mathrm{d}V \big|.
		\eeno
		By the cancellation Lemma \ref{cancellation-lemma-general-gamma-minus3-mu}, we have
		\beno \big|\int B^{s,\gamma, \eta} \mu_{*} (f^{2} -f^{\prime 2}) \mathrm{d}V \big|
		\lesssim  \eta^{\gamma} |f|^{2}_{L^{2}_{\gamma/2}}.
		\eeno
		Therefore, we have
		\ben \label{L-1-dominate}
		\langle \mathcal{L}^{s,\gamma,\eta}_{1} f,f \rangle \geq \frac{1}{10}(\mathcal{N}^{s,\gamma,\eta}(\mu^{1/2},f) + \mathcal{N}^{s,\gamma,\eta}(f,\mu^{1/2})) - C \eta^{\gamma} |f|^{2}_{L^{2}_{\gamma/2}}.
		\een
		By Proposition \ref{geq-eta-part-l2}, we get
		\ben \label{L-2-lower}
		|\langle \mathcal{L}^{s,\gamma, \eta}_{2} f,f \rangle| \lesssim \eta^{\gamma} |\mu^{1/8}f|_{L^2}^2 \lesssim \eta^{\gamma} |f|_{L^2_{\gamma/2}}^2.
		\een
		\eqref{L-1-dominate} and \eqref{L-2-lower} imply \eqref{equivalence-relation}.
		Then by applying Lemma \ref{two-parts-together}, we complete the proof of the theorem.
	\end{proof}

	\subsection{Spectrum-gap type estimate} In this subsection, we consider the coercivity estimates of $\mathcal{L}^{s,\gamma,\eta}$ in the microscopic space. This is also referred as the "spectral gap" estimate.

	Recall $\mathrm{ker}(\mathcal{L}^{s,\gamma}_{B})=\mathrm{ker}(\mathcal{L}^{\gamma}_{L})
	= \mathrm{span}\{\sqrt{\mu}, \sqrt{\mu}v_1, \sqrt{\mu}v_2,\sqrt{\mu}v_3, \sqrt{\mu}|v|^2 \}:=\mathrm{ker}$.
	An orthonormal basis of $\mathrm{ker}$ can be chosen as $$\{\sqrt{\mu}, \sqrt{\mu}v_1, \sqrt{\mu}v_2,\sqrt{\mu}v_3, \sqrt{\mu}(|v|^2-3) /\sqrt{6} \}
	:= \{e_{j}\}_{1 \leq j \leq 5}.$$ The projection operator $\mathbf{P}$ on the kernel space
	is defined as follows:
	\ben\label{DefProj} \mathbf{P}f:=\sum_{j=1}^{5}\langle f, e_{j}\rangle e_{j} =(a+b\cdot v+c|v|^2)\sqrt{\mu}, \een
	where for $1\le i\le 3$, \ben\label{Defabc}
	a=\int_{\R^3} (\frac{5}{2}-\frac{|v|^{2}}{2})\sqrt{\mu}f\mathrm{d}v; \quad b_i=\int_{\R^3} v_i\sqrt{\mu}f\mathrm{d}v; \quad c=\int_{\R^3} (\frac{|v|^2}{6}-\frac{1}{2})\sqrt{\mu}f\mathrm{d}v.
	\een
	We will show that the lower order term $|f|^{2}_{L^{2}_{\gamma/2}}$ in
	\eqref{rough-coercivity} can be dropped for $f \in \mathrm{ker}^{\perp}$.
	
	The idea is to firstly consider the case when $\gamma=0$ case and then to use mathematical induction for the general $\gamma<0$ case.
	
	\subsubsection{The case $\gamma=0$} This case is  clear, cf.  the explicit spectrum computation by Wang-Chang \cite{wang1952propagation}, in which the authors showed that the smallest positive eigenvalue is bounded from below by $\int b(\cos\theta) \sin^{2}\f{\theta}{2} \mathrm{d}\sigma$ upto a multiplicative factor.
	Recalling \eqref{mean-momentum-transfer}, it holds 
	for $f \in \mathrm{ker}^{\perp}$
	\beno
	\langle \mathcal{L}^{s,0,0}f, f\rangle \geq \lambda_{1} |f|_{L^{2}}^{2}.
	\eeno
	By the proof of Theorem \ref{strong-coercivity} for the case of $\gamma=0$, we can also take $\eta =0$ to get
	\beno
	\langle \mathcal{L}^{s,0,0}f, f\rangle + |f|_{L^{2}}^{2} \gtrsim |f|^{2}_{s,0}.
	\eeno
	Hence, there exists a generic  constant $\lambda_{2} > 0$ such that
	\ben \label{full-range-coercicity-gamma-0}
	\langle \mathcal{L}^{s,0,0}f, f \rangle  \geq \lambda_{2} |f|^{2}_{s,0}.
	\een
	
	We now show that $\mathcal{L}^{s,0,\eta}$ also satisfies the above estimate if $\eta$ is small enough.
	For this,
	we prove smallness of $\langle  \mathcal{L}^{s,0}_{\eta}f, f\rangle$ when $\eta$ is small.
	\begin{lem}\label{gamma-0-eat-lb} Let $0 \leq \eta \leq 1$, then it holds for $f \in \mathrm{ker}^{\perp}$
		that
		\beno \langle  \mathcal{L}^{s,0}_{\eta}f, f\rangle \lesssim \eta^{3}
		|f|^{2}_{H^{s}}. \eeno
	\end{lem}
	\begin{proof} Firstly, note that 
		$\langle \mathcal{L}^{s,0}_{\eta} f,f \rangle \leq 2\mathcal{N}^{s}_{\eta}(\mu^{1/2},f) + 2\mathcal{N}^{s}_{\eta}(f,\mu^{1/2})$, where
		the functional $\mathcal{N}^{s, \gamma}_{\eta}$ is defined in \eqref{definition-of-N-s-ga-eta-leq}.
		By \eqref{functional-N-lower-eta}, we have
		\beno
		\mathcal{N}^{s}_{\eta}(\mu^{1/2},f) \lesssim \eta^{2s+3} |f|^{2}_{H^{s}}.
		\eeno
		Recall
		\beno \mathcal{N}^{s}_{\eta}(f,\mu^{1/2}) = \int b^{s}(\theta)
		\psi_{\eta}(|v-v_{*}|)
		f^{2}_{*} ((\mu^{1/2})^{\prime}-\mu^{1/2})^{2} \mathrm{d}V. \eeno
		By using
		$|(\mu^{1/2})^{\prime}-\mu^{1/2}| \lesssim |v-v_{*}|\theta$ and  \eqref{mean-momentum-transfer}, we have
		\beno \mathcal{N}^{s}_{\eta}(f,\mu^{1/2}) \lesssim  \int b^{s}(\theta)
		\psi_{\eta}(|v-v_{*}|)
		f^{2}_{*} |v-v_{*}|^{2}\theta^{2} \mathrm{d}V
		\lesssim  \int \mathrm{1}_{|v-v_{*}| \leq 4\eta/3}
		f^{2}_{*} |v-v_{*}|^{2} \mathrm{d}v \mathrm{d}v_{*} \lesssim  \eta^{5} |f|^{2}_{L^{2}}. \eeno
		Combining  the above estimates completes the proof.
	\end{proof}
	
	From \eqref{full-range-coercicity-gamma-0}, by taking $\eta$ small enough in Lemma \ref{gamma-0-eat-lb},
	we get the following coercivity estimate.
	
	\begin{lem}\label{gamma-0-pure-coercivity} There is a generic constant $\eta_{2}>0$ such that
		 for $f \in \mathrm{ker}^{\perp}$, we have
		\beno  \langle \mathcal{L}^{s,0,\eta_2}f, f\rangle \geq \lambda_0
		|f|^{2}_{s,0}. \eeno
	\end{lem}

	\subsubsection{The general case $\gamma<0$}
	The coercivity estimate of $\mathcal{L}^{s,\gamma,\eta}$ in the $\mathrm{ker}^{\perp}$ space for $\gamma<0$ can be stated as follows.
	
	\begin{thm}\label{micro-dissipation} For  $-5 \leq \gamma \leq 0$, with
		the constant $\eta_1$ defined 
		in Theorem \ref{strong-coercivity} and the constants $\eta_2, \lambda_0$ defined 
		in Lemma \ref{gamma-0-pure-coercivity}, 
		let $\eta = \min \{\eta_1, \eta_2\}$.
		There is a generic constant $0<c<1$ such that for any $\gamma \in [-5, 0]$ satisfying
		$-ks \leq  \gamma< -(k-1) s$ for some integer $k \geq 0$,
		it holds for $f \in \mathrm{ker}^{\perp}$
		that
		\ben \label{gap-estimate}  \langle  \mathcal{L}^{s,\gamma, \eta}f, f\rangle \geq c^{2^{k}-1} s^{2^{k}-1} \lambda_0^{2^{k}}  |f|^{2}_{s,\gamma/2}. \een
	\end{thm}

\begin{rmk}\label{gamma-can-be-negative} Theorem  \ref{micro-dissipation} indeed holds for any $\gamma \leq 0$ even though  we only need it for $-5 \leq \gamma \leq 0$.
	The analysis can also be applied to the case when $\gamma>0$.
\end{rmk}	
	
	For later use, set
	$\lambda_{s,\gamma} := c^{2^{k}-1} s^{2^{k}-1} \lambda_0^{2^{k}}$. 
	The following remark is about the lower bound of $\lambda_{s,\gamma}$.

\begin{rmk}\label{lower-bound-on-lambda}	 For $
	-2s - 3\gamma \leq -2s$, as $-\gamma/s \leq 2+ 3/s$,	then
	\ben \label{lower-bound-on-lambda-explicit}
	\lambda_{s,\gamma} \geq c^{2^{\lceil 2+3/s \rceil}-1} s^{2^{\lceil 2+3/s \rceil}-1} \lambda_0^{2^{\lceil 2+3/s \rceil}} := \lambda_{s}.
	\een
	Here $\lceil a \rceil$ is the smallest integer no less than $a$. Note that $\lambda_{s}$ is non-decreasing with respect to $s$.
	\end{rmk}

	Motivated by  \cite{alexandre2012boltzmann,he2021boltzmann} about the exchanging the 
	kinetic component  in the cross-section with a weight of velocity on the function, we can
	 apply an induction argument based on the estimate for the case $\gamma=0$
	obtained in Lemma \ref{gamma-0-pure-coercivity} and the gain of moment of order $s$. As the first step, we reduce the case when $-s\le \gamma<0$  to  $\gamma=0$, and then by induction  to cover the whole range $-5 \leq \gamma \leq 0$. For this, we first introduce a scaled weight function 
	\ben\label{specialweightfun} U_{\delta}(v) := W(\delta v) = (1+ \delta^{2}|v|^{2})^{1/2} \geq \max\{\delta|v|,1\}. \een
	Here $\delta$ is a sufficiently small parameter to be chosen later. We now give two technical lemmas on some integrals involving $U_{\delta}$.

	\begin{lem}\label{difference-term-complication} Let $-5 \leq \alpha, \beta<0<s, \delta <1$.
		Recall $b^{s}(\theta)  := (1-s)
		\sin^{-2-2s}\frac{\theta}{2} \mathrm{1}_{0 \leq \theta \leq \pi/2}$.
		Set 
		\beno 
		X(\beta,\delta) := \delta^{-\beta}\left((U^{\beta/2}_{\delta})^{\prime}(U^{\beta/2}_{\delta})^{\prime}_{*}-
		U^{\beta/2}_{\delta}(U^{\beta/2}_{\delta})_{*}\right)^{2}.
		\eeno
		Then for $v \in \R^{3}$, 
		\ben \label{chi-W-difference-part} \int b^{s}(\theta) |v-v_{*}|^{\alpha} \psi^{\eta}(|v-v_{*}|)
		X(\beta,\delta) \mu_{*} \mathrm{d}\sigma \mathrm{d}v_{*}
		\lesssim  s^{-1} \delta^{2s} \eta^{\alpha} \langle v \rangle^{\alpha+\beta+2s}.\een
	\end{lem}
	\begin{proof}
		First, it is straightforward  to check
		\beno
		|v-v_{*}|^{\alpha} \psi^{\eta}(|v-v_{*}|) \mu_{*} \lesssim \eta^{\alpha} \langle 	v-v_{*}\rangle^{\alpha} \mu_{*} \lesssim \eta^{\alpha} \langle v \rangle^{\alpha} 
		\langle v_{*}\rangle^{|\alpha|} \mu_{*}
		\lesssim \eta^{\alpha} \langle v \rangle^{\alpha} 
		\mu_{*}^{\f12}.
		\eeno
		Note that 
		\beno X(\beta,\delta) \lesssim \delta^{-\beta}(U^{\beta}_{\delta})^{\prime}_{*}\left((U^{\beta/2}_{\delta})^{\prime}-(U^{\beta/2}_{\delta})\right)^{2}
		+\delta^{-\beta}U^{\beta}_{\delta}\left((U^{\beta/2}_{\delta})^{\prime}_{*}-(U^{\beta/2}_{\delta})_{*}\right)^{2}
		:= A_{1} + A_{2}. \eeno
		We only estimate $A_{1}$ because $A_{2}$ can be estimated similarly.
		
		For $a \leq 0$, one has
		\ben \label{derivative-bounds}
		|\nabla U^{a}_{\delta}| \lesssim  |a| \delta U^{a}_{\delta},
		\een
		which gives
		\beno \big((U^{a}_{\delta})^{\prime}-(U^{a}_{\delta})\big)^{2}  = \big|\int_{0}^{1} (\nabla U^{a}_{\delta})(v(\kappa))\cdot (v^{\prime}-v)\mathrm{d}\kappa\big|^{2}
		\lesssim a^{2} \delta^{2} \sin^{2}\f{\theta}{2}|v-v_{*}|^{2}  \int_{0}^{1}  U^{2a}_{\delta}(v(\kappa))\mathrm{d}\kappa.
		\eeno
		Thanks to $|v^{\prime}_{*}|^{2}+|v(\kappa)|^{2} \sim |v|^{2}+|v_{*}|^{2}$, we have
		\ben \label{natural-bound-using-conservation}
		\delta^{-2a}(U^{2a}_{\delta})^{\prime}_{*} U^{2a}_{\delta}(v(\kappa)) \lesssim \langle v \rangle^{2a},
		\een 
		which gives
		\ben \label{order-2-out}
		\delta^{-2a}(U^{2a}_{\delta})^{\prime}_{*}\big((U^{a}_{\delta})^{\prime}-(U^{a}_{\delta})\big)^{2} \lesssim a^{2} \delta^{2} \sin^{2}\f{\theta}{2}|v-v_{*}|^{2} \langle v \rangle^{2a}.
		\een

		Divide the integral $ \int b^{s}(\theta) |v-v_{*}|^{\alpha}
		A_{1} \mu_{*} \mathrm{d}\sigma \mathrm{d}v_{*}$ into two parts: $\mathcal{I}_{\leq}$ and $\mathcal{I}_{\geq}$ corresponding to $\delta|v-v_{*}|\leq 1$ and  $\delta|v-v_{*}|\geq 1$.
		When $\delta|v-v_{*}|\leq 1$, using \eqref{order-2-out} for $a = \beta/2$,
		we have
		\beno
		\mathcal{I}_{\leq} \lesssim  \eta^{\alpha}  \delta^{2} \langle v \rangle^{\alpha+\beta} \int \mathrm{1}_{|v-v_{*}|\leq \delta^{-1}} b^{s}(\theta) \sin^{2}\f{\theta}{2} |v-v_{*}|^{2} \mu^{\f12}_{*} \mathrm{d}\sigma \mathrm{d}v_{*}
		\lesssim \eta^{\alpha}\delta^{2s} \langle v \rangle^{\alpha+\beta+2s}.
		\eeno
		
		When $\delta|v-v_{*}|\geq 1$. We further divide  the integral $\mathcal{I}_{\geq}$ into two parts: 
		$\mathcal{I}_{\geq,\leq}$ and $\mathcal{I}_{\geq,\geq}$ corresponding to
		$\sin\f{\theta}{2} \leq \delta^{-1}|v-v_{*}|^{-1}$ and
		$\sin\f{\theta}{2} \geq \delta^{-1}|v-v_{*}|^{-1}$ respectively. 
		By using \eqref{order-2-out} for $a = \beta/2$,
		we have
		\beno
		\mathcal{I}_{\geq, \leq} \lesssim \eta^{\alpha}  \delta^{2} \langle v \rangle^{\alpha+\beta} \int  \mathrm{1}_{\sin\f{\theta}{2} \leq \delta^{-1}|v-v_{*}|^{-1}} b^{s}(\theta) \sin^{2}\f{\theta}{2} |v-v_{*}|^{2} \mu^{\f12}_{*} \mathrm{d}\sigma \mathrm{d}v_{*}
		\lesssim \eta^{\alpha}\delta^{2s} \langle v \rangle^{\alpha+\beta+2s}.
		\eeno
		
		For the remainder with $\sin\f{\theta}{2} \geq \delta^{-1}|v-v_{*}|^{-1}$, it holds from \eqref{natural-bound-using-conservation} that
		$A_{1} \lesssim \langle v \rangle^{\beta}$ and
		\beno
		\mathcal{I}_{\geq, \geq} &\lesssim&  \eta^{\alpha}  \langle v \rangle^{\alpha+\beta} \int \mathrm{1}_{\sin\f{\theta}{2} \geq \delta^{-1}|v-v_{*}|^{-1}} b^{s}(\theta)  \mu^{\f12}_{*} \mathrm{d}\sigma \mathrm{d}v_{*}
		\lesssim  s^{-1} \eta^{\alpha}\delta^{2s} \langle v \rangle^{\alpha+\beta+2s}.
		\eeno
		Combining the above estimates completes the proof of the lemma. 
	\end{proof}

	\begin{lem}\label{difference-term-complication-2} Let $-5 \leq \alpha, \beta<0<s, \eta, \delta <1$.
		Set $\varphi_{\beta,\delta}:= (1-U^{\beta/2}_{\delta})\mu^{\f{1}{2}}$, then
		\beno \mathcal{I} := \int b^{s}(\theta) |v-v_{*}|^{\alpha} \psi^{\eta}(|v-v_{*}|)
		(\varphi_{\beta,\delta}^{\prime} -  \varphi_{\beta,\delta})^{2} \mathrm{d}\sigma \mathrm{d}v
		\lesssim s^{-1} \delta^2 \eta^{\alpha} \langle  v_{*} \rangle^{\alpha+2s}.\eeno
	\end{lem}
	\begin{proof} By \eqref{derivative-bounds}, we get
		\ben \label{derivative-bounds-2}
		|\varphi_{\beta,\delta}| \lesssim  |\beta| \delta \mu^{\f{1}{4}}, \quad |\nabla \varphi_{\beta,\delta}| \lesssim  |\beta| \delta \mu^{\f{1}{4}}.
		\een
		From this, we first have
		\ben \label{varphi-uniform-ub-by-mu}
		(\varphi_{\beta,\delta}^{\prime} -  \varphi_{\beta,\delta})^{2}  \lesssim
		\delta^2  ((\mu^{1/2})^{\prime} + \mu^{1/2}).	\een
		By 1st-order Taylor expansion,
		we get
		\ben \label{varphi--by-taylor}
		(\varphi_{\beta,\delta}^{\prime} -  \varphi_{\beta,\delta})^{2}  \lesssim
		\delta^2 \sin^{2}\frac{\theta}{2} |v-v_{*}|^{2} \int_{0}^{1} \mu^{1/2}(v(\kappa))\mathrm{d}\kappa.\een
		combing these two estimates gives
		\ben \label{two-cases-combine}
		(\varphi_{\beta,\delta}^{\prime} -  \varphi_{\beta,\delta})^{2}   \lesssim \delta^2 \min \{1, \sin^{2}\frac{\theta}{2} |v-v_{*}|^{2}\}  \int_{0}^{1}  ((\mu^{1/2})^{\prime} + \mu^{1/2} + \mu^{1/2}(v(\kappa)))
		\mathrm{d}\kappa.
		\een
		By Lemma \ref{usual-change} and Lemma \ref{symbol},
		we get
		\beno 
		\mathcal{I} \lesssim \delta^2  \int b^{s}(\theta)  |v-v_{*}|^{\alpha} \psi^{\eta}(|v-v_{*}|) \min \{1, \sin^{2}\frac{\theta}{2} |v-v_{*}|^{2}\} \mu^{\f12} \mathrm{d}\sigma \mathrm{d}v
		\\ \lesssim s^{-1}\delta^2 \eta^{\alpha} 
		\int \langle v- v_{*} \rangle^{\alpha+2s}  \mu^{\f12} \mathrm{d}v 
		\lesssim s^{-1} \delta^2 \eta^{\alpha} \langle  v_{*} \rangle^{\alpha+2s},
		\eeno
	which completes the proof of the lemma.	
	\end{proof}

	We are now ready to  prove the coercivity estimate of $\mathcal{L}^{s,\gamma, \eta}$ for $-5 \leq \gamma \leq 0$ by induction.

	\begin{proof}[Proof of Theorem \ref{micro-dissipation}] In the proof, $0<\eta = \min \{\eta_1, \eta_2\}<1$ is a fixed constant. Hence,  $1 \leq \eta^{\gamma} \leq \eta^{-5}$ for $-5 \leq \gamma \leq 0$. 
		For brevity, set
		\beno J^{s,\gamma, \eta}(f) := 4 \langle  
		\mathcal{L}^{s,\gamma, \eta}f, f\rangle,\quad\mathbb{A}(g, h):=(g_{*}h + g h_{*} - g^{\prime}_{*}h^{\prime} - g^{\prime} h^{\prime}_{*}), \quad\mathbb{F}(g, h):= \mathbb{A}^{2}(g, h).\eeno 
		With these notations, we have
		$J^{s,\gamma, \eta}(f) = \int B^{s,\gamma, \eta} \mathbb{F}(\mu^{1/2}, f) \mathrm{d}V.$
		We divide the proof  into five steps.
		
		{\it Step 1: Localization of $\mathcal{J}^{s,\gamma,\eta}(f)$.}
		By  \eqref{specialweightfun} and if $a \le 0$,  we get \beno |v - v_{*}|^{-a}
		\leq  \max\{1, 2^{-a-1}\} \delta^{a} ((\delta|v|)^{-a}+(\delta|v_{*}|)^{-a})
		\leq 2 \max\{1, 2^{-a-1}\} \delta^{a} U^{-a}_{\delta}(v)U^{-a}_{\delta}(v_{*}),  \eeno
		which gives
		\beno |v - v_{*}|^{a} \geq C_a \delta^{-a} U^{a}_{\delta}(v)U^{a}_{\delta}(v_{*}), \eeno
		where $C_a = \f12 \min \{1, 2^{a+1}\}$.
		With $\gamma= \alpha + \beta, \gamma \leq \alpha, \beta \leq 0$, we have
		\beno \mathcal{J}^{s,\gamma,\eta}(f) \geq C_\beta \delta^{-\beta}\int
		B^{s,\alpha,\eta}  U^{\beta}_{\delta}(U^{\beta}_{\delta})_{*}
		\mathbb{F}(\mu^{1/2}, f) \mathrm{d}V.
		\eeno
		By setting  $h= U^{\beta/2}_{\delta}, \phi=\mu^{\f12}$ and commuting the weight function $U^{\beta}_{\delta}(U^{\beta}_{\delta})_{*}$ with $\mathbb{F}(\cdot,\cdot)$, we have
		\ben \label{move-inside} \nonumber
		&& U^{\beta}_{\delta}(U^{\beta}_{\delta})_{*}\mathbb{F}(\mu^{1/2}, f)  
		\\  &=&
		h^{2}_{*}h^{2}\mathbb{F}(\phi, f)
		=
		\left(h h_{*}\left(\phi_{*}f + \phi f_{*}\right) -
		h h_{*}\left(\phi^{\prime}_{*}f^{\prime} + \phi^{\prime} f^{\prime}_{*}\right)\right)^{2}
		\nonumber \\&=& \left(h h_{*}\left(\phi_{*}f + \phi f_{*}\right) -
		h^{\prime} h^{\prime}_{*}\left(\phi^{\prime}_{*}f^{\prime} + \phi^{\prime} f^{\prime}_{*}\right)
		+ \left(h^{\prime} h^{\prime}_{*}-
		h h_{*}\right)
		\left(\phi^{\prime}_{*}f^{\prime} + \phi^{\prime} f^{\prime}_{*}\right)
		\right)^{2}
		\nonumber \\  &\geq&  \frac{1}{2}  \left(h h_{*}\left(\phi_{*}f + \phi f_{*}\right) -
		h^{\prime} h^{\prime}_{*}\left(\phi^{\prime}_{*}f^{\prime} + \phi^{\prime} f^{\prime}_{*}\right) \right)^{2}\nonumber -\left(h^{\prime} h^{\prime}_{*}-
		h h_{*}\right)^{2}
		\left(\phi^{\prime}_{*}f^{\prime} + \phi^{\prime} f^{\prime}_{*}\right)^{2}\nonumber
		\\  &=&  \frac{1}{2}  \mathbb{F}(h \phi, h f)
		-\left(h^{\prime} h^{\prime}_{*}-
		h h_{*}\right)^{2}
		\left(\phi^{\prime}_{*}f^{\prime} + \phi^{\prime} f^{\prime}_{*}\right)^{2}.
		\een
		Thus, 
		\ben \label{move-inside-sym} \mathcal{J}^{s,\gamma,\eta}(f) &\geq&
		\frac{1}{2} C_{\beta} \delta^{-\beta}
		\int B^{s,\alpha,\eta}  \mathbb{F}(U^{\beta/2}_{\delta}\mu^{\f12}, U^{\beta/2}_{\delta}f) \mathrm{d}V
		\\  && - C_{\beta} \delta^{-\beta}\int B^{s,\alpha,\eta} \left(h^{\prime} h^{\prime}_{*}-
		h h_{*}\right)^{2}
		\left(\phi^{\prime}_{*}f^{\prime} + \phi^{\prime} f^{\prime}_{*}\right)^{2} \mathrm{d}V. \nonumber
		\een
		We further rewrite 
		$\mathbb{F}(U^{\beta/2}_{\delta}\mu^{\f12}, U^{\beta/2}_{\delta}f)$ 
		as  $\mathbb{F}(\mu^{\f12}, U^{\beta/2}_{\delta}f)$ plus some correction terms. That is,
		\ben \label{move-outside-asym} \mathbb{F}(U^{\beta/2}_{\delta}\mu^{\f12}, U^{\beta/2}_{\delta}f)  &=& \mathbb{A}^{2}(U^{\beta/2}_{\delta}\mu^{\f12}, U^{\beta/2}_{\delta}f)
		\nonumber \\&=&\left( \mathbb{A}(\mu^{\f12}, U^{\beta/2}_{\delta}f)  - \mathbb{A}\big((1-U^{\beta/2}_{\delta})\mu^{\f12}, U^{\beta/2}_{\delta}f\big) \right)^{2}
		\nonumber \\&\geq&\frac{1}{2} \mathbb{A}^{2}(\mu^{\f12}, U^{\beta/2}_{\delta}f)
		-\mathbb{A}^{2}\big((1-U^{\beta/2}_{\delta})\mu^{\f12}, U^{\beta/2}_{\delta}f\big)
		\nonumber \\&=&\frac{1}{2}\mathbb{F}(\mu^{\f12}, U^{\beta/2}_{\delta}f) -  \mathbb{F}((1-U^{\beta/2}_{\delta})\mu^{\f12}, U^{\beta/2}_{\delta}f).
		\een
		By symmetry and noting $\phi=\mu^{\f12}$, 
		we have
		\ben \label{sym-term}  \int B^{s,\alpha,\eta} \left(h^{\prime} h^{\prime}_{*}-
		h h_{*}\right)^{2}
		\left(\phi^{\prime}_{*}f^{\prime} + \phi^{\prime} f^{\prime}_{*}\right)^{2} \mathrm{d}V 
		\leq 4\int B^{s,\alpha,\eta} \left(h^{\prime} h^{\prime}_{*}-
		h h_{*}\right)^{2}\mu_{*}f^{2} \mathrm{d}V.
		\een
		By \eqref{move-inside-sym}, \eqref{move-outside-asym} and \eqref{sym-term}, we get
		\beno  \mathcal{J}^{s,\gamma,\eta}(f) &\geq&
		\frac{1}{4} C_{\beta}\delta^{-\beta}
		\int B^{s,\alpha,\eta} \mathbb{F}(\mu^{\f12}, U^{\beta/2}_{\delta}f) \mathrm{d}V
		\\  && - \frac{1}{2} C_{\beta} \delta^{-\beta}\int B^{s,\alpha,\eta} \mathbb{F}((1-U^{\beta/2}_{\delta})\mu^{\f12}, U^{\beta/2}_{\delta}f) \mathrm{d}V
		\\   && - 4 C_{\beta} \delta^{-\beta}\int B^{s,\alpha,\eta} \left(h^{\prime} h^{\prime}_{*}-
		h h_{*}\right)^{2} \mu_{*}f^{2} \mathrm{d}V
		\\	&:=& \frac{1}{4} C_{\beta} J_{1}^{\alpha, \beta} - \frac{1}{2} C_{\beta} J_{2}^{\alpha, \beta} -  4 C_{\beta} J_{3}^{\alpha, \beta}.
		\eeno
		
		We always choose $\beta$ in the range
		$-1 \leq -s \leq \beta \leq 0$. It is straightforward to check that
		 $C_{\beta} = \f{1}{2}$. Noting that $J_{1}^{\alpha, \beta} = \delta^{-\beta}  \mathcal{J}^{s,\alpha,\eta}(U^{\beta/2}_{\delta}f)$, we have
		\ben \label{separate-key-part} \mathcal{J}^{s,\gamma,\eta}(f)
		\geq  \frac{1}{8} \delta^{s}  \mathcal{J}^{s,\alpha,\eta}(U^{\beta/2}_{\delta}f) - \frac{1}{4}  J_{2}^{\alpha, \beta} -  2 J_{3}^{\alpha, \beta}. 
		\een

		{\it Step 2: Upper bound of $J_2^{\alpha, \beta}$.} 
		For simplicity of notations, set $\varphi_{\beta,\delta}= (1-U^{\beta/2}_{\delta})\mu^{\f12}, \psi_{\beta,\delta}=U^{\beta/2}_{\delta}f$. Then
		\ben  \label{J2-part}
		J_{2}^{\alpha, \beta} =
		\delta^{-\beta}\int B^{s,\alpha,\eta} \mathbb{F}(\varphi_{\beta,\delta}, \varphi_{\beta,\delta}) \mathrm{d}V
		 \lesssim \delta^{-\beta}
		\mathcal{N}^{s, \alpha, \eta}(\varphi_{\beta,\delta}, \varphi_{\beta,\delta}) + \delta^{-\beta}
		\mathcal{N}^{s, \alpha, \eta}(\varphi_{\beta,\delta},\varphi_{\beta,\delta}).
		\een
		By \eqref{derivative-bounds}, for $a \leq 0$, 
		\ben\label{lclfact1} 0 \leq  1-U^{a}_{\delta}(v) = U^{a}_{\delta}(0)-U^{a}_{\delta}(v)\lesssim |a|\delta |v|.\een
		By \eqref{lclfact1}, we have
		\ben \label{f-gamma-upper}
		(\varphi_{\beta,\delta}^{2})_{*} = ((1-U^{\beta/2}_{\delta})\mu^{\f12})^{2}_{*} \lesssim \delta^{2} \mu^{\f12}_{*}.
		\een
		From this and Prop. \ref{functional-N-up-geq-eta}, by using the fact that $\delta^{-\beta/2}U^{\beta/2}_{\delta} \in S^{\beta/2}_{1,0}$ is a radial  symbol of order $\beta/2$, we obtain
		\ben \label{J21-line2}
		\delta^{-\beta}
		\mathcal{N}^{s, \alpha, \eta}(\varphi_{\beta,\delta}, \varphi_{\beta,\delta}) \lesssim 
		\delta^{2}\delta^{-\beta} \mathcal{N}^{s, \alpha, \eta}(\mu^{1/4},\psi_{\beta,\delta})
		\lesssim s^{-1} \delta^{2}\delta^{-\beta}|U^{\beta/2}_{\delta}f|^{2}_{s,\alpha/2} \lesssim s^{-1} \delta^{2}|f|^{2}_{s,\gamma/2}.
		\een
		By  Lemma \ref{difference-term-complication-2}, we then have
		\ben \label{J22}
		J_{2,2}\lesssim s^{-1}
		\delta^2 \eta^{\alpha} 
		\delta^{-\beta} \int (U^{\beta/2}_{\delta}f)^{2}_{*} \langle  v_{*} \rangle^{\alpha+2s} \mathrm{d}v_{*}
		\lesssim s^{-1} \delta^{2}|W_{\gamma/2+s}f|^{2}_{L^{2}}
		\lesssim s^{-1} \delta^{2}|f|^{2}_{s,\gamma/2}.
		\een
		Plugging the estimates \eqref{J21-line2} and \eqref{J22} into \eqref{J2-part}, we get
		\ben \label{estimate-J-2}
		J_{2}^{\alpha, \beta} \lesssim s^{-1} \delta^{2}|f|^{2}_{s,\gamma/2}.
		\een

		{\it Step 3: Upper bound of $J_2^{\alpha, \beta}$.} 
		Lemma \ref{difference-term-complication} gives
		\ben \label{estimate-J-3}
		J_{3}^{\alpha, \beta} = \int B^{s,\alpha,\eta} 	X(\beta,\delta) \mu_{*}f^{2} \mathrm{d}V
		\lesssim s^{-1} \delta^{2s} |f|^{2}_{L^{2}_{\gamma/2+s}}.
		\een

		{\it Step 4: The case $-s \leq \gamma <0$.} We take $\alpha = 0, \beta =\gamma$.
		Recall $\mathcal{J}^{s,\alpha,\eta}(U^{\beta/2}_{\delta}f)  = 4 \langle \mathcal{L}^{s,\alpha,\eta}U^{\gamma/2}_{\delta}f, U^{\gamma/2}_{\delta}f\rangle$. 
		By Lemma \ref{gamma-0-pure-coercivity}, we have
		\ben \label{out-truncation-projection}
		\mathcal{J}^{s,0,\eta}(U^{\beta/2}_{\delta}f) \geq 4 
		\lambda_0 |(\mathbf{I}-\mathbf{P})U^{\gamma/2}_{\delta}f|^{2}_{s,0} \geq 4 
		\lambda_0 |(\mathbf{I}-\mathbf{P})U^{\gamma/2}_{\delta}f|^{2}_{L^{2}_{s}}
		. \een
		We claim that there exists $\delta_1>0$ such that if $0<\delta \leq \delta_1$, then for any $-5/2 \leq a \leq 0$,
		\ben \label{delta-small-projection-to-not-projection}
		|(\mathbf{I}-\mathbf{P})U^{a}_{\delta}f|^{2}_{L^{2}_{s}} \geq \f{1}{4} |f|^{2}_{L^{2}_{s+a}}.
		\een
		This yields
		\ben \label{J1-0-gamma}
		\mathcal{J}^{s,0,\eta}(U^{\beta/2}_{\delta}f) \geq  \lambda_0
		|f|^{2}_{L^{2}_{s+\gamma/2}}
		. \een
		
		We now prove \eqref{delta-small-projection-to-not-projection}. 
		Note that
		\beno 
		|(\mathbf{I}-\mathbf{P})(U^{a}_{\delta}f)|^{2}_{L^{2}_s} \geq  \frac{1}{2} |U^{a}_{\delta}f|^{2}_{L^{2}_s} - |\mathbf{P}(U^{a}_{\delta}f)|^{2}_{L^{2}_s}.
		\eeno
		Since $\delta \leq 1$ and $a \leq 0$,  $U^{a}_{\delta} \geq W_{a}$. Hence, 
		\ben \label{leading-term} |U^{a}_{\delta}f|^{2}_{L^{2}_s} \geq |f|^{2}_{L^{2}_{s+a}}.\een
		We now estimate $|\mathbf{P}(U^{a}_{\delta}f)|_{L^{2}}$ for $f \in \ker^{\perp}$. 
		Since
		\beno  \mathbf{P}(U^{a}_{\delta}f) = \sum_{i=1}^{5} e_{i}\int e_{i}U^{a}_{\delta}f \mathrm{d}v
		= \sum_{i=1}^{5} e_{i}\int e_{i}(U^{a}_{\delta}-1)f \mathrm{d}v ,\eeno
		then 
		\beno \big|\int e_{i}(U^{a}_{\delta}-1)f \mathrm{d}v\big| \lesssim |a| \delta |\mu^{\f18}f|_{L^{2}}.
		\eeno
		Therefore, 
		\ben \label{large-velocity-part-12}
		|\mathbf{P}(U^{a}_{\delta}f)|^{2}_{L^{2}_s} 
		\lesssim a^{2} \delta^{2} |\mu^{\f18}f|^{2}_{L^{2}}
		\lesssim \delta^{2} |f|^{2}_{L^{2}_{s+a}}
		.\een
		By combining the estimates \eqref{leading-term} and \eqref{large-velocity-part-12}
		and choosing $\delta_1$ suitably small, 
		we obtain \eqref{delta-small-projection-to-not-projection}.
		
		By plugging the estimates  \eqref{J1-0-gamma},   \eqref{estimate-J-2},  \eqref{estimate-J-3} into \eqref{separate-key-part},  for any $-s \leq \gamma \leq 0$ and
		$0<\delta \leq \delta_1$, for some generic constants $0< C_{1}$ and $ 1 \leq  C_{2}$,
		we have
		\ben \label{key-estimate-order-s-2s}   \mathcal{J}^{s,\gamma,\eta}(f) \geq C_{1} \delta^{s}|f|^{2}_{L^{2}_{\gamma/2+s}}-
		C_{2}s^{-1}\delta^{2s}|f|^{2}_{s,\gamma/2}. \een
		It is straightforward to check from above that 
		$C_{1} = \lambda_0 /8$.
		Recalling Theorem \ref{strong-coercivity},
		for some generic constants $0< C_{3} \leq 1 \leq  C_{4}$, we have
		\ben \label{known-estimate}    \mathcal{J}^{s,\gamma,\eta}(f) \geq C_{3}  |f|^{2}_{s,\gamma/2}- C_{4} |f|^{2}_{L^{2}_{\gamma/2+s}}.\een
		We can assume $\frac{C_{1}C_{3}}{2C_{4}C_{2}s^{-1}} \leq \delta_1^s$. Otherwise, we can take a larger $C_4$.

		Then the  combination $\eqref{known-estimate} \times C_{5}\delta^{2s} + \eqref{key-estimate-order-s-2s}$ gives
		\ben \label{combination}  (1+C_{5}\delta^{2s}) \mathcal{J}^{s,\gamma,\eta}(f) \geq (
		C_{1}-C_{4}C_{5}\delta^{s})\delta^{s}|f|^{2}_{L^{2}_{\gamma/2+s}}
		+(C_{3}C_{5}-C_{2}s^{-1})\delta^{2s}|f|^{2}_{s,\gamma/2}.  \een
		We can then  take $C_{5}$ large enough such that $C_{3}C_{5}-C_{2}s^{-1} \geq C_{2}s^{-1}$, for example  $C_{5} = 2C_{2}s^{-1}/C_{3} \geq 2$.
		And  then we choose  $\delta$ small enough such that $C_{1}-C_{4}C_{5}\delta^{s} \geq 0$, for example  $\delta^{s}=\frac{C_{1}}{C_{4}C_{5}} = \frac{C_{1}C_{3}}{2C_{4}C_{2}s^{-1}} \leq \delta_1^s$. Note that we can assume 
		$C_{5}\delta^{2s} = 
		\frac{C_{1}^2 C_3}{2 C_{4}^2 C_{2}s^{-1}} 
		\leq \delta_1^s \frac{C_{1}}{C_{4}}
		\leq 1$. Otherwise, we can take a larger $C_4$.
		Thus,  we get
		\ben \label{conclusion1-gamma-minus2}   \mathcal{J}^{s,\gamma,\eta}(f) \geq \f{1}{2}C_{2}s^{-1} \delta^{2s}|f|^{2}_{s,\gamma/2} = \f{1}{2} C_{2}s^{-1}\left(\frac{C_{1}C_{3}}{2C_{2}s^{-1}C_{4}}\right)^{2}  |f|^{2}_{s,\gamma/2}
		=   \f{1}{2} \left(\frac{C_{3}}{16(C_{2}s^{-1})^{1/2}C_{4}}\right)^{2} \lambda_0^2   |f|^{2}_{s,\gamma/2}
		. \een
		Recalling $\mathcal{J}^{s,\gamma,\eta}(f)  = 4 \langle \mathcal{L}^{s,\gamma,\eta}f,  f\rangle$, we get for $-s \leq \gamma \leq 0$ that
		\beno
		\langle \mathcal{L}^{s,\gamma,\eta}f,  f\rangle \geq c_s \lambda_0^2   |f|^{2}_{s,\gamma/2},
		\eeno
		where 
		\beno
		c_s = \f{1}{8} \left(\frac{C_{3}}{16C_{2}^{1/2}C_{4}}\right)^{2} s = 2^{-11} C_{3}^2 C_{2}^{-1}C_{4}^{-2} s.
		\eeno
		
		{\it Step 5: The case $-ks \leq  \gamma< -(k-1) s$  for $k \geq 2$.}
		In the previous step, starting from the $\gamma=0$ case
		by using	Lemma \ref{gamma-0-pure-coercivity}	
		where the constant is $\lambda_0$,
		to derive the $-s \leq \gamma<0$ case, we have a new constant $c_s \lambda_0^2$. 
		For
		$-ks \leq  \gamma< -(k-1) s$, we can choose 
		$\alpha = -(k-1)s$ and $\beta = \gamma + (k-1)s$ to apply the result of $\langle \mathcal{L}^{s,\alpha,\eta}f,  f\rangle$.
		Note that the constants $C_2, C_3, C_4$
		are generaic with respect to  $\alpha, \beta$ satisfying $\alpha + \beta = \gamma, -s \leq \beta \leq 0, -5 \leq \gamma \leq \alpha \leq 0$. 
		$\lambda_n = c_{s} \lambda_{n-1}^2$ implies that
		\beno
		\lambda_k = c_{s}^{2^{k}-1}  \lambda_0^{2^{k}}.
		\eeno
		For $-ks \leq  \gamma< (k-1) s$, by induction 
		we will have
		\ben \label{general-gamma-2-minus3}   \mathcal{J}^{s,\gamma,\eta}(f) \geq \lambda_k  |f|^{2}_{s,\gamma/2}. \een
		This completes the proof of the theorem by taking $c= 2^{-11} C_{3}^2 C_{2}^{-1}C_{4}^{-2}$.
	\end{proof}

	\section{Commutator estimates and weighted estimates}
	\label{Commutator-Estimate}
	In this section, we will study the  commutator estimates between the collision operators and the weight function $W_{l}$ for obtaining the  energy estimates in weighted Sobolev space. In this section, unless indicated otherwise, $0<s<1, -3-2s < \gamma \leq 0$ and $g,h,f$ are suitable smooth functions.

	\subsection{Commutator estimates for $Q^{s,\gamma}$}

	We first prove the following proposition.
	
	\begin{prop}\label{commutatorQepsilon}
		Let $l , l_1 \geq 0$. 
		Recall $C_{\delta,s,\gamma} = \delta^{-\f12}s^{-1} (\gamma+2s+3)^{-1}$.
		Let  $(a_{1},a_{2})=(\f{3}{2} + \delta, s)$ or $(0, \f{3}{2} + \delta)$, then
		\beno |\langle Q^{s,\gamma}(\mu^{1/2}g,W_{l}h)-W_{l}Q^{s,\gamma}(\mu^{1/2}g,h), f\rangle| &\lesssim_{l,l_1}
		s^{-1/2}  |\mu^{1/16} g|_{L^{2}} |h|_{L^{2}_{l+\gamma/2}} |f|_{s,\gamma/2} 
		\\ & \quad \quad+ C_{\delta, s, \gamma}  |\mu^{1/64} g|_{H^{a_1}}
		|h|_{H^{a_2}_{-l_1}} |f|_{s,\gamma/2}. \eeno
	\end{prop}
	\begin{proof}
		Note that
		\beno &&\langle Q^{s,\gamma}(\mu^{1/2}g,W_{l}h)-W_{l}Q^{s,\gamma}(\mu^{1/2}g,h), f\rangle  = \int B^{s,\gamma}(W_{l}-W^{\prime}_{l})\mu^{1/2}_{*}g_{*} h f^{\prime} \mathrm{d}V
		\\&=& \int B^{s,\gamma}(W_{l}-W^{\prime}_{l})\mu^{1/2}_{*}g_{*} h (f^{\prime}-f) \mathrm{d}V
		+\int B^{s,\gamma}(W_{l}-W^{\prime}_{l})\mu^{1/2}_{*}g_{*} h f \mathrm{d}V
		:=\mathcal{A}_{1} + \mathcal{A}_{2}.
		\eeno

		{\it Step 1: Estimate of $\mathcal{A}_{1}$.}
		We write $\mathcal{A}_{1} = \mathcal{A}_{1}^{\leq} + \mathcal{A}_{1}^{\geq}$ where $\mathcal{A}_{1}^{\leq}$ and $\mathcal{A}_{1}^{\geq}$ contains $B^{s,\gamma}_{1}$ and $B^{s,\gamma,1}$ respectively.
		By Cauchy-Schwarz inequality, we have
		\beno |\mathcal{A}_{1}^{\leq}| \leq \big(\int B^{s,\gamma}_{1} \mu^{1/2}_{*}(f^{\prime}-f)^{2} \mathrm{d}V\big)^{1/2}
		\big(\int B^{s,\gamma}_{1}(W_{l}-W^{\prime}_{l})^{2}\mu^{1/2}_{*} g^{2}_{*} h^{2}  
		\mathrm{d}V\big)^{1/2}
		:=(\mathcal{A}^{\leq}_{1,1})^{1/2}(\mathcal{A}^{\leq}_{1,2})^{1/2}. \eeno
		By \eqref{functional-N-lower-eta} and taking $\eta =1$, we have 	
		\beno \mathcal{A}^{\leq}_{1,1} =
		\mathcal{N}^{s,\gamma}_{1}(\mu^{1/4},f) \lesssim  C_{s,\gamma} |f|^{2}_{s,\gamma/2}.
		\eeno

		For any $\iota \in [0,1]$,
		note that 
		\ben \label{diff-wl-l-geq-0}
		|W^{\prime}_{l}-W_{l}| \lesssim l W_{l-1}(v) \left(\mathrm{1}_{0 \leq l < 1} W_{2-2l}(v_{*}(\iota)) + \mathrm{1}_{l \geq 1}W_{l-1}(v_{*}(\iota))\right) |v-v_{*}|\sin\frac{\theta}{2}.
		\een
		This and   \eqref{mean-momentum-transfer} give
		\ben \label{sigma-Wl-difference-leq-1}
		|v-v_{*}|^{\gamma}\psi_{1}(|v-v_{*}|)	\int \mu^{\f14}(v_{*}(\iota))
		b^{s}(\theta)(W_{l}-W^{\prime}_{l})^{2}\mathrm{d}\sigma \lesssim_{l}  \mathrm{1}_{|v-v_{*}| \leq 4/3} |v-v_{*}|^{\gamma + 2}
		\mu^{\f{1}{32}}(v) \mu^{\f{1}{32}}(v_*) 	
		, \een
		which yields for $b_1, b_2 \geq 0, b_1+b_2 =\f{3}{2}+\delta$ that
		\beno \mathcal{A}^{\leq}_{1,2}  
		\lesssim_{l} \delta^{-1} (\gamma+5)^{-1} 
		|\mu^{1/64}g|^{2}_{H^{b_1}} |\mu^{1/64}h|^{2}_{H^{b_2}}.
		\eeno
		Combining these two estimates gives
		$|\mathcal{A}_{1}^{\leq}| \lesssim_{l} C_{\delta,s,\gamma}  |\mu^{1/128}g|_{H^{b_1}} |\mu^{1/128}h|_{H^{b_2}}|f|_{s,\gamma/2}$.
		
		For $\mathcal{A}_{1}^{\geq}$, by  Cauchy-Schwarz inequality, we have
		\beno |\mathcal{A}_{1}^{\geq}| \leq \big(\int B^{s,\gamma,1} \mu^{1/2}_{*}(f^{\prime}-f)^{2} \mathrm{d}V\big)^{1/2}
		\big(\int B^{s,\gamma,1}(W_{l}-W^{\prime}_{l})^{2}\mu^{1/2}_{*} g^{2}_{*} h^{2}  \mathrm{d}V\big)^{1/2}
		:= (\mathcal{A}^{\geq}_{1,1})^{1/2}(\mathcal{A}^{\geq}_{1,2})^{1/2}. \eeno
		By \eqref{functional-N-lower-eta} and  taking $\eta =1$, we have 	
		\beno \mathcal{A}^{\geq}_{1,1} =
		\mathcal{N}^{s,\gamma,1}(\mu^{1/4},f) \lesssim  s^{-1} |f|^{2}_{s,\gamma/2}.
		\eeno
	 \eqref{sigma-Wl-difference-leq-1} and  \eqref{mean-momentum-transfer} give
		\ben \label{sigma-Wl-difference-geq-1}
		|v-v_{*}|^{\gamma}\psi^{1}(|v-v_{*}|)	\int \mu^{\f14}(v_{*}(\iota))
		b^{s}(\theta)(W_{l}-W^{\prime}_{l})^{2}\mathrm{d}\sigma \lesssim_{l}  \mathrm{1}_{|v-v_{*}| \geq 3/4} W_{2l+\gamma}
		, \een	
		which implies $ \mathcal{A}^{\geq}_{1,2}  
		\lesssim_{l}  |g|^{2}_{L^{2}} |h|^{2}_{L^{2}_{l+\gamma/2}}.
		$ Hence,
		$$|\mathcal{A}_{1}^{\geq}| \lesssim_{l} s^{-1/2}  |\mu^{1/8}g|_{L^{2}} |h|_{L^{2}_{l+\gamma/2}} |f|_{s,\gamma/2}.
		$$

		{\it Step 2: Estimate of $\mathcal{A}_{2}$.}
	Note that
		\beno
		|\mathcal{A}_{2}| = 
		| \langle Q^{s,\gamma}(\mu^{1/2}g, h f), W_{l} \rangle |
		= | \langle Q^{s,\gamma,1}(\mu^{1/2}g, h f), W_{l} \rangle + \langle Q^{s,\gamma}_1(\mu^{1/2}g, h f), W_{l} \rangle |.
		\eeno
		According to the proof of Lemma \ref{a-special-term}, for  $(a_{1},a_{2})=(\f{3}{2} + \delta, s)$ or $(0, \f{3}{2} + \delta)$, for $l, l_1 \geq 0$,
		it holds that
		\beno 
		|\langle Q^{s,\gamma}_{1}(\mu^{\f12}f_1, f_2 f_3), W_{l}\rangle|
		\lesssim_{l,l_1} C_{\delta, s,\gamma}  |\mu^{\f14}f_1|_{H^{a_1}}
		|f_2|_{H^{a_2}_{-l_1}} |f_3|_{H^{s}_{\gamma/2}},
		\eeno
		which gives
		\beno 
		|\langle Q^{s,\gamma}_1(\mu^{1/2}g, h f), W_{l} \rangle|
		\lesssim_{l,l_1} C_{\delta, s, \gamma}  |\mu^{\f14} g|_{H^{a_1}}
		|h|_{H^{a_2}_{-l_1}} |f|_{H^{s}_{\gamma/2}}.	\eeno

		By applying Taylor expansion \eqref{Taylor1}
		to 
		$W^{\prime}_{l} - W_{l}$,  we have
		\beno \langle Q^{s,\gamma,1}(\mu^{1/2}g, h f), W_{l} \rangle = \int B^{s,\gamma,1}(\nabla W_{l})(v)\cdot(v^{\prime}-v)\mu^{1/2}_{*}g_{*} h f \mathrm{d}V
		\\+\int B^{s,\gamma,1}(1-\kappa)(\nabla^{2}W_{l})(v(\kappa)):(v^{\prime}-v)\otimes(v^{\prime}-v)\mu^{1/2}_{*}g_{*} h f \mathrm{d}\kappa \mathrm{d}V
		:= \mathcal{A}_{2,1}+\mathcal{A}_{2,2}.\eeno
		By \eqref{cancell1}, $|(\nabla W_{l})(v)| \lesssim l \langle v \rangle^{l-1}$ and  \eqref{mean-momentum-transfer},  
		we have
		\beno |\mathcal{A}_{2,1}|  \lesssim   l \int \mathrm{1}_{|v-v_{*}|\geq 3/4} |v-v_{*}|^{\gamma+1}\langle v \rangle^{l-1}\mu^{1/2}_{*}|g_{*} h f| \mathrm{d}V
		\lesssim_{l}    |\mu^{1/16}g|_{L^{2}}|h|_{L^{2}_{l+\gamma/2}}|f|_{L^{2}_{\gamma/2}}.\eeno
		Since $|(\nabla^{2}W_{l})(v(\kappa))| \lesssim_{l} \langle v(\kappa) \rangle^{l-2} \lesssim_{l} \langle v \rangle^{l-2}\langle v_{*} \rangle^{|l-2|}$, by  \eqref{mean-momentum-transfer} again,
		we have
		\beno |\mathcal{A}_{2,2}| \lesssim_{l} \int \mathrm{1}_{|v-v_{*}|\geq 3/4} |v-v_{*}|^{\gamma+2}
		\langle v \rangle^{l-2} \mu^{1/4}_{*}|g_{*} h f |\mathrm{d}V
		\lesssim_{l} |\mu^{1/16}g|_{L^{2}}|h|_{L^{2}_{l+\gamma/2}}|f|_{L^{2}_{\gamma/2}}.
		\eeno
		Combining the above estimates completes the proof of the proposition.
	\end{proof}

	\subsection{Commutator estimates for $I^{s,\gamma}$} We now prove the following
	proposition.
	
	\begin{prop}\label{commutatorforI} Let $\gamma > -5, l \geq 0$,  $(a_{1},a_{2})=(\f{3}{2} + \delta, 0)$ or $(0, \f{3}{2} + \delta)$. Then
		\beno
		|\langle I^{s,\gamma}(g,W_{l}h;\beta) - W_{l} I^{s,\gamma}(g,h;\beta), f\rangle| &\lesssim_{l} & s^{-1/2}|g|_{L^{2}}|h|_{L^{2}_{l+\gamma/2}}|f|_{L^{2}_{\gamma/2+s}} 
		\\ && +
		\frac{1}{\gamma+5}
		|\mu^{1/64}g|_{H^{a_{1}}}|\mu^{1/64}h|_{H^{a_{2}}} |\mu^{1/16}f|_{L^{2}}.
		\eeno
	\end{prop}
	\begin{proof}
		We only give proof to the case when $\beta=0$ because  the argument can applied to the
		case when  $|\beta|>0$ as in \cite{he2021boltzmann}.
		
		By the definition \eqref{I-ep-ga-geq-eta} of $I^{s,\gamma}(g,h)$, we have
		\beno
		\mathcal{A} :=	\langle I^{s,\gamma}(g,W_{l}h) - W_{l} I^{s,\gamma}(g,h), f\rangle = \int B^{s,\gamma}((\mu^{1/2})_{*}^{\prime} - \mu_{*}^{1/2}) (W_{l}-W^{\prime}_{l})g_{*} h f^{\prime} \mathrm{d}V = \mathcal{A}^{\leq}
		+ \mathcal{A}^{\geq}
		.
		\eeno
		Here $\mathcal{A}^{\leq}$
		and $\mathcal{A}^{\geq}$  contain $B^{s,\gamma}_{1}$ and $B^{s,\gamma,1}$ respectively.
		
		Noting that $((\mu^{1/2})_{*}^{\prime} - \mu_{*}^{1/2}) =((\mu^{1/4})_{*}^{\prime} + \mu_{*}^{1/4})((\mu^{1/4})_{*}^{\prime} - \mu_{*}^{1/4})$,
		by Cauchy-Schwarz inequality, we have
		\beno |	\mathcal{A}^{\leq}| &\leq& \big(\int B^{s,\gamma}_{1} ((\mu^{1/4})_{*}^{\prime} - \mu_{*}^{1/4})^{2} f^{\prime 2} \mathrm{d}V\big)^{1/2}
		\\&&\times\big(\int B^{s,\gamma}_{1}((\mu^{1/4})_{*}^{\prime} + \mu_{*}^{1/4})^{2}(W_{l}-W^{\prime}_{l})^{2}g^{2}_{*} h^{2}  \mathrm{d}V\big)^{1/2}
		:=(\mathcal{A}^{\leq}_{1})^{1/2}(\mathcal{A}^{\leq}_{2})^{1/2}. \eeno
		By the change of variables $(v,v_{*},\sigma) \rightarrow (v_{*}^{\prime},v^{\prime},\sigma^{\prime})$ and Taylor expansion and by  \eqref{mean-momentum-transfer},
		we  obtain
		\beno \mathcal{A}^{\leq}_{1} = \int B^{s,\gamma}_{1} ((\mu^{1/4})^{\prime} - \mu^{1/4})^{2} f^{2}_{*} \mathrm{d}V = \mathcal{N}^{s,\gamma}_{1}(f,\mu^{1/4})
		\lesssim  \frac{1}{\gamma+5}|\mu^{1/16}f|^{2}_{L^{2}}.\eeno
		Using \eqref{sigma-Wl-difference-leq-1} for $\iota =0$ and $\iota =1$, we get
		\beno \mathcal{A}^{\leq}_{2}
		\lesssim \frac{1}{\gamma+5}
		|\mu^{1/64}g|^2_{H^{a_{1}}}|\mu^{1/64}h|^2_{H^{a_{2}}}.\eeno
		
		Similarly, 
		 we have
		\beno |	\mathcal{A}^{\geq}| &\leq& \big(\int B^{s,\gamma,1} ((\mu^{1/4})_{*}^{\prime} - \mu_{*}^{1/4})^{2} f^{\prime 2} \mathrm{d}V\big)^{1/2}
		\\&&\times\big(\int B^{s,\gamma,1}((\mu^{1/4})_{*}^{\prime} + \mu_{*}^{1/4})^{2}(W_{l}-W^{\prime}_{l})^{2}g^{2}_{*} h^{2}  \mathrm{d}V\big)^{1/2}
		:=(\mathcal{A}^{\geq}_{1})^{1/2}(\mathcal{A}^{\geq}_{2})^{1/2}. \eeno
		By the change of variables $(v,v_{*},\sigma) \rightarrow (v_{*}^{\prime},v^{\prime},\sigma^{\prime})$ and noting
	$
	|(\mu^{1/4})^{\prime} - \mu^{1/4}|\lesssim \min\{1,|v-v_{*}|\sin\frac{\theta}{2}\}$,
		 Lemma \ref{symbol} implies that 
		\beno \mathcal{A}^{\geq}_{1} = \int B^{s,\gamma,1} ((\mu^{1/4})^{\prime} - \mu^{1/4})^{2} f^{2}_{*} \mathrm{d}V = \mathcal{N}^{s,\gamma,1}(f,\mu^{1/4})
		\lesssim  s^{-1}|f|^{2}_{L^{2}_{\gamma/2+s}}.\eeno
		Using \eqref{sigma-Wl-difference-geq-1} for $\iota =0$ and $\iota =1$, we get
		\beno \mathcal{A}^{\geq}_{2}
		\lesssim 
		|g|^2_{L^{2}}|h|^2_{L^{2}_{l+\gamma/2}}.\eeno

	Combining  the above  estimates completes the proof  of the proposition.
	\end{proof}
	
	\subsection{Commutator esimates}  
	By Proposition \ref{commutatorQepsilon} and Proposition \ref{commutatorforI}, we have
	the following theorem.
	\begin{thm}\label{commutator-Gamma-geq-eta}
		Let $l , l_1 \geq 0$. Let  $(a_{1},a_{2})=(\f{3}{2} + \delta, s)$ or $(0, \f{3}{2} + \delta)$, then
		\beno |\langle \Gamma^{s,\gamma}(g,W_{l}h;\beta)-W_{l}\Gamma^{s,\gamma}(g,h;\beta), f\rangle| 
		\\ \lesssim_{l, l_1} 
		s^{-1/2}  |g|_{L^{2}} |h|_{L^{2}_{l+\gamma/2}} |f|_{s,\gamma/2} + C_{\delta, s, \gamma}  |\mu^{1/64} g|_{H^{a_1}}
		|h|_{H^{a_2}_{-l_1}} |f|_{s,\gamma/2}.				\eeno
	\end{thm}

	Theorem \ref{Gamma-full-up-bound} and Theorem \ref{commutator-Gamma-geq-eta} together give the following weighted upper bound estimate.
	\begin{col}\label{Gamma-full-up-bound-with-weight} Let $l \geq 0$. 	Let  $(a_{1},a_{2})=(\f{3}{2} + \delta, s)$ or $(s, \f{3}{2} + \delta)$.  Then
		\beno 
		|\langle \Gamma^{s,\gamma}(g,h;\beta), W_{2l}f\rangle| \lesssim_{l} s^{-1} |g|_{L^{2}}|h|_{s,l+\gamma/2}|f|_{s,l+\gamma/2}
		+	C_{\delta, s,\gamma} |g|_{H^{a_1}_{-5/2}}|h|_{H^{a_2}_{-5/2}}|f|_{s,l+\gamma/2}.
	\eeno
	\end{col}

	As an application of Theorem \ref{commutator-Gamma-geq-eta}, we have the following
	corollary.
	\begin{col}\label{commutator-L-full}Let $l,l_1\geq 0$,  there holds
		\beno  |\langle W_{l} \mathcal{L}^{s,\gamma}(g; \beta_0, \beta_1) - \mathcal{L}^{s,\gamma}(W_{l}g; \beta_0, \beta_1), f\rangle| \lesssim_{l, l_1} 
		s^{-1/2}  |g|_{L^{2}_{l+\gamma/2}} |f|_{s,\gamma/2} + C_{s, \gamma}  
		|g|_{H^{s}_{-l_1}} |f|_{s,\gamma/2}.\eeno
	\end{col}
	\begin{proof} Recall that $\mathcal{L}^{s,\gamma}_{1}(g; \beta_0, \beta_1) = - \Gamma^{s,\gamma}(\pa_{\beta_{1}}\mu^{1/2}, g;\beta_{0})$.
		Taking $\delta = \f12, a_{1}= \f{3}{2} + \delta, a_{2}=s$ in Theorem \ref{commutator-Gamma-geq-eta}, we get
		\beno |\langle W_{l} \mathcal{L}^{s,\gamma}_{1}(g; \beta_0, \beta_1) - \mathcal{L}^{s,\gamma}_{1}(W_{l}g; \beta_0, \beta_1), f\rangle| \lesssim_{l, l_1} 
		s^{-1/2}  |g|_{L^{2}_{l+\gamma/2}} |f|_{s,\gamma/2} + C_{s, \gamma}  
		|g|_{H^{s}_{-l_1}} |f|_{s,\gamma/2}.\eeno
		Recall that $\mathcal{L}^{s,\gamma}_{2}(g; \beta_0, \beta_1) = - \Gamma^{s,\gamma}(g,\pa_{\beta_{1}}\mu^{1/2};\beta_{0})$.
		Taking $h=\pa_{\beta_{1}}\mu^{1/2}$ in Proposition \ref{commutatorforI}, thanks to
		\eqref{mu-weight-result},
		we can also get a $\mu$-type weight for $g$.
		Then combining this with Proposition \ref{commutatorQepsilon} with $h=\pa_{\beta_{1}}\mu^{1/2}$, we get
		\beno |\langle W_{l} \mathcal{L}^{s,\gamma}_{2}(g; \beta_0, \beta_1) - \mathcal{L}^{s,\gamma}_{2}(W_{l}g; \beta_0, \beta_1), f\rangle| \lesssim_{l} C_{s, \gamma}  |\mu^{1/64}g|_{L^{2}}|f|_{s,\gamma/2}.\eeno
		Combining the above two estimates completes the proof of the corollary.
	\end{proof}

%

%
%
%

%
%
%
	
	\section{Well-posedness and grazing limit}
	\label{proof-main-theorem}
	
	In this section, we will prove Theorem \ref{asymptotic-result}. We divide the proof into
	three subsections. The first subsection is about the {\it a priori} estimates for a linear
	equation with a general source. In Subsection \ref{well-posedenss}, we  prove the global well-posedness result in Theorem \ref{asymptotic-result}. In
	Subsection \ref{asymptotic}, we derive the global asymptotic formula \eqref{error-function-uniform-estimate} stated  in Theorem \ref{asymptotic-result}.

	\subsection{A priori estimate} \label{general}
	
	We consider
	the following linear equation with a general source $g$:
	\ben \label{lBE-general} \partial_{t}f + v\cdot \nabla_{x} f + \mathcal{L}^{s,\gamma}f= g. \een

	A temporal energy functional $\mathcal{I}^{N}(f)$
	satisfying 	for some generic constant $C_{1}$ that
	\ben \label{temporal-bounded-by-energy}
	|\mathcal{I}^{N}(f)| \leq C_{1} \|\mu^{1/8} f\|^{2}_{H^{N}_{x}L^{2}}
	\een
	is used 
	to capture the dissipation of the macro components $\mathrm{M}(t,x) := (a(t,x), b(t,x), c(t,x))$ 
 of the solution $f$.
	\begin{lem}\label{estimate-for-highorder-abc}
		There exist  two generic constants $C_2, c_{0} > 0$ such that for any $N \geq 2$,
		\ben \label{solution-property-part2} \frac{\mathrm{d}}{\mathrm{d}t}\mathcal{I}^{N}(f) + c_{0}|\mathrm{M}|^{2}_{H^{N}_{x}} \leq C_2( C_{s,\gamma}^2
		\|\mu^{1/8} f_{2}\|^{2}_{H^{N}_{x}L^2} + \mathrm{NL}^{N}(g)), 
		\een
		where 
		\beno
		\mathrm{NL}^{N}(g) := \sum_{j=1}^{13} |\langle  \pa^{\alpha}g, \mu^{1/2} P_j\rangle|_{H^{N-1}_{x}}^2.
		\eeno
		Here, the standard thirteen moments polynomials $P_{j}$ are defined by
		\beno P_{1} = 1, P_{2} = v_{1}, P_{3} = v_{2},P_{4} = v_{3},  P_{5} = v_{1}^{2}, P_{6} = v_{2}^{2}, P_{7} = v_{3}^{2}, 
		\\ P_{8} = v_{1}v_{2}, P_{9} = v_{2}v_{3},P_{10} = v_{3}v_{1},  P_{11} = |v|^{2}v_{1}, P_{12} = |v|^{2}v_{2},P_{13} = |v|^{2}v_{3}. \eeno
	\end{lem}

	We refer readers to \cite{guo2003classical,duan2008cauchy,he2022asymptotic} for the detailed proof of Lemma \ref{estimate-for-highorder-abc}. 


	\begin{lem} \label{coercivity-L-alpha-beta}
		Let $|\alpha|+|\beta| \leq N$, $q\ge 0$, then
		\beno
		(\mathcal{L}^{s,\gamma} W_{q}\pa^{\alpha}_{\beta}f, W_{q}\pa^{\alpha}_{\beta}f) \geq
		(7/8)\lambda_{s}\| W_{q}\pa^{\alpha}_{\beta} f_2\|^2_{L^{2}_{x}L^2_{s,\gamma/2}} - C_{q,|\beta|}(\|
		\mu^{1/8}\pa^{\alpha} f_2\|^2_{L^{2}_{x}L^{2}}
		+|\partial^{\alpha}\mathrm{M}|^{2}_{L^{2}_{x}}).
		\eeno
	\end{lem}
	\begin{proof} By Theorem \ref{micro-dissipation}
		 and recalling the constant $\lambda_{s}$ in \eqref{lower-bound-on-lambda-explicit}, 
		\beno (\mathcal{L}^{s,\gamma} W_{q}\pa^{\alpha}_{\beta}f, W_{q}\pa^{\alpha}_{\beta}f)\ge \lambda_{s} \|(\mathbf{I}-\mathbf{P})W_{q}\pa^{\alpha}_{\beta}f\|_{L^{2}_{x}L^2_{s,\gamma/2}}^2. \eeno
		It is straightforward  to check for any $ 0 \leq  \alpha < 1$ that 
		\ben  \label{basic-inequality-1}
		|x|^{2} \geq \alpha |y|^{2}- \frac{\alpha}{1-\alpha}|x-y|^{2}.
		\een
		By the macro-micro decomposition $f=f_{1}+f_{2}$, we deduce that
		\beno (\mathcal{L}^{s,\gamma} W_{q}\pa^{\alpha}_{\beta}f, W_{q}\pa^{\alpha}_{\beta}f)\ge \lambda_{s}\|(\mathbf{I}-\mathbf{P})W_{q}\pa^{\alpha}_{\beta}(f_1+f_2)\|_{L^{2}_{x}L^2_{s,\gamma/2}}^{2}\\
		\ge (7/8)\lambda_{s} \| W_{q}\pa^{\alpha}_{\beta} f_2\|_{L^{2}_{x}L^2_{s,\gamma/2}}^2- C_{q,|\beta|}(\|\mu^{1/8} \pa^{\alpha} f_2\|_{L^{2}}^2+|\partial^{\alpha}\mathrm{M}|^{2}_{L^{2}_{x}}), \eeno
		where we have used  \eqref{basic-inequality-1} to take out $W_{q}\pa^{\alpha}_{\beta} f_2$ as the leading term. This completes the proof of the lemma.
	\end{proof}
	
	We now apply the commutator estimate obtained
	in
	Corollary \ref{commutator-L-full} to derive the following lemma.
	
	\begin{lem} \label{commutator-L-alpha-beta}
		Let $|\alpha|+|\beta| \leq N, \beta_{0} + \beta_{1} + \beta_{2} = \beta$, $q \geq 0$,  then for any $0<\delta \leq 1$, we have
		\beno
		&&	|\left(
		W_{q} \mathcal{L}^{s,\gamma}(\pa^{\alpha}_{\beta_{2}}f; \beta_0, \beta_1) - \mathcal{L}^{s,\gamma}(W_{q}\pa^{\alpha}_{\beta_{2}}f; \beta_0, \beta_1),
		W_{q}\pa^{\alpha}_{\beta}f\right)| 
		\\&\leq&
		\delta \|\pa^{\alpha}_{\beta}f_2\|_{L^{2}_{x}L^{2}_{s,q+\gamma/2}}^{2} + \delta^{-1} C_{q} C_{s,\gamma}^2 (\|\pa^{\alpha}_{\beta_{2}}f_2\|_{L^{2}_{x}L^{2}_{q+\gamma/2}}^2+\|\pa^{\alpha}_{\beta_{2}}f_2\|_{H^{s}_{\gamma/2}}^{2}) + \delta^{-1} C_{q,|\beta|} C_{s,\gamma}^2 |\partial^{\alpha}\mathrm{M}|^{2}_{L^{2}_{x}}.
		\eeno
	\end{lem}
	\begin{proof}
		By taking $l_1 = \gamma/2$ in
		Corollary \ref{commutator-L-full} and by using  the decomposition $f=f_1+f_2=\mathbf{P}f+f_2$,
	for any $0<\delta<1$, we have
		\beno && |\left(
		W_{q} \mathcal{L}^{s,\gamma}(\pa^{\alpha}_{\beta_{2}}f; \beta_0, \beta_1) - \mathcal{L}^{s,\gamma}(W_{q}
		\pa^{\alpha}_{\beta_{2}}f; \beta_0, \beta_1),
		W_{q}\pa^{\alpha}_{\beta}f\right)|  \\&\lesssim_{q}& 
		s^{-1/2}  \|\pa^{\alpha}_{\beta_{2}}f\|_{L^{2}_{x}L^{2}_{q+\gamma/2}} \|\pa^{\alpha}_{\beta}f\|_{L^{2}_{x}L^{2}_{s,q+\gamma/2}} + C_{s, \gamma}  
		\|\pa^{\alpha}_{\beta_{2}}f\|_{H^{s}_{\gamma/2}} \|\pa^{\alpha}_{\beta}f\|_{L^{2}_{x}L^{2}_{s,q+\gamma/2}}
		\\&\lesssim& \delta \|\pa^{\alpha}_{\beta}f\|_{L^{2}_{x}L^{2}_{s,q+\gamma/2}}^{2} + \delta^{-1} C_{q} C_{s,\gamma}^2 (\|\pa^{\alpha}_{\beta_{2}}f\|_{L^{2}_{x}L^{2}_{q+\gamma/2}}^2+\|\pa^{\alpha}_{\beta_{2}}f\|_{H^{s}_{\gamma/2}}^{2})
		\\&\leq&
		\delta \|\pa^{\alpha}_{\beta}f_2\|_{L^{2}_{x}L^{2}_{s,q+\gamma/2}}^{2} + \delta^{-1} C_{q} C_{s,\gamma}^2 (\|\pa^{\alpha}_{\beta_{2}}f_2\|_{L^{2}_{x}L^{2}_{q+\gamma/2}}^2+\|\pa^{\alpha}_{\beta_{2}}f_2\|_{H^{s}_{\gamma/2}}^{2}) + \delta^{-1} C_{q,|\beta|} C_{s,\gamma}^2 |\partial^{\alpha}\mathrm{M}|^{2}_{L^{2}_{x}}.
		\eeno
	This completes the proof of the lemma.
	\end{proof}
	
	The following lemma is about  the commutator $[\partial_{\beta}, \mathcal{L}^{s,\gamma}]$.
	
	\begin{lem} \label{commutator-L-with-beta}
		Let $|\alpha|+|\beta| \leq N, |\beta| \geq 1$, $q \geq 0$,  then
		\beno
		|(W_{q}[\partial_{\beta}, \mathcal{L}^{s,\gamma}]\pa^{\alpha}f ,W_{q}\pa^{\alpha}_{\beta}f)| &\leq&
		\delta \|\pa^{\alpha}_{\beta}f_2\|_{L^{2}_{x}L^{2}_{s,q+\gamma/2}}^{2} + \delta^{-1} C_{q,N} C_{s,\gamma}^2 |\partial^{\alpha}\mathrm{M}|^{2}_{L^{2}_{x}}
		\\&&	+ \delta^{-1} C_{q,N} C_{s,\gamma}^2 \sum_{\beta_{2}<\beta} (\|\pa^{\alpha}_{\beta_{2}}f_2\|_{L^{2}_{x}L^{2}_{q+\gamma/2}}^2+ \mathrm{1}_{q>0}
		\|\pa^{\alpha}_{\beta_{2}}f_2\|_{H^{s}_{\gamma/2}}^{2}).
		\eeno
	\end{lem}
	\begin{proof}
		By recalling
		$ \mathcal{L}^{s,\gamma}g = -\Gamma^{s,\gamma}(\mu^{1/2},g) - \Gamma^{s,\gamma}(g, \mu^{1/2}),$ \eqref{Gamma-beta}, \eqref{alpha-beta-on-Gamma} and \eqref{beta-version-L-s-gamma},
		we have
		\ben  \pa_{\beta}\mathcal{L}^{s,\gamma}g &=& \mathcal{L}^{s,\gamma}\pa_{\beta}g
		-\sum_{\beta_{0}+\beta_{1}+\beta_{2}= \beta, \beta_{2} < \beta} C^{\beta_{0},\beta_{1},\beta_{2}}_{\beta}
		[\Gamma^{s,\gamma}(\pa_{\beta_{1}}\mu^{1/2}, \pa_{\beta_{2}}g;\beta_{0}) + \Gamma^{s,\gamma}(\pa_{\beta_{1}}g, \pa_{\beta_{2}}\mu^{1/2};\beta_{0})]
		\nonumber \\ \label{alpha-beta-Lep}
		&=& \mathcal{L}^{s,\gamma}\pa_{\beta}g
		-\sum_{\beta_{0}+\beta_{1}+\beta_{2}= \beta, \beta_{2} < \beta} C^{\beta_{0},\beta_{1},\beta_{2}}_{\beta}
		\mathcal{L}^{s,\gamma}(\partial_{\beta_{2}}g;\beta_{0},\beta_{1}). 
		\een
		Note that $[\partial_{\beta}, \mathcal{L}^{s,\gamma}] h= \sum_{\beta_{2}<\beta} C^{\beta_{0},\beta_{1},\beta_{2}}_{\beta}\mathcal{L}^{s,\gamma}(\partial_{\beta_{2}} h; \beta_{0},\beta_{1})$. Thus
		\beno W_{q}[\partial_{\beta}, \mathcal{L}^{s,\gamma}]\pa^{\alpha}f &=& W_{q}\sum_{\beta_{2}<\beta} C^{\beta_{0},\beta_{1},\beta_{2}}_{\beta}\mathcal{L}^{s,\gamma}(\partial^{\alpha}_{\beta_{2}}f ;\beta_{0},\beta_{1}) 
		\\&=&\sum_{\beta_{2}<\beta} C^{\beta_{0},\beta_{1},\beta_{2}}_{\beta} \mathcal{L}^{s,\gamma}(W_{q}\partial^{\alpha}_{\beta_{2}}f; \beta_{0},\beta_{1})  +
		\sum_{\beta_{2}<\beta}C^{\beta_{0},\beta_{1},\beta_{2}}_{\beta}
		[W_{q},\mathcal{L}^{s,\gamma}(\cdot ;\beta_{0},\beta_{1})]\partial^{\alpha}_{\beta_{2}}f. \eeno
		By upper bound estimate in Prop. \ref{part-l}, we get
		\beno |\mathcal{L}^{s,\gamma}(W_{q}\partial^{\alpha}_{\beta_{2}}f; \beta_{0},\beta_{1}) ,W_{q}\pa^{\alpha}_{\beta}f)|
		\lesssim C_{s,\gamma}
		\|\partial^{\alpha}_{\beta_{2}}f\|_{L^{2}_{x}L^{2}_{s,q+\gamma/2}}\|\pa^{\alpha}_{\beta}f\|_{L^{2}_{x}L^{2}_{s,q+\gamma/2}}
	\\	\leq
		\delta  \|\pa^{\alpha}_{\beta}f_{2}\|_{L^{2}_{x}L^{2}_{s,q+\gamma/2}}^{2} + \delta^{-1} C_{s,\gamma}^2 C_{q,N}\|\pa^{\alpha}_{\beta_{2}}f_{2}\|_{L^{2}_{x}L^{2}_{s,q+\gamma/2}}^{2}
		+ \delta^{-1} C_{s,\gamma}^2 C_{q,N}|\partial^{\alpha}\mathrm{M}|_{L^{2}_{x}}^{2},
		\eeno
		where we have used $f = f_{1} + f_{2}$ and the definition of $a,b,c$. We apply  Lemma \ref{commutator-L-alpha-beta} for  $[W_{q},\mathcal{L}^{s,\gamma}(\cdot ;\beta_{0},\beta_{1})]$. Note that
		if $q=0$, the commutator $[W_{q},\mathcal{L}^{s,\gamma}(\cdot ;\beta_{0},\beta_{1})]=0$.
		Taking sum over $\beta_{2}<\beta$ completes the proof of the lemma.
	\end{proof}

	For any non-negative integers $n,m$, recall
	\beno \|f\|^{2}_{H^{n}_{x}\dot{H}^{m}_{l}} = \sum_{|\alpha|\leq n,|\beta| = m}\|\partial^{\alpha}_{\beta}f\|^{2}_{L^{2}_{x}L^{2}_{l}}, \quad
	\|f\|^{2}_{H^{n}_{x}\dot{H}^{m}_{s,l}} = \sum_{|\alpha|\leq n,|\beta| = m}\|\partial^{\alpha}_{\beta}f\|^{2}_{L^{2}_{x}L^{2}_{s,l}}. \eeno
	Let $N \geq 4, l\geq -N(\gamma+2s)$. For some generic constants $M$, $ L_{j}$ and $ K_{j}$ with  $0 \leq j \leq N$ (which may depend on $N, l, s, \gamma$ and will be determined later), we define
	\ben \label{def-energy-combination}
	\Xi_{N,l}^{s,\gamma}(f) &=& M\mathcal{I}^{N}(f)+ \sum_{j=0}^{N} L_{j}\|f\|^{2}_{H^{N-j}_{x}\dot{H}^{j}}+ \sum_{j=0}^{N} K_{j}\|f\|^{2}_{H^{N-j}_{x}\dot{H}^{j}_{l+j(\gamma+2s)}},
	\\ \label{def-dissipation-combination}
	\tilde{\mathcal{D}}^{s,\gamma}_{N,l}(f)&=& c_{0}M|\mathrm{M}|^{2}_{H^{N}_{x}} + \lambda_{s}
	\sum_{j=0}^{N} L_{j}\|f_2\|^{2}_{H^{N-j}_{x}\dot{H}^{j}_{s,\gamma/2}}
	+\lambda_{s}\sum_{j=0}^{N}K_{j}\|f_2\|^{2}_{H^{N-j}_{x}\dot{H}^{j}_{s,l+j(\gamma+2s)+\gamma/2}}. \een

	We are now ready to prove the a priori estimate of \eqref{lBE-general}.
	
	\begin{prop}\label{essential-estimate-of-micro-macro} Let  $N \geq 4, l\geq -N(\gamma+2s)$. Suppose $f$ is a solution to \eqref{lBE-general}, then 
		\ben \label{essential-micro-macro-result-2} \frac{\mathrm{d}}{\mathrm{d}t}\Xi_{N,l}^{s,\gamma}(f) +  \frac{1}{4} \tilde{\mathcal{D}}^{s,\gamma}_{N,l}(f) &\leq& M C_2 \mathrm{NL}^{N}(g)
		+ \sum_{j=0}^{N}2L_{j}\sum_{|\alpha|\le N-j,|\beta|=j}(\pa^{\alpha}_{\beta}g, \pa^{\alpha}_{\beta}f)  \\&&+\sum_{j=0}^{N}2K_{j}\sum_{|\alpha|\le N-j,|\beta|=j}(W_{l+j(\gamma+2s)}\pa^{\alpha}_{\beta}g, W_{l+j(\gamma+2s)}\pa^{\alpha}_{\beta}f).
		\nonumber \een
		The constants in \eqref{def-energy-combination} and
		\eqref{def-dissipation-combination} satisfy 
		\ben \label{c-relation-1}
		\max\{M, \{L_{j}\}_{0 \leq j \leq N}, \{K_{j}\}_{0 \leq j \leq N}\} = L_{0} = Z_{s,\gamma,N,l} := X_{s,\gamma} Y_{s,\gamma,l} (U_{s,\gamma,N,l} W_{s,\gamma,N,l})^{N},
		\\
		\label{c-relation-2}
		\min \{c_0 M, \{\lambda_{s} L_{j}\}_{0 \leq j \leq N}, \{\lambda_{s} K_{j}\}_{0 \leq j \leq N}\} = \lambda_{s},
		\een
		where
		\ben 
		\label{def-M1}
		X_{s,\gamma} =  \lambda_{s}^{-1} C_{s,\gamma}^2 C_{2},
		\\ \label{def-M2}
		Y_{s,\gamma,l} = 8^{l/s} \lambda_{s}^{-l/s} \left(C_{l} (\gamma+2s+3)^{-1} \right)^{1+l/s},
		\\
		\label{def-constant-W}
		W_{s,\gamma,N,l} =  \lambda_{s}^{-2} C_{s,\gamma}^2 C_{N},
		\\
		\label{def-constant-U}
		U_{s,\gamma,N,l} =    \max\{\lambda_{s}^{-2}C_{s,\gamma}^2 C_{N,l}, 8^{l/s} (C_{l}(\gamma+2s+3)^{-1}\lambda_{s}^{-1})^{1+l/s}
		\}.
		\een
		Here,
		$C_{N}, C_{l}, C_{N,l}$ are some large constants  depending only on the corresponding indices. Recalling the constants $\lambda_{s}$ from
		\eqref{lower-bound-on-lambda-explicit} and
		$C_{s, \gamma}$ from \eqref{def-C-s-gamma}, it is straightforward to check that  for any fixed $N,l$,
		there is a function $(x_1, x_2) \in (0,1) \times (0, 3] \to 
		Z_{N,l}(x_1, x_2) \in (0, \infty)$ satisfying
		\eqref{dependence-s-gamma} and
		\eqref{property-of-Z-N-l}.
	\end{prop}

	\begin{proof} We divide the proof into three steps to construct the energy functional $\Xi_{N,l}^{s,\gamma}(f)$ in \eqref{def-energy-combination}. 

		{\it Step 1: Propagation of $\|f\|^{2}_{H^{N}_{x}L^{2}}$}.
		By applying $\partial^{\alpha}$ to equation \eqref{lBE-general}, taking inner product with $\partial^{\alpha}f$, taking sum over $|\alpha|\leq N$, we have
		\ben \label{inner-product-with-pure-x}
		\frac{1}{2}\frac{\mathrm{d}}{\mathrm{d}t}\|f\|^{2}_{H^{N}_{x}L^{2}}  + \sum_{|\alpha| \leq N} (\mathcal{L}^{s,\gamma}\pa^{\alpha}f, \pa^{\alpha}f) = \sum_{|\alpha| \leq N}(\pa^{\alpha}g, \pa^{\alpha}f). \een
		By Theorem \ref{micro-dissipation} and using 
		$\partial^{\alpha}f_{2} = (\partial^{\alpha}f)_{2}$, we have $(\mathcal{L}^{s, \gamma} \partial^{\alpha}f, \partial^{\alpha}f) \geq \lambda_{s}\|\partial^{\alpha}f_{2}\|^{2}_{L^{2}_{x}L^2_{s,\gamma/2}}$, which yields
		\ben \label{solution-property-part-g}\frac{1}{2}\frac{\mathrm{d}}{\mathrm{d}t}\|f\|^{2}_{H^{N}_{x}L^{2}} + \lambda_{s} \|f_{2}\|^{2}_{H^{N}_{x}L^2_{s,\gamma/2}} \leq 
		\sum_{|\alpha| \leq N}
		(\pa^{\alpha}g, \pa^{\alpha}f).\een
		Multiplying \eqref{solution-property-part-g} by a large constant $2M_{1}$ and adding it to \eqref{solution-property-part2},
		we get
		\ben \label{essential-micro-macro-result} \frac{\mathrm{d}}{\mathrm{d}t}(M_{1}\|f\|^{2}_{H^{N}_{x}L^{2}}+\mathcal{I}^{N}(f))+ (c_{0}|\mathrm{M}|^{2}_{H^{N}_{x}}+
		M_{1}\lambda_{s}
		\|f_{2}\|^{2}_{H^{N}_{x}L^2_{s,\gamma/2}}) \leq 2M_{1}\sum_{|\alpha| \leq N}
		(\pa^{\alpha}g
		, \pa^{\alpha}f)+ C_2 \mathrm{NL}^{N}(g).  \een
		Here $M_{1}$ is large enough such that $M_{1} \geq 2C_{1}$ and $M_{1} \lambda_{s} \geq C_2 C_{s,\gamma}^2$ to insure from \eqref{temporal-bounded-by-energy} that
		\beno
		\f{1}{2}M_{1} \|f\|^{2}_{H^{N}_{x}L^{2}}
		\leq 		M_{1}\|f\|^{2}_{H^{N}_{x}L^{2}}+\mathcal{I}^{N}(f) \leq \f{3}{2}M_{1} \|f\|^{2}_{H^{N}_{x}L^{2}},
		\\
		M_{1}\lambda_{s}\|f_{2}\|^{2}_{H^{N}_{x}L^2_{s,\gamma/2}} \geq C_2 C_{s,\gamma}^2
		\|\mu^{1/8} f_{2}\|^{2}_{H^{N}_{x}L^2}.
		\eeno
		Note that 
 the term $C_2 C_{s,\gamma}^2
		\|\mu^{1/8} f_{2}\|^{2}_{H^{N}_{x}L^2}$ in \eqref{solution-property-part2} is 
		absorbed by the dissipation of the microscopic component  $ \|f_{2}\|^{2}_{H^{N}_{x}L^2_{s,\gamma/2}}$ in
		\eqref{solution-property-part-g}.
		We may assume $\lambda_{s} \leq 1$ and $C_2 \gg C_1$. Then we can take $M_1 = X_{s,\gamma}$ defined in
		\eqref{def-M1}.

		{\it Step 2: Propagation of $\|f\|^{2}_{H^{N}_{x}L^{2}_{l}}$}.
		By applying $W_{l}\partial^{\alpha}$ to equation \eqref{lBE-general}, taking inner product with $W_{l}\partial^{\alpha}f$, taking sum over $|\alpha|\leq N$, we have
		\ben \label{inner-product-with-pure-x-weight}
		\frac{1}{2}\frac{\mathrm{d}}{\mathrm{d}t}\|f\|^{2}_{H^{N}_{x}L^{2}_{l}}  + \sum_{|\alpha| \leq N} (W_{l}\mathcal{L}^{s,\gamma}\pa^{\alpha}f,W_{l}\pa^{\alpha}f) = \sum_{|\alpha| \leq N}(W_{l}\pa^{\alpha}g,W_{l}\pa^{\alpha}f). \een
		Using commutator to transfer weight gives
		\beno W_{l}\mathcal{L}^{s,\gamma}\pa^{\alpha}f = \mathcal{L}^{s,\gamma}W_{l}\pa^{\alpha}f + [W_{l}, \mathcal{L}^{s,\gamma}]\pa^{\alpha}f.
		\eeno
		By Lemma \ref{coercivity-L-alpha-beta}, we get
		\beno
		(\mathcal{L}^{s, \gamma} W_{l}\pa^{\alpha}f, W_{l}\pa^{\alpha}f) \geq
		(7/8)\lambda_{s} \|\pa^{\alpha}f_{2}\|_{L^{2}_{x}L^{2}_{s,l+\gamma/2}}^{2}
		-  C_{l}( \|\partial^{\alpha}f_{2}\|^{2}_{L^{2}_{x}L^2_{\gamma/2}}
		+|\partial^{\alpha}\mathrm{M}|^{2}_{L^{2}_{x}}).
		\eeno
		By	\eqref{sigma-Wl-difference-leq-1} and \eqref{sigma-Wl-difference-geq-1},
		we have 
		\ben \label{special-commutator-L}
		|	\langle [W_{l}, \mathcal{L}^{s,\gamma}]f,W_{l}f \rangle |
		\lesssim_{l} (\gamma+5)^{-1} |f|^{2}_{L^{2}_{l+\gamma/2}}.
		\een
		Thus, 
		\beno
		|([W_{l}, \mathcal{L}^{s,\gamma}]\pa^{\alpha}f,W_{l}\pa^{\alpha}f)| \lesssim_{l}  (\gamma+5)^{-1}
		\|\pa^{\alpha}f\|^{2}_{L^{2}_{x}L^{2}_{l+\gamma/2}} \leq C_{l} (\gamma+5)^{-1} 	\|\pa^{\alpha}f_2\|^{2}_{L^{2}_{x}L^{2}_{l+\gamma/2}} + C_{l} (\gamma+5)^{-1} 	|\partial^{\alpha}\mathrm{M}|^{2}_{L^{2}_{x}}
		.
		\eeno
		Since $|h|_{L^{2}_{q}}^2 \leq \delta |h|_{L^{2}_{q+s}}^2 + \delta^{-q/s} |h|_{L^{2}}^2$ for any $1>\delta>0$, we have
		\ben \label{reduce-to-gamma/2-2}
		\|\pa^{\alpha}f_{2}\|^{2}_{L^{2}_{x}L^{2}_{l+\gamma/2}} \leq \delta \|\pa^{\alpha}f_{2}\|^{2}_{L^{2}_{x}L^{2}_{s,l+\gamma/2}}
		+ \delta^{-q/s} \|\pa^{\alpha}f_{2}\|^{2}_{L^{2}_{x}L^{2}_{s,\gamma/2}}.\een
		By taking $\delta = \delta_{l, s, \gamma}$ where
		$\delta_{l, s, \gamma} C_{l} (\gamma+5)^{-1}
		= \lambda_{s}/8$, we get
		\ben \label{weighted-pure-x-2}
		\frac{\mathrm{d}}{\mathrm{d}t}\|f\|^{2}_{H^{N}_{x}L^{2}_{l}}  +  \f32 \lambda_{s}
		\|f_{2}\|^{2}_{H^{N}_{x}L^{2}_{s,l+\gamma/2}} \leq C_{l,s,\gamma}\|f_{2}\|^{2}_{H^{N}_{x}L^{2}_{s,\gamma/2}}+
		C_{l,\gamma}
		|\mathrm{M}|^{2}_{H^{N}_{x}}
		+2\sum_{|\alpha| \leq N}(W_{l}\pa^{\alpha}g
		,W_{l}\pa^{\alpha}f),  \een
for some constants $C_{l,\gamma}$ and $C_{l,s,\gamma}$
satisfying
		\ben \label{def-C-l-s-gamma}
		C_{l,\gamma} \leq C_{l} (\gamma+2s+3)^{-1}, \quad 
		C_{l,s,\gamma} \leq \left(C_{l} (\gamma+2s+3)^{-1} \right)^{1+l/s} (\lambda_{s}/8)^{-l/s}.
		\een

		We choose a constant $M_{2}$ large enough such that 
		\beno
		c_{0}M_{2}/2 \geq C_{l,\gamma},  \quad M_{2}M_{1}\lambda_{s}/2 \geq C_{l,s,\gamma}.
		\eeno 
		Recalling $M_1 = X_{s,\gamma}$ defined in
		\eqref{def-M1}, for simplicity, we can take $M_2 = Y_{s,\gamma,l}$ defined in
		\eqref{def-M2}. 
		Then the combination \eqref{essential-micro-macro-result}$\times M_{2}+$\eqref{weighted-pure-x-2} yields
		\ben \label{essential-micro-macro-result-3} &&\frac{\mathrm{d}}{\mathrm{d}t}(M_{2}\mathcal{I}^{N}(f)+M_{1}M_{2}\|f\|^{2}_{H^{N}_{x}L^{2}}+\|f\|^{2}_{H^{N}_{x}L^{2}_{l}})
		\\&&+ \frac{1}{2}(M_{2}c_{0}|\mathrm{M}|^{2}_{H^{N}_{x}}+M_{2}M_{1}\lambda_{s}\|f_{2}\|^{2}_{H^{N}_{x}L^2_{s,\gamma/2}}
		+ \lambda_{s}\|f_{2}\|^{2}_{H^{N}_{x}L^2_{s,l+\gamma/2}})
		\nonumber \\&\leq&  M_{2} C_2 \mathrm{NL}^{N}(g)
		+2M_{2}M_{1}\sum_{|\alpha| \leq N}
		(\pa^{\alpha}g
		, \pa^{\alpha}f)
		+2\sum_{|\alpha| \leq N}(W_{l}\pa^{\alpha}g
		,W_{l}\pa^{\alpha}f).
		\nonumber \een

		{\it Step 3: Propagation of $\|f\|^{2}_{H^{N-j}_{x}\dot{H}^{j}}$ and $\|f\|^{2}_{H^{N-j}_{x}\dot{H}^{j}_{l+j(\gamma+2s)}}$ for $j \geq 1$}. 	For notation convenience, set
		\beno \mathcal{X}^{i}(f) &:=& M^{i}\mathcal{I}^{N}(f)+ \sum_{0 \leq j \leq i}L^{i}_{j}\|f\|^{2}_{H^{N-j}_{x}\dot{H}^{j}}
		+\sum_{0 \leq j \leq i}K^{i}_{j}\|f\|^{2}_{H^{N-j}_{x}\dot{H}^{j}_{l+j(\gamma+2s)}},
		\\ \mathcal{Y}^{i}(f) &:=&
		M^{i}c_{0}|\mathrm{M}|^{2}_{H^{N}_{x}} +
		\lambda_{s}\sum_{j=0}^{i}L^{i}_{j}\|f_{2}\|^{2}_{H^{N-j}_{x}\dot{H}^{j}_{s,\gamma/2}}
		+\lambda_{s}\sum_{j=0}^{i}K^{i}_{j}\|f_{2}\|^{2}_{H^{N-j}_{x}\dot{H}^{j}_{s,l+j(\gamma+2s)+\gamma/2}},\\
		\mathcal{Z}^{i}(g, f) &:=& M^{i} C_2 \mathrm{NL}^{N}(g)
		+ \sum_{j=0}^{i}2L^{i}_{j}\sum_{|\alpha|\le N-j,|\beta|=j}(\pa^{\alpha}_{\beta}g, \pa^{\alpha}_{\beta}f) \nonumber \\&&+\sum_{j=0}^{i}2K^{i}_{j}\sum_{|\alpha|\le N-j,|\beta|=j}(W_{l+j(\gamma+2s)}\pa^{\alpha}_{\beta}g, W_{l+j(\gamma+2s)}\pa^{\alpha}_{\beta}f).
		\eeno
		We will use mathematical induction to prove that for any $0\leq i \leq N$, there are some constants $M^{i}, L^{i}_{j}, K^{i}_{j} \geq 1, 0 \leq j \leq i$ satisfying
		\beno
		M^{i}c_0 \geq 1, L^{i}_{j} \geq L^{i}_{j+1},
		K^{i}_{j} \geq K^{i}_{j+1},
		L^{i}_{j} \geq K^{i}_{j},
		\eeno
		such that
		\ben \label{essential-micro-macro-result-4} \frac{\mathrm{d}}{\mathrm{d}t}\mathcal{X}^{i}(f)+  2^{-1-i/N} \mathcal{Y}^{i}(f) \leq 	\mathcal{Z}^{i}(g, f). \een

		It is easy to check that \eqref{essential-micro-macro-result-4} is valid for $i=0$ thanks to \eqref{essential-micro-macro-result-3}. More precisely, 
		$M^{0} = M_2, L^{0}_{0} = M_1 M_2, K^{0}_{0}=1$.

		We obtain \eqref{essential-micro-macro-result-2} by taking $i=N$ in \eqref{essential-micro-macro-result-4}. 
		
		Assume \eqref{essential-micro-macro-result-4} is valid for $i = k$ for some $0 \leq k \leq N-1$. We now prove \eqref{essential-micro-macro-result-4} is also valid for $i=k+1$ by first considering  $\|f\|^{2}_{H^{N-j}_{x}\dot{H}^{j}}$ and then 
		$\|f\|^{2}_{H^{N-j}_{x}\dot{H}^{j}_{l+j(\gamma+2s)}}$.
		
		Let $\alpha$ and $\beta$ be multi-indices such that $|\alpha|\leq N-(k+1)$ and $|\beta|= k+1 \geq 1$. Applying $\pa^{\alpha}_{\beta}$ to both sides of \eqref{lBE-general} gives 
		\ben \label{weight-LBE-3} \partial_{t}\pa^{\alpha}_{\beta}f + v\cdot \nabla_{x} \pa^{\alpha}_{\beta}f +\sum_{\beta_{1}\leq \beta,|\beta_{1}|=1} \pa^{\alpha+\beta_{1}}_{\beta-\beta_{1}}f + \pa^{\alpha}_{\beta}\mathcal{L}^{s,\gamma}f = \pa^{\alpha}_{\beta}g. \een
		Taking inner product with $\pa^{\alpha}_{\beta} f$ over $(x,v)$, one has
		\ben \label{weight-mixed-derivative-no} \frac{1}{2}\frac{\mathrm{d}}{\mathrm{d}t}\|\pa^{\alpha}_{\beta}f \|^{2}_{L^{2}_{q}}  +  \sum_{\beta_{1}\leq \beta,|\beta_{1}|=1}(\pa^{\alpha+\beta_{1}}_{\beta-\beta_{1}}f,\pa^{\alpha}_{\beta}f) + (\pa^{\alpha}_{\beta}\mathcal{L}^{s,\gamma}f,\pa^{\alpha}_{\beta}f) = (\pa^{\alpha}_{\beta}g,\pa^{\alpha}_{\beta}f).  \een
		
		{\it Estimate of $(\pa^{\alpha+\beta_{1}}_{\beta-\beta_{1}}f,\pa^{\alpha}_{\beta}f)$.} By  Cauchy-Schwarz inequality and using $f=f_{1}+f_{2}$, we get
		\ben \label{weight-term-1-no}
		|(\pa^{\alpha+\beta_{1}}_{\beta-\beta_{1}}f, \pa^{\alpha}_{\beta}f)| &\leq& \|\pa^{\alpha+\beta_{1}}_{\beta-\beta_{1}}f\|_{L^2_{x}L^{2}_{-s-\gamma/2}}\|\pa^{\alpha}_{\beta}f\|_{L^2_{x}L^{2}_{s+\gamma/2}}
		\\ \nonumber  &\leq& \|\pa^{\alpha+\beta_{1}}_{\beta-\beta_{1}}f\|_{L^2_{x}L^{2}_{s,-(\gamma+2s)+\gamma/2}}\|\pa^{\alpha}_{\beta}f\|_{L^2_{x}L^{2}_{s,\gamma/2}}
		\\\nonumber  &\leq&
		\delta \|\pa^{\alpha}_{\beta}f_2\|_{L^2_{x}L^{2}_{s,\gamma/2}}^2 + C_{\delta}
		\|\pa^{\alpha+\beta_{1}}_{\beta-\beta_{1}}f_2\|_{L^2_{x}L^{2}_{s,l+|\beta-\beta_{1}|(\gamma+2s)+\gamma/2}}
		+ C_{\delta} C_{|\beta|} |\mathrm{M}|^{2}_{H^{N-k}_{x}}. \een
		Here $C_{\delta} \lesssim \delta^{-1}$.

	{\it Estimate of $(\pa^{\alpha}_{\beta}\mathcal{L}^{s,\gamma}f,\pa^{\alpha}_{\beta}f) $.}
		Using $	\pa^{\alpha}_{\beta}\mathcal{L}^{s,\gamma}f = \mathcal{L}^{s, \gamma}\pa^{\alpha}_{\beta}f +
		[\partial_{\beta}, \mathcal{L}^{s,\gamma}]\pa^{\alpha}f$, 
		by Lemma \ref{coercivity-L-alpha-beta} and Lemma \ref{commutator-L-with-beta} and
		Lemma \ref{commutator-L-with-beta} with  $\delta = \lambda_{s}/8$,
		we have
		\ben \label{weight-term-2-no}
	(\pa^{\alpha}_{\beta}\mathcal{L}^{s,\gamma}f, \pa^{\alpha}_{\beta}f)
	\geq
		(3/4)\lambda_{s} \| \pa^{\alpha}_{\beta} f_2\|^2_{L^{2}_{x}L^2_{s,\gamma/2}}
		-  \lambda_{s}^{-1} C_{s,\gamma}^2 C_{N}
		|\partial^{\alpha}\mathrm{M}|_{L^{2}_{x}}^{2}
		- \lambda_{s}^{-1} C_{s,\gamma}^2 C_{N} \sum_{\beta_{2}<\beta} \|\pa^{\alpha}_{\beta_{2}}f_{2}\|_{L^{2}_{x}L^{2}_{s,\gamma/2}}^{2} .
		\een
		By plugging \eqref{weight-term-1-no} and \eqref{weight-term-2-no} into \eqref{weight-mixed-derivative}, taking $\delta = \lambda_{s}/4N$ and 
		taking sum over $|\alpha|\leq N-(k+1), |\beta| = k+1$,
		we have
		\ben \label{essential-micro-macro-result-pure-no}&&\frac{\mathrm{d}}{\mathrm{d}t}\|f\|^{2}_{H^{N-k-1}_{x}\dot{H}^{k+1}}+ \lambda_{s} \|f_{2}\|^{2}_{H^{N-k-1}_{x}\dot{H}^{k+1}_{s,\gamma/2}}  \\&\leq& 2\sum_{|\alpha|\leq N-k-1,|\beta|=k+1}(\pa^{\alpha}_{\beta}g, \pa^{\alpha}_{\beta}f)
		+\lambda_{s}^{-1} C_{s,\gamma}^2 C_{N} \|f_{2}\|^{2}_{H^{N-k}_{x}
			\dot{H}^{k}_{s,l+k(\gamma+2s)+\gamma/2}}
		\nonumber \\&&+\lambda_{s}^{-1} C_{s,\gamma}^2 C_{N} \|f_{2}\|^{2}_{H^{N-k}_{x}
			H^{k}_{s,\gamma/2}}
		+ \lambda_{s}^{-1} C_{s,\gamma}^2 C_{N}
		|\mathrm{M}|_{H^{N-k}_{x}}^{2}. \nonumber\een

		By the induction assumption, \eqref{essential-micro-macro-result-4} is true when $i=k$, that is,	
		\ben \label{essential-micro-macro-result-k} \frac{\mathrm{d}}{\mathrm{d}t}\mathcal{X}^{k}(f)+  2^{-1-k/N} \mathcal{Y}^{k}(f) \leq 	\mathcal{Z}^{k}(g, f). \een		
		Note that $\mathcal{Y}^{k}(f)$ contains all the norms on the right hand side of
		\eqref{essential-micro-macro-result-pure-no}.

		We choose a constant $W_{k}$ large enough such that
		\beno
		W_{k} 2^{-1-\f{k}{N}-\f{1}{2N}} (2^{\f{1}{2N}} -1) \lambda_{s} \geq \lambda_{s}^{-1} C_{s,\gamma}^2 C_{N}.
		\eeno
		Note that this also gives
		\beno
		W_{k} 2^{-1-\f{k}{N}-\f{1}{2N}} (2^{\f{1}{2N}} -1)
		M^{k} c_{0} \geq \lambda_{s}^{-1} C_{s,\gamma}^2 C_{N}.
		\eeno
		Take 
		\ben \label{def-Wk}
		W_{k} =  \lambda_{s}^{-2} C_{s,\gamma}^2 C_{N}.
		\een
		Then 
		\eqref{essential-micro-macro-result-k}$\times W_{k}+$\eqref{essential-micro-macro-result-pure-no} yields
		\ben \label{k-plus1-with-k}&&\frac{\mathrm{d}}{\mathrm{d}t} (W_{k}\mathcal{X}^{k}(f) +
		\|f\|^{2}_{H^{N-k-1}_{x}\dot{H}^{k+1}})+ 2^{-1-k/N-1/2N} W_{k} \mathcal{Y}^{k}(f) + \lambda_{s} \|f_{2}\|^{2}_{H^{N-k-1}_{x}\dot{H}^{k+1}_{s,\gamma/2}}  
		\\ \nonumber &\leq& 2\sum_{|\alpha|\leq N-k-1,|\beta|=k+1}(\pa^{\alpha}_{\beta}g, \pa^{\alpha}_{\beta}f) + W_{k}\mathcal{Z}^{k}(g, f). \een

		Let $\alpha$ and $\beta$ be multi-indices such that $|\alpha|\leq N-(k+1)$ and $|\beta|= k+1 \geq 1$. Let $q=l+(k+1)(\gamma+2s)$. Applying $W_{q}\pa^{\alpha}_{\beta}$ to both sides of \eqref{lBE-general} gives
		\ben \label{weight-q-LBE-3} \partial_{t}W_{q}\pa^{\alpha}_{\beta}f + v\cdot \nabla_{x} W_{q}\pa^{\alpha}_{\beta}f +\sum_{\beta_{1}\leq \beta,|\beta_{1}|=1} W_{q}\pa^{\alpha+\beta_{1}}_{\beta-\beta_{1}}f + W_{q}\pa^{\alpha}_{\beta}\mathcal{L}^{s,\gamma}f = W_{q}\pa^{\alpha}_{\beta}g. \een
		Taking inner product with $W_{q}\pa^{\alpha}_{\beta} f$ over $(x,v)$ yields
		\ben \label{weight-mixed-derivative} \frac{1}{2}\frac{\mathrm{d}}{\mathrm{d}t}\|\pa^{\alpha}_{\beta}f \|^{2}_{L^{2}_{q}}  +  \sum_{\beta_{1}\leq \beta,|\beta_{1}|=1}(W_{q}\pa^{\alpha+\beta_{1}}_{\beta-\beta_{1}}f,W_{q}\pa^{\alpha}_{\beta}f) + (W_{q}\pa^{\alpha}_{\beta}\mathcal{L}^{s,\gamma}f,W_{q}\pa^{\alpha}_{\beta}f) = (W_{q}\pa^{\alpha}_{\beta}g,W_{q}\pa^{\alpha}_{\beta}f).  \een

		{\it Estimate of $(W_{q}\pa^{\alpha+\beta_{1}}_{\beta-\beta_{1}}f,W_{q}\pa^{\alpha}_{\beta}f)$.} 		
		Similar to \eqref{weight-term-1-no}, we have
		\ben \label{weight-term-1}
		|(W_{q}\pa^{\alpha+\beta_{1}}_{\beta-\beta_{1}}f, W_{q}\pa^{\alpha}_{\beta}f)| \leq 
		\delta \|\pa^{\alpha}_{\beta}f_2\|_{L^2_{x}L^{2}_{s,\gamma/2}}^2 + C_{\delta}
		\|\pa^{\alpha+\beta_{1}}_{\beta-\beta_{1}}f_2\|_{L^2_{x}L^{2}_{s,l+k(\gamma+2s)+\gamma/2}}
		+ C_{\delta} C_{l, |\beta|} |\mathrm{M}|^{2}_{H^{N-k}_{x}}. \een

		{\it Estimate of $(W_{q}\pa^{\alpha}_{\beta}\mathcal{L}^{s,\gamma}f,W_{q}\pa^{\alpha}_{\beta}f) $.}
	Observe that 
		\beno
		W_{q}\pa^{\alpha}_{\beta}\mathcal{L}^{s,\gamma}f = \mathcal{L}^{s, \gamma}W_{q}\pa^{\alpha}_{\beta}f + [W_{q}, \mathcal{L}^{s,\gamma}]\pa^{\alpha}_{\beta}f+
		W_{q}[\partial_{\beta}, \mathcal{L}^{s,\gamma}]\pa^{\alpha}f.
		\eeno
		By Lemma \ref{coercivity-L-alpha-beta},  \eqref{special-commutator-L} and Lemma \ref{commutator-L-with-beta}, taking $\delta = \lambda_{s}/8$ in Lemma \ref{commutator-L-with-beta},
		we have
		\beno
		&&(\mathcal{L}^{s, \gamma} W_{q}\pa^{\alpha}_{\beta}f+[W_{q}, \mathcal{L}^{s,\gamma}]\pa^{\alpha}_{\beta}f+
		W_{q}[\partial_{\beta}, \mathcal{L}^{s,\gamma}]\pa^{\alpha}f, W_{q}\pa^{\alpha}_{\beta}f)
		\\&\geq&
		\f{3}{4} \lambda_{s} \| W_{q}\pa^{\alpha}_{\beta} f_2\|^2_{L^{2}_{x}L^2_{s,\gamma/2}}
		+ \lambda_{s}^{-1} C_{q,N} C_{s,\gamma}^2 |\partial^{\alpha}\mathrm{M}|^{2}_{L^{2}_{x}}
		\\&&	+ \lambda_{s}^{-1} C_{q,N} C_{s,\gamma}^2 \sum_{\beta_{2}<\beta} (\|\pa^{\alpha}_{\beta_{2}}f_2\|_{L^{2}_{x}L^{2}_{q+\gamma/2}}^2+ 
		\|\pa^{\alpha}_{\beta_{2}}f_2\|_{H^{s}_{\gamma/2}}^{2}) + C_{q}(\gamma+5)^{-1}
		\|\pa^{\alpha}_{\beta}f\|^{2}_{L^{2}_{x}L^{2}_{q+\gamma/2}}.
		\eeno
		By using
		the decomposition $f=f_1 +f_2$ and
		\eqref{reduce-to-gamma/2-2} for  the last term, since $q \leq l$,
		we get
		\ben \label{weight-term-2}
		&&(W_{q}\pa^{\alpha}_{\beta}\mathcal{L}^{s,\gamma}f, W_{q}\pa^{\alpha}_{\beta}f)
		\\ \nonumber &\geq&
		\f{5}{8} \lambda_{s} \| W_{q}\pa^{\alpha}_{\beta} f_2\|^2_{L^{2}_{x}L^2_{s,\gamma/2}}
		+ \lambda_{s}^{-1} C_{N,l} C_{s,\gamma}^2 |\partial^{\alpha}\mathrm{M}|^{2}_{L^{2}_{x}}
		\\ \nonumber &&	+ \lambda_{s}^{-1} C_{N,l} C_{s,\gamma}^2 \sum_{\beta_{2}<\beta} (\|\pa^{\alpha}_{\beta_{2}}f_2\|_{L^{2}_{x}L^{2}_{q+\gamma/2}}^2+ 
		\|\pa^{\alpha}_{\beta_{2}}f_2\|_{H^{s}_{\gamma/2}}^{2}) + C_{l,s,\gamma}
		\|\pa^{\alpha}_{\beta}f_2\|^{2}_{L^{2}_{x}L^{2}_{s,\gamma/2}}.
		\een
		
		Plugging \eqref{weight-term-1} and \eqref{weight-term-2} into \eqref{weight-mixed-derivative}
		and 
		taking sum over $|\alpha|\leq N-(k+1), |\beta| = k+1$ give
		\ben \label{essential-micro-macro-result-pure}
		&&\frac{\mathrm{d}}{\mathrm{d}t}\|f\|^{2}_{H^{N-k-1}_{x}\dot{H}^{k+1}_{q}}+ \lambda_{s} \|f_{2}\|^{2}_{H^{N-k-1}_{x}\dot{H}^{k+1}_{s,q+\gamma/2}}  \\&\leq& 2\sum_{|\alpha|\leq N-k-1,|\beta|=k+1}(W_{q}\pa^{\alpha}_{\beta}g, W_{q}\pa^{\alpha}_{\beta}f)
		+ \lambda_{s}^{-1}
		C_{N}\|f_{2}\|^{2}_{H^{N-k}_{x}
			\dot{H}^{k}_{s,q+k(\gamma+2s)+\gamma/2}}
		\nonumber \\&&+ \lambda_{s}^{-1} C_{N,l} C_{s,\gamma}^2 \|f_{2}\|^{2}_{H^{N-k-1}_{x}H^{k}_{s,l+k(\gamma+2s)+\gamma/2}}
		+ \lambda_{s}^{-1} C_{N,l} C_{s,\gamma}^2 |\partial^{\alpha}\mathrm{M}|^{2}_{L^{2}_{x}} + C_{l,s,\gamma}
		\|f_2\|^{2}_{H^{N-k-1}_{x}\dot{H}^{k+1}_{s,\gamma/2}}
		. \nonumber\een
		Note that $2^{-1-k/N-1/2N} \mathcal{Y}^{k}(f) + \lambda_{s} \|f_{2}\|^{2}_{H^{N-k-1}_{x}\dot{H}^{k+1}_{s,\gamma/2}}$ in \eqref{k-plus1-with-k} contains all the norms  on the right hand side of
		\eqref{essential-micro-macro-result-pure}.
		We choose a constant $U_{k}$ large enough such that
		\beno
		U_{k} 2^{-1-\f{k}{N}-\f{1}{N}} (2^{\f{1}{2N}} -1) \lambda_{s} \geq \lambda_{s}^{-1} C_{s,\gamma}^2 C_{N,l},
		\quad U_{k} \lambda_{s} /2 \geq  C_{l,s,\gamma}.
		\eeno
		By recalling \eqref{def-C-l-s-gamma},
		we choose 
		\ben \label{def-Uk}
		U_{k} =   \max\{\lambda_{s}^{-2}C_{s,\gamma}^2 C_{N,l}, 8^{l/s} (C_{l}(\gamma+5)^{-1}\lambda_{s}^{-1})^{1+l/s}
		\}
		.
		\een		Then 
		\eqref{k-plus1-with-k}$\times U_{k}+$\eqref{essential-micro-macro-result-pure} yields	
		\beno
		&&\frac{\mathrm{d}}{\mathrm{d}t} (U_{k}W_{k}\mathcal{X}^{k}(f) + U_{k}
		\|f\|^{2}_{H^{N-k-1}_{x}\dot{H}^{k+1}} + \|f\|^{2}_{H^{N-k-1}_{x}\dot{H}^{k+1}_{q}})
		\\&&+  2^{-1-k/N-1/N} U_{k} W_{k} \mathcal{Y}^{k}(f) + 2^{-1} U_{k} \lambda_{s} \|f_{2}\|^{2}_{H^{N-k-1}_{x}\dot{H}^{k+1}_{s,\gamma/2}}  + \lambda_{s} \|f_{2}\|^{2}_{H^{N-k-1}_{x}\dot{H}^{k+1}_{s,q+\gamma/2}} 
		\\ &\leq& 2\sum_{|\alpha|\leq N-k-1,|\beta|=k+1}(W_{q}\pa^{\alpha}_{\beta}g, W_{q}\pa^{\alpha}_{\beta}f)+
		2U_{k}\sum_{|\alpha|\leq N-k-1,|\beta|=k+1}(\pa^{\alpha}_{\beta}g, \pa^{\alpha}_{\beta}f) + U_{k} W_{k}\mathcal{Z}^{k}(g, f).\eeno
		Hence  \eqref{essential-micro-macro-result-4} holds for $i=k+1$. 
		Precisely, we can set $M^{k+1} = U_{k}W_{k}M^{k}, L^{k+1}_{j}=U_{k}W_{k} L^{k}_{j}, K^{k+1}_{j}=U_{k}W_{k} K^{k}_{j}$ for $0 \leq j\leq k$ and
		$L^{k+1}_{k+1}=U_{k}, K^{k+1}_{k+1}=1$. Note that  $L^{N}_{0} = L^{0}_{0}\prod_{j=0}^{N-1} U_{j}W_{j} = M_1M_2\prod_{j=0}^{N-1} U_{j}W_{j}$.
		By taking  $i=N$ in \eqref{essential-micro-macro-result-4} and
		$M = M^N, L_{j} = L^{N}_{j}, K_{j} = K^{N}_{j}$ for $0 \leq j \leq N$, we get \eqref{essential-micro-macro-result-2}.
		It is straightforward to check the constants satisfy \eqref{c-relation-1}-\eqref{def-constant-U}.
		And this completes the proof of the proposition.
	\end{proof}

	\subsection{Global well-posedenss} \label{well-posedenss}

 We first derive the following a priori estimate for solutions to the Cauchy problem \eqref{Cauchy-linearizedBE}.
	\begin{thm}\label{a-priori-estimate-LBE} Let $N \geq 4, l\geq -N(\gamma+2s)$.
		If $f^{s,\gamma}$ is a solution of the Cauchy problem \eqref{Cauchy-linearizedBE}  satisfying
		$\sup_{0 \leq t \leq T} \| f^{s,\gamma}(t)\|_{H^{N}_{x,v}} \le \eta_{s,\gamma,N,l} := C_{N,l}^{-1}Z_{s,\gamma,N,l}^{-1} \lambda_{s} C_{s,\gamma}^{-1}$, then for any $t\in[0,T]$, the solution $f^{s,\gamma}$  satisfies	
		\ben \label{uniform-estimate-propagation-N-geq-5-l-big} \mathcal{E}_{N,l}^{s,\gamma}(f^{s,\gamma}(t)) + \frac{1}{8} \lambda_{s} \int_{0}^{t}  (|\mathrm{M}(\tau)|^{2}_{H^{N}_{x}} + \mathcal{D}^{s,\gamma}_{N,l}(f^{s,\gamma}_2)(\tau)) \mathrm{d}\tau \leq Z_{s,\gamma,N,l} \mathcal{E}_{N,l}^{s,\gamma}(f_{0}).
		\een	
	\end{thm}
	
	\begin{proof}[Proof of Theorem \ref{a-priori-estimate-LBE}]
		We apply Proposition \ref{essential-estimate-of-micro-macro} by taking $g = \Gamma^{s,\gamma}(f^{s,\gamma}
		,f^{s,\gamma})$ to have 
		\ben \label{g=gamma-f-f}
		\frac{\mathrm{d}}{\mathrm{d}t}\Xi^{N,l}(f^{s,\gamma}) +  \frac{1}{4} \tilde{\mathcal{D}}^{s,\gamma}_{N,l}(f^{s,\gamma}) &\leq& \sum_{j=0}^{N}2K_{j} \mathcal{A}^{N,j,l}_{s,\gamma}(f^{s,\gamma}
		,f^{s,\gamma},f^{s,\gamma}) \nonumber
		\\&&+
		\sum_{j=0}^{N}2L_{j} \mathcal{B}^{N,j,l}_{s,\gamma}(f^{s,\gamma}
		,f^{s,\gamma},f^{s,\gamma})+M C_2 \mathcal{C}^{N}_{s,\gamma}(
		f^{s,\gamma},f^{s,\gamma}), \een
		where
		\ben \label{A-N-j-l}
		\mathcal{A}^{N,j,l}_{s,\gamma}(g,h,f) := \sum_{|\alpha|\leq N-j,|\beta|=j}(W_{l+j(\gamma+2s)}\pa^{\alpha}_{\beta}\Gamma^{s,\gamma}(g,h), W_{l+j(\gamma+2s)}\pa^{\alpha}_{\beta}f),
		\\ \label{B-N-j-0}
		\mathcal{B}^{N,j}_{s,\gamma}(g,h,f) := \sum_{|\alpha|\leq N-j,|\beta|=j}(\pa^{\alpha}_{\beta}\Gamma^{s,\gamma}(g,h), \pa^{\alpha}_{\beta}f),
		\\ \label{C-N-epsilon-0}
		\mathcal{C}^{N}_{s,\gamma}(g,h) := \sum_{|\alpha| \leq N-1}\sum_{j=1}^{13}
		\int_{\mathbb{T}^{3}}|\langle  \pa^{\alpha}\Gamma^{s,\gamma}(g,h), \mu^{\f12} P_j\rangle|^{2} \mathrm{d}x.
		\een

		We will estimate $\mathcal{A}^{N,j,l}_{s,\gamma}, \mathcal{B}^{N,j}_{s,\gamma}, \mathcal{C}^{N}_{s,\gamma}$ in the following. 	Set
		\beno \|f\|^{2}_{H^{m}_{x,v}} := \sum_{|\alpha|+|\beta|\leq m}\|\pa^{\alpha}_{\beta}f\|^{2}_{L^{2}},  \,\, \|f\|^{2}_{D^{m}_{s,\gamma}} := \sum_{|\alpha|+|\beta|\leq m} \|\pa_{\beta}^{\alpha}f\|_{L^{2}_x
			L^{2}_{s,\gamma/2}}^2. \eeno
		Recall from \eqref{energy-functional} that the energy functional $\mathcal{E}_{N,l}^{s,\gamma}(f) =\sum_{j=0}^{N}\|f\|^{2}_{H^{N-j}_{x}
			\dot{H}^{j}_{l+j(\gamma+2s)}}$. Define the dissipation functional $\mathcal{D}_{N,l}^{s,\gamma}(f) =\sum_{j=0}^{N}\|f\|^{2}_{H^{N-j}_{x}
			\dot{H}^{j}_{s,l+j(\gamma+2s)+\gamma/2}}$.
		We claim
		\ben  \label{N-equals-4-x-v-weight-15}
		|\mathcal{A}^{N,j,l}_{s,\gamma}(g,h,f)| \lesssim_{N,l}
		C_{s,\gamma} (
		\|g\|_{H^{N}_{x,v}}
		\left(\mathcal{D}_{N,l}^{s,\gamma}(h)\right)^{\f12} +
		\|g\|_{D^{N}_{s,\gamma}} \|h\|_{H^{N}_{x,v}}) \|f\|_{H^{N-j}_{x}
			\dot{H}^{j}_{s,l+j(\gamma+2s)+\gamma/2}},
		\\ \label{N-geq-5-x-v-no-weight}
		|\mathcal{B}^{N,j}_{s,\gamma}(g,h,f)| \lesssim_{N}
		C_{s,\gamma} (\|g\|_{H^{N}_{x,v}} \|h\|_{D^{N}_{s,\gamma}} +
		\|g\|_{D^{N}_{s,\gamma}} \|h\|_{H^{N}_{x,v}}) \|f\|_{H^{N-j}_{x}
			\dot{H}^{j}_{s,\gamma/2}},
		\\ \label{only-x-with-mu-type-N-equals-4}
		\mathcal{C}^{N}_{s,\gamma}(g,f) \lesssim_{N} C_{s,\gamma}^2 \|g\|_{H^{N}_{x}L^2}^2 \|h\|_{H^{N}_{x}L^{2}_{s,\gamma/2}}^2.
		\een
		With the above  nonlinear estimates, by recalling \eqref{def-dissipation-combination}, \eqref{c-relation-1} and \eqref{c-relation-2}, if 
		\ben \label{assumption-1-on-Hnxv}
		C_{s,\gamma} 
		\sup_{0 \leq t \leq T} \| (f^{s,\gamma}(t))\|_{H^{N}_{x,v}} \leq 1,
		\een
		then 
		\ben \label{g=gamma-f-f-non-linear}
		&& \frac{\mathrm{d}}{\mathrm{d}t}\Xi^{N,l}(f^{s,\gamma}) +  \frac{1}{4} \lambda_{s} (|\mathrm{M}|^{2}_{H^{N}_{x}} + \mathcal{D}^{s,\gamma}_{N,l}(f^{s,\gamma}_2))
		\\ \nonumber  &\leq& C_{N,l}Z_{s,\gamma,N,l} (C_{s,\gamma} 
		\|f^{s,\gamma}\|_{H^{N}_{x,v}} + C_{s,\gamma}^2 
		\|f^{s,\gamma}\|_{H^{N}_{x,v}}^2) \mathcal{D}^{s,\gamma}_{N,l}(f^{s,\gamma})
		\\ \nonumber  &\leq& C_{N,l}Z_{s,\gamma,N,l} \lambda_{s}^{-1} C_{s,\gamma} 
		\|f^{s,\gamma}\|_{H^{N}_{x,v}}  \frac{1}{8} \lambda_{s} (|\mathrm{M}|^{2}_{H^{N}_{x}} + \mathcal{D}^{s,\gamma}_{N,l}(f^{s,\gamma}_2)), 
		\een
		where  we have used  $\mathcal{D}^{s,\gamma}_{N,l}(f) 
		\lesssim_{N,l} (|\mathrm{M}|^{2}_{H^{N}_{x}} + \mathcal{D}^{s,\gamma}_{N,l}(f_2))$
		in the last inequality.
		Now under the assumption
		\ben \label{assumption-2-on-Hnxv}
		C_{N,l}Z_{s,\gamma,N,l} \lambda_{s}^{-1} C_{s,\gamma} 
		\sup_{0 \leq t \leq T} \| f^{s,\gamma}(t)\|_{H^{N}_{x,v}} \leq 1,
		\een
		we have
		\ben \label{final-energy-inequality}
		\frac{\mathrm{d}}{\mathrm{d}t}\Xi^{N,l}(f^{s,\gamma}) +  \frac{1}{8} \lambda_{s} (|\mathrm{M}|^{2}_{H^{N}_{x}} + \mathcal{D}^{s,\gamma}_{N,l}(f^{s,\gamma}_2)) \leq 0, 
		\een
		which gives
		\ben \label{final-energy-inequality-integral}
		\Xi^{N,l}(f^{s,\gamma}(t)) +  \frac{1}{8} \lambda_{s} \int_{0}^{t}  (|\mathrm{M}(\tau)|^{2}_{H^{N}_{x}} + \mathcal{D}^{s,\gamma}_{N,l}(f^{s,\gamma}_2)(\tau)) \mathrm{d}\tau
		\leq \Xi^{N,l}(f_0).
		\een
		Recalling \eqref{c-relation-1}, we have
		\ben \label{energy-equivalence}
		\mathcal{E}_{N,l}^{s,\gamma}(f) \leq \Xi_{N,l}^{s,\gamma}(f) \leq Z_{s,\gamma,N,l} \mathcal{E}_{N,l}^{s,\gamma}(f).
		\een
		Therefore, we obtain \eqref{uniform-estimate-propagation-N-geq-5-l-big}.
		Note that \eqref{assumption-2-on-Hnxv} implies \eqref{assumption-1-on-Hnxv}. 
		
		Now it remains to prove \eqref{N-equals-4-x-v-weight-15}, \eqref{N-geq-5-x-v-no-weight} and  \eqref{only-x-with-mu-type-N-equals-4}. 
		We first consider 	$\mathcal{B}^{N,j}_{s,\gamma}(g,h,f)$ defined in \eqref{B-N-j-0}.
		By the binomial expansion \eqref{alpha-beta-on-Gamma}, we have
		\beno
		\pa^{\alpha}_{\beta}\Gamma^{s,\gamma}(g,h) = 
		\sum C(\alpha_{1},\alpha_{2},\beta_{0},\beta_{1},\beta_{2}) \Gamma^{s,\gamma}(\partial^{\alpha_{1}}_{\beta_{1}}g, \partial^{\alpha_{2}}_{\beta_{2}}h;\beta_{0}),
		\eeno
		where the sum is over $\alpha_{1}+\alpha_{2}=\alpha, \beta_{0}+\beta_{1}+\beta_{2} = \beta$.

		By taking $\delta=\f12$ in
		Theorem \ref{Gamma-full-up-bound}, for
		$(b_{1}, b_{2})=(2, s)$ or $(s, 2)$, we have	
		\beno
		|\langle \Gamma^{s,\gamma}(\partial^{\alpha_{1}}_{\beta_{1}}g, \partial^{\alpha_{2}}_{\beta_{2}}h;\beta_{0}), \pa^{\alpha}_{\beta}f \rangle| \lesssim C_{ s,\gamma}|\partial^{\alpha_{1}}_{\beta_{1}}g|_{H^{b_{1}}_{\gamma/2}}|\partial^{\alpha_{2}}_{\beta_{2}}h|_{H^{b_{2}}
			_{\gamma/2}}
		|\pa^{\alpha}_{\beta}f|_{H^{s}_{\gamma/2}}+ s^{-1} |\partial^{\alpha_{1}}_{\beta_{1}}g|_{L^{2}}|\partial^{\alpha_{2}}_{\beta_{2}}h|_{s,\gamma/2}|\pa^{\alpha}_{\beta}f|_{s,\gamma/2}.
		\eeno	
		Using the fact that $\int |g h f| \mathrm{d}x \lesssim |g|_{H^{a_1}} |h|_{H^{a_2}} |f|_{L^2}$	for $a_1 + a_2 =2, a_1,a_2 \geq 0$, we have
		\beno
		|\left(\Gamma^{s,\gamma}(\partial^{\alpha_{1}}_{\beta_{1}}g, \partial^{\alpha_{2}}_{\beta_{2}}h;\beta_{0}), \pa^{\alpha}_{\beta}f \right)| \lesssim C_{ s,\gamma}\|\partial^{\alpha_{1}}_{\beta_{1}}g
		\|_{H^{a_1}_{x}H^{b_{1}}_{\gamma/2}}\|\partial^{\alpha_{2}}_{\beta_{2}}h\|_{H^{a_2}_{x}H^{b_{2}}
			_{\gamma/2}}
		\|\pa^{\alpha}_{\beta}f\|_{L^{2}_{x}L^{2}_{s,\gamma/2}}
		\\+ C_{ s,\gamma} \|\partial^{\alpha_{1}}_{\beta_{1}}g\|
		_{H^{a_1}_{x}L^{2}}\|\partial^{\alpha_{2}}_{\beta_{2}}h\|
		_{H^{a_2}_{x}L^{2}_{s,\gamma/2}}\|\pa^{\alpha}_{\beta}f\|_{L^{2}_{x}L^{2}_{s,\gamma/2}}.
		\eeno	
		By suitably choosing $a_1, a_2$, 
		the second term in the above inequality  is directly bounded by 
		$C_{ s,\gamma} \|g\|_{H^{N}_{x,v}} \|h\|_{D^{N}_{s,\gamma}} \|\pa^{\alpha}_{\beta}f\|_{L^{2}_{x}L^{2}_{s,\gamma/2}}$. Next we will give
		the choices of  $a_1, a_2, b_1, b_2$ for the first  term. 
		
		In the following, we choose $a_{1}, a_{2} \in \{0,1,2\}$ with $a_{1}+a_{2}=2$ and $b_{1}, b_{2} \in \{s,2\}$ with $b_{1}+b_{2}=2+s$.
		For $N \geq 4$ and multi-indices $\alpha, \beta$ with $|\alpha|+|\beta| \leq N$, we consider all the combinations of $\alpha_{1},\alpha_{2},\beta_{1},\beta_{2}$ such that $\alpha_{1}+\alpha_{2}=\alpha, \beta_{1}+\beta_{2} \leq \beta$ in Table \ref{parameter-2} for the choices of $a_{1}, a_{2}, b_{1}, b_{2}$.
		
		\begin{table}[!htbp]
			\centering
			\caption{Parameter choices}\label{parameter-2}
			\begin{tabular}{ccccc}
				\hline
				$|\alpha_{1}|+|\beta_{1}|$  &$|\alpha_{2}|+|\beta_{2}|$ & $(a_{1}, a_{2}, b_{1}, b_{2})$ & $|\alpha_{1}|+a_{1}+|\beta_{1}|+b_{1}$ &
				$|\alpha_{2}|+a_{2}+|\beta_{2}|+b_{2}$   \\
				\hline
				0& $ \leq |\alpha|+|\beta|$& (2,0,2,s) & 4 &  $\leq |\alpha|+|\beta|+s$ \\
				1& $\leq |\alpha|+|\beta| -1$& (1,1,2,s) & 4 & $\leq |\alpha|+|\beta|+s$\\
				2& $\leq |\alpha|+|\beta| -2$& (0,2,2,s) & 4 & $\leq |\alpha|+|\beta|+s$\\
				3& $\leq |\alpha|+|\beta| -3$& (1,1,s,2) & 4 + s & $\leq |\alpha|+|\beta|$\\
				$|\alpha_{1}|+|\beta_{1}|\geq4$& $\leq |\alpha|+|\beta| -4$& (0,2,s,2) & N + s & $\leq |\alpha|+|\beta|$\\
				\hline
			\end{tabular}
		\end{table}
		With this, the part containing $s$ is bounded by dissipation functional $D^{N}_{s,\gamma}$, and the other part
		is bounded by energy functional $H^{N}_{x,v}$. As a result,
		\beno
		|\left(\Gamma^{s,\gamma}(\partial^{\alpha_{1}}_{\beta_{1}}g, \partial^{\alpha_{2}}_{\beta_{2}}h;\beta_{0}), \pa^{\alpha}_{\beta}f \right)| \lesssim C_{ s,\gamma} \|g\|_{H^{N}_{x,v}} \|h\|_{D^{N}_{s,\gamma}} \|\pa^{\alpha}_{\beta}f\|_{L^{2}_{x}L^{2}_{s,\gamma/2}}
		+C_{ s,\gamma} \|g\|_{D^{N}_{s,\gamma}} \|h\|_{H^{N}_{x,v}} \|\pa^{\alpha}_{\beta}f\|_{L^{2}_{x}L^{2}_{s,\gamma/2}}.
		\eeno			
		Taking sum yields	\eqref{N-geq-5-x-v-no-weight}.
		
		Similarly, we can use Corollary  \ref{Gamma-full-up-bound-with-weight} to derive	
		\eqref{N-equals-4-x-v-weight-15} and  use Prop.
		\ref{Gamma-g-h-mu}  to derive	
		\eqref{only-x-with-mu-type-N-equals-4}. This completes the proof of the Theorem.
	\end{proof}

	\begin{proof}[Proof of Theorem \ref{asymptotic-result}(Global well-posedness)]  Local well-posedness of the Cauchy problem \eqref{Cauchy-linearizedBE} and non-negativity of $\mu+\mu^{\f12}f$ can be proved by standard iteration. From this together with Theorem \ref{a-priori-estimate-LBE}, 	by taking $\delta_{s,\gamma,N,l} = \f12
		\eta_{s,\gamma,N,l}^2$,
		the standard continuity argument yields the global well-posedness result \eqref{uniform-controlled-by-initial} for the Boltzmann equation.
		Recalling the constants $\lambda_{s}$ from
		\eqref{lower-bound-on-lambda-explicit}, 
		$C_{s, \gamma}$ from \eqref{def-C-s-gamma}, 
		$Z_{s,\gamma,N,l}$ from \eqref{c-relation-1} and the constant $\eta_{s,\gamma,N,l}$ from
		Theorem \ref{a-priori-estimate-LBE},
		it is straightforward to check  that for any fixed $N,l$,
		there is a function $(x_1, x_2) \in (0,1) \times (0, 3] \to 
		\delta_{N,l}(x_1, x_2) \in (0, \infty)$ satisfying
		\eqref{dependence-s-gamma} and
		\eqref{property-of-delta-N-l}. Moreover, since all the estimates are uniform for $s \rightarrow i^{-1}$,  the global well-posedness result \eqref{uniform-controlled-by-initial-L}
		for the Landau equation follows by a similar argument.
	\end{proof}
	
	\subsection{Asymptotic formula for the limit} 
	\label{asymptotic}
	We prove \eqref{error-function-uniform-estimate} in this subsection.    Let $f^{s,\gamma}$ and $f^{\gamma}$ be the solutions to  \eqref{Cauchy-linearizedBE} and \eqref{Cauchy-linearizedLE} respectively with the initial data $f_0$. Set $F^{s,\gamma}_{R} :=  (1-s)^{-1} (f^{s,\gamma}-f^{\gamma})$, then it solves
	\ben \label{error-equation} && \partial_{t}F^{s,\gamma}_{R} + v \cdot \nabla_{x} F^{s,\gamma}_{R} + \mathcal{L}^{\gamma}_{L}F^{s,\gamma}_{R} 
	\\ \nonumber &=& (1-s)^{-1}[(\mathcal{L}^{\gamma}_{L}-\mathcal{L}^{s,\gamma}_{B})f^{s,\gamma}+(\Gamma^{s,\gamma}_{B}-\Gamma^{\gamma}_{L})(f^{s,\gamma},f^{\gamma})]
	+\Gamma^{s,\gamma}_{B}(f^{s,\gamma},F^{s,\gamma}_{R})
	+\Gamma^{\gamma}_{L}(F^{s,\gamma}_{R},f^{\gamma}). \een

	We will apply Proposition  \ref{essential-estimate-of-micro-macro} to the above equation for $F^{s,\gamma}_{R}$. 
	For brevity, we set
	\ben \label{three-terms-G1}
	G_{1} = (1-s)^{-1}[(\mathcal{L}^{\gamma}_{L}-\mathcal{L}^{s,\gamma}_{B})f^{s,\gamma}+(\Gamma^{s,\gamma}_{B}-\Gamma^{\gamma}_{L})(f^{s,\gamma},f^{\gamma})],
	\\ \label{three-terms-G1-G3} G_{2}=\Gamma^{s,\gamma}_{B}(f^{s,\gamma},F^{s,\gamma}_{R}), \quad G_{3}= \Gamma^{\gamma}_{L}(F^{s,\gamma}_{R},f^{\gamma}).
	\een
	By applying Proposition  \ref{essential-estimate-of-micro-macro} with $s=1, g = G_{1}+G_{2}+G_{3}$, since
	$|\langle \pa^{\alpha}g, \mu^{\f12}P_j\rangle|^{2} \leq 3 \sum_{i=1}^{3}
	|\langle \pa^{\alpha}G_{i}, \mu^{\f12}P_j\rangle|^{2} $, we have
	\ben \label{essential-micro-macro-error-function-2} \frac{\mathrm{d}}{\mathrm{d}t}\Xi_{N,l}^{1,\gamma}(F^{s,\gamma}_{R}) +  \frac{1}{4} \tilde{\mathcal{D}}^{1,\gamma}_{N,l}(F^{s,\gamma}_{R}) &\leq& \sum_{i=1}^{3} \sum_{j=0}^{N}2K_{j}\sum_{|\alpha|\le N-j,|\beta|=j}(W_{l+j(\gamma+2)}\pa^{\alpha}_{\beta}G_i, W_{l+j(\gamma+2)}\pa^{\alpha}_{\beta}F^{s,\gamma}_{R})
	\\&&
	+ \sum_{i=1}^{3} \sum_{j=0}^{N}2L_{j}\sum_{|\alpha|\le N-j,|\beta|=j}(\pa^{\alpha}_{\beta}G_i, \pa^{\alpha}_{\beta}F^{s,\gamma}_{R})  + 3 M C_2 \sum_{i=1}^{3} \mathrm{NL}^{N}(G_i).
	\nonumber \een

	Let us first estimate the terms containing $G_1$.
	Recalling \eqref{three-terms-G1} and
	\eqref{essential-micro-macro-error-function-2}, we need to estimate the following quantities
	\ben \label{I-1-i}
	\mathcal{I}_{1,i} := \sum_{|\alpha|\le N-j,|\beta|=j}(W_{l+j(\gamma+2)}\pa^{\alpha}_{\beta}
	G_{1,i}, W_{l+j(\gamma+2)}\pa^{\alpha}_{\beta}F^{s,\gamma}_{R}),
	\\ \label{I-2-3-i}
	\mathcal{I}_{2,i} :=
	\sum_{|\alpha|\le N-j,|\beta|=j}(\pa^{\alpha}_{\beta}G_{1,i}, \pa^{\alpha}_{\beta}F^{s,\gamma}_{R}),
	\quad \mathcal{I}_{3,i} :=
	\mathrm{NL}^{N}(G_{1,i}).
	\een
Here,  for $i=1,2$,
	\ben \label{G-1-1-1-2}
	G_{1,1} = (1-s)^{-1}(\mathcal{L}^{\gamma}_{L}-\mathcal{L}^{s,\gamma}_{B})f^{s,\gamma}, \quad
	G_{1,2} =
	(1-s)^{-1}(\Gamma^{s,\gamma}_{B}-\Gamma^{\gamma}_{L})(f^{s,\gamma},f^{\gamma}).
	\een
	These terms contain operator difference. 
We first establish 
	$Q^{s,\gamma}_{B} \to Q_{L}^{\gamma}, \Gamma^{s,\gamma}_{B} \to \Gamma_{L}^{\gamma}, \mathcal{L}^{s,\gamma}_{B} \to \mathcal{L}_{L}^{\gamma}$ as $s \to 1^{-}$. The results can be given in  weighted $L^2$-norm by using the estimates obtained in \cite{desvillettes1992asymptotics} and \cite{he2014well}.

	\begin{prop} \label{limit-Botltzmann-to-Landau} Let $-5 < \gamma \leq 0 $. Fix $l \geq 0$. 
		Let $a_1, a_2, b_1, b_2 \in \R$ satisfying
		$a_1 + a_2 = \gamma+ 6$ and $b_1+b_2 = \gamma+2$.
		If $-9/2 < \gamma \leq 0$,  then
		\ben \label{l2-result-case1}
		\left| \langle Q^{s,\gamma}_{B}(g,h) - Q_{L}^{\gamma}(g,h), W_{l}\varphi\rangle\right|\lesssim_{l,a_1,a_2,b_1,b_2}
		(1-s)
		|g|_{H^{2}_{|\gamma+2|+2}} |h|_{H^{2}_{l+b_1}}
		|\psi|_{L^{2}_{b_2}}
		\\ \nonumber + 
		(1-s) (\gamma+\f92)^{-1} |g|_{H^{3}_{l+|a_1|+|a_2|+2}} |h|_{H^{3}_{l+a_1}} |\psi|_{L^{2}_{a_2}}.
		\een
		If $-5 < \gamma \leq 0$, then 
		\ben \label{l2-result-case2}
		\left| \langle Q^{s,\gamma}_{B}(g,h) - Q_{L}^{\gamma}(g,h), W_{l}\varphi\rangle\right| \lesssim_{l,a_1,a_2,b_1,b_2}  (1-s) (\gamma+5)^{-1} 	|g|_{H^{2}_{|\gamma+2|+2}} |h|_{H^{2}_{l+b_1}}
		|\psi|_{L^{2}_{b_2}}
		\\\nonumber +(1-s) |g|_{H^{3+s_1}_{l+|a_1|+|a_2|+2}} |h|_{H^{3+s_2}_{l+a_1}} |\psi|_{L^{2}_{a_2}},	\een
		where $s_1, s_2 \geq 0$ satisfying
		$s_1 + s_2 = 1$.
	\end{prop}
For completeness, the proof of Proposition \ref{limit-Botltzmann-to-Landau} will be given in the Appendix.
Here, we  only concern about the dependence on the two physical parameters $\gamma, s$ and do not pursue the precise  dependence on $l, a_1, a_2, b_1, b_2$. Roughly speaking, the dependence on $l, a_1, a_2, b_1, b_2$ is of the form $c^{l}, c^{|a_1|}, c^{|a_2|}, c^{|b_1|}, c^{|b_2|}$ for some generic constant $c>1$.

	We can also get similar results for the non-linear terms $\Gamma^{s,\gamma}_{B}(g,h)$ and $\Gamma^{\gamma}_{L}(g,h)$ by slightly revising the proof of  Proposition \ref{limit-Botltzmann-to-Landau}.
	In this situation, there is no weight on $g$.
	\begin{prop} \label{Gamma-limit-Botltzmann-to-Landau} Let $-5 < \gamma \leq 0 $. Fix $l \geq 0$. 
		Let $a_1, a_2, b_1, b_2 \in \R$ satisfying
		$a_1 + a_2 = \gamma+ 6$ and $b_1+b_2 = \gamma+2$.
		If $-9/2 < \gamma \leq 0$,  then
		\ben \label{Gamma-l2-result-case1}
		\left| \langle \Gamma^{s,\gamma}_{B}(g,h) - \Gamma^{\gamma}_{L}(g,h), W_{l}\varphi\rangle\right|
		\lesssim_{l,a_1,a_2,b_1,b_2}
		(1-s)
		|g|_{H^{2}} |h|_{H^{2}_{l+b_1}}
		|\psi|_{L^{2}_{b_2}}
		\\ \nonumber + 
		(1-s) (\gamma+\f92)^{-1} |g|_{H^{3}} |h|_{H^{3}_{l+a_1}} |\psi|_{L^{2}_{a_2}}.
		\een
		If $-5 < \gamma \leq 0$, then 
		\ben \label{Gamma-l2-result-case2}
		\left| \langle \Gamma^{s,\gamma}_{B}(g,h) - \Gamma^{\gamma}_{L}(g,h), W_{l}\varphi\rangle\right|\lesssim_{l,a_1,a_2,b_1,b_2}  (1-s) (\gamma+5)^{-1} 	|g|_{H^{2}} |h|_{H^{2}_{l+b_1}}
		|\psi|_{L^{2}_{b_2}}
		\\\nonumber +(1-s)  |g|_{H^{3+s_1}} |h|_{H^{3+s_2}_{l+a_1}} |\psi|_{L^{2}_{a_2}},
		\een
		where $s_1, s_2 \geq 0$ satisfying
		$s_1 + s_2 = 1$. 
	\end{prop}

	Recalling $\mathcal{L}^{s,\gamma}_{B}f = - \Gamma^{s,\gamma}_{B}(\mu^{\f12}, f) - \Gamma^{s,\gamma}_{B}(f, \mu^{\f12}), \mathcal{L}^{\gamma}_{L}f = - \Gamma^{\gamma}_{L}(\mu^{\f12}, f) - \Gamma^{\gamma}_{L}(f, \mu^{\f12})$,
	as an application of 
	Proposition \ref{Gamma-limit-Botltzmann-to-Landau}, 
	we can put the higher regularity on $\mu^{\f12}$ as stated in the following proposition.
	
	\begin{prop} \label{linear-limit-Botltzmann-to-Landau} Let $-5 < \gamma \leq 0 $. Fix $l \geq 0$. 
		Let $a_1, a_2 \in \R$ satisfying
		$a_1 + a_2 = \gamma+ 6$.
		Then 
		\ben \label{linear-l2-result-case2}
		\left| \langle \mathcal{L}^{s,\gamma}_{B}f - \mathcal{L}^{\gamma}_{L}f, W_{l}\varphi\rangle\right| \lesssim_{l,a_1,a_2}  (1-s) (\gamma+5)^{-1}  |f|_{H^{3}_{(l+a_1)^{+}}} |\psi|_{L^{2}_{a_2}}.
		\een
	\end{prop}

	By \eqref{Gamma-l2-result-case2} and \eqref{linear-l2-result-case2},
	\ben \label{Gamma-l2-result-used}
	(1-s)^{-1}	\left| \langle \Gamma^{s,\gamma}_{B}(g,h) - \Gamma^{\gamma}_{L}(g,h), W_{l}\varphi\rangle\right|\lesssim_{l}  C_{\gamma}  |g|_{H^{3}} |h|_{H^{4}_{l+5+\gamma/2}} |\psi|_{L^{2}_{1+\gamma/2}},
\\ \label{linear-l2-result-case2-used}
	(1-s)^{-1}	\left| \langle \mathcal{L}^{s,\gamma}_{B}f - \mathcal{L}^{\gamma}_{L}f, W_{l}\varphi\rangle\right| \lesssim_{l}   C_{\gamma}  |f|_{H^{3}_{l+5+\gamma/2}} |\psi|_{L^{2}_{1+\gamma/2}}.
	\een
	
Recall \eqref{I-1-i} and	\eqref{G-1-1-1-2} for $\mathcal{I}_{1,1}$ and  $\mathcal{I}_{1,2}$.
	We now estimate  these two terms in details. By
	\eqref{linear-l2-result-case2-used}, we have
	\beno
	|\mathcal{I}_{1,1}| \lesssim_{N,l} C_{\gamma}
	\sum_{|\alpha|\le N-j,|\beta|=j} \sum_{\beta_1 \leq \beta} \|\pa^{\alpha}_{\beta_1}f^{s,\gamma}\|_{L^{2}_{x}H^{3}_{l+j(\gamma+2)+5+\gamma/2}} \|W_{l+j(\gamma+2)}\pa^{\alpha}_{\beta}
	F^{s,\gamma}_{R}\|_{L^{2}_{x}L^{2}_{1+\gamma/2}}
	\\ \leq C_{N,l} C_{\gamma} (\mathcal{D}^{s,\gamma}_{N+3,l+5-3(\gamma+2s)+N(2-2s)}(f^{s,\gamma}))^{1/2} \|F^{s,\gamma}_{R}\|_{H^{N-j}_{x}
		\dot{H}^{j}_{1,l+j(\gamma+2)+\gamma/2}},
	\eeno
	where we have used for any $0 \leq k \leq N+3$,
	\beno
	\mathcal{D}^{s,\gamma}_{N+3,
		l+5-3(\gamma+2s)+N(2-2s)}(f^{s,\gamma}) \geq 
	\|f^{s,\gamma}\|_{H^{N+3-k}_{x}
		\dot{H}^{k}_{s, l+5-3(\gamma+2s)+N(2-2s) + k(\gamma+2s)
			+\gamma/2}}^2.
	\eeno
	In particular, taking $k = j+3$ gives
	\beno
	\mathcal{D}^{s,\gamma}_{N+3,
		l+5-3(\gamma+2s)+N(2-2s)}(f^{s,\gamma}) 
	\geq 
	\|f^{s,\gamma}\|_{H^{N-j}_{x}
		\dot{H}^{j+3}_{s, l+5
			+j(\gamma+2)+\gamma/2}}^2.
	\eeno
	
	By
	\eqref{Gamma-l2-result-used}, we have
	\beno
	|\mathcal{I}_{1,2}| \lesssim_{N,l} C_{\gamma}
	\sum_{|\alpha|\le N-j,|\beta|=j} \sum_{\alpha_1+\alpha_2 = \alpha, \beta_1 \leq \beta}
	\int  |\pa^{\alpha_1}_{\beta_1}f^{s,\gamma}|_{H^{3}}
	|\pa^{\alpha_2}_{\beta_2}f^{\gamma}|_{H^{4}_{l+j(\gamma+2)+5+\gamma/2}} |W_{l+j(\gamma+2)}\pa^{\alpha}_{\beta}
	F^{s,\gamma}_{R}|_{L^{2}_{1+\gamma/2}} \mathrm{d}x
	\\ \leq C_{N,l} C_{\gamma} \|f^{s,\gamma}
	\|_{H^{N+3}_{x,v}} 
	(\sum_{|\alpha| + |\beta| \leq N+3, |\beta| \leq j+3}
	\|\pa^{\alpha}_{\beta}f^{\gamma}\|_{L^2_{x}L^{2}_{1,l+j(\gamma+2)+5+\gamma/2}}^{2})^{1/2}
	\|F^{s,\gamma}_{R}\|_{H^{N-j}_{x}
		\dot{H}^{j}_{1,l+j(\gamma+2)+\gamma/2}}
	\\ \\ \leq C_{N,l} C_{\gamma} \|f^{s,\gamma}
	\|_{H^{N+3}_{x,v}} 
	(\mathcal{D}^{1,\gamma}_{N+3,
		l+5-3(\gamma+2)}(f^{\gamma}))^{1/2}
	\|F^{s,\gamma}_{R}\|_{H^{N-j}_{x}
		\dot{H}^{j}_{1,l+j(\gamma+2)+\gamma/2}}.
	\eeno
Recalling	\eqref{I-2-3-i} and
	\eqref{G-1-1-1-2} for $\mathcal{I}_{2,1}$ and  $\mathcal{I}_{2,2}$,
	it is obvious that these two terms are also bounded by the  upper bounds of $\mathcal{I}_{1,1}$ and $\mathcal{I}_{1,2}$. Similarly, by \eqref{linear-l2-result-case2-used}, we have
	\beno
	|\mathcal{I}_{3,1}| \leq C_{N,l} C_{\gamma}^2 \mathcal{D}^{s,\gamma}_{N+3,l+5-3(\gamma+2s)+N(2-2s)}(f^{s,\gamma}),
	\\
	|\mathcal{I}_{3,2}| \leq C_{N,l} C_{\gamma}^2 \|f^{s,\gamma}
	\|_{H^{N+3}_{x,v}}^2 
	\mathcal{D}^{1,\gamma}_{N+3,
		l+5-3(\gamma+2)}(f^{\gamma}).
	\eeno

	Let us next estimate the terms containing $G_2$. Recall \eqref{A-N-j-l}, \eqref{B-N-j-0} and \eqref{C-N-epsilon-0} that
	\beno
	\sum_{|\alpha|\le N-j,|\beta|=j}(W_{l+j(\gamma+2)}\pa^{\alpha}_{\beta}G_2, W_{l+j(\gamma+2)}\pa^{\alpha}_{\beta}F^{s,\gamma}_{R})
	= \mathcal{A}^{N,j,l+2j-2sj}_{s,\gamma}(f^{s,\gamma},F^{s,\gamma}_{R},F^{s,\gamma}_{R}),
	\\
	\sum_{|\alpha|\le N-j,|\beta|=j}(\pa^{\alpha}_{\beta}G_2, \pa^{\alpha}_{\beta}F^{s,\gamma}_{R}) = \mathcal{B}^{N,j}_{s,\gamma}(f^{s,\gamma},F^{s,\gamma}_{R},F^{s,\gamma}_{R}),
	\\
	\mathrm{NL}^{N}(G_2) = \mathcal{C}^{N}_{s,\gamma}(f^{s,\gamma},F^{s,\gamma}_{R}).
	\eeno

	As for $-5<\gamma \leq -2, \f{1}{4} \leq \f{1}{2}(1 - \f{\gamma+3}{2}) \leq  s \leq 1$, it holds that
	\beno 
	C_{s,\gamma} = s^{-1} (\gamma+2s+3)^{-1} \leq 8 (\gamma+5)^{-1}.
	\eeno
	Therefore we can replace $C_{s,\gamma}$ with $C_{\gamma} := (\gamma+5)^{-1}$ in the rest of this section.
	
	Recalling \eqref{N-equals-4-x-v-weight-15}, 
	we have
	\ben  \label{N-equals-4-x-v-weight-15-G2}
&&|\mathcal{A}^{N,j,l+2j-2sj}_{s,\gamma}(f^{s,\gamma},F^{s,\gamma}_{R},F^{s,\gamma}_{R})| 
\\ \nonumber &\lesssim_{N,l}&	C_{\gamma} (
	\|f^{s,\gamma}\|_{H^{N}_{x,v}}
	\left(\mathcal{D}_{N,l}^{1,\gamma}(F^{s,\gamma}_{R})\right)^{\f12} +
	\|f^{s,\gamma}\|_{D^{N}_{s,\gamma}} \|F^{s,\gamma}_{R}\|_{H^{N}_{x,v}}) \|F^{s,\gamma}_{R}\|_{H^{N-j}_{x}
		\dot{H}^{j}_{s,l+j(\gamma+2)+\gamma/2}}.
	\een
By	\eqref{N-geq-5-x-v-no-weight} and  \eqref{only-x-with-mu-type-N-equals-4}, we have
\ben
 \label{N-geq-5-x-v-no-weight-G2}
|\mathcal{B}^{N,j}_{s,\gamma}(f^{s,\gamma},
	F^{s,\gamma}_{R},F^{s,\gamma}_{R})| \lesssim_{N}
	C_{\gamma} (\|f^{s,\gamma}\|_{H^{N}_{x,v}} \|F^{s,\gamma}_{R}\|_{D^{N}_{s,\gamma}} +
	\|f^{s,\gamma}\|_{D^{N}_{s,\gamma}} \|F^{s,\gamma}_{R}\|_{H^{N}_{x,v}}) \|F^{s,\gamma}_{R}\|_{H^{N-j}_{x}
		\dot{H}^{j}_{s,\gamma/2}},
	\\ \label{only-x-with-mu-type-N-equals-4-G2}
\mathcal{C}^{N}_{s,\gamma}(f^{s,\gamma},
	F^{s,\gamma}_{R}) \lesssim_{N} C_{\gamma}^2 \|f^{s,\gamma}\|_{H^{N}_{x}L^2}^2 \|F^{s,\gamma}_{R}\|_{H^{N}_{x}L^{2}_{s,\gamma/2}}^2.
	\een
	Note that \eqref{N-geq-5-x-v-no-weight-G2} and \eqref{only-x-with-mu-type-N-equals-4-G2}  follow exactly  from \eqref{N-geq-5-x-v-no-weight} and  \eqref{only-x-with-mu-type-N-equals-4}. The estimate \eqref{N-equals-4-x-v-weight-15-G2} takes account of the additional weight $2j -2sj$  over \eqref{N-equals-4-x-v-weight-15-G2} and is controlled by the dissipation norm of the linearized Landau operator.
	
	Let us estimate the terms containing $G_3$. By taking $s=1$ in \eqref{A-N-j-l}, \eqref{B-N-j-0} and \eqref{C-N-epsilon-0}, and replacing $\Gamma^{1,\gamma}$ by $\Gamma^{\gamma}_{L}$, we can  define $\mathcal{A}^{N,j,l}_{1,\gamma}(g,h,f), \mathcal{B}^{N,j}_{1,\gamma}(g,h,f), \mathcal{C}^{N}_{1,\gamma}(g,h)$ similarly. Then
	\beno
	\sum_{|\alpha|\le N-j,|\beta|=j}(W_{l+j(\gamma+2)}\pa^{\alpha}_{\beta}G_3, W_{l+j(\gamma+2)}\pa^{\alpha}_{\beta}F^{s,\gamma}_{R})
	= \mathcal{A}^{N,j,l}_{1,\gamma}(F^{s,\gamma}_{R}, f^{\gamma}, F^{s,\gamma}_{R}),
	\\
	\sum_{|\alpha|\le N-j,|\beta|=j}(\pa^{\alpha}_{\beta}G_3, \pa^{\alpha}_{\beta}F^{s,\gamma}_{R}) = \mathcal{B}^{N,j}_{1,\gamma}(
	F^{s,\gamma}_{R}, f^{\gamma},F^{s,\gamma}_{R}),
	\quad 
	\mathrm{NL}^{N}(G_3) = \mathcal{C}^{N}_{1,\gamma}(F^{s,\gamma}_{R}, f^{\gamma}).
	\eeno
	Note that these quantities contain the nonlinear term $\Gamma^{\gamma}_{L}$ of the Landau operator.
	By taking $s=1$ in the estimates of the nonlinear term $\Gamma^{s,\gamma}_{B}$ in previous sections, we can obtain estimates for $\Gamma^{\gamma}_{L}$. As a result, similarly to \eqref{N-equals-4-x-v-weight-15-G2}, \eqref{N-geq-5-x-v-no-weight-G2} and \eqref{only-x-with-mu-type-N-equals-4-G2}, we have
	\beno 
	|\mathcal{A}^{N,j,l}_{1,\gamma}(F^{s,\gamma}_{R}, f^{\gamma}, F^{s,\gamma}_{R})| \lesssim_{N,l}
	C_{\gamma} (
	\|F^{s,\gamma}_{R}\|_{H^{N}_{x,v}}
	\left(\mathcal{D}_{N,l}^{1,\gamma}(f^{\gamma})\right)^{\f12} +
	\|F^{s,\gamma}_{R}\|_{D^{N}_{1,\gamma}} \|f^{\gamma}\|_{H^{N}_{x,v}}) \|F^{s,\gamma}_{R}\|_{H^{N-j}_{x}
		\dot{H}^{j}_{1,l+j(\gamma+2)+\gamma/2}},
	\\ 
	|\mathcal{B}^{N,j}_{1,\gamma}(F^{s,\gamma}_{R}, f^{\gamma},F^{s,\gamma}_{R})| \lesssim_{N}
	C_{\gamma} (\|F^{s,\gamma}_{R}\|_{H^{N}_{x,v}} \|f^{\gamma}\|_{D^{N}_{1,\gamma}} +
	\|F^{s,\gamma}_{R}\|_{D^{N}_{1,\gamma}} \|f^{\gamma}\|_{H^{N}_{x,v}}) \|F^{s,\gamma}_{R}\|_{H^{N-j}_{x}
		\dot{H}^{j}_{1,\gamma/2}},
	\\ 
	\mathcal{C}^{N}_{1,\gamma}(F^{s,\gamma}_{R}, f^{\gamma}) \lesssim_{N} C_{\gamma}^2 \|F^{s,\gamma}_{R}\|_{H^{N}_{x}L^2}^2 \|f^{\gamma}\|_{H^{N}_{x}L^{2}_{1,\gamma/2}}^2.
	\eeno

	Plugging the above nonlinear estimates into \eqref{essential-micro-macro-error-function-2}, recalling 	\eqref{c-relation-1} and \eqref{c-relation-2}, we get
	\ben \label{g=gamma-f-f-non-linear-error}
	&& \frac{\mathrm{d}}{\mathrm{d}t}\Xi_{N,l}^{1,\gamma}(F^{s,\gamma}_{R}) +  \frac{1}{4} \lambda_{1} \left(|\mathrm{M}|^{2}_{H^{N}_{x}} + \mathcal{D}^{1,\gamma}_{N,l}((F^{s,\gamma}_{R})_2)\right)
	\\ \nonumber  &\leq& C_{N,l}Z_{1,\gamma,N,l} \bigg\{ 
	(C_{\gamma} \|f^{s,\gamma}\|_{H^{N}_{x,v}} +	C_{\gamma} \|f^{\gamma}\|_{H^{N}_{x,v}} + C_{\gamma}^2 \|f^{s,\gamma}\|_{H^{N}_{x}L^2}^2)
	\mathcal{D}_{N,l}^{1,\gamma}(F^{s,\gamma}_{R})
	\\ \nonumber && \quad \quad \quad \quad + \|f^{s,\gamma}\|_{D^{N}_{s,\gamma}} \|F^{s,\gamma}_{R}\|_{H^{N}_{x,v}} \left(\mathcal{D}_{N,l}^{1,\gamma}(F^{s,\gamma}_{R})\right)^{\f12}
	\bigg\}
	\\ \nonumber  &&+ C_{N,l}Z_{1,\gamma,N,l} \bigg\{ 
	C_{\gamma}
	\|F^{s,\gamma}_{R}\|_{H^{N}_{x,v}}
	\left(\mathcal{D}_{N,l}^{1,\gamma}(f^{\gamma})\right)^{\f12} \left(\mathcal{D}_{N,l}^{1,\gamma}(F^{s,\gamma}_{R})\right)^{\f12} + C_{\gamma}^2
	\|F^{s,\gamma}_{R}\|_{H^{N}_{x,v}}^2
	\mathcal{D}_{N,l}^{1,\gamma}(f^{\gamma})
	\bigg\}
	\\ \nonumber  &&+ C_{N,l}Z_{1,\gamma,N,l} \bigg\{ C_{\gamma} (\mathcal{D}^{s,\gamma}_{N+3,l+5-3(\gamma+2s)+N(2-2s)}(f^{s,\gamma}))^{1/2}  
	\\ \nonumber && \quad \quad \quad \quad  +
	C_{\gamma} \|f^{s,\gamma}
	\|_{H^{N+3}_{x,v}} 
	(\mathcal{D}^{1,\gamma}_{N+3,
		l+5-3(\gamma+2)}(f^{\gamma}))^{1/2}
	\bigg\} \left(\mathcal{D}_{N,l}^{1,\gamma}(F^{s,\gamma}_{R})\right)^{\f12}
	\\ \nonumber  &&+ C_{N,l}Z_{1,\gamma,N,l} \bigg\{ C_{\gamma}^2 \mathcal{D}^{s,\gamma}_{N+3,l+5-3(\gamma+2s)+N(2-2s)}(f^{s,\gamma})  +
	C_{\gamma}^2 \|f^{s,\gamma}
	\|_{H^{N+3}_{x,v}}^2 
	\mathcal{D}^{1,\gamma}_{N+3,
		l+5-3(\gamma+2)}(f^{\gamma})
	\bigg\}.
	\een
Recalling \eqref{lower-bound-on-lambda-explicit}, $\lambda_{1}$ is a generic constant
	for any $-5 \leq \gamma \leq 0$. By using
	\beno
	\mathcal{D}_{N,l}^{1,\gamma}(F^{s,\gamma}_{R}) \lesssim_{N,l} |\mathrm{M}|^{2}_{H^{N}_{x}} + \mathcal{D}^{1,\gamma}_{N,l}((F^{s,\gamma}_{R})_2),
	\eeno
	we have
	\ben \label{g=gamma-f-f-non-linear-error-2}
	&& \frac{\mathrm{d}}{\mathrm{d}t}\Xi_{N,l}^{1,\gamma}(F^{s,\gamma}_{R}) +  \frac{1}{4} \lambda_{1} (|\mathrm{M}|^{2}_{H^{N}_{x}} + \mathcal{D}^{1,\gamma}_{N,l}((F^{s,\gamma}_{R})_2))
	\\ \nonumber  &\leq& C_{N,l}Z_{1,\gamma,N,l}  
	(C_{\gamma} \|f^{s,\gamma}\|_{H^{N}_{x,v}} +	C_{\gamma} \|f^{\gamma}\|_{H^{N}_{x,v}} + C_{\gamma}^2 \|f^{s,\gamma}\|_{H^{N}_{x}L^2}^2)
	\mathcal{D}_{N,l}^{1,\gamma}(F^{s,\gamma}_{R})
	\\ \nonumber  &&+ C_{N,l}Z_{1,\gamma,N,l}^2  \bigg\{ C_{\gamma}^2
	\|F^{s,\gamma}_{R}\|_{H^{N}_{x,v}}^2
	\mathcal{D}_{N,l}^{s,\gamma}(f^{s,\gamma}) +
	C_{\gamma}^2
	\|F^{s,\gamma}_{R}\|_{H^{N}_{x,v}}^2
	\mathcal{D}_{N,l}^{1,\gamma}(f^{\gamma})
	\bigg\}
	\\ \nonumber  &&+ C_{N,l}Z_{1,\gamma,N,l}^2  \bigg\{ C_{\gamma}^2 \mathcal{D}^{s,\gamma}_{N+3,l+5-3(\gamma+2s)+N(2-2s)}(f^{s,\gamma})  +
	C_{\gamma}^2 \|f^{s,\gamma}
	\|_{H^{N+3}_{x,v}}^2 
	\mathcal{D}^{1,\gamma}_{N+3,
		l+5-3(\gamma+2)}(f^{\gamma})
	\bigg\} .
	\een
	By the assumption \eqref{initial-condition-smallness-for-asy} and Theorem \ref{a-priori-estimate-LBE},
	the solutions $f^{s,\gamma}$ and $f^{\gamma}$ satisfy
	\ben \label{small-f-s-gamma} \mathcal{E}_{N+3,l_*}^{s,\gamma}(f^{s,\gamma}(t)) + \frac{1}{8} \lambda_{s} \int_{0}^{t}  (|\mathrm{M}(\tau)|^{2}_{H^{N+3}_{x}} + \mathcal{D}^{s,\gamma}_{N+3,l_*}(f^{s,\gamma}_2)(\tau)) \mathrm{d}\tau \leq Z_{s,\gamma,N,l} \mathcal{E}_{N+3,l_*}^{s,\gamma}(f_{0}),
	\\ \label{small-f-gamma} \mathcal{E}_{N+3,l_*}^{1,\gamma}(f^{\gamma}(t)) + \frac{1}{8} \lambda_{1} \int_{0}^{t}  (|\mathrm{M}(\tau)|^{2}_{H^{N+3}_{x}} + \mathcal{D}^{s,\gamma}_{N+3,l_*}(f^{\gamma}_2)(\tau)) \mathrm{d}\tau \leq Z_{1,\gamma,N,l} \mathcal{E}_{N+3,l_*}^{1,\gamma}(f_{0}),
	\een
	where $l_* = l+2N-3\gamma+5 \geq l+5-3(\gamma+2s)+N(2-2s)$. By the smallness of the energy functional, the term containing $\mathcal{D}^{1,\gamma}_{N,l}(F^{s,\gamma}_{R})$ is absorbed by the left hand side. Then  
	the initial 
	condition $F^{s,\gamma}_{R}(0)=0$ implies
	\beno
	\sup_{t \geq 0}\Xi_{N,l}^{1,\gamma}(F^{s,\gamma}_{R}(t)) &\leq& 
	\exp \left(C_{N,l}Z_{s,\gamma,N,l}^3 C_{\gamma}^2 \mathcal{E}_{N+3,l_*}^{1,\gamma}(f_{0})\right) C_{N,l}Z_{s,\gamma,N,l}^3 C_{\gamma}^2 \mathcal{E}_{N+3,l_*}^{1,\gamma}(f_{0})
\\	&\leq&
	\exp \left(C_{N,l}Z_{s,\gamma,N,l}^3 C_{\gamma}^2 \mathcal{E}_{N+3,l_*}^{1,\gamma}(f_{0})\right).
	\eeno
	Recalling $F^{s,\gamma}_{R} :=  (1-s)^{-1} (f^{s,\gamma}-f^{\gamma})$, \eqref{property-of-Z-N-l} and \eqref{energy-equivalence}, we get \eqref{error-function-uniform-estimate}. This completes the proof of Theorem \ref{asymptotic-result}.

	\section{Appendix} \label{Operator-difference}

	We now prove Proposition \ref{limit-Botltzmann-to-Landau}.
	\begin{proof}[Proof of Proposition \ref{limit-Botltzmann-to-Landau}] The proof is based on \cite{desvillettes1992asymptotics} and \cite{he2014well}. Recall the Boltzmann operator  $Q^{s,\gamma}_{B}$ in \eqref{Boltzmann-operator} and the kernel $B^{s,\gamma}$ in \eqref{kernel-studied}. Following  the proof in based on \cite{desvillettes1992asymptotics} and \cite{he2014well}, we derive that
		\ben Q^{s,\gamma}_{B}(g, h) &=& \int_{\mathbb{R}^{3}}\left(\nabla_{v}-\nabla_{v_{*}}\right) \cdot\left[U^{s,\gamma}_{1}(v-v_{*})\left(\nabla_{v}-\nabla_{v_{*}}\right)\left(g_{*} h\right)\right] \mathrm{d}v_{*} \\ &&+\int_{\mathbb{R}^{3}}\left[U^{s,\gamma}_{2}(v-v_{*}):\left(\nabla_{v}-\nabla_{v_{*}}\right)^{2}\left(g_{*} h\right)\right] \mathrm{d}v_{*} \\ &&+\int_{\mathbb{R}^{3}} \int_{\mathbb{S}^{2}} R_{1}(v, v_{*},\sigma) B^{s,\gamma} \mathrm{d}v_{*} \mathrm{d}\sigma. 
		\een
		where
		\beno
		U^{s,\gamma}_{1} (v-v_{*}) &:=& \frac{1}{4} [\left|v-v_{*}\right|^{2} I_{3} - (v-v_{*}) \otimes (v-v_{*}) ] \int \sin^{2} \frac{\theta}{2} B^{s,\gamma} \mathrm{d}\sigma,
		\\
		U^{s,\gamma}_{2} (v-v_{*}) &:=&  ( \frac{3}{4} (v-v_{*})\otimes(v-v_{*}) -  \frac{1}{4} \left|v-v_{*}\right|^{2} I_{3}) \int  \sin^{4} \frac{\theta}{2} B^{s,\gamma} \mathrm{d}\sigma.
		\eeno
The function  $R_{1}(v, v_{*},\sigma)$ reads
		\ben
		R_{1}(v, v_{*},\sigma) = r_{1}(v, v_{*},\sigma)\left(g\left(v_{*}\right)-\frac{1}{2} A \cdot \nabla g\left(v_{*}\right)+\frac{1}{8} A \otimes A : \nabla^{2} g\left(v_{*}\right)+r_{2}(v, v_{*},\sigma) \right) \\+\frac{1}{8} A \otimes A : \nabla^{2} h\left(v\right)\left(-\frac{1}{2} A \cdot \nabla g\left(v_{*}\right)+\frac{1}{8} A \otimes A : \nabla^{2} g\left(v_{*}\right)+r_{2}(v, v_{*},\sigma)\right) \\+\frac{1}{2} A \cdot \nabla h\left(v\right)\left(\frac{1}{8} A \otimes A : \nabla^{2} g\left(v_{*}\right)+r_{2}(v, v_{*},\sigma)\right)+h\left(v\right) r_{2}(v, v_{*},\sigma),
		\een
		where $A = 2(v^{\prime} - v)$ and
		\beno   r_{1}(v, v_{*},\sigma) =  \f{1}{16}\sum_{1 \leq i,j,k\leq 3} \int_{0}^{1}(1-\kappa)^{2} A_{i} A_{j} A_{k}\pa^{3}_{ijk} h \left(v+\kappa\left(v^{\prime}-v\right)\right)  \mathrm{d}\kappa,
	\\   r_{2}(v, v_{*},\sigma) =  - \f{1}{16}\sum_{1 \leq i,j,k\leq 3} \int_{0}^{1}(1-\iota)^{2} A_{i} A_{j} A_{k}\pa^{3}_{ijk} g \left(v_*+\iota\left(v^{\prime}_*-v_*\right)\right)  \mathrm{d}\iota.
		\eeno
Note that $R_{1}(v, v_{*},\sigma)$ contains  $|A|^{k}$ for $k \geq 3$.

		Recalling \eqref{mean-momentum-transfer} and  \eqref{matrix} with $\Lambda = \pi$, it is
		staightforward to check that 
		\beno
		U^{s,\gamma}_{1} (z) = \left(\frac{1}{4} \int \sin^{2} \frac{\theta}{2} b^{s}(\theta) \mathrm{d}\sigma\right)
		|z|^{\gamma+2} \Pi(z) = 2^{s-1} \pi |z|^{\gamma+2} \Pi(z) = 2^{s-1} a^{\gamma}(z).
		\eeno
			Recall the Landau operator $Q_{L}^{\gamma}$ given by \eqref{oroginal-definition-Laudau-oprator} and  \eqref{matrix} with $\Lambda = \pi$. In another form,
		\beno Q_{L}^{\gamma}(g, h) = \int_{\R^3} \left(\nabla_{v}-\nabla_{v_{*}}\right) \cdot \left[a^{\gamma}\left(v- v_{*}\right)\left(\nabla_{v}-\nabla_{v_{*}}\right)\left(g_{*} h\right)\right] \mathrm{d}v_{*},\eeno 
which gives
		\ben \label{Q-into-three-terms}
	 Q^{s,\gamma}_{B}(g, h)=& 2^{s-1} Q_{L}^{\gamma}(g, h) +\int_{\mathbb{R}^{3}}\left[U^{s,\gamma}_{2}(v-v_{*}):\left(\nabla_{v}-\nabla_{v_{*}}\right)^{2}\left(g_{*} h\right)\right] \mathrm{d}v_{*} \\ \nonumber  &+\int_{\mathbb{R}^{3}} \int_{\mathbb{S}^{2}} R_{1}(v, v_{*},\sigma) B^{s,\gamma} \mathrm{d}v_{*} \mathrm{d}\sigma. 
		\een
		We now have
		\beno
		{Q^{s,\gamma}_{B}(g, h)-Q_{L}^{\gamma}(g, h)}&=& (2^{s-1} - 1)Q_{L}^{\gamma}(g, h)  +\int_{\mathbb{R}^{3}}\left[U^{s,\gamma}_{2}(v-v_{*}):\left(\nabla_{v}-\nabla_{v_{*}}\right)^{2}\left(g_{*} h\right)\right] \mathrm{d}v_{*} \\ &&+\int_{\mathbb{R}^{3}} \int_{\mathbb{S}^{2}} R_{1}(v, v_{*},\sigma) B^{s,\gamma} \mathrm{d}v_{*} \mathrm{d}\sigma := \sum_{i=1}^{3}E_{i}.
		\eeno
		Note that  for $0<s<1$, 
		\ben \label{small-factor}
		|2^{s-1} - 1| \leq 1-s.
		\een
	For showing validity for  $\gamma>-5$, we rewrite the Landau operator
		$Q_{L}^{\gamma}(g, h)$. Recall that
		\beno 
		Q_{L}^{\gamma}(g, h) = \nabla \cdot
		\int_{\R^3}  a^{\gamma}\left(v- v_{*}\right)
		(g_* \nabla h - (\nabla g)_* h)
		\mathrm{d}v_{*} = \nabla \cdot [(a^{\gamma} * g) \nabla h - (a^{\gamma} * \nabla g) h].
		\eeno
		In order not to  have any  derivatives on the kernel function $a^{\gamma}$, we write 
		\beno 
		Q_{L}^{\gamma}(g, h) = (a^{\gamma} * g) : \nabla^2 h - (a^{\gamma} *: \nabla^2 g) h,
		\eeno
		where $A : B := tr(A B)$ for the two matrices $A, B$. More precisely, 
		\beno 
		Q_{L}^{\gamma}(g, h) = \sum_{i, j=1}^{3} (a^{\gamma}_{ij} * g)  \pa^2_{ij} h - \sum_{i, j=1}^{3} (a^{\gamma}_{ij} * \pa^2_{ij} g)   h.
		\eeno
		Note that 
		\ben \label{a-order-gamma-plus-2}
		|a^{\gamma}(v-v_{*})| \lesssim |v-v_{*}|^{\gamma+2}.
		\een
		To estimate 
	$
		\left|\left\langle Q_{L}^{\gamma}(g, h), W_{l} \psi\right\rangle\right| 
	$,
		it suffices to consider the following type of integral
		\ben \label{gamma-plus-2-kept}
		\int |v-v_{*}|^{\gamma+2} |(\pa^{\alpha_1}g)_* \pa^{\alpha_2} h W_{l} \psi| \mathrm{d}v \mathrm{d}v_*,
		\een
		where $(|\alpha_1|, |\alpha_2|) = (2, 0)$ or $(|\alpha_1|, |\alpha_2|) = (0, 2)$. 

		Note that
		\beno
		\int  \sin^{4} \frac{\theta}{2} B^{s,\gamma} \mathrm{d}\sigma = 8 \pi  (1-s) |v-v_{*}|^{\gamma} \int_{0}^{1/\sqrt{2}} t^{3-2s} d t \lesssim (1-s) |v-v_{*}|^{\gamma},
		\eeno
		which gives
		\ben \label{small-factor-U2}
		|U^{s,\gamma}_{2} (v-v_{*})| \lesssim (1-s) |v-v_{*}|^{\gamma+2}.
		\een
		This shows that 
		in order to estimate 
$\left|\left\langle E_2, W_{l} \psi \right\rangle\right| 
	$,
		it suffices to consider the integral \eqref{gamma-plus-2-kept} for $|\alpha_1| + |\alpha_2| = 2$.
		In general, we consider
		\beno
		\int |v-v_{*}|^{\gamma+2} |g_* h \psi| \mathrm{d}v \mathrm{d}v_*, 
		\eeno
		for $\gamma>-5$. Note that the integral has singularity as $\gamma \to (-5)^{+}$.
		It is obvious that $|v-v_{*}| \leq W(v) W(v_*)$.
		If $\gamma+2 \geq 0$, then
		\beno
		|v-v_{*}|^{\gamma+2} \leq W_{\gamma+2}(v) W_{\gamma+2}(v_*),
		\eeno
		which gives
		\beno
		\int |v-v_{*}|^{\gamma+2} |g_* h \psi| \mathrm{d}v \mathrm{d}v_* \leq 
		|g|_{L^{1}_{\gamma+2}} |h|_{L^{2}_{b_1}}
		|\psi|_{L^{2}_{b_2}},
		\eeno
		where $b_1, b_2 \in \R$ satisfying
		$b_1 + b_2 = \gamma +2$.
		If $\gamma+2 < 0$, 
		then
		\beno
		|v-v_{*}|^{\gamma+2} \lesssim \mathrm{1}_{|v-v_*| \leq 1} |v-v_{*}|^{\gamma+2} W_{\gamma+2}(v) W_{|\gamma+2|}(v_*) + \mathrm{1}_{|v-v_*| \geq 1} W_{\gamma+2}(v) W_{|\gamma+2|}(v_*),
		\eeno
		which gives
		\beno
		\int |v-v_{*}|^{\gamma+2} |g_* h \psi| \mathrm{d}v \mathrm{d}v_* \lesssim
		\f{1}{\gamma+5} |g|_{L^{p}_{|\gamma+2|}} |h|_{L^{q}_{b_1}}
		|\psi|_{L^{2}_{b_2}} + 
		|g|_{L^{1}_{|\gamma+2|}} |h|_{L^{2}_{b_1}}
		|\psi|_{L^{2}_{b_2}},
		\eeno
		where $2 \leq p, q \leq \infty$ satisfying
		$1/p + 1/q = 1/2$. 
		Here we have used $\int \mathrm{1}_{|v-v_*| \leq 1} |v-v_{*}|^{\gamma+2} \mathrm{d}v_* \lesssim \f{1}{\gamma+5}$.
		In summary, by using the basic inequality $|g|_{L^{1}} \lesssim |g|_{L^{2}_2}$ and the embedding 
		$H^{2} \hookrightarrow L^{\infty}$ and 
		$H^{s} \hookrightarrow L^{p}$ where $1/p = 1/2 - s/3$,
		for $-5<\gamma \leq 0$,
		we have
		\ben \label{gamma-plus-2-case}
		\int |v-v_{*}|^{\gamma+2} |g_* h \psi| \mathrm{d}v \mathrm{d}v_* \lesssim \f{1}{\gamma+5}
		|g|_{H^{s_1}_{|\gamma+2|+2}} |h|_{H^{s_2}_{b_1}}
		|\psi|_{L^{2}_{b_2}},
		\een
		where $0 \leq s_1, s_2 \leq 2$ satisfying
		$s_1 + s_2 = 2$.

		By applying \eqref{gamma-plus-2-case} for estimation on  \eqref{gamma-plus-2-kept}, and
		by  recalling \eqref{small-factor} and \eqref{small-factor-U2},
		we obtain 
		\ben \label{E-1-and-2-result}
		\left|\left\langle (2^{s-1} - 1)Q_{L}^{\gamma}(g, h), W_{l} \psi\right\rangle\right| + \left|\left\langle E_2, W_{l} \psi \right\rangle\right| \lesssim \f{1-s}
		{\gamma+5}
		|g|_{H^{2}_{|\gamma+2|+2}} |h|_{H^{2}_{l+b_1}}
		|\psi|_{L^{2}_{b_2}},
		\een
		where $b_1+b_2 = \gamma+2$.

		We now turn to estimate $E_{3}$.
		By the fact $|A| \lesssim \sin \frac{\theta}{2} |v-v_{*}|$,
		one has $\max \{|A|^{3},|A|^{4},|A|^{5}, |A|^{6}\} \lesssim \sin^{3} \frac{\theta}{2} |v-v_{*}|^{3} W_{3}(v-v_*)$. Plugging this into the definition of $R_{1}(v, v_{*},\sigma)$, one has
		\begin{equation}
			\begin{array}{l}{|R_{1}(v, v_{*},\sigma)| \lesssim  \sin^{3} \frac{\theta}{2} |v-v_{*}|^{3} W_{3}(v-v_*) \sum_{i=1}^{4}R_{1,i}(v, v_{*},\sigma)}, \\ {R_{1,1}=
					\sum_{i=0}^{2} \sum_{j= 3-i}^{2}|\nabla^{i}g(v_{*})||\nabla^{j}h(v)|},
				\\   {R_{1,2}= \sum_{i=0}^{2} |\nabla^{i}g(v_{*})| \int_{0}^{1}(1-\kappa)^{2}\left|\nabla^{3} h \left(v(\kappa)\right) \right| \mathrm{d}\kappa},
				\\   {R_{1,3}=\sum_{i=0}^{2} |\nabla^{i}h(v)| \int_{0}^{1}(1-\iota)^{2}\left|\nabla^{3} g \left(v_{*}(\iota)\right) \right| \mathrm{d}\iota},
				\\   {R_{1,4}=\int_{0}^{1}(1-\iota)^{2}\left|\nabla^{3} g \left(v_{*}(\iota)\right) \right| \mathrm{d}\iota \int_{0}^{1}(1-\kappa)^{2}\left|\nabla^{3} h \left(v(\kappa)\right)\right| \mathrm{d}\kappa}.
			\end{array}
		\end{equation}
		Then we have
		$\left|\left\langle E_{3},W_{l} \psi\right\rangle\right| \lesssim \sum_{i=1}^{4}J_{i}$,
		where \beno J_{i} =
		\int B^{s,\gamma}  R_{1,i} (v, v_{*},\sigma) \sin^{3} \frac{\theta}{2} |v-v_{*}|^{3}
		W_{3}(v-v_*) W_{l}(v) |\psi(v)| \mathrm{d}V. \eeno
		In general, for $0 \leq \iota, \kappa \leq 1$,
		we consider 
		\beno \mathcal{I}(g,h) =
		\int B^{s,\gamma} \sin^{3} \frac{\theta}{2} |v-v_{*}|^{3}
		W_{3}(v-v_*) W_{l}(v) 
		|g(v_{*}(\iota))  h(v(\kappa))
		\psi(v)| \mathrm{d}V. \eeno
If $\gamma+3 \geq 0$, then 
		\ben \label{gamma-plus-3-ge-0}
		|v-v_{*}|^{\gamma+3} W_{3}(v-v_{*}) W_{l}(v) \lesssim_{l,a_1,a_2}  W_{l+|a_1| + |a_2|}(v_{*}(\iota)) W_{l+a_1}(v(\kappa)) W_{a_2}(v).
		\een
		If $\gamma+3 < 0$, we have
		\ben \label{gamma-plus-3-le-0}
		|v-v_{*}|^{\gamma+3} W_{3}(v-v_{*}) W_{l}(v) \lesssim_{l,a_1,a_2} \mathrm{1}_{|v-v_*| \leq 1} |v-v_{*}|^{\gamma+3} W_{l+|a_1| + |a_2|}(v_{*}(\iota)) W_{l+a_1}(v(\kappa)) W_{a_2}(v) 
		\\ \nonumber + \mathrm{1}_{|v-v_*| \geq 1} W_{l+|a_1| + |a_2|}(v_{*}(\iota)) W_{l+a_1}(v(\kappa)) W_{a_2}(v).
		\een
	By the above estimates,
		we have
		\beno |\mathcal{I}(g,h)| \lesssim_{l, a_1, a_2}  \int b^{s}(\theta) \sin^{3} \frac{\theta}{2} 
		|\tilde{g}(v_{*}(\iota))  \tilde{h}(v(\kappa))
		\tilde{\psi}(v)| \mathrm{d}V
		\\+
		\mathrm{1}_{\gamma+3<0} \int \mathrm{1}_{|v-v_*| \leq 1} |v-v_{*}|^{\gamma+3} b^{s}(\theta) \sin^{3} \frac{\theta}{2} 
		|\tilde{g}(v_{*}(\iota))  \tilde{h}(v(\kappa))
		\tilde{\psi}(v)| \mathrm{d}V := \mathcal{I}_1(g,h) + \mathcal{I}_2(g,h),
		\eeno
		where $\tilde{g} = W_{l+|a_1| + |a_2|} g, \tilde{h} = W_{l+a_1} h, \tilde{\psi} = W_{a_2} \psi$.

		We now consider the functional $\mathcal{I}_1(g,h)$ where there is no singularity. By  Cauchy-Schwarz inequality, applying
		the change of variable \eqref{change-of-variable-2} and 
		using the fact $1 \leq \psi_{a}(\theta) \leq \sqrt{2}$, 
		we have
		\beno \mathcal{I}_1(g,h) &\lesssim& \left(
		\int b^{s}(\theta) \sin^{3} \frac{\theta}{2} 
		|\tilde{g}(v_{*}(\iota))  \tilde{h}^2(v(\kappa))| \mathrm{d}V \right)^{1/2}
		\\ &&\times \left(
		\int b^{s}(\theta) \sin^{3} \frac{\theta}{2} (\psi_{\kappa+\iota}(\theta))^{3}
		|\tilde{g}(v_{*}(\iota)) 
		\tilde{\psi}^2(v)| \mathrm{d}V \right)^{1/2}
		\\ &=& \left(
		\int b^{s}(\theta) \sin^{3} \frac{\theta}{2} (\psi_{\kappa+\iota}(\theta))^{3}
		|\tilde{g}(v_{*})  \tilde{h}^2(v)| \mathrm{d}V \right)^{1/2}
		\\ && \times \left(
		\int b^{s}(\theta) \sin^{3} \frac{\theta}{2} (\psi_{\iota}(\theta))^{3}
		|\tilde{g}(v_{*}) 
		\tilde{\psi}^2(v)| \mathrm{d}V \right)^{1/2}
		\\ &\lesssim& \left(
		\int b^{s}(\theta) \sin^{3} \frac{\theta}{2}
		|\tilde{g}(v_{*})  \tilde{h}^2(v)| \mathrm{d}V \right)^{1/2}
	 \left(
		\int b^{s}(\theta) \sin^{3} \frac{\theta}{2} 
		|\tilde{g}(v_{*}) 
		\tilde{\psi}^2(v)| \mathrm{d}V \right)^{1/2}. \eeno
		Note that
		\ben \label{order-3-small-factor}
		\int  \sin^{3} \frac{\theta}{2} b^{s}(\theta)  \mathrm{d}\sigma = 8 \pi  (1-s)  \int_{0}^{1/\sqrt{2}} t^{2-2s} \mathrm{d}t \lesssim (1-s), 
		\een
		which gives
		\beno \mathcal{I}_1(g,h) \lesssim (1-s)
		\left(
		\int
		|\tilde{g}(v_{*})  \tilde{h}^2(v)| \mathrm{d}v \mathrm{d}v_{*}  \right)^{1/2}
		\left(
		\int
		|\tilde{g}(v_{*}) 
		\tilde{\psi}^2(v)| \mathrm{d}v \mathrm{d}v_{*}  \right)^{1/2}
		\lesssim (1-s) |\tilde{g}|_{L^{1}} |\tilde{h}|_{L^{2}} |\tilde{\psi}|_{L^{2}}. \eeno
		
		We now consider the functional $\mathcal{I}_2(g,h)$ where there is singularity as $|v-v_*| \to 0$.  By Cauchy-Schwarz inequality, applying
		the change of variable \eqref{change-of-variable-2},
		using the fact $1 \leq \psi_{\kappa}(\theta) \leq \sqrt{2}$ and  \eqref{order-3-small-factor},
		we have
		\beno \mathcal{I}_2(g,h) \lesssim 
		(1-s)\left(
		\int \mathrm{1}_{|v-v_*| \leq 1} |v-v_{*}|^{\gamma+3}
		|\tilde{g}(v_{*})  \tilde{h}^2(v)| \mathrm{d}v \mathrm{d}v_{*}  \right)^{1/2}
		\\ \times \left(
		\int \mathrm{1}_{|v-v_*| \leq 1} |v-v_{*}|^{\gamma+3}
		|\tilde{g}(v_{*}) 
		\tilde{\psi}^2(v)| \mathrm{d}v \mathrm{d}v_{*} \right)^{1/2}
		. \eeno
		If $\gamma>-9/2$, using $\int \mathrm{1}_{|v-v_*| \leq 1} |v-v_{*}|^{\gamma+3} |g_*| \mathrm{d}v_* \lesssim \f{1}{\gamma+ 9/2} |g|_{L^{2}}$ to have 
		\ben \label{I2-gh-case1}
		\mathcal{I}_2(g,h) \lesssim 
		(1-s) (\gamma+\f92)^{-1}  |\tilde{g}|_{L^{2}} |\tilde{h}|_{L^{2}} |\tilde{\psi}|_{L^{2}}.
		\een
		If $\gamma>-11/2$, then
		\ben \label{I2-gh-case2}
		\mathcal{I}_2(g,h) \lesssim 
		(1-s) (\gamma+\f{11}{2})^{-5/6}  |\tilde{g}|_{L^{p}} |\tilde{h}|_{L^{q}} |\tilde{\psi}|_{L^{2}},
		\een
		where $2 \leq p,q \leq 6$ satisfying $1/p + 1/q = 2/3$.
		Indeed, putting together $\tilde{g}$ and
		$\tilde{h}$, we can get
		\beno \mathcal{I}_2(g,h) \lesssim 
		(1-s)\left(
		\int \mathrm{1}_{|v-v_*| \leq 1} |v-v_{*}|^{\f{4}{5}(\gamma+3)}
		|\tilde{g}^2(v_{*})  \tilde{h}^2(v)| \mathrm{d}v \mathrm{d}v_{*}  \right)^{1/2}
		\\ \times \left(
		\int \mathrm{1}_{|v-v_*| \leq 1} |v-v_{*}|^{\f{6}{5}(\gamma+3)}
		| 
		\tilde{\psi}^2(v)| \mathrm{d}v \mathrm{d}v_{*} \right)^{1/2}. \eeno
		Using $
		\int \mathrm{1}_{|v-v_*| \leq 1} |v-v_{*}|^{\f{6}{5}(\gamma+3)}  \mathrm{d}v_* \lesssim \f{1}{\gamma+ 11/2}$,
		the latter integral is bounded by
		\beno
		\int \mathrm{1}_{|v-v_*| \leq 1} |v-v_{*}|^{\f{6}{5}(\gamma+3)}
		| 
		\tilde{\psi}^2(v)| \mathrm{d}v \mathrm{d}v_{*} \lesssim (\gamma+ 11/2)^{-1} |\tilde{\psi}|_{L^{2}}^2.
		\eeno
		Let $k(z) = 1_{|z| \leq 1} |z|^{\f{4}{5}(\gamma+3)}$, then $|k|_{L^{3/2}} \lesssim (\gamma+ 11/2)^{-2/3}$. Thus,
		the first integral is bounded by 
		\beno
		|(k * \tilde{g}^2) \tilde{h}^2|_{L^{1}}
		\lesssim |(k * \tilde{g}^2)|_{L^{r}}
		|\tilde{h}^2|_{L^{r'}}
		\lesssim |k|_{L^{3/2}}
		|\tilde{g}^2|_{L^{q}}
		|\tilde{h}^2|_{L^{r'}}
		\\ \lesssim |k|_{L^{3/2}}
		|\tilde{g}|_{L^{2q}}^2
		|\tilde{h}|_{L^{2r'}}^2
		\lesssim  (\gamma+ 11/2)^{-2/3} 
		|\tilde{g}|_{L^{2q}}^2
		|\tilde{h}|_{L^{2r'}}^2,
		\eeno
		where $1/r + 1/r' = 1, 1+ 1/r = 2/3 + 1/q$. Then we get $1/2q + 1/2r' = 2/3$. 
		Combining these  two estimates yields \eqref{I2-gh-case2}.

		We conclude that 
		if $-9/2 < \gamma \leq 0$, 
		\beno |\mathcal{I}(g,h)| \lesssim 
		(1-s) (\gamma+\f92)^{-1} |\tilde{g}|_{L^{2}_2} |\tilde{h}|_{L^{2}} |\tilde{\psi}|_{L^{2}};
		\eeno
		if $-11/2 < \gamma \leq 0$, by the Sobolev embedding,
		\beno |\mathcal{I}(g,h)| \lesssim 
		(1-s) (\gamma+\f{11}{2})^{-5/6}  |\tilde{g}|_{H^{s_1}_2} |\tilde{h}|_{H^{s_2}} |\tilde{\psi}|_{L^{2}}.
		\eeno

		Therefore, if $-9/2 < \gamma \leq 0$, 
		\ben \label{E3-case1}
		\left|\left\langle E_{3},W_{l} \psi\right\rangle\right|
		\lesssim 
		(1-s) (\gamma+\f92)^{-1} |g|_{H^{3}_{l+|a_1|+|a_2|+2}} |h|_{H^{3}_{l+a_1}} |\psi|_{L^{2}_{a_2}};
		\een
		if $-11/2 < \gamma \leq 0$, 
		\ben \label{E3-case2}
		\left|\left\langle E_{3},W_{l} \psi\right\rangle\right|
		\lesssim  (1-s) (\gamma+\f{11}{2})^{-5/6} |g|_{H^{3+s_1}_{l+|a_1|+|a_2|+2}} |h|_{H^{3+s_2}_{l+a_1}} |\psi|_{L^{2}_{a_2}}.
		\een
		By combining \eqref{E-1-and-2-result}, \eqref{E3-case1} and \eqref{E3-case2},   the proof of Proposition \ref{limit-Botltzmann-to-Landau} is completed.
	\end{proof}

	{\bf Acknowledgments.}
	The authors would like to thank the support by the Centre for Nonlinear Analysis, The Hong Kong
	Polytechnic University. 
	The research was partially supported by the National Key Research and
	Development Program of China project no. 2021YFA1002100. The research of Tong Yang was
	supported by a fellowship award from the Research Grants Council of the Hong Kong Special
	Administrative Region, China (Project no. SRF2021-1S01). The research of Yu-Long Zhou was
	supported by the NSFC project no. 12001552, the Science and Technology Project in Guangzhou
	no. 202201011144.
	
	\normalem
	
	\bibliographystyle{siam}
	\bibliography{Boltzmann}

\end{document}